\documentclass[11pt,a4paper]{book}
 
\usepackage[a4paper]{geometry}
\usepackage{a4wide}

\usepackage{epigraph}  

\usepackage[english]{babel}
\usepackage[utf8]{inputenc}

\usepackage{pifont}

\usepackage{mathtools}

\usepackage{float}

 \usepackage{libertine,libertinust1math} 

\usepackage{enumitem}
\usepackage{manfnt} 

\usepackage{chngcntr} 

\usepackage{graphicx,longtable,multirow,amsmath,amssymb, doi, bm, bbm, fancyhdr}

\usepackage{marvosym}

\usepackage[usenames,dvipsnames]{xcolor}
\usepackage{mathrsfs}

\usepackage{hyperref} 
\hypersetup{
bookmarks=true,         
    unicode=false,          
    pdftoolbar=true,        
    pdfmenubar=true,        
    pdffitwindow=false,     
    pdfstartview={FitH},    
    citecolor=blue,
    colorlinks=true,       
    linkcolor=blue,          
    filecolor=blue,      
    urlcolor=blue           
}

\usepackage[Sonny]{fncychap}

\usepackage{amsfonts}
\usepackage{amssymb}
\usepackage{graphicx} 

\usepackage{fancybox}     
\usepackage{bbm}
\usepackage{framed}

\usepackage{tikz} 
\usepackage{tikz-qtree}

\usepackage[framemethod=tikz]{mdframed}

\usepackage{lipsum}

\newtheorem{assumption}{Assumption}

\newtheorem{corollary}{Corollary}
\newtheorem{condition}{Condition}

\newenvironment{proof}[1][Proof]
  {\par\noindent\normalfont{\em \noindent #1.}\par\nopagebreak%
  \begin{mdframed}[
     linewidth=1pt,
     linecolor=black,
     bottomline=false,topline=false,rightline=false,
     innerrightmargin=0pt,innertopmargin=0pt,innerbottommargin=0pt,
     innerleftmargin=1em,
     skipabove=.5\baselineskip,
     skipbelow=.5\baselineskip
   ]}
  {\end{mdframed}}

\usepackage{esint}
\usepackage{dsfont}

\newcommand{\ra}{\rightarrow}{}

\def\1{1\!{\rm l}}

\newcommand{\leqa}{\lesssim}
\newcommand{\geqa}{\gtrsim}

\newcommand{\al}{\alpha}
\newcommand{\be}{\beta}
\newcommand{\ep}{\epsilon}
\newcommand{\ga}{\gamma}
\newcommand{\Ga}{\Gamma}

\newcommand{\ka}{\kappa}
\newcommand{\la}{\lambda}
\newcommand{\La}{\Lambda}

\newcommand{\te}{\theta}

\newcommand{\ta}{\tau}
\newcommand{\eps}{\varepsilon}
\newcommand{\veps}{\varepsilon}
\newcommand{\vphi}{\varphi}
\newcommand{\pli}{+\infty}

\newcommand{\EM}{\ensuremath}
\newcommand{\cA}{\EM{\mathcal{A}}}
\newcommand{\cB}{\EM{\mathcal{B}}}
\newcommand{\cC}{\EM{\mathcal{C}}}
\newcommand{\cD}{\EM{\mathcal{D}}}
\newcommand{\cE}{\EM{\mathcal{E}}}
\newcommand{\cF}{\EM{\mathcal{F}}}
\newcommand{\cG}{\EM{\mathcal{G}}}
\newcommand{\cH}{\EM{\mathcal{H}}}

\newcommand{\cJ}{\EM{\mathcal{J}}}

\newcommand{\cL}{\EM{\mathcal{L}}}
\newcommand{\M}{{\mathcal M}}
\newcommand{\cM}{\EM{\mathcal{M}}}
\newcommand{\cN}{\EM{\mathcal{N}}}
\newcommand{\cP}{\EM{\mathcal{P}}}

\newcommand{\cR}{\EM{\mathcal{R}}}
\newcommand{\cS}{\EM{\mathcal{S}}}
\newcommand{\cT}{\EM{\mathcal{T}}}
\newcommand{\cU}{\EM{\mathcal{U}}}
\newcommand{\cV}{\EM{\mathcal{V}}}

\newcommand{\cX}{\EM{\mathcal{X}}}

\newcommand{\psg}{{\langle}}
\newcommand{\psd}{{\rangle}}

\newcommand{\oli}{\overline}
\newcommand{\uli}{\underline}
 
\newcommand{\qed}{\hfill \square}

\definecolor{pennblue}{rgb}{0.2, 0.2, 0.7}
\definecolor{pennred}{rgb}{0.89, 0.04, 0.36}

\DeclareMathAlphabet{\mathpzc}{OT1}{pzc}{m}{it}

\newcommand{\noi}{\noindent}

\newcommand{\RR}{\mathbb{R}}

\newcommand{\mb}{\mathbb{B}}
\newcommand{\mh}{\mathbb{H}}
\newcommand{\bH}{{\mh}}

\newcommand{\given}{\,|\,}

\definecolor{blendedblue}{rgb}{0.2,0.2,0.7}

\newcommand{\re}[1]{{\color{red}{#1}}}

\newcommand{\sbl}[1]{{\color{blendedblue}{#1}}}
\newcommand{\ma}[1]{{\color{magenta}{#1}}}
\newcommand{\pgr}[1]{{\color{ForestGreen}{#1}}}

\newcommand{\bfi}{\bar{\Phi}}

\newcommand{\R}{\mathds{R}}
\newcommand{\mbf}{\mathbf}
\newcommand{\ind}[1]{\mbf{1}\{ #1 \}}

\newcommand{\rn}{\sqrt{n}}
\newcommand{\nm}{n^{-1/2}}

\newcommand{\ixn}{\mathbb{X}}
\newcommand{\ix}{\mathbb{X}}
\newcommand{\dob}{\mathbb{W}}
\newcommand{\ikl}{I_k^L}

\newcommand{\ti}[1]{\noindent {\sc #1}.}
 
\newcommand{\ol}{\overline}
\newcommand{\ols}{\ol{s}}

\def\mT{\mathcal{T}}

\newcommand{\vepn}{\veps_n}

\newcommand{\argmax}[1]{\underset{#1}{\, \text{argmax}\ }} 
\newcommand{\argmin}[1]{\underset{#1}{\, \text{argmin}\ }}

\newcommand{\Var}{\operatorname{Var}}

\newcommand{\infoe}{\tilde{I}_{\eta_0}}

\newcommand{\sqe}{\textsf{SqExp}}

\newcommand{\expo}{\textsf{Exp}}

\newcommand{\ba}{\pmb{\alpha}}
\newcommand{\bp}{\mathbf{p}}

\newcommand{\bi}{\begin{itemize}}
\newcommand{\ei}{\end{itemize}}
\newcommand{\bc}{\begin{center}}
\newcommand{\ec}{\end{center}}
\newcommand{\bfr}{\begin{framed}}
\newcommand{\efr}{\end{framed}}

\newcommand{\bpr}{\begin{proof}}
\newcommand{\epr}{\end{proof}}
  
\newcommand{\vp}{\vspace{.25cm}}
\newcommand{\vm}{\vspace{.5cm}}

\newcommand{\tlk}{\theta_{lk}}
\newcommand{\tolk}{\theta_{0,lk}}

\newcommand{\sil}{\sigma_{l}}
\newcommand{\elk}{\veps_{lk}}

\newcommand{\an}{{\alpha_n}}
\newcommand{\ap}{$\an-$posterior }

\newcommand{\effinf}{\tilde{\psi}_{f_0}}

\DeclareMathOperator{\FDR}{FDR}
\newcommand{\FNR}{\mbox{FNR}}

\DeclareMathOperator{\lc}{L_C}

\DeclarePairedDelimiter{\abs}{\lvert}{\rvert}
\DeclarePairedDelimiter{\braces}{\{}{\}}
\DeclarePairedDelimiter{\norm}{\lVert}{\rVert}

\DeclarePairedDelimiter{\brackets}{(}{)}
\DeclarePairedDelimiter{\floor}{\lfloor}{\rfloor}

\newcommand{\fr}{\mathfrak{R}}

\usepackage{bm} 
\newcommand{\bb}{\bm{b}}

\usepackage{pgfplots} 

\numberwithin{figure}{chapter}

\newcounter{defcounter}
\counterwithin{defcounter}{chapter}
\newenvironment{definition}{
\setlength{\topsep}{0pt}
\noindent
\refstepcounter{defcounter}
\bfr
\sbl{Definition \thedefcounter.}
}{\efr\vm }

\newcounter{thmcounter}
\counterwithin{thmcounter}{chapter} 
\newenvironment{thm}{
\setlength{\topsep}{0pt}
\noindent
\refstepcounter{thmcounter}
\bfr  
\re{Theorem \thethmcounter.}
}{\efr\vm }

\newcounter{propcounter}
\counterwithin{propcounter}{chapter}
\newenvironment{prop}{
\setlength{\topsep}{0pt}
\noindent
\refstepcounter{propcounter}
\bfr
\re{Proposition \thepropcounter.}
}{\efr\vm }

\newcounter{lemcounter}
\counterwithin{lemcounter}{chapter}
\newenvironment{lem}{
\setlength{\topsep}{0pt}
\noindent
\refstepcounter{lemcounter}
\bfr
\sbl{Lemma \thelemcounter.}
}{\efr\vm }

\title{
\includegraphics[scale=1.2]{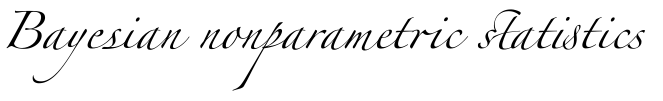}
St-Flour lecture notes
} 
 
\author{Isma\"el Castillo}

\date{}

\numberwithin{equation}{chapter} 

\begin{document} 
 
\maketitle

\epigraph{\ti{Principe} Si un \'evénement peut être produit par un nombre  $n$ de causes différentes, les probabilités de l'existence de ces causes prises de l’événement sont entre elles comme les probabilités de l’événement prises de ces causes, et la probabilité de l'existence de chacune d'elles est égale à la probabilité de l’événement prise de cette cause, divisée par la somme de toutes les probabilités de l’événement prises de chacune de ces causes.}{\textit{Laplace, Mémoire sur la probabilité de causes par les évènements}} 
 
\newpage 
 
\tableofcontents 
 
\newpage

\section*{Foreword}
\addcontentsline{toc}{chapter}{Foreword}

Bayesian methods are a prominent tool  in statistics, machine learning, and practical applications of statistics. Let us give a few examples from different application fields
\begin{enumerate}
\item medical imaging: often it is not possible to observe directly tissues in the interior of the human body, and many imaging medical devices use indirect measurements. For instance, in PET (Positron Emission Tomography) scans, one observes at the surface of the body the signal resulting of emissions of a radioactive liquid previously absorbed by the patient; in EIT (Electrical Impedance Tomography), electric currents are applied and measured at the level of the skin. In these examples, scanning machines built in the medical industry use Bayesian algorithms; in particular, measures of Bayesian uncertainty quantification are reported. We refer to \cite{stuart10} for a review on the use of Bayesian methods for  statistical {\em inverse problems}, of which PET and EIT are two specific examples.
\item astrostatistics and gravitational wave detection: in 2015--2016, a hundred years after Einstein's prediction of their existence through general relativity, LIGO and Virgo interferometers detected for the first time gravitational waves emitted from the interaction of a pair of black holes. One of the signal processing tools used to reconstruct signals is a Bayesian model, used to extract the waveform from the noisy data; here again Bayesian uncertainty quantification is reported in the form of credible bands on the wave form; see \cite{gravi, gravi14, gravi19}.
\item genomics and microarray data: in the analysis of DNA chips, it is common to have to handle a very high number of tests simultaneously (this is a {\em multiple testing} problem); beyond the famous Benjamini-Hochberg procedure (which can be seen as an empirical Bayes procedure), it is particularly interesting to use Bayesian prior distributions that take into account the existence of a possible {\em structure} in the data: one may think for instance of a Markov dependence structure that takes into account proximity along the DNA strip;  see for instance \cite{efronetal01, sc09, acg22}.
\item clustering in statistics and machine learning: often the statistician needs to classify a number of items such as texts in a number of different classes, taking into account that texts may share topics. Two influential works  \cite{steph00, bnj03} use a hierarchical Bayesian model based on  Dirichlet priors to do inference in such settings, the first in genetics for inference of population structures using multilocus genotype data and the second in the mentioned context of document modelling and text classification. 
\end{enumerate}

{\em Goal and outline.} The purpose of these lectures is to provide a set of tools to understand the behaviour of Bayesian posterior distributions in possibly complex settings (that is those where the parameter is for instance a function, or if there is a large number of unknown parameters), using what is called the {\em frequentist analysis of posterior distributions.} 

In practice many aspects of posterior distributions are used for inference. For instance,  regions with large probability under the posterior (so-called credible regions) are often used as confidence regions; for priors allowing for variable selection, posterior inclusion probabilities (the $\ell$--values in Chapter \ref{chap:mtc}) are routinely used to decide whether a variable should be included in the model or not.  Nevertheless,  mathematical guarantees for doing this are often lacking or limited. Understanding the properties of Bayesian posterior distributions in this context can both serve as a theoretical backup for algorithms used in practice, but also for providing guidelines for prior choices, as certain priors can sometimes be proved to be suboptimal and/or less convenient than others.

We try to cover ideas enabling to deal with the classical statistical trilogy of problems: {\em estimation} (mainly) and also {\em confidence sets} (for $1/\rn$--estimable functionals mostly) and {\em testing} (a bit).

We start from and take as basis the generic results of Ghosal, Ghosh and van der Vaart \cite{ggv00} in Chapter \ref{chap:intro}. We will build from there and discuss a number of examples such as Gaussian processes, tree priors, priors for deep neural networks, and topics such as statistical adaptation to smoothness, sparsity and structure in Chapters \ref{chap:rate1} to \ref{chap:ada2}. 
We then move to study the limiting shape of posterior distributions through results known as Bernstein--von Mises theorems in Chapters \ref{chap:bvm1} and \ref{chap:bvm2}. Bayesian multiple testing questions are considered in Chapter \ref{chap:mtc}. While the previous tools enable to provide theoretical back-up already for a number of practical algorithms, for instance those based on Gaussian processes or on $\ell$--values for testing, it is often the case that one needs to resort to some simulation algorithms that approximate the posterior distribution. While not the main focus of these lectures, we provide some convergence results in Chapter \ref{chap:vb} for Variational Bayes approximations to posteriors, which form a popular alternative to algorithms such as MCMC (Markov Chain Monte Carlo), especially for large models and deep neural networks.\\

{\em Scope.} For simplicity of exposition we work mostly in the setting of well-specified models. In practical applications the degree in which the model can be considered as known or nearly known varies, so depending on the setting, one may have to take into account additional term(s) to account for possible model misspecification. We deal with infinite-dimensional models, so allow for (some) model misspecification in that the precise specification of functions or parameters need not be known, although for convenience we often make assumptions (such as the one of independent and identically distributed data -- i.i.d. or simply iid for short -- in density estimation, or of Gaussian noise in regression). There are ways to make the Bayesian approach more `robust', for instance by replacing the likelihood by some other function; we do not go in this (interesting) direction here, except perhaps when we derive results for tempered posteriors, where the likelihood is raised to some small power.  In a minimal view, the results presented here can be seen as a solution for what can be done if the model is reasonably well known.  An advantage is that we can derive quite sharp results in terms of optimality, related to the fact that we work with the likelihood function. That is, the presented theory can be seen as a benchmark of what can be achieved in case we have already a reasonable amount of knowledge about the model. 

One restriction we make is that we work in the `dominated setting', that is in particular one where there is a Bayes formula (described in the finite case by Laplace in the citation above from \cite{laplace1774}) for the conditional distribution. The general non-dominated case is also very interesting (and relevant among others for the clustering applications mentioned as point 4. above); typical objects arising in that setting are processes with jumps at random locations such as Dirichlet processes or completely random measures, but the theory requires quite different techniques. A unifying theory has not emerged yet and would be of great interest.

Finally, we work mostly in an asymptotic setting to make the statements and arguments more transparent, although often arguments can be made non-asymptotic with fairly explicit constants (some of these arising from the theory may then be quite `large' though, but this is not specific to the Bayesian analysis presented here).\\

{\em On proofs and simplifications.} We have chosen to focus for the main part of these lectures on relatively simple (yet central) models from nonparametric statistics: the Gaussian white noise model, density estimation, sparse sequences and linear regression for high-dimensional models. We treat sometimes slightly more complex settings such as compositional structures; some other relevant settings where the presented theory can be applied include inverse problems, survival analysis, diffusion models, graphical models or matrix estimation problems, to name a few. 
 
We try when possible to present simple arguments that will still be robust to a complexification of the model and setting. Sometimes this requires some slight adaptations: we have at times rewritten arguments from papers  adding one or two assumptions to make the proof simpler. Sometimes we use tempered posteriors to focus on prior mass conditions only, and occasionally we use  boundedness assumptions in high-dimensional models. We have tried to comment on it in the text. \\

{\em Research areas and open problems.} While there is an elegant and general approach via prior mass (plus possibly testing and entropy control) for estimation using Bayesian posteriors in terms of certain losses, presented in Chapters \ref{chap:intro}-\ref{chap:rate1}, there is a lot to understand yet beyond this. We discuss a few other losses particularly relevant for applications, such as the supremum loss, multiple testing and classification losses, although for now in specific settings: a general theory for these is still lacking. In particular, there is a strong potential for testing methods using Bayesian posteriors, but no general theory. 
 There is also a particular need for developing asymptotic normality type results among others in high-dimensional models, which in a sense correspond to a blend of the material presented in Chapters \ref{chap:ada2} and \ref{chap:bvm1}. More generally, uncertainty quantification is a key problem in data science today: there is much to do to understand the theoretical boundaries of what can be done already for any (computable) method, and using a Bayesian posterior distribution as a measure of uncertainty in particular.\\

{\em Acknowledgements.}  These lectures were given at the $51$th \'Ecole d'été de Probabilités de Saint--Flour in 2023. Part of the material is based on joint work with co-authors and it is my great pleasure to thank them all here. I am very grateful for the Scientific Board for having given me this opportunity to lecture at the Summer School. I would like to thank Felix Otto and Ivan Corwin for their inspiring lectures, as well as all the participants,  and especially the local organisers Christophe Bahadoran, Hacène Djellout and Boris Nectoux, for two intense  weeks full of activities including a memorable ascension to the Puy Marie. 
 
I am indebted to Eddie Aamari and Pierre Alquier (both present on spot!), to Kweku Abraham, Sergios Agapiou, Julyan Arbel, Diarra Fall, Matteo Giordano, Guillaume Kon Kam King, Thibault Randrianarisoa, \'Etienne Roquain and Stéphanie van der Pas for comments on the text, as well as to Clarisse Boinay, Lucas Broux, Gabriel Clara,  Mauricio Daros Andrade, Paul Egels, Sascha Gaudlitz, Alessandro Gubbiotti, Mikolaj Kasprzak, Alice L'Huillier,  Felix Otto,  Guillaume Le Mailloux,  Raphaël Maillet, Mathieu Molina, Rémi Peyre, Samis Trevezas and Sumit Vashishtha for their questions and comments during the school.

\chapter{Introduction, rates I} \label{chap:intro}

\section{Statistical models} \label{sec:models}

Let  $\cX^{(n)}$ be a metric space equipped with a $\sigma$-algebra $\cA^{(n)}$, where $n$ is an integer corresponding to the amount of information, for instance the number of available observations.  A  {\em statistical experiment} is a collection of probability measures $\{P_{\eta}^{(n)}\}$ on $\cX^{(n)}$ indexed by a parameter $\eta$ which belongs to some measurable space  $\cH$ to be specified. We use the generic notation $X^{(n)}$ to denote observations from this experiment, which generally (although not always) will mean that $X^{(n)}$ is a draw from $P_{\eta}^{(n)}$. 
 From now on $n\ge 1$ is a given integer, which we may let tend to $\infty$. \\ 

In all following examples, the statistical model is indexed by, either an infinite-dimensional parameter, for example a real-valued function, denoted e.g. by $\eta=f$, over some space, 
or a {\em high-dimensional} parameter, denoted by $\eta=\te$. In the latter case, the model is parametric for each fixed $n$ but the dimension of the parameter increases with $n$. The quantities $f$ or (/and) $\te$ are unknown and the goal of the statistician is to say something about them after having observed data from the model. Many statistical questions can be classified as belonging to one of the following trilogy: {\em estimation}, {\em testing} and {\em confidence sets}. Here we will be mostly interested in estimation and  confidence sets, but will also occasionally mention testing, which sometimes plays an 
important role in proofs. 

\subsubsection{Nonparametric models}

In the next models, the unknown parameter is a single function $\eta=f$, although there could be several functions to estimate, or a combination of a function and a finite-dimensional parameter (which is rather called a semiparametric model, as mentioned below).\\

\ti{Fixed--design nonparametric regression model} One observes, for $f$ (say) a continuous function,
\[ X_i = f(i/n) + \veps_i,\quad 1\le i\le n,\]
where $\veps_i$ are iid $\cN(0,1)$. In this case $\eta=f$ and \[ P_f^{(n)}=\bigotimes_{i=1}^n \cN(f(t_i),1).\] 
This model has independent but not identically distributed observations. \\ 

\ti{Gaussian white noise model}
Let $L^2:=L^2([0,1])$ be the space of square integrable functions on $[0,1]$. For $f \in L^2$, $dW$ standard white noise, consider observing
\begin{equation} \label{mod.gwn}
 dX^{(n)}(t) = f(t)dt + \frac{1}{\rn}dW(t),\quad t\in [0,1].
\end{equation}
Two possible meanings of `observations'   in this context are:  observing the path 
 $X^{(n)}(x)=\int_0^x f(t)dt + n^{-1/2}W(x)$, where $W$ is standard Brownian motion on $[0,1]$; or, given a collection of orthonormal functions $\{\vphi_k,\ k\ge 1\}$ in $L^2$ forming a basis of $L^2$, observing the collection 
\begin{equation}
 X_k = f_k + \frac{1}{\rn}\xi_k, \qquad k\ge 1,
\end{equation} 
where $f_k=\int_0^1 \vphi_k f$ and $\{\xi_k\}_{k\ge 1}$ are independent $\cN(0,1)$ variables. This last version of the model is often also called the (infinite) {\em Gaussian sequence model}. In the first case $\eta=f$ and in the second case $\eta=(f_k)$ the collection of coefficients of $f$ onto the basis.\\

\ti{Density estimation}
On the unit interval, the density model consists of observing  independent identically distributed  data
\begin{equation} \label{mod.dens}
X_1,\ldots,X_n \quad {i.i.d.}\ \sim f,
\end{equation}
with $f$ a density function on the interval $[0,1]$. In this case the common law of the $X_i$s is the distribution $P_f$ of density $f$ with respect to Lebesgue measure on $[0,1]$.  Then $P_f^{(n)}$ is the product measure $\otimes_{i=1}^n P_f$ on $[0,1]^n$. 
For models with i.i.d.~data (and those only) for simplicity in the sequel we denote $P_f^{n}:=P_f^{(n)}$.\\

\ti{Geometric spaces} 
It is of interest to generalise the previous models to the case where data `sits' on a 
geometrical object, say a compact metric space  $\M$. One may think of a torus, a sphere, a manifold, or maybe even a discrete structure such as a tree, a graph etc. 

For instance, the density estimation model on a manifold $\M$ consists in observing 
\begin{equation} \label{gdens}
X_1,\ldots,X_n \quad {i.i.d.}\ \sim f,  
\end{equation}
where $X_i$ are $\M$-valued random variables with positive density function $f$ on $\M$. \\

\subsubsection{High-dimensional models}

High-dimensional models are those where the number of unknown parameters may grow to infinity with the number of data points. In such models, in order for estimation to be possible it is often necessary to assume that only a relatively small number of parameters are truly significant, which is  a {\em sparsity} assumption. In the next two examples, the parameter $\eta=\theta$ is a vector of high dimension; it could also be a matrix as in estimation of low-rank matrices; we refer to e.g. \cite{saravdg-stf} for an overview.\\

\ti{Needles and straw in a haystack}
Suppose that we observe 
\begin{equation} \label{mod.spa}
 X_i = \theta_i + \veps_i,\quad i=1,\ldots,n,
\end{equation}
for independent standard normal random variables $\veps_i$ and
an unknown vector of means $\te=(\te_1,\ldots,\te_n)$.  Suppose $\theta$ is {\em sparse} in that it belongs to the class of {\em nearly black vectors}
\begin{equation} \label{def-nb}
 \ell_0[s_n] = 
\left\{ \theta\in\RR^n:\ \text{Card}\{i: \te_i  \neq 0 \} \le s_n \right\}. 
\end{equation}
Here $s_n$ is a given number, which in theoretical investigations
is typically assumed to be $o(n)$, as $n\ra\infty$. Sparsity may also
mean that many means are small, but possibly not exactly zero. \\

\ti{High-dimensional linear regression}
Consider estimation of a parameter $\te\in\RR^p$ in the linear regression model 
\begin{equation} \label{mod.spahd}
Y = X\te + \veps
\end{equation}
where $X$ is a given, deterministic $(n\times p)$ matrix, and $\ep$ is an $n$-variate standard normal vector. As for the previous model, we are interested in the \emph{sparse} setup, where $n\leq p$, and possibly $n\ll p$, and most of the coefficients $\be_i$ of the parameter vector
are zero, or close to zero.  Model \eqref{mod.spa} is a special case with $X$ the identity matrix of size $n$. Model \eqref{mod.spahd} shares some
features with this special case, but is different in 
that it must take account of the noninvertibility
of $X$ and its interplay with the sparsity assumption, and
does not allow a factorization of the model along the coordinate axes.\\

There are of course further links between all the above models. For instance, the study of the sparse Gaussian sequence model \eqref{mod.spa} is related to the study of certain sparse 
nonparametric classes, namely sparse Besov spaces. 

\subsubsection{Semiparametric models}

Slightly informally and broadly speaking, one may define semiparametric models as those models where the parameter of interest is a (often, but not necessarily) finite-dimensional {\em  aspect} of the parameter $\eta$ of the model, where $\eta$ is typically infinite-dimensional. We give two first examples. \\

\ti{Separated semiparametric models} A model $\{P_\eta^{(n)},\ \eta=(\theta,f),\ \theta\in\Theta,\ f\in \cF\}$, where $\Theta$ is a subset of $\RR^k$ for given $k\ge 1$ and $\cF$ a nonparametric set, is called {\em separated} semiparametric model. The full parameter $\eta$ is a pair $(\theta,f)$, with $\theta$ called parameter {\em of interest} and $f$ {\em nuisance} parameter. Despite the terminology, this does not exclude $f$ to be of interest too.

For instance, 
the following is called shift or {\em translation model}: 
one observes sample paths of the process $X^{(n)}$ such that, for $W$ standard Brownian motion,
\begin{equation} \label{transla}
dX^{(n)}(t) = f(t-\te) dt + \frac{1}{\rn} dW(t), \quad t\in[-1/2,1/2],
\end{equation}
where the unknown function $f$ is {\em symmetric} (that is $f(-x)=f(x)$ for all $x$), smooth and say $1$-periodic, and the unknown parameter of interest $\theta$ is the center of symmetry of the {\em signal} $f(\cdot-\te)$. This model has a 
very specific property: estimation of $\theta$ can be done as efficiently as in the parametric case where $f$ would be known, at least asymptotically. This is called a model {\em without loss of information}. An example of model with {\em loss of information} is the famous Cox proportional hazards model, see Appendix \ref{app:models} for a definition.\\

\ti{Functionals} More generally, given a model $\{P_\eta^{(n)},\ \eta\in\cH\}$, one may be interested in estimating a function $\psi(\eta)$ of the parameter $\eta$, for some function $\psi$  on $\cH$. For instance, in the density model \eqref{mod.dens}, one may consider the estimation of linear functionals of the density $\psi(f)=\int_0^1 a(u)f(u)du$, for $a(\cdot)$ a bounded measurable function on $[0,1]$.\\

Nonparametric models are recurrent in these notes, semiparametric models and functionals will be discussed in Chapters \ref{chap:bvm1}-\ref{chap:bvm2}, while part of Chapter \ref{chap:ada2} and Chapter \ref{chap:mtc} are concerned with high-dimensional models. Sometimes there will be a further unknown `structure' underlying the unknown parameter and that can also be of interest: in high-dimensional model this is often the sparsity pattern (i.e. which coordinates are truly active), or it can be a collection of active variables in compositions, such as in deep neural networks models in Chapter \ref{chap:ada2}.  

\section{The Bayesian paradigm}

To introduce the Bayesian framework and the generic theorems in the first two Chapters, we will use the classical notation $\te$ instead of $\eta$ for the parameter (although $\te$ may be a function). In nonparametric models below we will mostly use $f$, and will go back to $\eta$ for semiparametric models.\\ 

\ti{The Bayesian approach} Given a statistical model $\{P_\te^{(n)},\ \te\in \Theta\}$ and observations $X^{(n)}$,  a Bayesian defines a (new, called `Bayesian') model by attributing a probability distribution to the pair $(X^{(n)},\te)$. To do so, \sbl{first}, one chooses a probability distribution $\Pi$ on $\Theta$, called the \sbl{prior distribution}.  The distribution $P_\te^{(n)}$ is then viewed as the conditional law  $X^{(n)} \given \te$.  
Combining both distributions indeed gives a law on the pair  $(X^{(n)},\te)$. This is often written
\begin{align*}
X^{(n)}\given \te & \sim P_{\te}^{(n)} \\
\te & \sim \Pi.
\end{align*}
The estimator of $\te$, in the Bayesian sense, is then the conditional distribution $\te  \given X^{(n)}$, called {\em a posteriori} distribution or simply {\em posterior}. The posterior is a data-dependent probability {\em measure}
 on $\Theta$ and is denoted $\Pi(\cdot\given X^{(n)})$. One often writes
  \[ \te \given X^{(n)} \sim \Pi(\cdot\given X^{(n)}). \]

Informally, in order to estimate a given parameter $\te$, one first makes it random by choosing a prior distribution on the set of possible parameters.  Next one updates this a priori knowledge by conditioning on the observed data, obtaining the posterior distribution. Of course, many choices of prior are in principle possible, and one can expect this choice to have an important  impact on how the posterior distribution looks like. As the number of observations grows however, one may expect  that the influence of the prior becomes less and less eventually. We shall see through all three next Chapters that in nonparametric and high dimensional models this typically cannot be achieved without special care. \\

Technically speaking, the standard (and broadest) definition of conditional distributions is via desintegration of measures. Here we shall restrict ourselves to a specific setting where the posterior distribution is given by Bayes' formula and hence can be (somewhat) explicitly written.\\

\ti{Dominated framework}  In the sequel, we suppose we are in the following dominated framework: for $\mu^{(n)},\nu$ sigma--finite measures, suppose 
\begin{align*}
dP_\te^{(n)} & = p_\te^{(n)}d\mu^{(n)} \qquad \forall\theta\in\Theta,\\
d\Pi & = \pi d\nu.
\end{align*}
Note that the measure $\mu^{(n)}$ has to dominate all measures $P_\te^{(n)}$, for any possible value of $\theta$. The second line is present for convenience: often we just take $\nu=\Pi$ (and then $\pi=1$),  but sometimes there is a natural measure $\nu$ (for instance Lebesgue measure if one works with continuous priors on $\RR$) to work with.\\

In this setting, the distribution of $(X^{(n)},\theta)$ has density $(x^{(n)},\te)\to p_\te^{(n)}(x^{(n)})\pi(\te)$ with respect to $\mu^{(n)}\otimes\nu$. We will always assume (without mentioning it) that this mapping is measurable for suitable choices of $\sigma$--fields on the space of $X$'s and $\theta$'s, so that the next definition makes sense. 

\begin{definition}
The \sbl{posterior distribution}, denoted $\Pi[\cdot\given X^{(n)}]$, is the conditional distribution $\cL(\te\given X^{(n)})$ of $\te$ given $X^{(n)}$ in the Bayesian setting as above. It is a distribution on $\Theta$, that depends on the data $X^{(n)}$. In the dominated framework as assumed above, it has a density with respect to $\nu$ given by Bayes' formula 
\[ \te \to \frac{p_\te^{(n)}(X^{(n)})\pi(\te)}{\int p_\te^{(n)}(X^{(n)})\pi(\te)d\nu(\te)}.\]
\end{definition}

\sbl{Example: fundamental model.} Consider the model $\cP=\{\cN(\te,1)^{\otimes n},\ \te\in\RR\}$ with observations $X^{(n)}=(X_1,\ldots,X_n)$. Suppose we take a normal prior $\Pi=\cN(0,\sigma^2)$ on $\te$ (with $\sigma^2>0$): as an Exercise it is easy to check that the posterior $\Pi[\cdot\given X^{(n)}]$  is also a Gaussian distribution -- any Gaussian prior leads to a Gaussian posterior, so the class of all such priors is said to be {\em conjugate} -- given by, for $\oli{X}$ the empirical mean $\sum_{i=1}^n X_i/n$,
\begin{equation} \label{mfpost}
  \cL(\te\given X^{(n)})=\cN\left( \frac{n\oli{X}}{n+\sigma^{-2}}, \frac{1}{n+\sigma^{-2}} \right). 
\end{equation}

\vp

The statistical models we have introduced in Section \ref{sec:models} are all dominated: for fixed design regression, density estimation on $[0,1]$, and the high-dimensional models, one can  take the product-Lebesgue measure on $[0,1]^n$ as dominating measure. For the Gaussian white noise and sequence models, one uses the measure induced by the observations when the parameter is zero, see Appendix \ref{app:models}.  \\ 

Compared to standard (point-)estimators such as the maximum likelihood estimator (MLE), the Bayesian posterior distribution is a more complex object, having both a `center' (this can be e.g. the mean, or the median of the posterior) and a measure of `uncertainty' or `spread' (such as the posterior standard deviation or variance if they exist). This can be useful for the purpose of \sbl{uncertainty quantification}.

\begin{definition}
A \sbl{credibility region} of level (at least) $1-\al$, for $\al\in[0,1]$, is a measurable set $A\subset\Theta$ (typically depending on the data $A=A(X)$), such that
\[ \Pi[A\given X] = (\ge) 1-\al.\]
\end{definition}

Natural questions at this point are: how does $\Pi[\cdot\given X^{(n)}]$ behave as $n\to\infty$ and in which sense? Is there convergence? A limiting distribution? Are credibility regions linked in some way to confidence regions? We will attempt to give some answers in these lectures. \\

\noindent {\sc{Why Bayesian estimators?}}
Often, priors have a natural probabilistic interpretation and insights from the construction of stochastic processes in probability theory can be helpful to understand how the prior distribution spreads its mass across  the parameter set. Additional `smoothing' parameters may themselves get a prior, thus leading to  natural constructions of priors via hierarchies. 

Since the posterior is a measure, it can serve various purposes at the same time: for estimation one may use a point estimator deduced from the posterior, for uncertainty quantification one may use credible sets, while testing hypotheses can in principle be done by just comparing their probabilities under the posterior. The fact that one is using the likelihood suggests the posteriors may inherit certain `optimality' properties thereof; this combined with the flexibility of the choice of the prior distribution should make it possible to achieve optimality in many situations. 
Of course proving that the previous steps are legitimate and that certain optimality properties hold is not always an easy task, especially in high dimensional models. 

There are other attractive aspects of the Bayesian approach that we do not discuss here: for instance the fact that there are natural priors corresponding to exchangeable data, as developed among others by the Italian school after de Finetti. 

From the practical perspective, implementation methods of posterior distributions based on e.g. Markov Chain Monte Carlo techniques have been very much developed since the mid-90's, and somewhat more recently Variational Bayes approximations, that we discuss briefly in Chapter \ref{chap:vb} have witnessed an important interest in particular in the machine learning community. This in turn 
leads  to the need of developing theoretical tools to understand convergence properties of the corresponding posterior distributions, and of their approximations.

\section{Nonparametric priors, examples}

Maybe the most natural idea to build a prior on a nonparametric object such as a function is to decompose the object into simple, finite-dimensional, `pieces'. Next put a prior distribution on each piece and finally `combine' the pieces together to form a prior on the whole object. 

If $\eta=f$ is an element of $L^2[0,1]$, one may first decompose $f$ into its coefficients $\{f_k\}$ onto an orthonormal basis $\{\vphi_k\}$ of $L^2[0,1]$, such as the Fourier basis, a wavelet basis etc. Next, draw real-valued independent variables as prior on each coefficient. A natural requirement is to choose the individual laws so that the so-formed function $f$ almost surely belongs to $L^2$. This can be easily accommodated by taking the coordinate variances going to $0$ fast enough. This leads us to set
\begin{equation} \label{gen.prior}
f(\cdot) = \sum_{k=1}^{\infty} \sigma_k A_k \vphi_k(\cdot),
\end{equation} 
where $\{A_k\}$ is a sample from a centered distribution with finite second moment and $\{\sigma_k\}$ is a deterministic sequence in $\ell^2$. This gives ample room for choosing sequences $\{\sigma_k\}$ and the common law of the $\{A_k\}$. And, anticipating slightly, the variety of behaviours of the corresponding posterior distributions in such simple models as white noise \eqref{mod.gwn} is already quite broad. 
\\

\ti{Gaussian process priors} Specialising the previous construction to Gaussian distributions for the law of $A_k$, one obtains particular instances of Gaussian processes taking values in $L^2[0,1]$. 

Another way of building a, say centered, Gaussian process prior $(Z_t)$ on the interval $[0,1]$ is via a covariance kernel $K(s,t)=\mathbb{E}(Z_sZ_t), (s,t)\in [0,1]^2$. The choice $K(s,t)=s\wedge t$ gives Brownian motion, which can be see to have (a version with) paths of H\"older--regularity $1/2-\delta$ for any $\delta>0$.  The choice $K(s,t)=e^{-(s-t)^2}$ corresponds to the so-called squared-exponential Gaussian process, which it turns out induces much smoother paths than those of Brownian motion.

Starting from Brownian motion, one can define a new Gaussian process by integrating it a fractional number $(\al-1/2)$ of times. This leads to the  Riemann-Liouville process of parameter $\al>0$ 
\begin{equation} \label{defrl}
R_t^{\al}=\int_0^t (t-s)^{\al-1/2}dW(s),\quad t\in[0,1],
\end{equation}
where $W$ is standard Brownian motion. One further defines a {\em Riemann-Liouville type process} (RL-type process) as, for $\underline{\al}$ the largest integer smaller than $\al$, 
\begin{equation} \label{defrlt}
 X_t^\al=R_t^{\al}+\sum_{k=0}^{\underline{\al}+1}Z_k t^k, \quad t\in[0,1], 
\end{equation} 
where $Z_0,\ldots,Z_{\underline{\al}+1},R_t$ are independent, $Z_i$ is standard normal and $R_t^{\al}$
is the Riemann-Liouville process of parameter $\al$. If $\al=1/2$ then $R_t^{\al}$ is simply standard Brownian motion and if $\{\al\}=1/2$, with
$\{\al\}\in[0,1)$ the fractional part of $\al$, then $R_t^{\al}$ is a $k$-fold integrated Brownian motion.
The reason for adding the polynomial part to form the RL-type process $X_t^{\al}$ is that the support in $\cC^0[0,1]$ of $X_t^{\al}$ is the whole space $\cC^0[0,1]$, see Chapter \ref{chap:rate1} for more on this.

Yet another, slightly more abstract, way of building a Gaussian prior is by defining it as a Gaussian measure on a separable Banach space $\mb$ (e.g. $L^2[0,1]$, $\cC^0[0,1]$ etc.) with a norm denoted $\|\cdot\|_\mb$ or simply $\|\cdot\|$ if no confusion can arise. 
It can be shown that, in general, this construction coincides with the one starting from a covariance kernel as above. We refer to \cite{rkhs} for a comprehensive review. \\

\ti{Priors on density functions} Now consider the question of building a prior distribution on a density $f$ on the interval $[0,1]$. A difficulty is the presence of two constraints on $f$, that is 
$f\ge 0$ and $\int_0^1 f=1$, which prevents the direct use of a prior such as \eqref{gen.prior}.  
We briefly present some approaches. Although arguably not the first to have been considered historically, a simple possible approach consists in applying a transformation to a given function on $[0,1]$ to make it a density, such as an exponential link function \cite{Leonard78, Lenk88}. Given a, say, continuous function $w$ on $[0,1]$, consider the mapping $w\to p_w$ defined by
\begin{equation} \label{exp-transfo}
 p_w(s) = \frac{ e^{w(s)} }{ \int_0^1 e^{w(u)} du },\quad s\in[0,1]. 
\end{equation} 
Now any prior on continuous functions, such as a random series expansion \eqref{gen.prior} or a Gaussian process prior on $[0,1]$ as before, gives rise to a prior on densities by taking the image measure under the transform \eqref{exp-transfo}.

 A different yet perhaps more `canonical' approach is to build the random density directly via the construction of a random probability measure on $[0,1]$, absolutely continuous with respect to Lebesgue measure. This connects this question to the central topic of construction of random measures. A landmark progress in that area was the construction of the Dirichlet process by Ferguson (1973) \cite{ferguson73}. In terms of density estimation however, samples from the Dirichlet process cannot be used directly since the corresponding random measure is discrete. However, the Dirichlet process turns out to be a particular case of some more general random structures: so called tail-free processes, which where introduced by Freedman (1963) 
 \cite{freedman63} and Fabius (1964) \cite{fabius64}. For well-chosen parameters, the so-obtained random probability measures have a density. This way one obtains as particular case the P\'olya tree processes \cite{msw92}, \cite{lavine92} introduced in Chapter \ref{chap:rate1}. 
  
Other ways to build random densities include random histograms (see  Chapters \ref{chap:rate1}, \ref{chap:ada1} and \ref{chap:bvm2}), random kernel mixtures (Chapter \ref{chap:ada1}) such as Bernstein polynomials \cite{petrone02}, Beta mixtures \cite{r10}, location scale mixtures  \cite{dejongevz10} etc. \\
 
\ti{Priors in semiparametric models} In a separated semiparametric model $\{\cP_{\te,f}\}$, a natural way to build a prior on the pair $(\te,f)$ is simply via a product prior $\pi_\theta\otimes \pi_f$ on each coordinate.

\section{Convergence of the posterior distribution}

Recall that we work with a model $\cP=\{P_\te^{(n)},\ \te\in\Theta\}$, where $\theta$ is typically a nonparametric quantity (a function $\te=f$, a pair $\te=(\ta,f)$,...). We equip the set $\Theta$ of possible parameters with a (semi-)metric $d$. Example of metrics include those induced by metrics on probability measures e.g. $d(\te,\te'):=d(P_{\te}^{(n)},P_{\te'}^{(n)})$ (see Appendix \ref{app-dist}). There are many other distances of interest. For instance, if $\te=f$ a function, one may think of $L^p, p\ge 1$, metrics $d(f,g)=\|f-g\|_p$. The Bayesian approach sets
\begin{align}
\te & \sim \Pi,  \label{prior}\\
X^{(n)} \given \te & \sim P_\te^{(n)},\label{posterior}
\end{align} 
and Bayes' formula (under the domination and measurability assumptions) explicitly gives the mass of any  $B\in\cB$ under the posterior distribution
\begin{equation}\label{bayesfor}
\Pi[B\given X^{(n)}]=\frac{\int_B p_\te^{(n)}(X^{(n)})d\Pi(\te)}{\int p_\te^{(n)}(X^{(n)})d\Pi(\te)}.
\end{equation}
Note that if $\Pi[B]=0$ for a set $B$, then $\Pi[B\given X^{(n)}]=0$. \\

There are several ways in which one can use the Bayesian modelling \eqref{prior}--\eqref{posterior}. The first obvious way is to assume that the  distribution of our actually observed data $X^{(n)}$ is the marginal distribution of  $X^{(n)}$ arising from \eqref{prior}--\eqref{posterior}. Under this setting one can investigate optimality properties with respect to so--called Bayesian risks $E \ell(\te,T(X^{(n)}))$ for a loss function $\ell$, estimators $T(X^{(n)})$ and $E$ the expectation under the previous distributions.  However everything then depends on the actual choice of prior $\Pi$ and two statisticians with two different priors (even if they are `close') may get completely different answers. Another, more `objective' way is to assume that there is a `true' value $\te_0$ of the parameter and that the actually observed data follows $X^{(n)}\sim P_{\te_0}$, and study \eqref{bayesfor} in probability under 
this assumption. This is called frequentist analysis of posterior distributions and is the framework we consider in these notes. This is not to say that we will not use the Bayesian model in itself, it can in fact be very helpful to suggest optimal estimators or procedures (see Chapter \ref{chap:mtc} for an example). \\

\ti{Notation} For simplicity we drop the dependence in $n$ in the notation and write $X=X^{(n)}, p_\te=p_\te^{(n)}$. One should keep in mind that typically all quantities below depend on `$n$'. We denote by $E_{\te}$ the expectation under the law $P_{\te}=P_{\te}^{(n)}$ and $\text{Var}_{\te}$ the variance under $P_{\te}$. \\

\ti{Frequentist analysis of posteriors.} In what follows we study the behaviour of $\Pi[\cdot\given X]$ in probability under $P_{\te_0}$. By dividing by $p_{\te_0}(X)$, which is independent of $\te$, one may rewrite \eqref{bayesfor} as
\[ \Pi[B\given X] = \frac{\int_B \frac{p_\te}{p_{\te_0}}(X) d\Pi(\te)}{\int \frac{p_\te}{p_{\te_0}}(X) d\Pi(\te)}. \]
In order for the ratio in the last display to be well--defined under $P_{\te_0}$, it will be silently assumed that $P_{\te_0}[\int  p_\te(X) d\Pi(\te)>0]=1$, which will always be the case for the priors we shall consider.

\begin{definition}\sbl{[Contraction rate]} We say that a sequence $\veps_n$  (often  tending to $0$ as $n\to\infty$) is a contraction rate around $\te_0$ for $\Pi[\cdot\given X]$,   for $d$ a suitable metric over $\Theta$, if as $n\to\infty$,
\[ E_{\te_0}\Pi[ d(\te,\te_0)>\veps_n \given X] =o(1). \]
\end{definition}
What will be our target rate $\veps_n$? This will depend on $\te_0$, $\Theta$ and $d$. Often, we shall assume that $\te_0$ belongs to some regularity set $S_\be(L)$ (say a Sobolev ball of order $\be$ and radius $L$) and we will try to take $\veps_n$ to be of the order (or as close as possible to) of the minimax rate
\[ \bar{\veps}_n = \inf_T \sup_{\te \in S_\be(L)} E_\te d(T,\te), \]
where the infimum is taken over all possible estimators $T=T(X)$ of $\te$. For standard regularity classes and distances, $\bar{\veps}_n$ will often be of the order  $C(\be,L)n^{-\be/(2\be+1)}$, possibly up to logarithmic factors. \\

For probability measures $P,Q$ and $P\ll Q$ (otherwise set the quantities below to  $+\infty$), let us set
\begin{align*}
K(P,Q)  = \int \log\frac{dP}{dQ} dP,\qquad
V(P,Q) & = \int \left(\log\frac{dP}{dQ}-K(P,Q)\right)^2 dP,
\end{align*}
respectively the Kullback--Leibler (KL) divergence and its ``variance". 

\begin{definition}\sbl{[KL--type neighborhood]}$\,$
For any $\veps>0$, we define  
\begin{equation} \label{bkln}
B_K(\te_0,\veps) =
\left\{ \te:\ K(P_{\te_0},P_\te)\le n\veps^2,\ V(P_{\te_0},P_\te)\le n\veps^2 \right\}.
 \end{equation}
\end{definition}
This neighborhood of $\te_0$ plays an important role. To illustrate this, we start by a key lemma that demonstrates how to use $B_K(\te_0,\veps)$ to bound the denominator in Bayes formula from below. 
\begin{lem} \label{lem:tr1}
For any probability distribution $\Pi$ on $\Theta$, for any $C,\veps>0$, with $P_{\te_0}$--probability at least $1-1/(C^2n\veps^2)$,
\begin{equation} 
\int   \frac{p_\te}{p_{\te_0}}(X) d\Pi(\te) \ge \Pi[ B_K(\te_0,\veps) ]
e^{-(1+C)n\veps^2}.
\end{equation}
\end{lem}
\begin{proof}[Proof of Lemma \ref{lem:tr1}] 
Let $B:=B_K(\te_0,\veps)$ and suppose $\Pi(B)>0$, otherwise the result is immediate. Let us denote $\oli\Pi(\cdot)=\Pi(\cdot\cap B)/\Pi(B)$. Next let us bound from below
\[  \int  \frac{p_\te}{p_{\te_0}}(X) d\Pi(\te) \ge 
\int_B  \frac{p_\te}{p_{\te_0}}(X) d\Pi(\te) = \Pi(B) \int  \frac{p_\te}{p_{\te_0}}(X) d\oli\Pi(\te). \]
As  $\oli\Pi(\cdot)$ is a probability measure on $B$, Jensen's inequality applied to the logarithm gives
\begin{align*}
\log \int  \frac{p_\te}{p_{\te_0}}(X) d\oli\Pi(\te)
& \ge \int_B \log  \frac{p_\te}{p_{\te_0}}(X) d\oli\Pi(\te) \\
& \ \ = - \int_B \left[\log  \frac{p_\te}{p_{\te_0}}(X) -K(P_{\te_0},P_\te)\right] d\oli\Pi(\te)
- \int_B K(P_{\te_0},P_\te)d\oli\Pi(\te)\\
& \ \ \ \ \ge - Z - n\veps^2,
\end{align*}
where we have set $Z:= \int_B \left[\log  \frac{p_\te}{p_{\te_0}}(X) -K(P_{\te_0},P_\te)\right]d\oli\Pi(\te)$ and used that $K(P_{\te_0},P_\te)\le n\veps^2$ on the set $B$ by definition. Define the event $\cB_n := \{ \left| Z \right| \le Cn\veps^2 \}$. 
 By Tchebychev's inequality
\[ P_{\te_0}\left[ \cB_n^c \right] =
P_{\te_0}\left[\left|  Z \right| > Cn\veps^2 \right]  \le \frac{1}{(Cn\veps^2)^2} \text{Var}_{\te_0} Z.\]
Use Cauchy-Schwarz' inequality,  
Fubini's theorem and the fact that $V(f_0,f)\le \veps^2$ on $B$ to deduce
\begin{align*}
\text{Var}_{\te_0}Z & = E_{\te_0}\left[ \left(  
\int_B \left[\log \frac{p_{\te_0}}{p_{\te}}(X) -K(P_{\te_0},P_\te)\right] d\oli\Pi(\te)
\right)^2  \right] \\
&\le E_{\te_0}   
\int_B \left[\log \frac{p_{\te_0}}{p_\te}(X) -K(P_{\te_0},P_\te)\right]^2 d\oli\Pi(\te) 
 = \int_B V(P_{\te_0},P_\te) d\oli\Pi(\te) \le n\veps^2 \oli\Pi(B)=n \veps^2.
\end{align*}
Deduce $P_{\te_0}\left[\cB_n^c\right]\le 1/(C^2 n\veps^2)$. The previous bounds imply that on $\cB_n$,
\[ \log \int   \frac{p_{\te_0}}{p_{\te}}(X) d\oli\Pi(\te) \ge -(C+1)n\veps^2, \]
from which the claimed result follows by taking exponentials and renormalising by $\Pi(B)$.
\end{proof}

Lemma \ref{lem:tr1} is key for proving the next result, which gives a more refined version of the statement $\Pi[B]=0$ implies $\Pi[B\given X]=0$, with $0$ replaced by some suitable $o(1)$. The message is that if the prior distribution puts very little prior mass on some (sequence of) set(s), then the posterior distributions puts little mass over such set(s). \\
\begin{lem} \label{lem:priorm}
Let $A_n$ be a measurable set such that, if $\veps_n$ verifies $n\veps_n^2\to\infty$,  as $n\to\infty$
\begin{equation}  \label{tr:ratio}
\frac{\Pi[A_n]}{e^{-2n\veps_n^2} \Pi[B_K(\te_0,\veps_n)]} = o(1).
\end{equation}
Then we have, as $n\to\infty$,
\[ E_{\te_0}\Pi[A_n\given X] = o(1). \]
\end{lem}
\begin{proof}[Proof of Lemma \ref{lem:priorm}] 
As a preliminary remark, note that, since $p_\te$ is by definition a density,
\[ E_{\te_0} \left[\frac{p_\te}{p_{\te_0}}(X)\right]= \int_{p_{\te_0}>0}  \frac{p_\te}{p_{\te_0}}(x) 
p_{\te_0}(x) d\mu(x) +0\le \int p_\te(x) d\mu(x) =1. \]
Bayes' formula as in \eqref{bayesfor} for the set $A_n$,  is $\Pi[A_n\given X]=N/D$ with $D= \int  \frac{p_\te}{p_{\te_0}}(X) d\Pi(\te)$. Lemma \ref{lem:tr1} implies, on an event $E_n$ with probability at least $1-(Cn\veps_n^2)^{-1}$, 
\[ D \ge \Pi[ B_K(\te_0,\veps_n) ] e^{-(1+C)n\veps_n^2}.\]
Let us now bound $N/D$ from above by
\[ \frac{N}{D} \le \frac{ e^{-(1+C)n\veps_n^2}}{\Pi[ B_{K}(\te_0,\veps_n)] }
\int_{A_n} \frac{p_\te}{p_{\te_0}}(X)  d\Pi(\te) \1_{E_n} + \1_{E_n^c}, \]
where the bound for the last term is obtained noting that $N/D=\Pi[A_n\given X]\le 1$.
 Taking expectations (first note $\1_{E_n}\le 1$), and invoking first Fubini's theorem and then the preliminary remark, 
 \begin{align*}
 E_{\te_0} \frac{N}{D} & \le \frac{ e^{-(1+C)n\veps_n^2}}{\Pi[ B_{K}(\te_0,\veps_n)] }
\int_{A_n}E_{\te_0}\frac{p_\te}{p_{\te_0}}(X) 
 d\Pi(\te)  + E_{\te_0}\1_{E_n^c} \\
 & \le \frac{ e^{-(1+C)n\veps_n^2}}{\Pi[ B_{K}(\te_0,\veps_n)] }\Pi[A_n]
 + P_{\te_0}\left(E_n^c\right).
 \end{align*}
 The last display goes to $0$ by invoking Lemma \ref{lem:tr1} with $C=1$ and $\veps=\veps_n$, \eqref{tr:ratio} and  $n\veps_n^2\to\infty$. 
\end{proof}

\section{A generic result, first version} \label{sec:ggv1}

Let us start with a brief historical perspective. Doob (1949) \cite{doob} showed that posteriors are (nearly) always consistent in a $\Pi$--almost sure sense, which is interesting but prior--dependent. Schwartz (1965) \cite{Schwartz} proved consistency in the sense of the definition above under some sufficient conditions of existence of certain tests and of enough prior mass around the true $f_0$. Diaconis and Freedman (1986) \cite{df86} exhibited an example of a seemingly natural prior whose posterior distribution is not consistent. Ghosal, Ghosh and van der Vaart (2000) \cite{ggv00}, Shen and Wasserman (2001) \cite{sw01} and Ghosal and van der Vaart (2007) \cite{gvni} gave sufficient conditions for rates of convergence. These references are mostly concerned with nonparametric problems, which along with more precise results on $\sqrt{n}$--functionals will be the main focus of these lectures. For this first Chapter we follow mostly \cite{ggv00, gvni} (also presented in the book \cite{gvbook}). 

We note that although with a somewhat different focus, the theory of PAC--Bayes bounds is another relevant theory in this context that has developed since the end of the 90's: roughly speaking, it is more turned towards machine learning applications (in particular classification and regression) for which one does not necessarily wish to assume much on the statistical model -- and 
where therefore one expects somewhat less precise results, in particular in terms or rates or optimal constants --. Early contributions include works by McAllester \cite{mcallester98} and Catoni \cite{catoni01, catoni07}. We refer to the survey paper \cite{alquier21} for an overview of its applications.   \\

A {\em test} based on observations $X$ is a measurable function $X\to\vphi(X)$ taking values in $\{0,1\}$. \\

Given the statistical model $\cP=\{P_\te,\ \te\in\Theta\}$,  let $\Pi$ be a prior distribution on $(\Theta,\cB)$. Suppose also that $\Theta$ is equipped with a distance $d$ (examples will be given below).
 We denote by $\Theta_n^c=\Theta\setminus\Theta_n$ the complement of $\Theta_n\subset\Theta$.
 The next result, based on \cite{ggv00, gvni}, is referred to as GGV theorem below.
 
\begin{thm}[\re{GGV, version with tests}] \label{thm:ggvt}
Let $(\veps_n)$ be a sequence with $n\veps_n^2\to\infty$ as $n\to\infty$. \\
Suppose there exist $C>0$ and measurable sets $\Theta_n\subset \Theta$ such that
\begin{enumerate}
\item[i)] there exist tests $\psi_n=\psi_n(X)$ with
\[ E_{\te_0}\psi_n =o(1),\qquad \sup_{\te\in\Theta_n:\ d(\te,\te_0) >M\veps_n}
E_\te(1-\psi_n)\le e^{-(C+4)n\veps_n^2},
\]
\item[ii)] \[ \Pi\left[\Theta_n^c\right] \le e^{-n\veps_n^2(C+4)},\]
\item[iii)] \[ \Pi[B_K(\te_0,\veps_n)]\ge e^{-Cn\veps_n^2}. \]
\end{enumerate}
Then for large enough $M$, the posterior distribution converges at rate $M\veps_n$ towards $f_0$:  as $n\to\infty$,
\[ E_{\te_0}\Pi[\{\te:\ d(\te,\te_0) > M\veps_n\}\given X ] =o(1).\]
\end{thm}

Let us briefly comment on the conditions. Assumption iii) is natural: there should be enough prior mass around the true $\te_0$. Indeed, recall by Lemma \ref{lem:priorm} above that if the prior mass of a set is too small, its posterior mass will be too: having a too small prior probability of the KL--neighborhood would mean its posterior mass is vanishing, so there could not be convergence at rate $\veps_n$, at least in terms of the `divergence' defined by the KL--type  neighborhood.

 Assumption ii) allows to work on a subset $\Theta_n$, so it gives some flexibility, especially if $\Theta$ is a `large' set: indeed, combining ii) with iii) 
\[ \frac{\Pi[\Theta_n^c]}{\Pi[B_{K}(\te_0,\veps_n)]} \le e^{-4n\veps_n^2}, \]
which leads to $E_{\te_0}\Pi[\Theta_n^c\given X]=o(1)$ using Lemma \ref{lem:priorm}. 

Assumption i) is so far a little more mysterious. It can be seen more as a `meta--condition', that makes the proof of the result quite quick. We will see below another version of the result, where i) is replaced by another, more interpretable, condition. 
Let us just note that the distance $d$ in i) is the same as in the result: one needs to find tests with respect to this distance.\\

{\em Important point about uniformity.} In order to be able to compare to usual optimality results in the minimax sense, it is important to verify the above not only for a single $\te_0$, but rather for all $\theta_0$ in a certain set. For instance, by verifying that the conditions of Theorem \ref{thm:ggvt} hold uniformly over $\te_0\in S(\be,L)$, for some set $S(\be,L)$ (e.g. a Sobolev ball), one gets
\[ \sup_{\te_0\in S(\be,L)} E_{\te_0} \Pi[\{\te:\ d(\te,\te_0)\ge M\veps_n\}\given X ] = o(1). \]
In order not to surcharge notation, we sometimes omit the supremum in stating the results, but it can be verified that they hold uniformly over the relevant classes depending on the context.\\

\begin{proof}[Proof of Theorem \ref{thm:ggvt}]
Since $E_{\te_0}\Pi[\Theta_n^c\given X]=o(1)$ as noted above, is is enough to 
prove that $E_{\te_0}\Pi[\cC_n\given X]=o(1)$, where 
\[ \cC_n = \{\te\in\Theta_n,\ d(\te,\te_0)\ge M\veps_n\}. \]
Using the tests $\psi_n$ from Assumption i), one decomposes
\[ \Pi[\cC_n\given X] = \Pi[\cC_n\given X] \psi_n + \Pi[\cC_n\given X](1-\psi_n). \]
With $\Pi[\cC_n\given X]\le 1$, one gets $E_{\te_0} \Pi[\cC_n\given X] \psi_n
\le E_{\te_0}\psi_n=o(1)$ thanks to i). For the second term, we write, recalling $\psi_n=\psi_n(X_1,\ldots,X_n)=\psi_n(X)$ is a function of the data,
\[ \Pi[\cC_n\given X](1-\psi_n) = \frac{\int_{\cC_n}   \frac{p_\te}{p_{\te_0}}(X) (1-\psi_n(X) d\Pi(\te)}{\int \frac{p_\te}{p_{\te_0}}(X) d\Pi(\te)}=:\frac{N}{D}.\]
In order to bound the denominator from below, let us introduce the event
\[ \cB_n=\left\{ \int   \frac{p_\te}{p_{\te_0}}(X) d\Pi(\te) \ge \Pi[ B_{K}(\te_0,\veps_n) ]
e^{-2n\veps_n^2} \right\}.\]
Lemma \ref{lem:tr1} tells us that $P_{\te_0}[\cB_n]\ge 1-(n\veps_n^2)=1-o(1)$ using $n\veps_n^2\to\infty$. Deduce
\[ \frac{N}{D} \le \frac{e^{2n\veps_n^2}}{\Pi[B_K(\te_0,\veps_n) ]}
\int_{C_n} \frac{p_\te}{p_{\te_0}}(X) (1-\psi_n(X)) d\Pi(\te) +\1_{\cB_n^c}.
\]
Observe, arguing as in the proof of Lemma \ref{lem:priorm}, 
\begin{align*} E_{\te_0}\left[ \frac{p_\te}{p_{\te_0}}(X) (1-\psi_n)(X) \right] 
& =\int_{p_{\te_0}>0} \frac{p_\te}{p_{\te_0}}(x) (1-\psi_n(x)) 
 p_{\te_0}(x) d\mu(x) \\
& \le \int  (1-\psi_n(x)) p_\te(x) d\mu(x) 
= E_\te [1-\psi_n(X)].
\end{align*}
By taking expectations and using Fubini's theorem,
\begin{align*}
 E_{\te_0}\frac{N}{D} & \le \frac{e^{2n\veps_n^2}}{\Pi[B_{K}(\te_0,\veps_n) ]}
\int_{C_n}  E_{\te_0}\left[\frac{p_\te}{p_{\te_0}}(X) (1-\psi_n(X) \right] d\Pi(\te) + P_{\te_0}[\cB_n^c] \\
& \le e^{(C+2)n\veps_n^2}
\int_{C_n}  E_{\te}\left[(1-\psi_n(X) \right] d\Pi(\te) + P_{\te_0}[\cB_n^c]\\
& \le e^{(C+2)n\veps_n^2} e^{-(C+4)n\veps_n^2} + P_{\te_0}[\cB_n^c] \le 
e^{-2n\veps_n^2} + o(1)=o(1). \qquad \qed
\end{align*}
\end{proof}

\section{Testing and entropy, a second generic result}

In Theorem \ref{thm:ggvt}, the testing condition i) requires to be able to test a `point' $\te_0$ versus the `complement of a ball' $\{\te\in\Theta_n,\ d(\te,\te_0)>M\veps_n\}$. The latter set has not a very simple structure (one would prefer a ball for instance instead of a complement!). Let us see how one can simplify this through combining tests of `point' versus `ball'.

\bfr
\sbl{Testing condition (T).} \label{testt} Suppose one can find constants $K>0$ and $a\in(0,1)$ such that  for any $\veps>0$, if $\te_0,\te_1\in\Theta$ are such that $d(\te_0,\te_1)>\veps$, then there exist tests $\vphi_n$ with
\begin{align}
E_{\te_0}\vphi_n & \le e^{-Kn\veps^2} \\
\sup_{\te:\ d(\te,\te_1)<a\veps} E_\te(1-\vphi_n) & \le  e^{-Kn\veps^2}.
\end{align}
\efr
This condition is in fact always verified for certain distances and models. The next result, due to Lucien Le Cam and Lucien Birgé, proves that (T) holds in density estimation for two specific distances.  Regression-type models are considered in Appendix \ref{app:testing}  together with $L^2$--type distances.
\begin{thm} \label{thm:tests}
The testing condition (T) is always verified in the density estimation model for $d$ the $L^1$--distance or the Hellinger distance $h$ (see Definition \ref{defhel}).
\end{thm}
We prove this result in Appendix \ref{app:testing} for the $L^1$--distance. For the Hellinger distance, we refer to \cite{gvbook}, Proposition D.8.
 
\begin{definition}
The $\veps$--covering number of a set $\cE$ for the distance $d$, denoted $N(\veps,\cE,d)$, is the minimal number of $d$--balls of radius $\veps$ necessary to cover $\cE$.
\end{definition}
The entropy of a set measures its `complexity'/`size'. 
Let us give a few examples
\begin{itemize}
\item If $\cE=[0,1]$ and $d(x,y)=|x-y|$, then $N(\veps,\cE,d)$ is of order $1/\veps$.
\item If $\cE$ is the unit ball in $\RR^k$
\[ B(0,1)=\left\{\te\in\RR^k,\ \|\te\|_2^2:=\sum_{i=1}^k\te_i^2 \le 1\right\}, \]
then $N(\veps,\cE,\|\cdot\|_2)$ is of order $\veps^{-k}$. Note that this number grows exponentially with the dimension $k$. This classical result is recalled in Appendix \ref{app:testing}. 
\item As will be seen in the sequel, there are many results available for balls in various function spaces (histograms, Sobolev or H\"older balls etc.). 
\end{itemize}

\begin{lem} \label{lem:testent}
Suppose that the testing condition (T) holds (with constants $a,K$) for a distance $d$ on $\Theta$ and that, for a sequence of measurable sets $\Theta_n$, and a sequence $(\veps_n)$ with $n\veps_n^2\ge 1$,
\[ \log N(\veps_n,\Theta_n,d) \le  Dn\veps_n^2. \]
Then for a given $c>0$ there exists $M=M(a,K,c)$ large enough and tests
 $\psi_n=\psi_n(X)$ such that 
\[ E_{\te_0}\psi_n =o(1),\qquad \sup_{\te\in\Theta_n:\ d(\te,\te_0)>M\veps_n}
E_\te(1-\psi_n)\le e^{-cn\veps_n^2}. \]
\end{lem} 
\begin{proof}
Let us consider the set 
\[ G_n = \{\te\in\Theta_n,\ d(\te,\te_0)>4M\veps_n\} \]
and partition it in `shells' $\cC_j$ as follows
\[ G_n = \bigcup_{j\ge 1} \, \{\te\in\Theta_n,\ 4Mj\veps_n<d(\te,\te_0)\le 4M(j+1)\veps_n\}
=  \bigcup_{j\ge 1}\, \cC_j. \]
Now let us cover each shell $\cC_j$ by balls.
\begin{itemize}
\item Let $\veps=jM\veps_n$ and consider a minimal covering of $\cC_j$ by balls $B_{ij}$ of radius $a\veps$, for $a\in(0,1)$ the constant appearing in condition (T): by definition of the covering number, the number of these balls is $N(a\veps,\cC_j,d)$.
\item Let us denote by $g_{ij}$ the centers of the balls of the previous covering. Since $B_{ij}$ must intersect $\cC_j$ (otherwise it could be removed from the covering which would then not be minimal), we have, as $a\in(0,1)$,
\[ d(\te_0,g_{ij})\ge 4Mj\veps_n-2a\veps=4Mj\veps_n-2ajM\veps_n\ge
2Mj\veps_n>\veps. \]
So, for each $g_{ij}$ there exists a test $\vphi_{ij}$ satisfying the properties given by condition (T). 
\item On the other hand, we also have for any $j\ge 1$, if $M\ge a^{-1}$,
\begin{align*}
N(a\veps,\cC_j,d) = N(ajM\veps_n,\cC_j,d) 
& \le N(ajM\veps_n,\Theta_n,d) \\
& \le N(\veps_n,\Theta_n,d).
\end{align*}
\item Let us now combine the just--contructed tests $\vphi_{ij}$ by setting
\[ \psi :=  \sup_{i,j\ge 1} \vphi_{ij}. \]
\end{itemize}
Let us now verify that the test $\psi$ satisfies the desired properties. First, recalling $\veps=jM\veps_n$,
\begin{align*}
E_{\te_0}\psi \le E_{\te_0}\left(\sum_{i,j}\vphi_{ij} \right) 
&\le \sum_{j\ge 1} \sum_{i} E_{\te_0}\vphi_{ij}
\le \sum_{j\ge 1} N(\veps_n,\Theta_n,d) e^{-Kn\veps^2}\\
& \le \sum_{j\ge 1} N(\veps_n,\Theta_n,d) e^{-KnM^2\veps_n^2 j^2}
 \le N(\veps_n,\Theta_n,d) \frac{e^{-KnM^2\veps_n^2}}{1-e^{-KnM^2\veps_n^2}}\\
& \le Ce^{cn\veps_n^2-KnM^2\veps_n^2}
\end{align*}
which is $o(1)$ if $c<KM^2/2$ say. On the other hand, uniformly for $\te\in\Theta_n$ such that $d(\te,\te_0)>4M\veps_n$,
\begin{align*}
E_\te(1-\psi) & \le \sup_{j,i}\sup_{\te\in B_{ij}} E_\te(1-\psi)
 \le \sup_{j,i} \sup_{\te\in B_{ij}} E_\te(1-\vphi_{ij}) \\
& \le  \sup_{j,i} e^{-Kn(jM\veps_n)^2}  
 \le e^{-Kn(M\veps_n)^2} \le e^{-cn\veps_n^2}
\end{align*}
as soon as $KM^2>c$, which concludes the proof.
\end{proof}

We now state a generic result with the testing condition replaced by an entropy condition, and where we also allow for possibly different rates for the sieve and prior mass conditions (which go together, recalling they originate from applying Lemma \ref{lem:priorm}) and the entropy condition.

\begin{thm}[\re{GGV, entropy version}] \label{thm:ggve}$\ $
Let $(\oli\veps_n,\uli\veps_n)$ 
be sequences with $n(\oli\veps_n^2\wedge \uli\veps_n^2)\to\infty$ as $n\to\infty$. 
Suppose $d$ is a distance on $\Theta$ such that the testing condition \sbl{(T)} holds with constants $a, K>0$.  Assume there exist $C,D>0$ and measurable sets $\Theta_n\subset \Theta$ such that, for $B_K$ as in \eqref{bkln},
\begin{enumerate}
\item[i)] $ \log N(\oli\veps_n,\Theta_n,d)\le Dn\oli\veps_n^2$,
\item[ii)] $ \Pi[\Theta_n^c] \le e^{-n\uli\veps_n^2(C+4)}, $
\item[iii)] $ \Pi[B_K(\te_0,\uli\veps_n)]\ge e^{-Cn\uli\veps_n^2}. $
\end{enumerate}
Set $\veps_n=\oli\veps_n\vee \uli\veps_n$. Then for $M=M(a,K,C,D)$ large enough, the posterior distribution converges at rate $M\veps_n$ towards $f_0$:  as $n\to\infty$,
\[ E_{\te_0}\Pi[\{\te:\ d(\te,\te_0)\ge M\veps_n\}\given X ] =o(1).\]
\end{thm}

\begin{proof}
We start by noting that, for given $n\ge 1$, the maps
\[ \veps\to \log N(\veps,\Theta_n,d),\qquad \veps\to n\veps^2 \]
are respectively non--increasing and increasing: for the first, note that if $\veps'>\veps$, a covering of $\Theta_n$ with $\veps$--balls gives rise to a covering with $\veps'$--balls using the same centers. 
Combining this monotonicity property with the entropy condition i), one now can apply Lemma \ref{lem:testent} with $\veps_n=\oli\veps_n\vee \uli\veps_n$ and $c=C+4$. Indeed, the entropy condition required is also valid with the slower rate $\veps_n$ which gives for large enough $M$ the existence of tests $\psi_n$ with
\[ E_{\te_0}\psi_n =o(1),\qquad \sup_{\te\in\Theta_n:\ d(\te,\te_0)>M\veps_n}
E_\te(1-\psi_n)\le e^{-cn\veps_n^2}, \]
that is, the first condition of Theorem \ref{thm:ggvt}. 
Next by combining ii) and iii) one obtains $E_{\te_0}\Pi[\Theta_n^c\given X]=o(1)$ by using Lemma \ref{lem:priorm}, that requires $n\uli\veps_n^2\to\infty$. The prior mass condition iii)  is automatically verified if one replaces $\uli\veps_n$ by $\veps_n$: indeed by doing so the prior mass does not decrease and the exponential term decreases.  

Now working on the set $\cC_n=\{\te\in\Theta_n,\ d(\te,\te_0)\ge M\veps_n\}$, one can follow line by line the proof of Theorem \ref{thm:ggvt}, which concludes the proof.
\end{proof}

\section{An example for which explicit computations are possible} \label{sec:explicit}

{\em Model.} Consider the Gaussian sequence model as above, with $X=(X_1,\ldots)$ and, for $k\ge 1$,
\[ X_k = \te_k + \veps_k/\rn, \]
where the distribution of $X$ given $\te$ is given by 
\[ P_{\te}^{(n)} = \bigotimes_{k=1}^{\infty} \cN(\te_k,1/n). \]
{\em Prior.} Suppose that as a prior $\Pi$ on $\te$s one takes, for some $\al>0$,
\begin{equation} \label{gsm_prior}
 \Pi = \Pi_\al = \bigotimes_{k=1}^{\infty} \cN(0,\sigma_k^2),\qquad \text{with }
\sigma_k^2:=k^{-1-2\al}. 
\end{equation}
If  working with infinite product distributions looks intimidating at fist, one can just consider truncated versions of both model and prior at $k=n$. All what follows can then be computed in $\RR^n$, and the statistical interpretation remains similar, noticing that the bias induced from not considering the coordinates $k\ge n+1$ is at most $\sum_{k\ge n+1} f_{0,k}^2\leqa n^{-2\be}$ which can be seen to be always negligible compared to the rate $\veps_n$ obtained below. \\
 
{\em Posterior distribution.} Bayes' formula gives that the posterior distribution of $\te_k$ given $X$ only depends on $X_k$ and 
\[ \cL(\te_k\given X_k) = 
\cN\left( \frac{n}{n+\sigma_k^{-2}}X_k,\frac{1}{n+\sigma_k^{-2}} \right).
\]
Furthermore, the complete posterior distribution of $\te$ is
\[ \Pi[\cdot\given X] = 
\bigotimes_{k=1}^{\infty} \cN\left( \frac{n}{n+\sigma_k^{-2}}X_k,\frac{1}{n+\sigma_k^{-2}} \right).
 \]

{\em The true $\te_0$.} We assume the following smoothness condition, for some $\be, L>0$,
\begin{equation} \label{ex1_smooth}
 \te_0 \in S_\be(L):=\left\{\te\in\ell^2:\ \sum_{k=1}^\infty k^{2\be}\te_k^2\le L\right\}.\\
\end{equation}

{\em Posterior convergence under $\te_0$.} Considering a frequentist analysis of the posterior with a fixed truth $\te_0$, it is natural to wonder whether $\Pi[\cdot\given X]$ is consistent at $\te_0$ and if so at which rate it converges for, say, the $\|\cdot\|_2^2$ loss, given by (setting $\|\cdot\|=\|\cdot\|_2$) 
\[ \|\te-\te'\|^2 = \sum_{k\ge 1} (\te_k-\te_k')^2. \]

Let us consider the posterior mean, that is the sequence of means over coefficients
\[ \oli{\theta}(X)=\left(\int \te_k d\Pi(\te_k\given X_k)\right)_{k\ge 1} =
\left( \frac{nX_k}{n+\sigma_k^{-2}} \right)_{k\ge 1}. \]

\paragraph{First step: reduction to a mean/variance problem.}

Using Markov's inequality,
\begin{align*}
\Pi[\|\te-\te_0\|>\veps_n\given X] & \le
\frac{1}{\veps_n^2} \int \|\te-\te_0\|^2 d\Pi(\te\given X) \\
&\le \frac{1}{\veps_n^2} \sum_{k\ge 1} 
\int (\te_k-\te_{0,k})^2 d\Pi(\te\given X).
\end{align*}
The ``bias--variance decomposition" is (observe that the crossed term is zero because we have centered around the posterior mean)
\begin{align*}
 \int (\te_k-\te_{0,k})^2 d\Pi(\te\given X)
& =\int (\te_k-\oli{\te}_k)^2 d\Pi(\te\given X)+
\int (\oli{\te}_k-\te_{0,k})^2 d\Pi(\te\given X) \\
& = \int (\te_k-\oli{\te}_k)^2 d\Pi(\te\given X)+
 (\oli{\te}_k-\te_{0,k})^2.
\end{align*}
as the last term does not depend on $\theta$. Note that the first term in the last sum is Var$(\te_k\given X_k)$. In order to show that, for some $\veps_n=o(1)$ to be determined, 
\[ E_{\te_0}\Pi[\|\te-\te_0\|>\veps_n\given X] =o(1),
\]
it is enough to study the behaviour of the two terms
\begin{align*}
(a) & := \sum_{k\ge 1} E_{\te_0}\text{Var}(\te_k\given X_k)\\
(b) & := \sum_{k\ge 1} E_{\te_0}(\oli\te_k-\te_{0,k})^2.
\end{align*}

\paragraph{Study of the terms (a) and (b).} 
For both terms, we distinguish the regimes $\sigma_k^2< 1/n$ and $\sigma_k^2\ge 1/n$, or equivalently $k> N_\al$ and $k\le N_\al$ respectively, with 
\[ N_\al:=\lfloor n^{\frac{1}{1+2\al}} \rfloor. \]
We can now use the bounds
\begin{align*}
(a) & \le \sum_{k\ge 1} \frac{1}{n+\sigma_k^{-2}} \\
& \le \sum_{k\le N_\al} \frac{1}{n} +  
 \sum_{k> N_\al} \sigma_k^{2}\le \frac{N_\al}{n}
 + CN_\al^{-2\al} \leqa n^{-\frac{2\al}{2\al+1}}.
\end{align*}
For the second term, by using the explicit expression of $\oli\te_k$, a little computation shows
\begin{align*}
E_{\te_0}(\oli\te_k-\te_{0,k})^2
& = \frac{\sigma_k^{-4}}{(n+\sigma_k^{-2})^2}\te_{0,k}^2
+\frac{n}{(n+\sigma_k^{-2})^2}\\
& = \qquad (I) \qquad +\qquad (II).
\end{align*}
The term (II) is the easiest to bound. Its sum is bounded by
\[ \sum_{k\ge 1} \frac{n}{n+\sigma_k^{-2}} \frac{1}{n+\sigma_k^{-2}}\le 
 \sum_{k\ge 1} \frac{1}{n+\sigma_k^{-2}}\leqa n^{-\frac{2\al}{2\al+1}}, \]
by the same reasoning as before. The sum of the term (I) is bounded by, with $a\vee b=\max(a,b)$,
\begin{align*}
\sum_{k\le N_\al} \frac{k^{2+4\al}}{n^2}\te_{0,k}^2
+ \sum_{k> N_\al}  \te_{0,k}^2 & 
\le  n^{-2}\sum_{k\le N_\al} k^{2+4\al-2\be}k^{2\be}\te_{0,k}^2+\sum_{k> N_\al}  k^{-2\be} k^{2\be}\te_{0,k}^2\\
& \le n^{-2} \sum_{k\le N_\al} N_\al^{(2+4\al-2\be)\vee 0}k^{2\be}\te_{0,k}^2 + N_{\al}^{-2\be} L\\
& \le n^{-2} (N_\al^{2+4\al-2\be}\vee 1) L+N_{\al}^{-2\be}L
\leqa (n^{-2}+N_{\al}^{-2\be}) L.
\end{align*}
This last term is at most of order $n^{-2\be/(2\al+1)}+n^{-2}$.
Conclude that one can take  $\veps_n$ to be
\[ \veps_n = M_n n^{-\frac{\al\wedge \be}{2\al+1}}, \]
where $M_n$ is an arbitrary sequence going to infinity. This rate is the fastest for the choice $\al=\be$, but this would then require the statistician to know the smoothness $\be$. The question of {\em adaptation} to the smoothness is discussed in Chapter \ref{chap:ada1}.

\section*{Exercises}
\addcontentsline{toc}{section}{{\em Exercises}}

\begin{enumerate}
\item In the fundamental model $\cP=\{\cN(\te,1)^{\otimes n},\ \te\in\RR\}$ with Gaussian prior $\Pi=\cN(0,\sigma^2)$ on $\theta$,
\begin{enumerate}
\item Show that the posterior distribution $\Pi[\cdot\given X^{(n)}]$ is given by \eqref{mfpost}.
\item Using the explicit form of the posterior, show that, in probability under $(X_1,\ldots,X_n)\sim P_{\te_0}^{\otimes n}$ (i.e. in the frequentist sense), the posterior $\Pi[\cdot\given X^{(n)}]$ converges at rate $M_n/\rn$ towards $\te_0$ in terms of the distance $d(\te,\te')=|\te-\te'|$ on $\RR$, where $M_n$ is an arbitrary sequence going to infinity.
\item Define a $(1-\al)$--credible interval by using the quantiles of the posterior distributions at levels $\al/2=.05$ and $1-\al/2$ respectively. What is the center of this interval? Show that it is of minimum width among intervals of credibility at least $1-\al$. 
\end{enumerate}
\item In the nonparametric example of Section \ref{sec:explicit} with a prior $\Pi=\Pi_\al$ as in \eqref{gsm_prior}, show the following lower-bound-type result: for any $\te_0\in\ell^2$
\[ E_{\te_0} \int \|\te-\te_0\|_2^2d\Pi(\te\given X) \geqa n^{-\frac{\al}{2\al+1}}. \]
Give an interpretation of this result.
\end{enumerate}

\chapter{Rates II and first examples} \label{chap:rate1}

As in the previous chapter we work with  $\cP=\{ P_{\te}^{(n)},\ \te\in \Theta\}$ a dominated model with observations $X=X^{(n)}$. That is $dP_\te^{(n)}=p_\te^{(n)} d\mu^{(n)}$ for any $\te\in\Theta$, for a dominating measure $\mu^{(n)}$ independent of $\te$. When there is no ambiguity we drop the dependence in $n$ and simply write $p_\te=p_{\te}^{(n)}$ and $\mu=\mu^{(n)}$. 

\section{Extensions}

The aim when stating the previous two theorems on posterior concentration was to give simple -- yet general and already fairly broadly applicable -- statements and proofs. These results can in turn be refined in a number of ways. Many refinements are described in the book \cite{gvbook}. We only briefly mention a few
\begin{enumerate}
\item {\em Coupling numerator and denominator when studying Bayes' formula.} 
The previous formulations of the convergence theorem treat denominator and numerator separately. This can be suboptimal, especially when the parameter space is large or unbounded: this situation arises for instance in high-dimensional models, discussed in more details in Chapter \ref{chap:ada2}.

\item {\em Other notions of entropy.} It is already clear from the proof of Lemma \ref{lem:testent} that upper-bounds are possibly generous there, and indeed one can provide more precise conditions. Instead of looking at a `global' entropy, one can also look at a more `local' versions of the entropy.

\item {\em Other distances.} As such the application of the GGV theorem is limited to distances for which certain tests exist. Although there are often natural distances for which such tests exist (e.g. $L^1$-- or Hellinger distances for independent data, $L^2$--distances for Gaussian regression), it may be difficult (or even impossible) to find such tests for other distances of interest. One example relevant in applications is the supremum norm distance between functions, see Chapter \ref{chap:bvm2}.

\end{enumerate}

\section{Fractional posteriors}

 A popular generalisation (in particular in machine learning and PAC--Bayesian theory) of the posterior distribution is the so--called {\em $\alpha$--posterior}, where given a prior $\Pi$ on $\theta$, for $\al>0$ one defines the distribution, for every measurable $B$,
\[ \Pi_\al[B\given X]=\frac{\int_B L_{n,\al}(\te)d\Pi(\te)}{\int L_{n,\al}(\te)d\Pi(\te)},\quad L_{n,\al}(\te)=\left[p_\te^{(n)}(X)\right]^\alpha.\]
In the limiting case $\al\to 0$, one simply obtains the prior distribution itself (so, the data $X$ plays no role), while if one lets $\al_n\to\infty$ one gets close to ``maximum likelihood". 

In the sequel we consider the case where $0<\al<1$, which tempers the influence of the data in the obtained distribution. A technical advantage of working with an $\alpha$--posterior, $\al<1$, is that convergence rate results can be obtained under prior--mass conditions only, without requiring entropy--type bounds, as in Theorem \ref{alpost} below. Results of this type on rates date back to \cite{tzhang06} (see \cite{ltcr23} for the present version; and \cite{walkerhjort01} for an earlier result on consistency). A drawback is that $L_{n,\al}(X)$ is not a likelihood anymore, so the original Bayesian interpretation is lost: optimality properties related to the use of the likelihood may then be lost. Typically, statistical efficiency is lost and `credible' sets from the $\al$--posterior will also often be larger for $\al<1$ than in the posterior case $\al=1$, see Chapter \ref{chap:bvm1} for more on this and possible remedies.

For any $\al\in(0,1)$, the $\al$--R\'enyi divergence between distributions $P,Q$ having densities $p$ and $q$ with respect to $\mu$ is defined as 
\begin{align*}
		D_\alpha(P, Q) = -\frac{1}{1-\alpha}\log \left( \int p^\alpha q^{1-\alpha}d\mu \right).
	\end{align*}
It is related to the standard $L^1$--distance via Pinsker's inequality (see e.g. \cite{vanerven}, Theorem 31)
\[ D_\alpha(P, Q) \ge \al \|P-Q\|_1^2/2. \]
Let us recall the definition \eqref{bkln} of the Kullback--Leibler type neighborhood $B_K(\te,\veps)$ of $\te_0\in\Theta$, and that in the next result we use the simplified notation $P_\te=P_\te^{(n)}$, $P_{\te_0}=P_{\te_0}^{(n)}$. 
  
\begin{thm}\label{alpost} For any non negative sequence $\eps_n$ and $0<\al<1$ such that $n\al\eps_n^2 \rightarrow \infty$ and 
		\begin{align}\label{equation_thm_1}
			\Pi(B_K(\te_0, \eps_n)) \geq e^{-n\al\eps_n^2},
		\end{align}
		there exists a constant $C>0$ independent of $\al$ such that as $n\to\infty$, for $P_0=P_{\te_0}^{(n)}$,
		\begin{align*}
			\Pi_{\al}\left( \te: \: \frac{1}{n}D_{\al}(P_\te, P_{\te_0}) \geq C \frac{\al\eps_n^2}{1-\al}  \, \given\, X\right) = o_{P_0}(1),
		\end{align*}	
where the term $o_{P_0}(1)$ is independent of $\al$.
	\end{thm}

In Theorem \ref{alpost}, one may choose $\al=\al_n$ that possibly goes to $0$ or $1$. For $\al\to 1$, the constant in front of the rate blows up, which suggests that the assumptions do not suffice to get a rate of order $\veps_n$ (this is indeed the case, see \cite{bsw99} for a counterexample showing an inconsistent posterior under a prior mass condition only, whereas in the same setup the previous result yields rate $n^{-1/3}$ for the $\al$--posterior and $\al$ bounded away from $1$). 

As written the result is in terms of the normalised divergence $D_\al(P_{\te}, P_{\te_0})/n$ which still depends both on $\al_n$ and $n$. In case $P_\te$'s are products $P_\te=Q_\te^{\otimes n}$,  the (immediate) tensorisation property of $D_\al$ combined with Pinsker's inequality leads to
\[ D_\al(P_{\te}, P_{\te_0}) = n D_\al(Q_{\te}, Q_{\te_0})
\ge n\al \|Q_{\te} -  Q_{\te_0} \|_1^2/2.
\]
Therefore in the iid setting Theorem \ref{alpost} automatically implies convergence of the posterior in terms of the {\em squared--}$L^1$ distance at rate $\veps_n^2/(1-\al_n)$, for any $\veps_n$ that verifies the stated prior mass condition. 

Note the $\al_n$ inside the exponential in the prior mass condition: this makes it quite different from the related   condition in the GGV theorem in the regime when $\al_n$ tends to $0$. More precisely, one then typically obtains a rate similar to the one obtained from applying the GGV theorem, but with $n$ replaced by $n'=n\al_n$ (precisely due to this extra $\al_n$ factor in the prior mass condition). 
For instance, nonparametric squared rates $n^{-2\be/(2\be+d)}$ typically become $(n\al_n)^{-2\be/(2\be+d)}$. This only changes the constant if $\al_n$ is bounded away from $0$, but otherwise the rate is  slower. \\

\begin{proof}[Proof of Theorem \ref{alpost}]
By Lemma \ref{lem:mino_den_posterior}, on a subset $C_n$ of $P_0$-probability at least $1-\frac{1}{n\eps_n^2}$,  for any measurable set $A \subset \Theta$,
		\begin{equation}\label{eq:Bayes_formula}  
		\begin{split}
			E_0 \Pi_{\al}(A|X)  = E_0 \frac{\int_{A} \frac{p_\te^\al}{p_{\te_0}^\al}(X) d\Pi(\te)}{\int  \frac{p_\te^\al}{p_{\te_0}^\al}(X) d\Pi(\te)}
			& \leq E_0 \frac{\int_{A} \frac{p_\te^\al}{p_{\te_0}^\al}(X) d\Pi(\te)}{\Pi(B_K(\te_0, \eps_n)) e^{-2{\al}n\eps_n^2}}1_{C_n} + P_0(C_n^c) \\
			& =  \frac{\int_{A} \int p_\te(x)^{\al} p_{\te_0}(x)^{1-\al} d\mu(x)d\Pi(\te)}{\Pi(B_K(\te_0, \eps_n)) e^{-2{\al}n\eps_n^2}} +o(1),
		\end{split}		
		\end{equation}
		where the last equality follows from Fubini's theorem. Set
		\begin{align*}
			A=A_n&:= \left\{\te: \: \int p_\te(x)^{\al} p_{\te_0}(x)^{1-\al} d\mu(x) \leq e^{- 4n\al\eps_n^2} \right\} \\
			&= \left\{\te: \: -\frac{1}{n(1-\al)}\log(\int p_\te(x)^{\al} p_{\te_0}(x)^{1-\al} d\mu(x)) \geq 4\frac{\al\eps_n^2}{1-\al}\right\}\\
			& = \left\{\te: \: \frac{1}{n} D_{\al}(p_\te, p_{\te_0}) \geq 4\frac{\al\eps_n^2}{1-\al}\right\}.
		\end{align*}		
Substituting $A_n$ into the second-last display and using the prior mass condition \eqref{equation_thm_1} yields
		\begin{align*}
			E_0 \Pi_{\al}(A_n \given X) 
			& \leq \frac{\int_{A_n} e^{- 4n\al\eps_n^2} d\Pi(\te)}{\Pi(B_K(\te_0, \eps_n)) e^{-2{\al}n\eps_n^2}} +o(1)  \leq e^{- n\al\eps_n^2}  + o(1) = o(1),
		\end{align*}
		since $n\al\eps_n^2 \rightarrow \infty$.

\end{proof}

\begin{lem}\label{lem:mino_den_posterior} For any distribution $\Pi$ on $\Theta$, any $C, \eps>0$ and $0 < \alpha \leq 1$, with $P_{0}$-probability at least $1-(C^2n\eps^2)^{-1}$, we have
		\begin{align*}
			\int \frac{p_\te(X)^\alpha}{p_{\te_0}(X)^\alpha}d\Pi(\te) \geq \Pi(B_K(\te_0, \eps))e^{-\alpha(C+1)n\eps^2}.
		\end{align*}
	\end{lem}

	\begin{proof} The proof is (almost) the same as that of Lemma \ref{lem:tr1}. 
		 Let $B:=B_K(\te_0, \eps)$. 
		Suppose $\Pi(B)>0$ (otherwise the result is immediate), and denote by $\bar{\Pi} = \Pi(\cdot \cap B)/\Pi(B)$.
		One bounds from below
		\begin{align*}\label{eq:ELBO_lb}
			\int \frac{p_\te^\alpha}{p_{\te_0}^\alpha}(X)d\Pi(\te) \geq \int_{B} \frac{p_\te^\alpha}{p_{\te_0}^\alpha}(X) d\Pi(\te) = \Pi(B) \int \frac{p_\te^\alpha}{p_{\te_0}^\alpha}(X)d\bar{\Pi}(\te). 
		\end{align*}
		Since $\bar{\Pi}$ is a probability measure on $\Theta$, Jensen's inequality applied to the logarithm gives
		\begin{align*}
				\log\left(\int \frac{p_\te^\alpha}{p_{\te_0}^\alpha}(X) d\bar{\Pi}(\te)\right) 
				& \geq  \alpha  \int \log\left(\frac{p_\te}{p_{\te_0}}(X)\right)d\bar{\Pi}(\te) \\
& \ge - \al Z -\al \int_B K(P_{\te_0},P_\te)d\bar{\Pi}(\te)\ge 
-\al Z -\al n\veps^2,			
		\end{align*}	
with the random variable $Z := \int \left[ \log\frac{p_\te}{p_{\te_0}}(X) - K(P_{\te_0},P_\te)\right] d\bar{\Pi}(\te)$. In the proof of Lemma \ref{lem:tr1}, we have shown that the event $\{|Z|\le Cn\veps^2\}$ has probability at least $1-1/(C^2n\veps^2)$. On this event  the first display of the proof is larger than $\Pi(B)e^{-\al(C+1)n\veps^2}$, which concludes the proof.
	\end{proof}

\section{Lower bounds} \label{sec:lb}

The following definition mirrors the one given for an upper-bound rate in Chapter \ref{chap:intro}.  

\begin{definition} \label{def-lb}
For $d$ a distance on the parameter set $\Theta$, we say that $\zeta_n$ is a \sbl{lower bound} for the posterior
 $\Pi[\cdot\given X]$ contraction rate, in terms of the distance $d$, if for $X=X^{(n)}$, as $n\to\infty$,
\[ \Pi\left[\, \{\te:\ d(\te,\te_0) \le \zeta_n\} \given X\right] \to 0,\]
in probability under $P_{\te_0}$. 
\end{definition}

The interpretation is that if one looks with a magnifying glass `too close' to a given point $\te_0$ then asymptotically there is no posterior mass around it. The definition may look surprising at first since it may sound counterintuitive that a converging posterior  puts no mass asymptotically on small balls  around $\te_0$. However, this just reflects that $\zeta_n$ is too fast a scaling to capture mass asymptotically. Imagine for instance a situation where the posterior equals a normal variable of variance $v_n$ with $v_n\asymp 1/n$. Then any ball of radius $\zeta_n=o(1/\sqrt{n})$ receives vanishing mass asymptotically, see Exercises. 

A simple yet quite efficient way to show lower bound results for posteriors, introduced in \cite{ic08}, is to apply Lemma \ref{lem:priorm} in combination with the choice of set $A_n=\{\te:\ d(\te,\te_0)\le \zeta_n \}$. This typically handles situations where there is a lack of prior  mass around $\te_0$ at scale $\zeta_n$.

\section{Random Histogram priors}

{\em Histogram prior on $[0,1]$ with deterministic number of jumps.} Let $K=K_n$ be an integer, a number of `jumps' -- to be chosen later --, and let us subdivide $[0,1]$ in $K$ equally spaced intervals: 
 for $I_k=[(k-1)/K, k/K)$, let us set
\begin{equation} \label{priorhist}
 f = \sum_{k=1}^K h_k \1_{I_k},\qquad (h_1,\ldots,h_k)\sim P_\psi^{\otimes K},
\end{equation}
where $P_\psi$ is the common distribution of the (random) histogram heights. For simplicity in what follows we take $P_\psi=\text{Lap}(1)$ the standard Laplace distribution, which has density $x\to e^{-|x|}/2$ on $\RR$, although many other choices are possible.\\

{\em Statistical model.} Let us consider one of the canonical nonparametric models: it turns out the simplest to verify the conditions of Theorem \ref{thm:ggve} is the Gaussian white noise model, but the proof is quite easily adapted for the regression and density models. We shall come back to the density model later. Recall that in the white noise model one observes $X=X^{(n)}$ with $dX(t)=f(t)dt+dW(t)/\sqrt{n}$.  \\

{\em Bayesian setting.} We put as prior on $f\in L^2[0,1]$ a histogram prior $\Pi$ defined as in \eqref{priorhist}, which combined with the law of $X\given f$ in the white noise model gives a posterior distribution $\Pi[\cdot\given X]$. The model is dominated by $P_0^{(n)}$ (the distribution of the data when $f=0$), see Appendix \ref{app:models}, and Bayes formula can be written, for any $B$ in the Borel $\sigma$-field of $\cC^0[0,1]$,
\[ \Pi(B\given X) = \frac{\int_B \exp\left\{n\int_0^1 f(t)dX(t)-n\|f\|_2^2/2\right\} d\Pi(f)}{\int  \exp\left\{n\int_0^1 f(t)dX(t)-n\|f\|_2^2/2\right\} d\Pi(f)}. \]
 \vp

{\em Frequentist study of $\Pi[\cdot\given X]$ and regularity condition on $f_0$.} To study the frequentist behaviour of the posterior, it is usual to impose some regularity conditions on $f_0$, which will then typically influence the expression of the convergence rate one obtains. For $\al\le 1$ define a H\"older--ball $\cC^\al(L)=\{g:[0,1]\to\RR,\ \forall x,y\in[0,1],\ |g(x)-g(y)|\le L|x-y|^\al\}$. We assume that the true $f_0$ belongs to $\cF=\cF(\al, L, M)$, for $L,M>0$ and $\al\le 1$, with
\begin{equation} \label{classhisto}
\cF= \{f: [0,1]\to\RR:\ f\in\cC^\al(L),\ \ \|f\|_\infty\le M\}. 
\end{equation}

{\em Posterior convergence rate.} The specific form of the model will actually not matter much for Theorem \ref{thm:ggve}: it is enough to know we can apply it, since the white noise model verifies the testing condition (T) with $d=\|\cdot\|_2$ as noted earlier. So it is enough to verify the conditions i), ii), iii) of Theorem \ref{thm:ggve} with this distance and suitably chosen sets $\cF_n$. If one can do so, we will obtain a  posterior contraction rate for $\Pi[\cdot\given X]$ in terms of $d=\|\cdot\|_2$. For simplicity in this section we denote $\|\cdot\|=\|\cdot\|_2$ the $L^2$--norm on $[0,1]$.

\begin{thm} \label{thm:hist}
In the Gaussian white noise model, suppose the true $f_0\in\cF(\al,L,M)$ for some $L,M>0$ and $\al\in(0,1]$. Let $\Pi$ be a random histogram prior as above with a number of jumps
\[ K\asymp \left(\frac{n}{\log{n}} \right)^{\frac{1}{2\al+1}}. \] 
Then for $m>0$ a large enough constant, as $n\to\infty$ 
\[ E_{f_0}\Pi[\|f-f_0\|_2\le m\veps_n\given X]\to 1,\qquad 
\veps_n\asymp \left(\frac{\log{n}}{n} \right)^{\frac{\al}{2\al+1}}.
\]
\end{thm}
This result says that if $K$ is appropriately chosen in terms of the smoothness $\alpha$ of $f_0$, then the posterior achieves the optimal rate $n^{-\al/(2\al+1)}$ for $d=\|\cdot\|_2$ up to a logarithmic term.

\subsubsection*{Basic histogram facts}

Let $\cV_K=\text{Vect}_{L^2}(\1_{I_1},\ldots,\1_{I_K})$ denote the subspace of $L^2=L^2[0,1]$ spanned by histograms over the partition $(I_k)$. For a sequence $(u_1,\ldots,u_K)\in\RR^K$, let us denote $\|\cdot\|_K$ the euclidean norm in $\RR^K$
\[ \|u\|_K^2= \sum_{k=1}^K u_k^2.\]
{\em Fact 1.} The orthogonal projection of $f\in L^2$ onto $\cV_K$ is
\[ f^{[K]}=\sum_{k=1}^K \oli{f}_k \1_{I_k},\qquad \text{with }\ \oli{f}_k=K\int_{I_k}f. \]
Let us denote $\oli f:=(\oli f_1,\ldots,\oli f_K)$. For any $f\in L^2$, 
\[ \| f^{[K]}\|^2 = \frac{1}{K}\sum_{k=1}^K \oli f_k^2
= \frac{1}{K} \|\oli f\|_K^2.\]
Thus, up to a factor $K^{-1}$, the $L^2$--norm of $f^{[K]}$ coincides with the $\|\cdot\|_K$--norm of the sequence $\oli f$.\\

\noindent {\em Fact 2.} Let $f_0\in \cC^\alpha(L)$ with $\al\in(0,1]$. Then 
\[ \| f_0 - f_0^{[K]}\|_\infty \le L K^{-\al}. \]
Indeed, by the mean--value theorem $\oli{f}_{0,k}=K\int_{I_k} f_0=f_0(c_k),$ for a  $c_k\in I_k$. For $t\in I_k$, we have $|f_0(t)-f_0^{[K]}(t)|=|f_0(t)-f_0(c)|\le L|t-c|^\al \le L K^{-\al}$. This gives the result by making $k$ range from $1$ to $K$.\\

\subsubsection*{Verifying the conditions of Theorem \ref{thm:ggve}}

\begin{proof}[Proof of Theorem \ref{thm:hist}]
Let us choose some {\em sieve} sets $\cF_n$ as follows
\[ \cF_n = \left\{f\in\cV_K:\ f= \sum_{k=1}^K h_k \1_{I_k},\ \ (h_1,\ldots,h_K)\in\cH_n\right\}, \]
where $\cH_n$ is the set of sequences $h=(h_k)_{1\le k\le K}$ defined as 
\[ \cH_n=\left\{ h=(h_k),\ \ \max_{1\le k\le K}|h_k| \le n\right\}. \]
The upper bound on the heights turns helpful to verify the entropy condition.\\

{\em Entropy condition i).} Since any $f\in\cV_K$ is equivalently characterised by its height sequence $h$ and $\|f\|^2=K^{-1}\|h\|_K^2$, it is enough to cover the set of sequences $\cH_n$. By using Fact 1 above,
\[ N(\veps,\cF_n,\|\cdot\|)=N(\sqrt{K}\veps,\cH_n,\|\cdot\|_K). \]
Now note that $\|h\|_K^2\le K\max_{1\le k\le K} h_k^2\le K n^2$ for any $h\in\cH_n$, so that $\cH_n\subset B_{\RR^K}(0,\sqrt{K}n)$.\\

Lemma \ref{lem:entub} gives $N(\delta,B_{\RR^K}(0,M),\|\cdot\|_K) \le \left(3M/\delta\right)^K$   for $3M/\delta\ge 1$. This implies, for $\veps\le 1$,
\begin{align*}
N(\veps,\cF_n,\|\cdot\|)  \le N(\sqrt{K}\veps,B_{\RR^K}(0,\sqrt{K}n),\|\cdot\|_K) \le (3n/\veps)^K.
\end{align*}
In order to fulfill i), one obtains the condition $K\log(3n/\oli\veps_n)\le Dn\oli\veps_n^2$.\\

{\em Sieve condition ii).} By definition of $\cF_n$, using that $P[|\text{Lap}(1)|>n]= e^{-n}$,
\[ \Pi[\cF_n^c]\le \Pi[\exists k\in\{1,\ldots,K\}:\ |h_k|>n] 
\le Ke^{-n}\le \exp\{\log{K}-n\}.\]
In order to fulfill i), one obtains the condition $\log{K}-n\le -n\uli\veps_n^2(C+4)$. This is always satisfied for large enough $n$ provided $K=K_n$ is chosen so that $K_n=o(n)$. \\

{\em Prior mass condition iii).} Recall that in the white noise model $B_n(f_0,\veps)$ is just the $L^2$--ball $\{f:\ \|f-f_0\|< \veps\}$. Pythagoras theorem gives $\|f-f_0\|^2=\|f-f_0^{[K]}\|^2+\|f_0-f_0^{[K]}\|^2$. So  for any $\eta>0$
\[ \Pi[\|f-f_0\|<\eta] = \Pi[\|f-f_0^{[K]}\|^2<\eta^2-\|f_0-f_0^{[K]}\|^2].\]
By Fact 2 above, $\|f_0-f_0^{[K]}\|\le \|f_0-f_0^{[K]}\|_\infty\le LK^{-\al}$. This means that provided $K$ is chosen large enough in terms of $\eta$ (the condition involving the rate is given below), one can always make sure that $\eta^2-\|f_0-f_0^{[K]}\|^2\le \eta^2/2$. It is thus enough to consider, for $\veps>0$,
\begin{align*}
\Pi[\|f-f_0^{[K]}\|<\veps] & =\Pi\left[K^{-1}\sum_{k=1}^K (\oli f_k-\oli f_{0,k})^2\le \veps^2 \right] = \Pi\left[K^{-1}\sum_{k=1}^K (h_k-\oli f_{0,k})^2\le \veps^2 \right]  \\
& \ge \Pi\left[ \bigcap_{k=1}^K\ \{|h_k-\oli f_{0,k}|\le \veps \right] 
 \ge \prod_{k=1}^K \Pi\left[ |h_k-\oli f_{0,k}|\le \veps \right], 
\end{align*}
using the independence of the heights $h_k$ under the considered prior distribution.
To further bound from below the last display, note that $\Pi\left[ |h_k-\oli f_{0,k}|\le \veps \right]$ is the probability that a standard Laplace variable belongs to a certain interval of length $2\veps$. Since $|\oli f_{0,k}|\le \|f_0\|_\infty\le M$ by assumption, this interval is included in $[-M-\veps,M+\veps]\subset [-2M,2M]$ if $\veps\le 1$. On the latter interval, the standard Laplace density is at least $e^{-2M}/2$. Deduce that 
\[\Pi\left[ |h_k-\oli f_{0,k}|\le \veps \right] \ge 2\veps\cdot e^{-2M}/2=\veps e^{-2M}, \]
so that $\Pi[\|f-f_0^{[K]}\|<\veps]\ge \veps^K \exp\{-2MK\}$. Putting the previous bounds together,  if $LK^{-\al}\le \uli\veps_n/2$
\[ \Pi[\|f-f_0\|_2<\uli\veps_n] \ge \Pi[\|f-f_0^{[K]}\|<\uli\veps_n/2]
\ge  (\uli\veps_n/2)^K \exp\{-2MK\}.\]
So the prior mass condition is verified if 
\begin{align*}
LK^{-\al}& \le \uli\veps_n/2\\
 2MK+K\log(2/\uli\veps_n)& \le Cn\uli\veps_n^2. 
\end{align*} 
It is now easy to verify that conditions i) up to iii) are satisfied for the choices
\[ \uli\veps_n\asymp \oli\veps_n\asymp \left(\frac{\log{n}}{n} \right)^{\frac{\al}{2\al+1}},\qquad K\asymp \left(\frac{n}{\log{n}} \right)^{\frac{1}{2\al+1}}.\]
\end{proof} 
\vp

\section{Gaussian process priors} \label{sec:GPs}

The elements of  theory of Gaussian processes (GPs) needed here are summarised in Appendix \ref{app:gps}. GPs are natural candidates for prior distributions on functions -- we stick for simplicity to functions defined on $[0,1]$ -- ; here they will be seen as random elements taking values in a separable Banach space $\mb$ with norm $\|\cdot\|_\mb$. We consider only the following two cases  in the sequel: $\mb=L^2[0,1]$ the Hilbert space of squared-integrable functions on $[0,1]$ and $\mb=\cC^0[0,1]$ the space of continuous functions on $[0,1]$, equipped with their respective canonical norms $\|\cdot\|_2$ and $\|\cdot\|_\infty$. 

We give a few examples to start with, more will be given along the way. We hint at what their `regularity' (in a sense we do not make explicit here) is, since it  helps interpreting the results below. \\

{\em Example: Brownian motion (BM).} Brownian motion is the GP $(B_t)$ with zero mean and covariance $K(s,t) = E[B_sB_t]=s \wedge t$. It can be shown that there is a version of Brownian motion with sample paths that are $1/2-\veps$ H\"older, for any $\veps>0$, so in a sense its regularity is $1/2$. \\

{\em Example: Riemann-Liouville process $R^\al$.} Consider, for $\al>0$ and $(B_s)$ Brownian motion 
\[ R_t^\al = \int_0^t (t-s)^{\al-1/2} dB(s).\]
For $\al=1/2$ one gets Brownian motion while more generally this can be seen as a $(\al-1/2)$--integrated Brownian motion, having thus `regularity' close to $\al$.\\

{\em Gaussian series prior.} For $\zeta_j$  iid $\cN(0,1)$ variables, and $(e_j)$ an orthonormal basis of $L^2[0,1]$, 
\begin{equation} \label{seriesgp}
 W_t = \sum_{j=1}^\infty j^{-\frac12-\alpha} \zeta_j e_j(t)
\end{equation} 
is, for $\al>0$, again a process whose regularity (here in a Sobolev type sense) is nearly $\al$.\\

All these processes can be seen as random variables in $\mb=L^2$ or $\cC^0$, and there are fairly explicit characterisations of their RKHS, as we see below. The next definition is key in the analysis of GP posterior rates. We denote by 
$\overline{\mh}^\mb$, the closure in $\mb$ (with respect to the norm of $\mb$) of $\mh$.

\begin{definition}
Let $W$ be a Gaussian random variable taking its values in $\mb$ separable Banach space, with RKHS $\mh$. Let $w$ belong to $\overline{\mh}^\mb$. For any $\veps>0$, define
\begin{align*}
 \vphi_w(\veps)
& = \inf_{h\in\mh,\ \|h-w\|_\mb<\veps} \frac12\|h\|_\mh^2
- \log P[ \|W\|_\mb<\veps ] \\
& \ =: \vphi^A_{w}(\veps) + \vphi_0(\veps) 
\end{align*}
The function $\vphi_w(\cdot)$ is called the \sbl{concentration function} of the process $W$. 
\end{definition}

\subsubsection*{Pre-concentration theorem}

Using the just defined concentration function $\vphi_w$ of a Gaussian process, it turns out that one can verify conditions  close to the ones of the main Theorem \ref{thm:ggve}. This is the goal of the next Theorem, due to \cite{vvvz}; we see further below how this is then used to obtain rates in specific models.

\begin{thm}\cite{vvvz} \label{thmvv}$\ $
 Let $W$ be a Gaussian random variable taking values in $\mb$ separable Banach space, with RKHS $\mh$. 
Let $w_0\in \overline{\mh}^{\mb}$ and let $\veps_n>0$ be such that
\begin{equation} \label{conceq}
\vphi_{w_0}(\veps_n)\le n\veps_n^2.
\end{equation}
Then for any $C>1$ with $Cn\veps_n^2>\log{2}$, there exists 
$B_n\subset \mb$ measurable sets such that
\begin{align*}
\sbl{(i)}\quad \log N(3\veps_n,B_n,\|\cdot\|_\mb) & \le 6Cn\veps_n^2 \\
\sbl{(ii)}\quad \ \ \quad\qquad P[ W\notin B_n ] & \le e^{-Cn\veps_n^2}\\
\sbl{(iii)}\quad P[\|W - w_0\|_\mb < 2\veps_n] & \ge e^{-n\veps_n^2}.
\end{align*}
\end{thm}

\begin{proof}
The inequality \sbl{(iii)} is a consequence of Theorem \ref{thm:usb} on probability of balls for Gaussian processes and their link to the concentration function: 
\[ P[\|W - w_0\|_\mb < 2\veps_n] \ge e^{-\vphi_{w_0}(\veps_n)},\] 
which combined with \eqref{conceq} leads to \sbl{(iii)}. \\

In order to prove \sbl{(ii)}, we define
\[ B_n = \veps_n \mb_1 + M_n \mh_1, \]
where $M_n$ is to be chosen. By Borell's inequality, with $\Phi(x)=\int_{-\infty}^x (e^{-u^2/2}/\sqrt{2\pi}) du=P[\cN(0,1)\le u]$, 
\[ P[W\notin B_n ] \le 1 - \Phi( \Phi^{-1}(e^{-\vphi_0(\veps_n)})+M_n). \]
By definition of the concentration function as a sum of two nonnegative terms, 
$\vphi_{0}(\veps_n)\le \vphi_{w_0}(\veps_n)\le n\veps_n^2$ using \eqref{conceq}. 
Let us set, for some $C>1$,
\[ M_n = -2\Phi^{-1}(e^{-Cn\veps_n^2}). \]
Then we have, by monotonicity of $\Phi^{-1}$ and definition of $M_n$,
\[ \Phi^{-1}(e^{-\vphi_0(\veps_n)})\ge \Phi^{-1}(e^{-n\veps_n^2})\ge -M_n/2.\]
Inserting this back into the previous upper-bound on $P[W\notin B_n ]$ leads to
\[ P[W\notin B_n ] \le 1 - \Phi(M_n/2)= 1 - \Phi(-\Phi^{-1}(e^{-Cn\veps_n^2}))=
e^{-Cn\veps_n^2},\]
using that $\Phi(-x)=1-\Phi(x)$ for any real $x$, so \sbl{(ii)} is established.\\

It now remains to check \sbl{(i)}. Let $h_1,\ldots, h_N$ be elements of  $M_n\mh_1$ separated by at least $2\veps_n$ in terms of the $\|\cdot\|_\mb$ norm, and suppose this set of points is {\em maximal} (in the sense that $N$ is the maximal number of $2\veps_n$--separated points in $M_n\mh_1$; the argument below shows that $N$ is necessarily finite).  
 
The balls $h_1+\veps_n\mb_1, \ldots, h_N+\veps_n\mb_1$ are disjoint since the $h_i$'s are $2\veps_n$--separated. This implies
\[ 1 \ge P\left[W\in \bigcup_j \, (h_j+\veps_n\mb_1) \right] = \sum_{j=1}^N 
P[W\in h_j+\veps_n\mb_1].  \]
Applying inequality \eqref{usmallb}, since $h_j$'s belong to $\mh$, and using $\|h_j\|_\mh\le M_n$, one gets
\[ P[W\in h_j+\veps_n\mb_1] \ge e^{-\|h_j\|_\mh^2/2} P[ W\in \veps_n\mb_1]\ge 
e^{-M_n^2/2 -\vphi_0(\veps_n)}. \]
Inserting this into the previous inequality leads to 
\[ 1 \ge Ne^{-M_n^2/2 -\vphi_0(\veps_n)}, \]
from which one sees in particular that $N$ must be finite. Deduce 
\[  N(2\veps_n,M_n\mh_1,\|\cdot\|_\mb) \le N \le e^{M_n^2/2 + \vphi_0(\veps_n)}.\]
This implies 
\[ N(3\veps_n,\veps_n\mb_1+M_n\mh_1,\|\cdot\|_\mb)\le e^{M_n^2/2 + \vphi_0(\veps_n)}.\]
By a standard inequality on $\Phi^{-1}$, the inverse of the Gaussian CDF $\Phi$, we have
\[ 0>\Phi^{-1}(y)\ge -\sqrt{\frac52 \log(1/y)}, \qquad 0<y<1/2.\]
Deduce, using  $Cn\veps_n^2>\log{2}$, that
\[ M_n=-2\Phi^{-1}(e^{-Cn\veps_n^2})\le 2 \sqrt{\frac52 \log(e^{Cn\veps_n^2})}.
\]
Combining with the previous inequality on $N$, one obtains
\[ N(3\veps_n,B_n,\|\cdot\|_\mb)\le e^{5Cn\veps_n^2+\vphi_0(\veps_n)}
\le e^{6Cn\veps_n^2}, \]
using once again \eqref{conceq}, which leads to \sbl{(i)} and concludes the proof.
\end{proof}

\subsubsection*{Application: rates for GP priors in regression}

Recall that the Gaussian white noise model is 
\[ dX^{(n)}(t)=f(t)dt+ \frac{1}{\rn}dW(t),\quad t\in[0,1]\]
and that in this model, tests verifying condition (T) exist for the $\|\cdot\|_2$--norm, and that the neighborhood $B_{KL}$ of the GGV theorem is just the  $L^2$ ball $\{f:\, \|f-f_0\|_2<\veps_n\}$. 

\begin{thm} \label{thmgwn}
Let $X^{(n)}$ be observations from  the Gaussian white noise model. 
Let $\Pi$ be a prior distribution on $f\in L^2[0,1]$, defined as the distribution  of a centered Gaussian random variable in $\mb=L^2$, with RKHS $\mh$. 
Suppose the true $f_0\in \overline{\mh}^\mb$ and let $\veps_n$ be such that
\[ \vphi_{f_0}(\veps_n)\le n\veps_n^2, \]
where $\vphi_{f_0}$ is the concentration function of $W$ in $\mb=L^2$. 
Then for $M$ large enough, as $n\to\infty$,
\[ E_{f_0} \Pi[ \|f-f_0\|_2>M \veps_n\given X^{(n)}] \to 0.\]
\end{thm}

\begin{proof}
It is enough to note that the conclusion of Theorem \ref{thmvv} matches exactly the conditions of the GGV Theorem, noting that $d=\|\cdot\|_2$ and that the neighborhood $B_{KL}$ of the GGV theorem is the  $L^2$ ball $\{f:\, \|f-f_0\|_2<\veps_n\}$. The Theorem thus follows from the GGV theorem (up to setting $\veps_n'=2\veps_n$ and noting that $C>1$ can be taken arbitrarily large).
\end{proof}

\subsubsection*{Application: rates for GP priors in density estimation}

In the density estimation model on $[0,1]$,
\[ X^{(n)}=(X_1,\ldots,X_n) \sim P_f^{\otimes n}, \]
where $P_f$ is the distribution of density $f$ on $[0,1]$. For the next result, we work in $\mb=\cC^0[0,1]$ space of continuous functions on $[0,1]$, equipped with the supremum norm $\|\cdot\|_\infty$. The next result implicitly assumes that $\log{f_0}$ is well--defined, that is, that $f_0$ is bounded from below.  

\begin{thm} \label{thmde}
Let $X^{(n)}$ be observations from the density estimation model. 

Let $\Pi$ be a prior distribution on $f\in \cC^0[0,1]=\mb$, defined as the distribution  of 
\begin{equation}\label{normgp}
t\to \frac{e^{W_t}}{\int_0^1 e^{W_u} du},
\end{equation}
with $(W_t, t\in[0,1])$ 
a centered Gaussian process with continuous sample paths, with RKHS $\mh$. 

Let $w_0:=\log f_0$. Suppose $w_0\in \overline{\mh}^\mb$. Suppose, for some $\veps_n>0$, we have 
\[ \vphi_{w_0}(\veps_n)\le n\veps_n^2, \]
with $\vphi_{f_0}$  the concentration function of $W$ in $\mb=\cC^0[0,1]$. 
Then for $M$ large enough, as $n\to\infty$,
\[ E_{f_0} \Pi[ h(f,f_0)>M \veps_n\given X^{(n)}] \to 0.\]

\end{thm}

\begin{proof}
One can apply Theorem \ref{thmvv} to the function $w_0$: there exist sets $B_n$ such that the conclusions (i)--(ii)--(iii) of that Theorem are satisfied. 

Our goal is to verify the conditions of the GGV theorem with the Hellinger distance $d=h$. 
For such $B_n$, let us set
\[ \cF_n:=\left\{ f= \frac{e^w}{\int_0^1 e^{w(u)} du},\quad \text{ for } w\in B_n\right\}.\]
By (ii), we have $\Pi[\cF_n^c]=P_W[\mb\setminus B_n]\le e^{-Cn\veps_n^2}$, so the second condition of GGV is satisfied (we denote by $P_W$ the distribution of the Gaussian process at the level of $w$'s, while $\Pi$ is the induced distribution at the level of densities $f$).\\

In order to verify the entropy and prior mass conditions of the GGV theorem, one needs to link the distance on $w$'s to the distance on densities. This is done in Lemma \ref{lem:dist}. \\

From the first inequality in Lemma \ref{lem:dist}, one deduces that a covering of $B_n$ by $3\veps_n$--balls using the $\|\cdot\|_\infty$--metric induces a covering of $\cF_n$ by $3\veps_n e^{3\veps_n/2}$--balls for the Hellinger distance $h$. For $\veps_n\to 0$ and large $n$, this implies 
\[ \log N(4\veps_n,\cF_n,h)\le \log N(3\veps_n,B_n,\|\cdot\|_\infty),\]
which means using (i) of Theorem \ref{thmvv} that the entropy condition of the GGV theorem is satisfied.\\

The second and third inequalities in Lemma \ref{lem:dist} imply, for a large enough constant $K>0$,
\[ \Pi[B_{KL}(f_0,K\veps_n)] \ge P_W[\|W-w_0\|_\infty\le 2\veps_n],\]
which is larger than $e^{-n\veps_n^2}$ using (iii) of Theorem \ref{thmvv}, which shows the prior mass condition is satisfied. The result now follows from the GGV theorem. 
\end{proof}

\subsubsection*{Examples}

{\em Brownian motion.} 
Consider Brownian motion $W_t=B_t$ in the setting $(\mb,\|\cdot\|_\mb)=(\cC^0[0,1],\|\cdot\|_\infty)$ (the results are the same up to constants in the $L^2$--setting). The small ball probability of Brownian motion is well--known from the probability literature: one can show (we admit it), as $\veps\to 0$,
\[ \vphi_0(\veps)=- \log P[\|B\|_\infty<\veps] \asymp \veps^{-2}.\]
It remains to study the approximation term in the concentration function.  The RKHS of Brownian motion on $[0,1]$ is $\{\int_0^\cdot g(u)du,\ g\in L^2[0,1]\}$, equipped with the Hilbert norm $\| \int_0^\cdot g\|_\mh^2=\|g\|_2^2$. 

\begin{lem} \label{lem:bmapp}
Let $(\mh,\|\cdot\|_\mh)$ be the RKHS of Brownian motion.  Suppose $w_0\in\cC^\be[0,1]$, for some $\be\ge 0$ and $w_0(0)=0$. Then
\[ \inf_{h\in\mh:\, \|h-w_0\|_\infty<\veps} \|h\|_\mh^2 \leqa \veps^{\frac{2\be-2}{\be}}\vee 1. \]
\end{lem}

\begin{proof}
We define a sequence $h\in\mh$ that approximates $w_0$. The idea is to use a convolution. First, one can restrict to the case $\be\le 1$, otherwise $w_0$ already belongs to the RKHS so one can take $h=w_0$. Also, $w_0$ can be extended to $\mathbb{R}$ while keeping the H\"older-type property $|w_0(x)-w_0(y)|\leqa |x-y|^\be$ (just take $w_0$ the appropriate constant outside of $[0,1]$) \\

Let $\phi_\sigma(u)=\phi(u/\sigma)/\sigma$, for $\sigma>0$, and $\phi(u)=e^{-u^2/2}/\sqrt{2\pi}$ the Gaussian density. Let 
\[ h_\sigma(t) := (\phi_\sigma*w_0)(t) - (\phi_\sigma*w_0)(0), \]
with $\phi_\sigma*w_0(t)=\int_\RR \phi_\sigma(t-u)w_0(u)du$. Note that $h_\sigma(0)=0$ and $h_\sigma$ is a $\cC^\infty$ map (because it is a convolution by a smooth function), so $h_\sigma$ belongs to $\mh$. We now evaluate
\begin{align*}
| \phi_\sigma*w_0(t)- w_0(t) | & = |\int \phi_\sigma(u)(w_0(t-u)-w_0(t)) du \\
& \leqa \int \phi_\sigma(u) |u|^\be du \leqa \sigma^\be \int |v|^\be\phi(v)dv\leqa \sigma^\be.
\end{align*}
Since $w_0(0)=0$, we get a similar bound for $\phi_\sigma*w_0(0)$ by setting $t=0$ in the previous inequality. This shows $\|h_\sigma-w_0\|_\infty\leqa \sigma^\be$. \\

On the other hand, $\|h_\sigma\|_\mh^2=\int_0^1 (h_\sigma)'(t)^2 dt$, where, using 
$\int \phi' =0$,
\begin{align*}
|(h_\sigma)'(t)| & =  | \int w_0(t-u) \frac{1}{\sigma^2} \phi'(u/\sigma) du| 
 = | \int (w_0(t-u)-w_0(t)) \frac{1}{\sigma^2} \phi'(u/\sigma) du|\\
& \leqa  \sigma^{-2} \int |u|^\be |\phi'(u/\sigma)| du  \leqa  \sigma^{\be-1}.
\end{align*}
The result follows by taking $\sigma\asymp \veps^{1/\be}$. 
\end{proof}
It follows also from the proof of Lemma \ref{lem:bmapp} that $\overline{\mh}^\mb$ is the set of continuous functions $f$ such that $f(0)=0$ (one uses the proof for $w_0$ continuous and $w_0(0)=0$, replacing the H\"older condition by absolute continuity of $w_0$). That is, almost all of $\mb$ except for the restriction $f(0)=0$. One can show that to obtain all of $\mh$, it suffices to consider `Brownian motion released at zero'
\[ Z_t = B_t + Y, \]
with $Y$ an $\cN(0,1)$ variable independent of $(B_t)$. The RKHS of $(Z_t)$ can be shown to be $H=\{c+\int_0^\cdot g(u)du,\ g\in L^2[0,1], c\in \RR\}$, for which $\overline{\mh}^\mb=\mb$.

By gathering the small ball probability estimate and Lemma \ref{lem:bmapp}, one gets, with $a\vee b=\max(a,b)$,
\[ \vphi_{w_0}(\veps_n)\leqa \veps_n^{-2} + C \vee \veps_n^{(2\be-2)/\be}.\]
By equating this rate to $n\veps_n^2$, one obtains, with $a\wedge b=\min(a,b)$,
\[ \veps_n \asymp n^{-1/4} \vee n^{-\be/2} = n^{-\left\{\frac14\wedge \frac{\be}{2}\right\}}. \]
The rate is the fastest if $\be=1/2$, for which $\veps_n \asymp n^{-1/4}$. When $\be<1/2$, the rate is $\veps_n \asymp n^{-\be/2}$: the approximation term (the `bias') dominates in the contribution from the concentration function. When $\be\ge 1/2$, the small ball probability term (analog of the `variance') dominates.\\

By using Theorem \ref{thmde} in the density estimation model, with $W_t$ in \eqref{normgp} a Brownian motion released at $0$, one obtains that for any true density $f_0$ such that $w_0=\log{f_0}$ belongs $\cC^\be[0,1]$ then an upper-bound on the posterior convergence rate is given by $\veps_n$ in the last display.

It can be shown that the above rate cannot be improved for Brownian motion: it is the best one that one can get with this prior. From the minimax perspective, the rate $\veps_n$ above matches the minimax rate for estimating $\cC^\be$ functions, that is $n^{-\be/(2\be+1)}$ if and only if $\be=1/2$. \\

{\em Riemann-Liouville process $W_t=R_t^\alpha$.} One can show a similar result as for Brownian motion (again, modulo proper `release' of the process at $0$ so that $\overline{\mh}^\mb$), with the `regularity' $1/2$ of Brownian motion replaced by $\al$. Up to a possible logarithmic factor, the obtained rate is then \cite{vvvz, ic08}
\[ \veps_n \asymp n^{-\frac{\al\wedge \be}{2\al+1}}.\]
Again, the rate is the optimal one (from the minimax perspective) if $\al=\be$, but sub-optimal otherwise.\\

{\em GP series prior.} Recall the random series GP prior,  for $\zeta_j$ iid $\cN(0,1)$ variables
\[ W_t = \sum_{j=1}^\infty j^{-\frac12-\alpha} \zeta_j e_j(t),\]
 and $(e_j)$ an orthonormal basis of $L^2[0,1]$. For $\mb=L^2[0,1]$, it can be shown that a rate solving the concentration function equation $\vphi_{w_0}(\veps_n)\le n\veps_n^2$ is again
 \[ \veps_n \asymp n^{-\frac{\al\wedge \be}{2\al+1}}.\]

\subsubsection*{Take-away message}

The main take-away message from Theorem \ref{thmvv} and its applications in Theorems \ref{thmgwn} and \ref{thmde} is that, when a Gaussian process is used as prior distribution (and provided the $\|\cdot\|_\mb$--norm is easily related to the testing distance $d$, KL and V), the rate of convergence of the posterior distribution in terms of $d$ is essentially determined by solving the equation $\vphi_{w_0}(\veps_n)\leqa n\veps_n^2$, where $w_0=f_0$ in the white noise model (respectively $w_0=\log{f_0}$ in density estimation).  

As applications, we have seen here only a few examples of GPs, but the results can be applied to many others, including `squared-exponential' (which we study in the next Chapter) \cite{vvvz09}, Mat\'ern \cite{vvvz11}... 
Similarly, the results apply much more broadly in terms of  statistical models (e.g. to binary classification, random design regression etc.), see e.g. \cite{vvvz09, vvvz11, gvbook}.  \\

From the rates $\veps_n$ obtained above, we see that one always obtain a convergence rate going to zero (the posterior is said to be consistent), which is optimal if (and only if \cite{ic08}) the prior `regularity' matches $\be$, the regularity of the function to be estimated. As $\be$ is rarely known in practice, this shows that the Gaussian process priors  have to be made more complex if one wishes to derive {\em adaptation}, i.e. obtaining a prior for which the posterior achieves the optimal rate regardless of the actual value of $\be$. This question is considered in Chapter \ref{chap:ada1}. 

\section{P\'olya trees and further basis-related priors} \label{sec:polya}

There are many more interesting prior constructions for nonparametrics, we mention only two others here, several others will appear in the next chapters.\\

\ti{$p$-exponential priors} It is natural to ask what happens if the Gaussian distribution in the series prior \eqref{seriesgp} is replaced by another one. If the distribution of the iid variables $\zeta_i$ has density proportional to $\exp(-|x|^p/p)$, called $p$--exponential or Subbotin distribution ($p=1$ gives the Laplace law, $p=2$ the Gaussian), with $1\le p\le 2$, then \cite{agapiouetal21} develops a theory based on a generalisation of the concentration function $\vphi_w(\veps)$ for Gaussian processes. It then follows that the posterior contraction rate around a $\beta$--smooth unknown for a $p$--exponential series prior (i.e. \eqref{seriesgp} with Gaussian replaced by $p$-exponential) is given by
\begin{equation} \label{ratep}
 \veps_n\leqa 
\begin{cases}
n^{-\frac{\be}{1+2\be+p(\al-\be)}},&\qquad \text{if } \al\ge \be,\\
n^{-\frac{\al}{1+2\al}},&\qquad \text{if } \al \le \be. 
\end{cases} 
\end{equation}
The obtained rate features a similar break-point at $\al=\be$ as for GPs, and the rate is the same in the {\em undersmoothing} case $\al\le \be$, for which it is the prior's own regularity that drives the rate. In the {\em oversmoothing} case $\al>\be$, the rates improves from $n^{-\be/(1+2\al)}$ (Gaussian case $p=2$) to $n^{-\be/(1+\be+\al)}$ (Laplace case $p=1$). The case of even heavier tails for the $\zeta_i$'s is particularly interesting and discussed in details in Chapter \ref{chap:ada1}.\\

\ti{P\'olya trees} Now we explain a way to construct random probability measures using a regular dyadic partition and a tree.  First we introduce some notation relative to dyadic partitions. For any fixed indexes $l\geq 0$ and $0 \leq k < 2^l$, the number $r = k2^{-l}$ can be written in a unique way as $\veps(r) := \veps_1(r) \ldots \veps_l(r)$, its finite expression of length $l$ in base $1/2$ (it can end with one or more $0$'s). That is, $\veps_i \in \{0, 1\}$ and \[k2^{-l} = \sum_{i=1}^l \veps_i(r)2^{-i}.\] Let $\mathcal{E} := \bigcup_{l \geq 0}\{0,1\}^l\cup \{\emptyset\}$ be the set of finite binary sequences. We write $|\veps| = l$ if $\veps \in  \{0,1\}^l$ and $|\emptyset| = 0$. For $\veps=\veps_1\veps_2\ldots\veps_{l-1}\veps_l$, we also use the notation $\veps'=\veps_1\veps_2\ldots\veps_{l-1}(1-\veps_l)$.

Let us introduce a sequence of partitions $\mathcal{I} = \{ (I_\veps)_{|\veps|=l}, l \geq 0\}$ of the unit interval. Set $I_\emptyset = (0, 1]$ and, for any $\veps \in \mathcal{E}$ such that $\veps= \veps(l;k)$  is the expression in base $1/2$ of $k2^{-l}$, set 
\[I_\veps := \left(\frac{k}{2^l},\frac{k+1}{2^l}\right]:=I_k^l. \]	
For any $l \geq 0$, the collection of all such dyadic intervals is a partition of $(0, 1]$.

Suppose we are given a collection of random variables $(Y_{\veps},\,\veps\in\cE)$ with values in $[0,1]$ such that 
\begin{align}
Y_{\veps1} & = 1 - Y_{\veps0},\qquad \forall\, \veps\in\cE, \label{conservative} \\
E[Y_\veps Y_{\veps0} Y_{\veps00}\cdots]&=0
\qquad\qquad\quad \forall\, \veps\in\cE,\ 
\end{align}
Let us then define a random probability measure on dyadic intervals by 
\begin{equation}\label{treepr}
P(I_\veps) = \prod_{j=1}^l Y_{\veps_1\ldots\veps_j}.
\end{equation}
By (a slight adaptation, as we work on $[0,1]$ here, of) Theorem 3.9 in \cite{gvbook}, the measure $P$ defined above extends to a random probability measure on Borel sets of $[0,1]$ almost surely, that we call {\em tree--type prior}. 

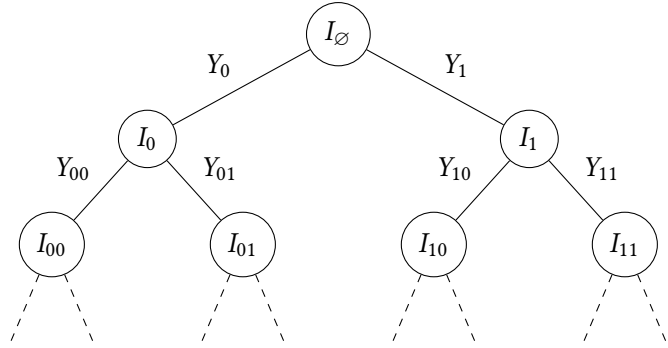
\begin{figure}
\begin{center}
\begin{tikzpicture}[
   level distance=1.4cm,sibling distance=1cm, 
   edge from parent path={(\tikzparentnode) -- (\tikzchildnode)}]
\Tree [.\node[draw,circle] {$I_\varnothing$}; 
    \edge node[auto=right] {$Y_{0}$}; 
    [.\node[draw,circle]{$I_0$};  
      \edge node[auto=right] {$Y_{00}$};  
      [.\node[draw,circle]{$I_{00}$}; 
      \edge[dashed]; {}
      \edge[dashed]; {}
      ]  
      \edge node[auto=left] {$Y_{01}$}; [.\node[draw,circle]{$I_{01}$}; 
      \edge[dashed]; {} 
      \edge[dashed]; {}
      ] 
          ]
     \edge node[auto=left] {$Y_{1}$};      
    [.\node[draw,circle]{$I_1$};
    \edge node[auto=right] {$Y_{10}$};  
      [.\node[draw,circle]{$I_{10}$}; 
       \edge[dashed]; {}
      \edge[dashed]; {}
      ] \edge node[auto=left] {$Y_{11}$}; 
    [.\node[draw,circle]{$I_{11}$};
     \edge[dashed]; {}
      \edge[dashed]; {}
     ] 
    ] ]
\end{tikzpicture}
\caption{Indexed binary tree with levels  $l\le 2$ represented. The nodes index the intervals $I_\veps$. Edges are labelled with random variables $Y_\veps$.} \label{figtree}
\end{center}
\end{figure}

\textit{Paths along the tree.} The distribution of mass in the construction \eqref{treepr}  can be visualised using a tree representation: to compute the random mass that $P$ assigns to the subset $I_\veps$ of $[0, 1]$, one follows a binary tree along the expression of $\veps$ : $\veps_1; \veps_1\veps_2,\ldots,\veps_1\veps_2\ldots\veps_l = \veps$. The mass $P(I_\veps)$ is the product of variables $Y_{\veps0}$ or $Y_{\veps1}$ depending on whether one goes `left' ($\veps_j = 0$) or `right' ($\veps_j = 1$) along the tree :\begin{equation}\label{defp}P(I_\veps) = \prod^l_{j=1, \veps_j=0} Y_{\veps_1,\ldots,\veps_{j-1}0} \times \prod^l_{j=1, \veps_j=1}(1 - Y_{\veps_1,\ldots,\veps_{j-1}0}).
\end{equation}
see Figure \ref{figtree}.  A given $\veps = \veps_1, \ldots, \veps_l \in \mathcal{E}$ gives rise to a path $\veps_1\rightarrow \veps_1\veps_2 \rightarrow \veps_1\veps_2 \ldots \veps_l$. We denote $I^{[i]}_\veps := I_{\veps_1\ldots \veps_i}$, for any $i$ in $\{1,\ldots, l\}$. Similarly, denote $$Y^{[i]}_\veps =  Y_{\veps_1\ldots \veps_i}.$$ 
  
One can continue this construction for $\veps$ of arbitrary length, obtaining an infinite binary tree. One can also truncate at a given level  $|\veps|=L$.

{\em Link with the Haar basis.} Denoting by $(\psi_{lk})$ the standard Haar basis, a simple calculation shows 
\begin{equation} \label{polyac}
\int_0^1 \psi_{lk} dP = 
 2^{l/2}P(I_\veps)(1-2Y_{\veps0}),
\end{equation}
which we may interpret as a coefficient  on the Haar basis.  
  
\begin{definition} \label{pt}
A random probability measure $P$ follows a \sbl{P\'olya tree} distribution PT$(\mathcal{A})$ with parameters $\mathcal{A} =
\{\alpha_\veps ; \veps \in \mathcal{E}\}$ on the sequence of partitions $\mathcal{I}$ if it is a tree prior distribution as in \eqref{treepr} with variables $Y_{\veps}$ that, for  $\veps \in \mathcal{E}$, are mutually independent and follow a Beta distribution
\begin{equation}\label{priorbeta}
Y_{\veps0}\sim \text{Beta}(\alpha_{\veps0}, \alpha_{\veps1}).
\end{equation}
\end{definition}
A standard assumption is that the parameters $\alpha_\veps$ only depend on the depth $|\veps|$, so that $\alpha_\veps = a_l$ for all $\veps$ with $|\veps|=l$,  any $l \geq 1$, and a sequence $(a_l)_{l\geq 1}$ of positive numbers. The class of P\'olya tree distributions is quite flexible: different behaviours of the sequence of parameters $(a_l)$ give P\'olya trees with remarkably different properties. For instance (e.g. \cite{gvbook}, Chapter 3)
\begin{itemize}
\item if $\sum_{l} a_l^{-1}$ converges, $P$ is a.s. absolutely continuous with respect to Lebesgue measure, with density $f$. This happens if $a_l$ increases fast enough, e.g. for the choice $a_l=2^{2\al l}$ with $\al>0$; 
\item the special choice $a_l=2^{-l}$ gives the Dirichlet process DP \cite{ferguson73} with uniform base measure, which verifies 
$P(A_1, \ldots, A_p) \sim \text{Dir}(|A_1|,\ldots,|A_p|)$ for any measurable partition $A_1,\ldots,A_p$ of the unit interval if $P$ is a draw from the DP, with Dir$(\cdot)$ denoting the discrete Dirichlet distribution;
\item the case $a_l=1$  gives a.s. a fractal-type measure $P$ that has a continuous distribution function but is {\em not} absolutely continuous: one obtains a type of Mandelbrot's multiplicative cascade.
\end{itemize}

We will mostly discuss the first of the three cases above, but let us mention that the Dirichlet process is a central object in the study of discrete random structures and of Bayesian nonparametrics in particular. A draw $P$ from a DP is a discrete measure with atoms at random iid locations and a special distribution for the atom's probabilities. It is a canonical prior on distributions, although it cannot be used directly in the dominated framework we consider (since the model of all distributions is not dominated). Nevertheless, it can still often be deployed for instance as a mixing random distribution (see Chapter \ref{chap:ada1}).

As it turns out, the choice  of parameters $a_l=2^{2\al l}$ with $\al>0$ is a `right one' \cite{c17} in order to model $\al$--smooth functions: this can be seen from the fact that $P$ has in this case a density $f$ so that \eqref{polyac} gives the Haar wavelet coefficients of $f$ as
\[ f_{lk} = 2^{l/2}P(I_\veps)(1-2Y_{\veps0}),\]
where $P(I_\veps)$ is a product (`cascade') of Beta variables. In particular, one can show \cite{c17} that in the density estimation model, taking as prior the one this `$\al$--regular' PT induces on densities (recall that by the first point above  PT draws have a density), the corresponding posterior distribution contracts at rate in the $L^2$ sense (up to logarithmic factors)
\[ \veps_n \leqa n^{-\frac{\al\wedge \be}{2\al+1}}, \] 
if the true density is  $\be$--H\"older and bounded away from $0$. 
This can be shown by using the following remarkable conjugacy property of PTs (stated here in a general non-dominated framework \cite{gvbook})

\begin{prop}[conjugacy of PTs in iid sampling model]$\ $ Suppose $X_1,\ldots,X_n\given P$ are iid of law $P$, and let us endow $P$ with a PT$(\cA)$ prior. Then
\[ P\given X_1,\ldots, X_n \sim \text{PT}(\cA_X),\]
where the updated parameters of the Beta variables are $\al_\veps^X:=\al_\veps + N_X(I_\veps)$, where $N_X(I_\veps)$ is the number of points in the sample that fall in $I_\veps$.
\end{prop}

\ti{Gaussian processes and P\'olya trees: an analogy} In view of the results of Section \ref{sec:GPs}, where similar contraction rates are obtained, and of similar conjugacy properties of GPs in Gaussian regression, it is natural to view $\al$--regular PTs as above as `density-estimation-analogues' of $\al$--regular GPs in regression. Pushing this analogy a bit further, the Dirichlet process can be interpreted as the analogue in the iid sampling model of a Gaussian white noise in Gaussian regression. Both objects have ``regularity $-1/2$" (recall the DP corresponds to $a_l=2^{2\cdot(-1/2)\cdot l}=2^{-l}$ as noted above), something that for white noise will be relevant in Chapter \ref{chap:bvm2}.

\vm 

\section*{Exercises}

\begin{enumerate}
\item {\em Lower bounds in parametric models.} Consider a model $\cP=\{P_{\te}^{\otimes n},\ \te\in\Theta\}$ with $\Theta=\RR$. 
\begin{enumerate}
\item Consider the fundamental model $\cP_G:=\{\cN(\te,1)^{\otimes n},\ \te\in\RR\}$ with a Gaussian $\cN(0,1)$ prior. By using the explicit expression of the posterior, show that for any $m_n\to 0$, the rate $\zeta_n=m_n/\sqrt{n}$ is a lower bound for the posterior rate.
\item Still in model $\cP_G$, verify that the following property $(P)$ holds: there exists a constant $c>0$ such that for any $\veps>0$ and $\theta_0\in\Theta$,
\[ B_{K}(\te_0,\veps) \supset \{\te:\ |\te-\te_0|\le c\veps\}. \]
That is, the model is {\em ``regular"} in that $KL$--neighborhoods are `comparable' to intervals.
\item Turning now to the general case, suppose that the model verifies property (P) defined in (b) and that the prior has a continuous and positive density with respect to Lebesgue measure on $\RR$. Prove that for any vanishing sequence $(m_n)$
\[ E_{\te_0}\Pi[|\te-\te_0|\le \frac{m_n}{\rn}\given X] \to 0, \]
that is, $m_n/\rn$ is a lower bound for the posterior contraction rate.
\end{enumerate}
\end{enumerate}

\addcontentsline{toc}{section}{{\em Exercises}}

\chapter{Adaptation I:  smoothness} \label{chap:ada1}

As we have seen, one limitation of Gaussian processes (GP) for statistical inference is that the optimal statistical rate, for instance in a regression setting, is attained only if the GP's parameter is well-chosen in view of  the smoothness of the function to be estimated. However, in practice the latter is typically unknown, leading to an {\em adaptation} problem. Below we see that, at least for the canonical distances used in the previous generic results, the adaptation question can be addressed in conceptually simple ways, both in terms of construction and proofs.

\section{General principles}

To fix ideas suppose the model is $\cP=\{P_f, f\in \cF\}$, where $f$ is a function to be estimated, and that we have a family of prior distributions $\{\Pi_\al\}$ 
 indexed by a parameter $\al\in\cA$: for instance $\al=K$ with $K$ the number of bins for regular random histograms, or $\al$ indexing the decrease of variances in the Gaussian series prior 
\begin{equation} \label{gpser}
W(\cdot) = \sum_{j\ge 1} j^{-\frac12-\alpha} \zeta_j e_j(\cdot).
\end{equation}

In this section we present two main possibilities: {\em hierarchical Bayes}, where $\al$ is itself given a prior distribution and {\em empirical Bayes}, where $\al$ is `estimated' by a data-driven quantity $\hat{\al}$ to be chosen. The former is probably the most Bayesian in spirit, and has typically the most flexibility, although the latter can be sometimes easier to compute.\\

{\em Hierarchical Bayes.} The prior $\Pi$ on $f$ takes the form
\begin{align*}
\al & \sim \pi \\
f\given \al & \sim \Pi_\al,
\end{align*}
where $\pi$ is a distribution on $\cA$. The prior is then a {\em mixture} $\Pi[f\in \cdot] = \int \Pi_\al[f\in \cdot] d\pi(\al)$. Of course, it is a special case of the usual Bayesian setting, with the prior taking this specific mixture form.\\

{\em Empirical Bayes.} One sets $\Pi=\Pi_{\hat\al}$, where $\hat\al=\hat\al(X)$ is an `estimator' of $\alpha$ to be chosen. In principle $\hat\al$ can be any measurable function of the data, although we present here a general principle that is often employed in practice, namely  empirical Bayes {\em marginal maximum likelihood} (MMLE). The idea is that the {\em marginal distribution} of $X$ given $\al$ in the Bayesian framework, whose density is the denominator in Bayes' formula written for the prior $\Pi_\al$ for fixed $\al$ i.e. 
\[ \Pi_\al[B\given X] = \frac{\int_B p_f(X) d\Pi_\al(f)}{\int p_f(X) d\Pi_\al(f)}  \]
can serve as a likelihood for $\al$. One then maximises it, setting
\[ \hat\al(X) = \underset{\al\in\cA}{\text{argmax}}\, \int p_f(X) d\Pi_\al(f). \]
Sometimes the set of maximisation is made slightly smaller to avoid `boundary' problems. Some examples will be given below.

In terms of proofs, for hierarchical Bayes priors one can use the generic approach to rates presented in Chapter \ref{chap:intro}. For empirical Bayes (EB), the situation is somewhat more complicated, as the prior depends on the data, so the arguments do not go through as such. Rousseau and Szab\'o \cite{rs17} provide a set of sufficient conditions in the spirit of the generic theorems as before to deal with EB marginal maximum likelihood. In case of simple models and when a closed-form expression of the marginal likelihood is available, it is also possible to use direct arguments.

\section{Random histograms}

In the setting of the Gaussian white noise model, in order to make the random histogram with deterministic number $K=K_n$ of bins we considered earlier {\em adaptive} to smoothness, we can simply take $K$ itself random by setting $\Pi=\Pi_H$ the hierarchical prior, with $I_{k,K}=((k-1)/K,k/K]$,
\begin{align*}
K & \sim \pi_K(\cdot),\qquad \text{with } \pi_K(k)\propto e^{-k\log{k}},\\
f\given K & \sim \cL\left( f = \sum_{k=1}^K h_{k,K} \1_{I_{k,K}},\quad (h_{1,K},\ldots,h_{k,K})\sim P_\psi^{\otimes K} \right).
\end{align*}

\begin{thm} \label{thm:randomkhist}
In the Gaussian white noise model, suppose the true $f_0$ belongs to  $\cF(\be,L,M)$ as in \eqref{classhisto} for some $\be\in(0,1]$ and $L,M>0$. Then for $\Pi=\Pi_H$ the prior with random $K$ as above 
\[ E_{f_0}\Pi[\|f-f_0\|_2\le m\veps_n\given X]\to 1,\qquad 
\veps_n\asymp \left(\frac{\log{n}}{n} \right)^{\frac{\be}{2\be+1}},
\]
as $n\to\infty$, where $m>0$ is a large enough constant.
\end{thm}
Other choices of the prior on $K$ are possible. For instance, one could take $\pi_K(k)\propto e^{-k}$. This would lead to a similar result, but with a slightly different log--factor in the rate. \\

\begin{proof}
One defines a sieve as, with $K_n=Dk_n$ for $D$ large enough to be chosen, and $k_n=(n/\log{n})^{1/(2\be+1)}$, 
\[ \cF_n = \bigcup_{k=1}^{K_n} \left\{f=\sum_{j=1}^k h_{j,k} \1_{I_{j,K}},\ 
\max_{j} |h_{j,k}| \le n \right\}=\bigcup_{k=1}^{K_n} \cF_{n,k}.     \]
The entropy condition is easily verified using $\|u\|_2\le \sqrt{k}\max_j |u_j|$ and, arguing as in the proof of Theorem \ref{thm:hist},  the estimate
$N(\veps,\cF_{n,k},\|\cdot\|_2) \le (3n\sqrt{k}/\veps)^k$, so that
\[ N(\veps,\cF_{n},\|\cdot\|_2)\le \sum_{k=1}^{K_n} (3nK_n/\veps)^k \leqa (nK_n/\veps)^{k_n+1}\]
which gives $\log N(\oli\veps_n,\cF_{n},\|\cdot\|_2)\leqa K_n\log(n/\oli\veps_n)+K_n\log K_n$. 

Also, the complement of the sieve has small prior mass as
\begin{align*}
\Pi[\cF_n^c] & \le \Pi(K>K_n) 
+ \sum_{k=1}^{K_n} \Pi[\cF_n^c\given K=k]\Pi[K=k] \\
& \leqa e^{-K_n\log{K_n}} + \sum_{k=1}^{K_n} ke^{-n}\Pi[K=k]
\leqa e^{-K_n\log{K_n}} +  e^{-n}E[K]\leqa e^{-K_n\log{K_n}} +  e^{-n},
\end{align*}
where we use that $K$ has finite expectation. Finally, for the prior mass condition,
\begin{align*}
 \Pi[\|f-f_0\|_2\le \uli\veps_n] & \ge \Pi[\{\|f-f_0\|_2\le \uli\veps_n\}\cap\{K=k_n\}] \\
& \ \ = \Pi[\|f-f_0\|_2\le \uli\veps_n\given K=k_n] \Pi[K=k_n]
\end{align*}
and we can now used the bound for fixed $K=k_n$ used in the previous chapter, which gives, provided 
 $Lk_n^{-\be}\le \uli\veps_n/2$, that 
$\Pi[\|f-f_0\|_2<\uli\veps_n\given K=k_n] 
\ge  (\uli\veps_n/2)^{k_n} \exp\{-2Mk_n\}$. \\

Putting everything together, we see that we need: $K_n\log(nK_n/\oli\veps_n)\leqa n\oli\veps_n^2$, and $Lk_n^{-\be}\le \uli\veps_n/2$ as well as $k_n\log(2/\uli\veps_n)+2Mk_n\leqa n\uli\veps_n^2$. This is satisfied for $\oli\veps_n\asymp\uli\veps_n\asymp\veps_n$ as in the statement of the result (choose first $\uli\veps_n$ with a large enough constant to verify prior mass, then $D$ large enough to verify the second condition and finally the constant in front of $\uli\veps_n$ to be large enough).
\end{proof}

\section{Adaptation for Gaussian priors}

In the next two Sections, we give constructions based on Gaussian process priors that lead to adaptation to smoothness in a conceptually simple way. We mostly refer to original papers for the proofs: these can be based on the tools for GPs using the concentration function introduced in Chapter \ref{chap:rate1}, at least for fixed hyperparameter, and making the dependence of the concentration function explicit in this hyperparameter.

\subsubsection*{Series priors with adaptive choice of $\al$}

Consider the hierarchical prior $\Pi$, for $\zeta_i$ iid standard normal and $\expo$ a standard exponential variable
\begin{equation} \label{prior-hgp}
\begin{aligned}
\al & \sim \expo(1) \\
f\given \al & \sim \Pi_\al \qquad \text{law of } \sum_{j\ge 1} j^{-\frac12-\al}\zeta_j e_j(\cdot)
\end{aligned}
\end{equation}
The exponential prior for $\alpha$ is just one possibility, Gamma or heavier tailed densities are also possible. 

For the empirical Bayes approach, one can follow a marginal maximum likelihood approach. 
Projecting the white noise model onto the basis $(e_j)$, setting $Y_j=\int e_j(u)dX^{(n)}(u)$ and $Y=(Y_j)$, the marginal distribution of $Y\given \al$ can be shown to be \cite{ksvv16}
\[ Y\given \al \sim  \bigotimes_{j=1}^\infty \cN\left(0,j^{-1-2\al}+\frac1n\right), \]
and the relative log-likelihood (with respect to an infinite product of $\cN(0,1/n)$ variables) is 
\[ \ell_n(\al) = -\frac12 \sum_{j=1}^\infty 
\left( \log\left(1+\frac{n}{j^{1+2\al}} \right)-\frac{n^2}{j^{1+2\al}+n}Y_j^2 \right).
\]
One may then set, choosing $\cA_n=[0,\log{n}]$, 
\begin{equation} \label{aleb}
\hat\al = \underset{\al\in\cA_n}{\text{argmax}}\, \ell_n(\al).
\end{equation}
Let $\cS(\be,R)$ be the Sobolev ball $\{f:\ \sum_j j^{2\be} f_j^2 \le R\}$ with $f_j=\int_0^1 f(u) e_j(u)du$.

\begin{thm}{\cite{ksvv16}} \label{thm-gpa} Suppose $f_0$ belongs to $S(\be,R)$ for some $\be,R>0$.\\

{\em Hierarchical Bayes.} Consider the prior $\Pi$ in  \eqref{prior-hgp} on $f$ in the Gaussian white noise regression model.  Then there exists a positive constant $l$ such that
\[ E_{f_0} \Pi[ \|f-f_0\|_2 > (\log{n})^l n^{-\frac{\be}{2\be+1}} \given Y] = o(1). \]

{\em Empirical Bayes.} Let $\hat\al$ be defined by \eqref{aleb}. Then the plug-in posterior $\Pi_{\al}[\cdot\given Y]$ verifies, for $l$ large enough,
\[ E_{f_0} \Pi_{\hat{\al}}[ \|f-f_0\|_2 > (\log{n})^l n^{-\frac{\be}{2\be+1}} \given Y] = o(1).  \]
\end{thm}
The original proof in \cite{ksvv16} uses a direct argument based on the explicit form of $\ell_n(\al)$;   \cite{as23} for $p$-exponential  $\zeta_j$'s uses an approach based on generic arguments as in \cite{rs17}.

\subsubsection*{Abstract GPs with random scaling}

Suppose we start with a smooth Gaussian process on $[0,1]$. It turns out that smoothness of the GP is related to smoothness of its covariance: a popular choice is the {\em squared-exponential} GP, denote $\sqe(a)$ in the sequel $(a>0)$ which is the mean-zero GP $Z$ with covariance 
\begin{equation}\label{sqk}
 K(s,t) = E[Z(s)Z(t)] = e^{-a^2(s-t)^2}, \qquad s,t\in[0,1], 
\end{equation} 
where $a^2$ is called the inverse length scale parameter. This process has same distribution as $t\to W(at)$ where $W$ is a $\sqe(1)$ process, and can be seen to have analytic sample paths. Its small ball behaviour is obtained in \cite{vvvz09}
\begin{equation}\label{sqe-sb}
 \vphi_0(\veps) = - \log P\Big( \sup_{t\in[0,1]} |Z(t)| \le \veps\Big) 
\le Ca \left(\log\frac{a}{\veps}\right)^2. 
\end{equation}
The small-ball probability decreases much slower than for e.g. Brownian motion, reflecting that the process varies much less, which does not sound very attractive for adaptation: in fact, by using a lower bound argument as in Section \ref{sec:lb} it can be shown that the convergence rate of the posterior associated to $\sqe(a)$ cannot be better than $1/(\log{n})^{-l}$ for some constant $l$ \cite{vvvz11}! However, one may counter this by making $a$ random. Observe indeed in Figure \ref{plotsqe} that taking $a$ large makes the paths more wiggly by `accelerating time'. On the contrary, making $a$ smaller would tend to `freeze' the paths, an effect we will use later in Chapter \ref{chap:ada2}.
\begin{center} \label{plotsqe}
\includegraphics[height=4cm, width=8cm]{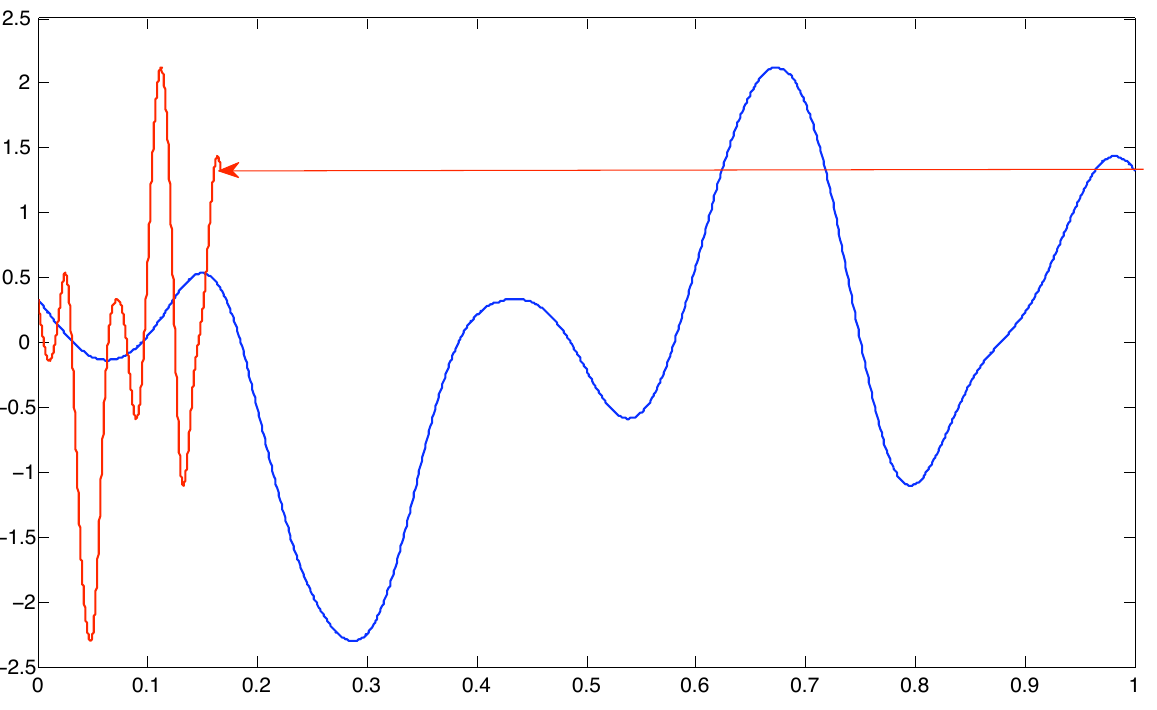}
\end{center}
Consider the following hierarchical prior $\Pi$ on smooth functions
\begin{equation} \label{prior-sqa}
\begin{aligned}
A & \sim \expo(1) \\
Z\given A & \sim \sqe(A),
\end{aligned} 
\end{equation}
for $\expo(1)$ a standard exponential distribution (other Gamma laws work, too).

\begin{thm}{\cite{vvvz09}} \label{thm-sqa}
Consider the prior $\Pi$ in \eqref{prior-sqa} on $f$ in the Gaussian white noise regression model. Suppose $f_0$ belongs to $\cH(\be,L)$ for some $\be,L$. Then there exist a positive constant $l$ such that
\[ E_{f_0} \Pi[ \|f-f_0\|_2 > (\log{n})^l n^{-\frac{\be}{2\be+1}} \given X] = o(1). \]
\end{thm}

The proof of this result in \cite{vvvz09} is based on obtaining a precise dependence in $a$ for the concentration function of the GP, in order to apply the theory from Chapter \ref{chap:rate1} when verifying the conditions of the generic Theorem \ref{thm:ggve}. First, one bounds from above the entropy of the GP's RKHS unit ball $\mh_1$ in terms of $a$. This enables one to derive the dependence in $a$ for the small ball probability \eqref{sqe-sb}, using general links between the entropy of $\mh$ and the small ball probability \cite{KuelbsLi93}. The approximation term of the concentration function is then handled separately. Finally, one checks that the choice of the prior on $a$ is compatible with the conditions of Theorem \ref{thm:ggve}.

\section{Geometric priors}
 
\ti{Geometric spaces} Consider a geometric framework such as density estimation \eqref{gdens} on a compact metric space $\M$ with metric $\rho$ and equipped with a Borel measure $\mu$,  or a corresponding white noise model 
\begin{equation} \label{ggwn}
dX^{(n)}(x) = f(x)dx + \frac{1}{\rn} dZ(x),\quad x\in \M, 
\end{equation}
where $f$ is square-integrable on $\M$ and $Z$ is a white noise on $\M$.

There is a simple reason why the squared-exponential kernel cannot be used in such a context. Although \eqref{sqk} admits the immediate generalisation, for $\rho$ the metric on $\M$,
\begin{equation} \label{sqexp-rho} 
\kappa_{\rho}(s,t) = e^{-\rho(s,t)^2},\qquad (s,t)\in\M^2, 
\end{equation} 
it can be shown that this function is {\it not} positive definite in general 
  already for the simplest examples such as $\M$ taken to be the sphere in $\R^k$, $k\ge 2$. Yet, we shall see below that \eqref{sqk} admits a natural generalisation  to this context, but it is not as simple as \eqref{sqexp-rho}.\\

For simplicity we present in an overview of the construction in \cite{ckp13}, giving pointers to the paper for details when appropriate. We take the case of the sphere $\M=\mathbb{S}^2$ as recurrent illustration. Let $B(x,r)$ denote the ball of center $x$ and radius $r$ for the metric $\rho$ on $\M$. 
Suppose that $\M$ verifies the so-called {\it Ahlfors property}: there exist positive $  c_1, c_2,  d$ such that 
\begin{equation}\label{poly-case} \hbox{ for all } x \in \M, \;  \hbox{for all } \;   0 <r  \leq 1, \;\ \  c_1 r^d \leq | B(x,r)| \leq c_2 r^d.
\end{equation} 
In the case of the sphere $\mathbb{S}^2$, $d=2$. More generally $d$ in the sequel can be thought of as the `dimension' of $\M$ although in general $d$ could be non-integer. \\

\ti{Operator $L$, Laplacian and heat kernel}
The starting point is a self-adjoint positive operator $L$ on functions  on $\M$ 
(more precisely on a domain $D$ dense in $L^2(\M)$, the space of square integrable functions with respect to the measure $\mu$). 
When defined, {\em minus} the {\em Laplacian} on $\M$, that is $L=-\Delta_\cM$ is typically appropriate.
Suppose $L$ admits a discrete spectrum with finite dimension spectral spaces 
$\cH_k=\text{Vect}\{(e_k^l),\ 1\le l\le \text{dim}(\cH_k)\}$ and that its eigenfunctions $e_k^l$ are continuous functions on $\M$. The numbering is chosen so that the eigenvalues are ordered in an increasing order. Also, in all this section, sums over $k$ and $l$ range over $1\le k\le\text{dim}(\cH_k)$ and $1\le l < \infty$. Under some conditions, the following series converges to a continuous function
\begin{equation} \label{heat-kernel}
P_t(x,y) := \sum_{k} e^{-t\la_k} \sum_{l} e_k^l(x) e_k^l(y),
\end{equation}
on $\cM\times \cM$, called the {\em heat kernel}. Let us justify this terminology in an informal way when $L=-\Delta_\cM$. By informally differentiating under the series sign, we see that $P_t(\cdot,y)$ for any fixed $y$ is a solution in $g$ of the heat equation 
\begin{equation}\label{heat}
\frac{\partial g}{\partial t}= (-L)g = \Delta_\cM g.
\end{equation}
For a more formal characterisation of $P_t$, in particular the connection to semi-groups, see \cite{ckp13} Section 2.4 and references therein. 

The orthonormal basis of $L^2(\M)$ generated by $\{e_k^l\}$ can be interpreted as a {\em harmonic analysis} over $\M$. In the case of the sphere
 $\M=\mathbb{S}^2$, the eigenvectors of the Laplacian $\Delta_{\mathbb{S}^2}$ are well-known: these are the {\em spherical harmonics}, which have explicit expressions in terms of homogeneous polynomials of three variables and in this case $\la_k=k(k+1)$ (see \cite{ckp13}, Section 3). \\

\ti{Decoupling time and space} Let us  note the presence of the indexing variable $t$, the `time', in \eqref{heat-kernel}. For the squared-exponential kernel on the real-line, the prior can be made more flexible by stretching the path along the `$x$'-axis, that is the {\em space} domain. Since in general there is no natural analogue of stretching on a geometric object, a natural idea is to stretch {\em time} instead. Indeed, our procedure puts a prior on time as we describe below. Let us now discuss a further property of $P_t$.\\

\ti{Estimates for the heat kernel} 
The following Gaussian-like estimates of the heat kernel $P_t$ are satisfied 
in a surprisingly large variety of situations, in particular on all compact manifolds without boundary, see e.g. Grigor'yan \cite{Grigoryan}, for instance on the sphere. We assume them to hold:  suppose that  there exist $C_1, C_2>0, c_1, c_2 >0$, {such that}, $ \hbox{for all } t \in ]0,1[, $ and any $x,y\in \M$,
\begin{equation}\label{borngauss}
 \frac{C_2e^{-\frac{c_2 \rho^2(x,y)}t}}{|B(x, \sqrt t)|^{1/2} |B(y, \sqrt t)|^{1/2}} 
\leq P_t(x,y) \leq \frac{C_1e^{-\frac{c_1 \rho^2(x,y)}t}}{|B(x, \sqrt t)|^{1/2} |B(y, \sqrt t)|^{1/2}} 
,
\end{equation}
where $|B(x,r)|$ denotes the volume of the ball $B(x,r)$. \\

\ti{Geometric prior} 
A prior on functions from $\M$ to $\mathbb{R}$ is constructed hierarchically as follows. 
 
First, generate a collection of independent standard normal variables 
$\{X_k^l \}$ with indexes $k\ge 0$ and $1 \leq l \leq dim(\cH_{\lambda_k})$. Set, for $x\in\M$ and any $t\in(0,1]$, 
\begin{equation} \label{priorwt}
  W^t(x) = \sum_{k} \sum_{l} e^{-\lambda_kt/2}  X_k^l e^l_k(x).  
\end{equation}
This process is centered and has covariance kernel precisely $P_t$, as follows by direct computation,
\[ \mathbb{E}(W^t(x)W^t(y)) = P_t(x,y). \]

Second, draw a positive random variable $T$ according to a density $g$ on $(0,1]$. This variable can be interpreted as a random scaling, or random `time'. It turns out that convenient choices of $g$ are deeply connected to the geometry of $\M$. We choose the density $g$ of $T$ such that, for a positive constant $q$, with $d$   defined in \eqref{poly-case}, 
\begin{equation} \label{priorT}
  g(t) \propto e^{-t^{-d/2}\log^q(1/t)}, \quad t\in (0,1].
\end{equation}
We show below that the choice $q=1+d/2$ leads to sharp rates.

 The full (non-Gaussian) prior we consider is $W^T$, where $T$ is random with density $g$. That is,
\begin{equation} \label{priorw}
 W^T(x) =  \sum_k \sum_{l} e^{-\lambda_kT/2}  X_k^l e^l_k(x),
  \end{equation} 
and we define $\Pi$ as the  prior on functions on $\cM$ induced by $W^T$.\\

\noindent {\sc Does the prior \eqref{priorT} relate to the square-exponential GP ?}
So far there does not seem to be a direct connection between our construction and that of  \cite{vvvz09}. However, such a connection becomes apparent when taking another look at equation \eqref{borngauss}. We see that the covariance kernel of $W^t$ for a given $t$ very closely relates to $e^{-c\rho^2(x,y)/t}$, which however is not itself in general a covariance kernel as noted above. In this sense, the heat kernel {\em is} the natural generalisation of
the squared-exponential kernel $e^{-C(x-y)^2}$ to geometric spaces.\\
  
\ti{Sketch of required arguments}
To obtain convergence rates corresponding to the prior $\Pi$ and derive Theorem 
\ref{thm-sph} below, we use the general rate Theorem \ref{thm:ggve}. As seen in Chapter \ref{chap:rate1} for Gaussian processes a rate is obtained by solving in $\veps_n$ the equation $\vphi_{f_0}(\veps_n)\le n\veps_n^2$, with $\vphi$ the concentration function of the process. Here $\Pi$ is not Gaussian, but conditionally on a given value of $T$, say $T=t$, the prior induced by $W^t$ is Gaussian by construction. So, an important step in the proof is the study of the concentration function $\vphi_{f_0}$ of $W^t$ at the true function $f_0$, which involves an approximation term as well as the small ball probability of $W^t$.

The approximation part of $\vphi$ requires some regularity condition on $f_0$; it turns out that it is particularly natural to work with a scale of Besov spaces, which  may precisely be defined in terms of quality of approximation. Define first the `low frequency' functions from the eigenspaces $\cH_\la$ as
\[\Sigma_t=\bigoplus_{\la\le\sqrt{t}}\cH_\la.\]
Next, let $\cE_t(f)_p:=\inf_{g\in\Sigma_t}\|f-g\|_p$ denote the best approximation of
$f \in L^p=L^p(\cM)$ from $\Sigma_t$. Then the Besov space $B_{pq}^s(\cM)$ 
is defined as
\begin{equation} \label{geo-besov}
 B_{pq}^s(\cM):=\{ f\in L^p,\ \ \|f\|_{A_{pq}^s}:= \|f\|_p +
\Big(\sum_{j\ge 0}\big(2^{s j}\cE_{2^j}(f)_p \big)^q\Big)^{1/q}< \infty \}. 
\end{equation}
Assuming a  $B_{2,\infty}^s(\cM)$-regularity in the white noise case and a $B_{\infty,\infty}^s(\cM)$-regularity in the density estimation case enables a control of the approximation part.
 
The study of the small ball probability of the process $W^t$ is more delicate, especially since we look for sharp rates. We achieve this by using the general very precise link existing for Gaussian processes between small ball probability and entropy of the RKHS, as established in \cite{KuelbsLi93}. For this, we need first the expression of the RKHS say $\mh^t$ of $W_t$.\\
 
\ti{The prior $W^t$ and its RKHS $\mh^t$} For any $t>0$, it follows from the expression of $W^t$ that
\begin{equation} \label{rkhs-ht}
 \bH^t = \Big\{\, h = \sum_k \sum_l a_k^l  e^{-\lambda_kt/2}e_k^l, \qquad \sum_{k,l}|a_k^l|^2 <\infty\, \Big\},
\end{equation}
equipped with the inner product 
\[ \langle\ \sum_k \sum_{l} a_k^l  e^{-\lambda_kt/2}e_k^l \ ,\  
 \sum_k \sum_{l} b_k^l  e^{-\lambda_kt/2}e_k^l \ \rangle_{\bH^t} 
 = \sum_k \sum_{l} a_k^l  b_k^l.\]
Let us further denote $\bH^t_1$ the unit ball of $\bH^t$.\\

\ti{Key estimates} The next result is a sharp entropy estimate of the RKHS unit ball $\mh^t_1$, uniform in a range of time parameters $t$. The statement brings together {\em geometry} via the
covering number $N(\epsilon, \M,\rho) $ of the space $\M$,  {\em probability} via 
the RKHS of the process $W^t$ and {\em approximation}, via the entropy of $\bH^t_1$, denoted $H(\cdot,\bH^t_1,D)=\log N(\cdot,\bH^t_1,D)$, for a given distance $D$ on $\bH$.

\begin{thm} \label{entropyconnection} Suppose the space $\M$, the operator $L$  and its eigenfunctions $e_k^l$ verify the properties listed above. For $t>0$, let $\mh^t$ be defined by \eqref{rkhs-ht}. Let  us fix $a>0, \nu>0$. There exists $\epsilon_0>0$ such that for $\epsilon, t$ with $ \epsilon^\nu \leq at$
and  $ 0<\epsilon\leq \epsilon_0,$ 
\[ H(\epsilon, \bH^t_1, \|\cdot\|_2) \asymp   H(\epsilon, \bH^t_1, \|\cdot\|_\infty) \asymp  N(\delta(t,\epsilon), \M, \rho)  \cdot \log \frac 1\epsilon,
\quad \hbox{with}\quad \frac 1{\delta(t,\epsilon)}:=
  \sqrt{\frac 1t  \log \frac 1\epsilon}. \] 
\end{thm}
Under the assumption \eqref{poly-case} that balls have a polynomially increasing volume in terms of their radius, the covering number $N(\eta,\M,\rho)$ of $\M$ is shown to be $N(\eta,\M,\rho) \asymp \eta^{-d}$, which yields the estimate $\delta(t,\epsilon)^{-d}\log(1/\epsilon)$ for the entropy in Theorem \ref{entropyconnection}.  From this one can deduce an estimate of the same order  $-\log\mathbb{P}(\|W^t\|_2<\epsilon)\asymp-\log\mathbb{P}(\|W^t\|_\infty<\epsilon)\asymp t^{-d/2} \log^{1+d/2}(1/\epsilon)$ for the small ball probabilities, both in terms of the $L^2$- and $L^\infty$-norms. \\ 
    
\ti{Convergence rate for the geometric prior} 
The following theorem states a result for the white noise and density estimation problems on $\M$. In the first case, the prior is directly the law on $L^2(\cM)$  induced by $W^T$ in \eqref{priorw}. In density estimation, the prior is the image measure of the law of $W^T$ viewed as a random element of $\cC^0(\cM)$ under the exponential transformation $w\to p_w^{[\cM]}=e^w/\int_\cM e^w$ on $\cM$. Recall  the definition of the Besov spaces from \eqref{geo-besov}.

\begin{thm} \label{thm-sph}
Let the set $\M$ and the operator $L$ satisfy the properties listed above. 
Consider the white noise model \eqref{ggwn} on $\cM$.  Suppose that $f_0$ is in the Besov space $B_{2,\infty}^\beta(\M)$ with $\beta>0$ and that the prior $\Pi$ 
on $f$ is $W^T$ given by \eqref{priorw}. Let $q=1+d/2$ in \eqref{priorT}. Set $\veps_n = (\log{n}/n)^{2\be/(2\be+d)}$.
 For $M$ large enough, as $n\to\infty$,
\[ \Pi( \|f - f_0\|_2 \ge M\veps_n\ |\ X) \to^{P_0} 0. \]
Consider the density model \eqref{gdens} on $\cM$. Suppose that $\log f_0$ is in the Besov space $B_{\infty,\infty}^\beta(\M)$ with $\beta>0$ and that the prior $\Pi$ 
on $f$ is the law induced by $p_{W^T}^{[\cM]}$ with $W^T$ as in \eqref{priorw}. With $q, \veps_n$ as before and $h$ the Hellinger distance between densities on $\cM$, for $M$ large enough, as $n\to\infty$,
\[ \Pi( h(f,f_0) \ge M\veps_n\ |\ X) \to^{P_0} 0. \]
\end{thm}  
Uniformity in the results can be obtained on balls of the considered Besov spaces.

\section{Heavy tailed series priors}

Let us come back in this Section to series priors of the type, for $(e_j)$ an orthonormal basis of $L^2[0,1]$,
\begin{equation*} 
 W_t = \sum_{j=1}^\infty \sigma_j \zeta_j e_j(t),
\end{equation*} 
where now $\zeta_j$ are iid  variables with some distribution to be specified and $\sigma_j$ some {\em deterministic} decreasing sequence to be chosen. 

Here we present a recent observation made in \cite{ac23}: heuristically taking $p\to 0$ in the posterior contraction rate for $p$--exponential priors \eqref{ratep} (note that the  result  in \cite{agapiouetal21} relies on log-concavity and thus established only for $1\le p\le 2$) suggests that adaptation to the regularity may be obtained ``for free" with series priors (at least in the oversmoothing regime $\al>\be$) if one takes a heavy-tailed density for the common density $h$ of the $\zeta_j$'s. We show in \cite{ac23} that this is actually the case in some generality, and we prove here a special sub-case of the results: take to fix ideas a symmetric continuous density $h$ for the $\zeta_j$'s that is decreasing on $[0,+\infty)$ and has heavy tails in the sense that,  for $x\ge 1$ and some $a\ge 4$ and $c, C>0$, 
\begin{equation} \label{ht:conddens}
h(x)\le Cx^{-4},\qquad \text{and}\qquad h(x)\ge cx^{-a}.
\end{equation}
For instance, many Student densities verify this. 
Further take the sequence $(\sigma_k)$ to be, for $k\ge 1$,
\begin{equation} \label{defsigr} 
\sigma_k = e^{-(\log k)^2}.  
\end{equation} 
Define, for $\be, L>0$, the Sobolev ball
 \[ S_L(\be)=\Big\{f=(f_k),\quad \sum_{k\ge 1} k^{2\be} f_k^2\le L^2 \Big\}.\]  
 
\begin{thm} \label{thm-ht}
In Gaussian white noise regression, let $\Pi$ be a series prior with $\zeta_j$ having a symmetric density satisfying \eqref{ht:conddens} and $(\sigma_k)$ as in \eqref{defsigr}. Suppose $f_0\in S_L(\be)$ for some $\be, L>0$. Then
\[ E_{f_0}\Pi[ \{f:\ \|f-f_0\|_2 > \cL_n  n^{-\frac{\be}{2\be+1}} \} \given X] \to 0, \]
as $n\to\infty$, where $\cL_n= (\log{n})^{d}$ for some $d>0$. 
\end{thm} 

Hence heavy-tailed series priors yield automatic adaptation to smoothness! The choice of $(\sigma_k)$ as in \eqref{defsigr} is important: it `forces' the prior to be in the `oversmoothing' regime (for which one suspected using the heuristics above that adaptation could happen). The idea is then that the heavy tails enable the posterior to pick the presence of signal in the data, even though the baseline prior is very `smooth' (through $\sigma_k$ decreasing quite fast). \\  

Under the setting of Theorem \ref{thm-ht}, it can also be shown, to make the link with \eqref{ratep} more precise, that if $h$ is chosen as above and with now $\sigma_k=k^{-1/2-\al}$ for some $\al>0$ then adaptation occurs ``for free" if $\al\ge \be$ (oversmoothing case), i.e. the rate in this case is $n^{-\be/(2\be+1)}$ while for $\be>\al$ the rate is $n^{-\al/(2\al+1)}$, both up to log terms: this is indeed exactly \eqref{ratep} (up to logs) where one has set $p=0$.   \\
 
\begin{proof}[Proof of Theorem \ref{thm-ht}]
We just show here the prior mass property: thanks to Theorem \ref{alpost} and the comments below it for white noise regression, this already immediately implies a rate for the tempered $\rho$-posterior, $0<\rho<1$. A proof for the classical posterior can be found in \cite{ac23}.

Let $K\ge 2$ be an integer, and for a function $f$ in $L^2$, let $f^{[K]}$ denote its projection onto the linear span of $e_1,\ldots,e_K$ and $f^{[K^c]}=f-f^{[K]}$. Then 
\begin{align*}
\Pi&[\|f-f_0\|_2 < \veps] \ge 
\Pi\left[ \|f^{[K]}-f_0^{[K]}\|_2 < \veps/2 \,,\, \|f^{[K^c]}-f_0^{[K^c]}\|_2 < \veps/2\right] \\
& \ge \Pi\left[\forall\, k\le K,\ \  |f_k - f_{0,k} | \le \frac{\veps}{2\sqrt{K}}\ ;\ \forall\, k>K,\ \  |f_k| \le \frac{\veps}{D\sqrt{k}\log{k}} \right] \1_{\|f_0^{[K^c]}\|_2<\veps/4}\\
& = \prod_{k=1}^{K} \Pi\left[ |f_k - f_{0,k} | \le \frac{\veps}{2\sqrt{K}} \right] \cdot \Pi\left[\forall\, k>K,\ \  |f_k| \le \frac{\veps}{D\sqrt{k}\log{k}} \right] \1_{\|f_0^{[K^c]}\|_2<\veps/4},
\end{align*}
using independence under the prior, the fact that $k^{-1/2}/\log(k)$ is a square-summable sequence and with $D$  a large constant. Denote by $\cP_2$ the prior probability before the indicator in the last display.

Suppose the indicator in the last display equals one, which imposes $\|f_0^{[K^c]}\|_2<\veps/4$, for which a sufficient condition is
\begin{equation}
K^{-2\be} L^2 <\veps^2/16
\end{equation}
  if $f_0$ is in $\cS(\be,L)$. Let us now bound each individual term $p_k:=\Pi[ |f_k-f_{0,k}|\le \veps/(2\sqrt{K})]$. By symmetry, for any $k\le K$, one can assume $f_{0,k}\ge 0$ and
\begin{align*}
 p_k & \ge \int_{f_{0,k}}^{f_{0,k}+\veps/(2\sqrt{K})} \sigma_k^{-1}h(x/\sigma_k) dx  \ge \frac{\veps}{2\sqrt{K}} h(C/\sigma_K) \\
 & 
        \ge c\veps e^{-C_1\log(C/\sigma_K)},
\end{align*} 
using that $(\sigma_k)$ is decreasing as well as $x\to h(x)$ on $[0,\infty)$ by assumption, that $\log(1/\{2\sqrt{K}\})\ge \log(\sigma_K/C)$, and $f_{0,k}+\veps/(2\sqrt{K})\le C$ since $|f_{0,k}|$ are bounded by $L$ for $f_0\in \cS(\be,L)$.   So 
\begin{align*}
 \prod_{k=1}^{K} p_k & \ge \veps^K \exp\left\{ -C_1K\log(C/\sigma_K) \right\}  \ge \veps^K\exp\{ -C_2K\log^{2}{K}\}.
\end{align*} 
On the other hand, we also have, for $\overline{H}(x)=\int_{x}^{+\infty}h(u)du$ and $\cP_2$ as defined above,
\begin{align*}
\cP_2 &=\prod_{k>K} (1-2 \overline{H}(\veps/\{D\sigma_k\sqrt{k}\log{k}\}))\\
& = \prod_{k>K} (1-2 \overline{H}(\veps e^{\log^2{k}}/\{D\sqrt{k}\log{k}\})).
\end{align*}
Set $\veps=DK^{-\be}$, then for large enough $K$, we have 
$\veps e^{\log^2{k}}/\{D\sqrt{k}\log{k}\}\ge 1$ for $k>K$ and 
\[\overline{H}(\veps e^{\log^2{k}}/\{D\sqrt{k}\log{k}\})
\le c_2(\veps e^{\log^2{k}}/\{D\sqrt{k}\log{k}\})^{-2}
\le C_3 e^{-\log^2{k}},
 \]
so that $\cP_2\ge \exp\{ -C\sum_{k>K} e^{- .5\log^2{k}} \}\ge \exp\{-C' e^{-.5\log^2{K}}\}$ which is bounded from below by a constant, so the final bound obtained for the probability at stake is $\exp\{-C'K\log^{2}{K}\}$. By identifying the latter with $\exp(-Cn\veps_n^2)$, one obtains $K=n^{1/(1+2\be)}(\log{n})^{-2/(1+2\be)}$ and $\veps_n^2=n^{-2\be/(2\be+1)}(\log{n})^{4\be/(1+2\be)}$. 
\end{proof} 
 
\section{Adaptive priors: further options} 

We now briefly discuss a few other possibilities to derive adaptation to smoothness: the first concerns so-called mixture priors that are very popular in density estimation (and can be seen as Bayesian versions of kernel density estimators). The two other concern priors that model (directly or indirectly) wavelet coefficients, so can be viewed as building up on series priors as presented before.

\subsubsection*{Mixtures for density estimation}

To fix ideas, we consider a very specific class of prior distributions on densities on $\RR$ called {\em location mixtures}, and restrict for simplicity to one specific kernel, namely the Gaussian kernel. For $\sigma>0$,  denote by $\phi_\sigma$ the density of a $\cN(0,\sigma^2)$ variable.

For a given $\sigma>0$ and a given distribution $F$, let 
\begin{equation} \label{mixtg}
p_{F,\sigma}(x) = \int_{-\infty}^{\infty} \phi_\sigma(x-z)dF(z).
\end{equation}
Note that $p_{F,\sigma}$ is itself a density on the real line. 
 
A natural way to build a prior distribution on densities is to draw independently $\sigma$ and $F$ at random. For instance, one may take a Gamma distribution on $\sigma$, and a Dirichlet process prior (see Chapter \ref{chap:rate1}) for $F$. In the model of density estimation on $\RR$, under some mild conditions on the tails of the true density $f_0$, the posterior distribution can be shown to contract \cite{krv10} (see also \cite{gvbook}, Section 9.4) around $\be$--smooth $f_0$'s at rate 
\[ \veps_n \asymp (\log n)^q n^{-\frac{\be}{2\be+1}}, \]
for some $q>0$, and this {\em for any given} $\be>0$. On top of being adaptive to the smoothness of the density, a  remarkable property of the posterior here is that, even if the kernel is Gaussian (and thus of `order 2', i.e. $\int x\phi_\sigma(x)dx=0$ but $\int x^2\phi_\sigma(x)dx\neq 0$), the optimal rate is obtained for any $\be>0$, and not only for $\be\le 2$ as in the case of the standard kernel density estimator with Gaussian kernel. This phenomenon was first observed by Rousseau \cite{r10} for mixtures of Beta densities, and relies on the fact that approximation properties through posteriors are richer than with the simple classical kernel estimator, that must approximate $f_0$ through its expectation, i.e. through the convolution $f_0*\phi$.

\subsubsection*{Spike-and-slab priors}

Let us turn back (once again!) to series priors defined on, say, a wavelet basis $(\psi_{lk})$. Instead of putting a continuous distribution on coefficients as considered before, let us allow the prior to set some coefficients to $0$, which leads to 
\begin{equation} \label{sasnp}
 f_{lk} \sim (1-w) \delta_0 + w \Gamma,
\end{equation} 
independently for all $l,k$, for some weight $w\in(0,1)$ and $\Gamma$ a distribution with a density on $\RR$ (one generally stops at $l=n$ and set all coefficients for larger $l$'s to $0$). 
The Dirac mass at $0$ is called the {\em spike}, while $\Ga$ is the {\em slab} part of the prior. This prior was first introduced in a quite different setting (e.g. \cite{mb88} in the linear regression model) with the goal of model selection in mind, and indeed it plays a central role in high-dimensional models when one wishes to induce sparsity, see Chapter \ref{chap:ada2}. Let us just mention that this prior is connected to thresholding methods (not very surprisingly, as it sets some coefficients to $0$).
In the Gaussian white noise model, the choices of prior parameters
\[ w=w_n=1/n,\qquad \Gamma=\text{Lap},\]
for $\text{Lap}$ the standard Laplace distribution lead to a posterior that converges to $\beta$--H\"older $f_0$'s  in the $\|\cdot\|_2$--sense at the near-minimax rate \cite{hrs15} 
\[ \veps_n = \left( \frac{\log{n}}{n} \right)^{\frac{\be}{2\be+1}}.\]
This can be proved using the relatively explicit characterisation of the posterior in white noise: since coordinates are independent, the posterior is a product and also of spike-and-slab form. In fact, it can be shown that convergence in the $\|\cdot\|_\infty$--sense at the same rate also holds. Similar results hold in  density estimation for exponentiated-and-renormalised spike-and-slab priors \cite{zn22}. \\

In density estimation, we have seen in Chapter \ref{chap:rate1} that a natural way to build a prior on densities is via P\'olya trees. It is possible to define a {\em spike-and-slab P\'olya} tree distribution. To do so, the idea proposed in \cite{cm21} is to replace the Dirac mass at $0$ in \eqref{sasnp} by the Dirac mass at $1/2$, leading to a new distribution on the variables $Y_{\veps0}$ along the tree in \eqref{priorbeta} given by
\[ Y_{\veps0} \sim (1- w_\veps)\delta_{1/2}+w_\veps \text{Beta}(a_{\veps0}, a_{\veps1}). \]
independently along the tree. It is shown in \cite{cm21} that, truncating the tree at depth $n/\log^2(n)$, and with choice of weights $w_\veps=w_{|\veps|}\propto e^{-\kappa |\veps|}$, for $\ka$ a large enough constant, the posterior converges at the minimax rate in terms of the supremum norm $(\log{n}/n)^{\be/(2\be+1)}$ for any $\be\in(0,1]$.

\subsubsection*{Tree priors}

Starting from a binary tree as in Figure \ref{figtree}, and noting that the idea of spike-and-slab priors is to `activate' (i.e. to set to a non-zero value) a certain (arbitrary) set of wavelet coefficients, one can instead opt for a more `structured' prior that activates coefficients only if they belong to a certain {\em finite} binary tree, e.g. the one with colored nodes in Figure \ref{treecart}. The tree can itself be chosen at random, e.g. according to a Galton--Watson process $\Pi_{\mT}$ with a.s. extinction. This gives the scheme
\begin{align*}
\cT & \sim \Pi_\mT \\
f\equiv(f_{lk}) \given \cT & \sim \Pi_{f\given \mT},
\end{align*}
where one possible choice for $\Pi_{f\given \mT}$ is simply one that makes coordinates independent
\[ \Pi_{f\given \mT}=\bigotimes_{l,k} \pi_{lk},\qquad \pi_{lk}=\ind{(l,k)\in\cT}\Gamma + \ind{(l,k)\notin\cT}\delta_0. \]
It can also be interesting to have dependencies between the coordinates, as we mention next.

One interest of this construction is its link with random tree methods such as CART or Bayesian CART. For instance, if one restricts to binary splits, Bayesian CART \cite{cart1, cart2} induces a prior on functions that corresponds to the above prior with $(\psi_{lk})$ the Haar basis, with a certain (dependent) prior on Haar-wavelet coefficients given $\cT$, see \cite{cr21} for a detailed construction and proofs. It is shown in \cite{cr21} that this construction leads, in white noise regression, to a posterior that converges to $\beta$--H\"older $f_0$'s  in the $\|\cdot\|_2$--sense at the near-minimax rate \cite{cr21} 
\[ \veps_n = \left( \frac{\log^2{n}}{n} \right)^{\frac{\be}{2\be+1}},\]
with the same rate being also achieved for the $\|\cdot\|_\infty$--norm (this holds for $\be\in(0,1]$ if the Haar basis is chosen, and up to a given arbitrary smoothness $S>0$ if an  $S$--smooth wavelet basis is chosen). 

\begin{figure}
\begin{center}
\tikzset{every node/.style={rectangle,  fill=gray, text=black}
  }
\begin{tikzpicture}[
   level distance=1.4cm,sibling distance=1cm, 
   edge from parent path={(\tikzparentnode) -- (\tikzchildnode)}]
\Tree [.\node[draw,circle,label={(0,0)},color=pennblue] {}; 
    \edge ;
    [.\node[draw,rectangle,label={(1,0)},color=pennred]{};  
          ]
     \edge;
    [.\node[draw,circle,label={(1,1)},color=pennblue]{};
    \edge ;
      [.\node[draw,circle,label={(2,2)},color=pennblue]{}; 
       \edge ;
      [.\node[draw,rectangle,label={(3,4)},color=pennred]{}; 
            ]  
      \edge ;
      [.\node[draw,rectangle,label={(3,5)},color=pennred]{};     
           ] 
] 
\edge ;
    [.\node[draw,rectangle,label={(2,3)},color=pennred]{};
         ] 
    ] ]
\end{tikzpicture}
\caption{Finite tree of activated coefficients} \label{treecart}
\end{center}
\end{figure}
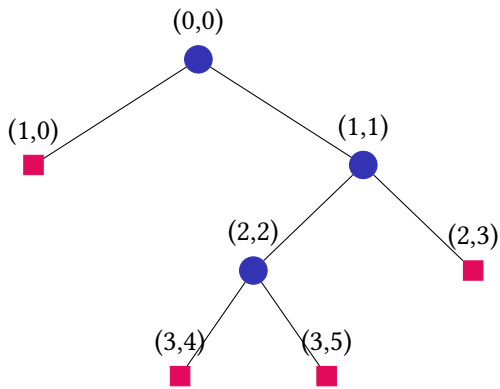
 
\section{A glimpse at adaptive confidence sets}

So far we have discussed mainly convergence rates of the posterior distribution, without yet investigating whether one could use that $\Pi[\cdot \given X]$ is a distribution with a {\em spread} that could possibly help the purpose of uncertainty quantification. We will consider this question in more details for finite-dimensional functionals (more generally $\sqrt{n}$--estimable functionals, possibly function--valued) in Chapters \ref{chap:bvm1} and \ref{chap:bvm2}. Let us discuss briefly slightly informally now the case where one wishes to obtain a confidence set for, say, a function $f$ (either in a regression or in a density estimation context).
 
There are two main desiderata: for a given distance $d$, a collection of regularity classes $\cC(s), s>0$ (think for instance of balls of $s$--H\"older functions), a distance $d$ (e.g. $\|\cdot\|_2, \|\cdot\|_\infty$,...) over functions, one wishes to find a (data-dependent) set $\cC_n(X)$ such that, for given level $1-\al>0$, for large enough $n$,
\begin{equation} \label{cs_cov}
\inf_{f\in\cC=\cup_s \cC(s)} P_f[f\in\cC_n(X)] \ge 1-\al. 
\end{equation} 
as well as, for any possible $s$, for $|\cC_n|_d$ the diameter of $\cC_n$ in terms of $d$,
\begin{equation} \label{cs_diam}
 \sup_{f\in\cC(s)} E_f[|\cC_n|_d] \asymp r_n(s),
\end{equation}
for $r_n(s)$ the minimax rate (possibly up to log terms) of estimation over $\cC(s)$ in terms of $d$.

A potential Bayesian solution to this problem is to pick a ball $B=B(X)$ for $d$ that has large posterior probability (it can be centered e.g. at a certain aspect of the posterior, for instance the posterior mean), say larger than $1-\al$. That is, its credibility is at least $1-\al$. Can such a ball verify \eqref{cs_cov}--\eqref{cs_diam}? 

The short answer is {\em no} in general, for a reason independent of Bayesian procedures: simply, constructing adaptive confidence sets is typically a difficult problem, which may or may not have a solution depending on the classes $\cC(s)$, the distance $d$, and the model. For instance, for regression or density models and the supremum norm $\|\cdot\|_\infty$, there are impossibility results showing that \eqref{cs_cov}--\eqref{cs_diam} cannot hold simultaneously, even if one restricts the adaptation problem to only two possible regularities. For the $\|\cdot\|_2$--norm, the problem admits a solution only on certain `windows' of regularities $[s_0,2s_0]$ for given $s_0>0$. The existence of adaptive confidence sets in general is linked to a delicate interplay between certain rates of estimation and rates of testing, and possibly also the dimension of the ambient space (e.g. as for Wasserstein-type losses as in \cite{dr23}). We refer to \cite{ginenicklbook}, Chapter 8, for an in-depth account of these phenomena. 
  
If one is willing to take smaller classes $\cC(s)$ by imposing extra constraints on the functions, then the problem may become possible again. A natural assumption in this context is that of {\em self-similarity}, which roughly means that the function has a similar behaviour across frequencies in its spectral analysis. Under such a structural assumption, the construction of adaptive confidence sets often becomes possible again. We refer to the discussion paper \cite{svv15} for an overview and discussion. In particular, the authors prove that in the setting of series priors as in Theorem \ref{thm-gpa} earlier in this Chapter (both in the EB and hierarchical Bayes case) that slightly enlarged (by a logarithmic factor) credible balls centered at the posterior mean and with $1-\al$ credibility are frequentist confidence balls for the $\|\cdot\|_2$--norm under self-similarity conditions. Similar results for {\em confidence bands} (i.e. for the $\|\cdot\|_\infty$--loss) are derived in regression in \cite{ray17} (spike--and--slab priors) and \cite{cr21} (tree priors), and in density estimation in \cite{cm21} (spike--and--slab) and \cite{cr22} (tree priors). Generic results for sieve priors and derived in \cite{rs20}. Negative results for the marginal maximum likelihood EB methods for the $L^2$ norm on `windows' of regularities (without self-similarity) are derived in \cite{svvl2}. This can be remedied by replacing the MMLE by an ad hoc estimate estimating the risk (\cite{svvl2}). There are many interesting further questions in this direction.

\section*{{\em Exercises}}
\addcontentsline{toc}{section}{{\em Exercises}} 
 
\begin{enumerate} 
\item Taking \eqref{sqe-sb} as granted, as well as the following approximation result, valid for $a$ large enough,
\[ \inf \{ \|h\|_{\mh^a}^2:\ \|h-f_0\|_\infty\le Ca \} \le Da,\]
for some universal constants $C,D$, where $\mh^a$ is the RKHS of the $\sqe(a)$ process,  
 prove an analog of Theorem \ref{thm-sqa} for $\rho$-posteriors, $\rho<1$. 
\item Find a formula for the posterior distribution in Gaussian white noise for a spike--and--slab prior as in \eqref{sasnp} in terms of an updates spike--and--slab distribution, with updated parameters depending on a certain convolution between $\Gamma$ and the noise density. 
\end{enumerate}

\chapter{Adaptation II: high-dimensions and deep neural networks} \label{chap:ada2}

In the first three sections of this chapter, we study sparse high-dimensional models and adaptation to the corresponding {\em sparsity} parameter. Our purpose is to introduce broadly used priors in this setting and explain a few key results with relatively simple proofs. We focus mostly on spike and slab-type priors and the sequence or linear regression models. We refer to \cite{bcg21} for an overview of results in the field of Bayesian high-dimensional models. In the second part of the chapter, we focus on priors on `deep' structures such as deep neural networks and deep Gaussian processes and discuss {\em adaptation to structure} for compositional classes.

\section{Priors in high dimensions} 

Since the 2000's practical applications where the number of unknown  parameters is `large', even possibly much larger than the number of observations, have become commonplace. Although it may seem paradoxical at first to be able to solve or even say something in such `difficult' settings, a key pattern that has emerged in the study of these models is that of {\em sparsity}. Namely, although the number of parameters is very large, possibly only a few are really significant. \\

{\em Sparsity.} A common sparsity assumption is the following: the true $\theta_0$ belongs to 
 the {\em nearly-black} class 
 \begin{equation} \label{ellclass}
 \ell_0[s] = \left\{\theta\in\mathbb{R}^n: \#\{i: \theta_i\neq 0\}\le s\right\}
 \end{equation}
 for $0\le s\le n$,  where $\#$ stands for the cardinality of a finite set. This means that only $s$ out of $n$ coordinates of $\te$ are nonzero (but we do not know which ones), and typically one assumes that  only a very small number of coordinates of $\te$ have 'signal', that is are nonzero. 
 In the sequel we assume $s=s_n\to\infty$ and $s=o(n)$ as $n\to\infty$.\\

{\em Some high-dimensional sparse models.} The simplest high-dimensional model is given by the normal sequence  model
\begin{equation} \label{seqmod}
 X_i=\theta_i+\epsilon_i,\quad i=1,\ldots,n,
 \end{equation}
 where $\varepsilon_i$ are i.i.d. $\cN(0,1)$, the parameter set $\Theta$ for $\theta=(\theta_1,\ldots,\theta_n)$ is $\mathbb{R}^n$ but $\theta$ is assumed to be sparse in the sense that it belongs to one of the sets $\ell_0[s]$ for some $0\le s\le n$. 
The optimal minimax rate in model \eqref{seqmod} over $\ell_0[s]$, in terms of  squared error loss $\|\te\|^2=\sum_{i=1}^n\te_i^2$  is
 \[ \inf_T \sup_{\te_0\in\ell_0[s]} E_{\te_0}[\|T(X)-\te\|^2]=2s\log(n/s)\cdot (1+o(1)),\]
 where $T=T(X)$ is an estimator of $\te$ based on the observation of $X=(X_1,\ldots,X_n)$. \\

 This model is a special case (set $p=n$) of the high-dimensional Gaussian linear regression model
 \begin{equation} \label{hdreg}
  Y = X\te + \veps, 
 \end{equation}
where $\te\in \RR^p$, the noise vector $\veps$ follows a $\cN(0,\sigma^2 I_{n})$ distribution and $X$ is a $n\times p$ matrix with real coefficients. The `high-dimensional case' corresponds to $n\le p$, possibly $n=o(p)$. In that case, one generally assumes  $\te\in\ell_0[s]$ (with $\RR^n$ replaced by $\RR^p$ in the definition), for some $s=o(n)$.\\

{\em The need for prior modelling.}  In the case where  $\Theta$ is a subset of $\RR^n$, the simplest prior that comes to mind is $\Pi=\otimes_{i=1}^n G$, making the coordinates of $\te$ independent of distribution $G$ on $\RR$. However, from the point of view of the posterior distribution, this unstructured prior is often not suitable. Consider for instance model \eqref{seqmod} and let us endow $\te$ with the a product of Laplace (double-exponential) priors 
\[ \Pi_{\la}=\bigotimes_{i=1}^n \text{Lap}(\la/2), \quad \la>0. \]  
For this choice, the posterior mode (that is, the mode of the posterior density) is [exercise: check it]
\begin{equation*}
\hat{\te}^{L}_\la= \argmin{\te\in\RR^n} \bigl[ \|X-\te\|_2^2 + \la \|\te\|_1 \bigr].
\end{equation*}
This is nothing but the classical LASSO estimator. 
In the special case of model \eqref{seqmod}, for the choice $\la=\la^*\asymp \sqrt{\log{n}}$, the LASSO achieves the minimax rate 
\[ \sup_{\te\in\ell_0[s]}  E_\te[\| \hat{\te}^{L}_{\la^*} - \te\|^2]\leqa s\log{n}, \] up to the form of the log factor.  However, if the true $\te_0=0$, for small $\delta>0$, one can show that 
\begin{equation} \label{lblasso}
 E_0 \Pi_{\la^*}\Bigl[\|\te\|^2\le \delta \frac{n}{\log{n}} \given Y\Bigr] \rightarrow 0.
\end{equation} 
This means that the ``LASSO--posterior distribution" $\Pi_{\la^*}[\cdot \given X]$  is suboptimal over sparse classes $\ell_0[s]$ for $s\ll n/\log^2{n}$. The intuition behind this result is that, although its mode is the LASSO and is thus sparse,  the LASSO--posterior as a probability distribution is not sparse. A sample from $\Pi_{\la^*}[\cdot \given X]$ almost surely sets no coordinate of $\te$ to $0$. From the Bayesian perspective, this means that one needs to take structural assumptions such as sparsity into account when proposing a prior distribution. \\

{\em Notation and setting.}  For a vector $\te\in\RR^p$ (or $\RR^n$ in the sequence model) and  a set $S\subset\{1,2,\ldots,p\}$, let
$\te_S$ be the vector $(\te_i)_{i\in S}\in\RR^S$, and $|S|$  the cardinality of $S$.
The \emph{support} of the parameter $\te$ is  the set $S_\te=\{i: \te_i\not=0\}$.
The support of the true  $\te_0$ is denoted $S_0$,
with cardinality $s_0:=|S_0|$.  
Moreover, we write $s=|S|$ if there is no ambiguity to which set $S$ is referred to. 

\subsubsection*{Spike--and--slab-type priors}

{\em Spike--and--slab priors}. For $\al\in[0,1]$ and $\Gamma$ a distribution on $\RR$, the prior
\begin{equation} \label{sasprior}
 \Pi_\al = \Pi_{\al,\Gamma} = \bigotimes_{i=1}^n\, (1-\al)\delta_0 + \al \Gamma,
\end{equation}  
where $\delta_0$ is the Dirac mass at $0$, is called \sbl{spike and slab (SAS) prior} with parameter $\al\in[0,1]$ and slab distribution $\Gamma$. To inforce sparsity, one may choose a deterministic $\al$: the choice $\al=1/n$ is possible (but slightly `conservative') and implies that under the prior, the expected number of nonzero coefficients is of the order of a constant. To obtain an improved data fit, options performing better in practice include an empirical Bayes choice $\hat \al$ of $\al$, or hierarchical Bayes, where $\al$ is itself given a prior, for instance a Beta distribution, e.g. $\text{Beta}(1,n+1)$. We  discuss this in more details below.\\   
  
{\em Subset--selection priors}. For $\pi_n$ a prior on the set $\{0,1,2,\ldots, n\}$ and $\cS_k$ the collection of all subsets of $\{1,\ldots,n\}$ of size $k$, let $\Pi$ be constructed as
\begin{equation*}
k \sim \pi_n, \qquad  S\given k  \sim \text{Unif}(\cS_k), \qquad 
\te\given S \sim \bigotimes_{i\in S} \Gamma \,\otimes\, \bigotimes_{i\notin S} \delta_0.
\end{equation*}
The spike--and--slab prior \eqref{sasprior} is a particular case where $\pi_n$ is the binomial $\text{Bin}(n,\al)$ distribution. Through the prior $\pi_n$, it is possible to chose dimensional priors that `penalise' more large dimensions than the binomial, for instance the complexity prior $\pi(k)\propto \exp(-ak\log(bn/k))$. 

\subsubsection{Continuous shrinkage priors} 

While subset selection priors are particularly appealing in view of their naturally built-in model selection, one may instead use prior distributions that do not  put any coefficient exactly to $0$ but instead draw either very small values or intermediate/strong ones. A way to do so is to replace the Dirac mass in \eqref{sasprior} by an absolutely continuous distribution with density having a high or infinite density at zero. \\

{\em Spike and slab LASSO}. This prior replaces $\delta_0, \Gamma$ by two Laplace distributions $\text{Lap}(\la_0), \text{Lap}(\la_1)$ with $\la_0$ large, typically going to $\infty$ with $n$ to enforce (near--)sparsity and $\la_1$ a constant, that is 
\[  \Pi_\al = \Pi_{\al,\Gamma} = \bigotimes_{i=1}^n\, (1-\al)\text{Lap}(\la_0) + \al \text{Lap}(\la_1),\]
where $\al$, as above for the spike and slab prior, to be chosen. \\

{\em Horseshoe prior.} Leaving finite mixtures, one may also consider continuous mixtures: a popular choice is the {\em horseshoe} prior of Carvalho, Polson and Scott \cite{carvalhopolsonscott}, which is  a continuous scale--mixture of Gaussians.

\begin{definition} \sbl{[Horseshoe prior].}  \label{def:hs}
The horseshoe prior  with parameter $\ta>0$ is the distribution on $\RR$ of the variable $X_\ta$ defined as 
 \begin{align*}
 \lambda &\sim C(\tau), \\ 
 \te \given \la &\sim \mathcal{N}(0, \lambda^2)
 \end{align*}
where $C(\tau)$ is a centered Cauchy distribution with scale parameter $\tau$.
\end{definition}

It can be checked that the density $\pi_\tau$ of the horseshoe prior verifies for any nonzero real $t$
 \begin{equation*}
 \frac{1}{(2\pi)^{3/2}\tau}\log\left(1+\frac{4\tau^2}{t^2}\right)<\pi_\tau(t)<\frac{1}{\sqrt{2\pi^3}\tau}\log\left(1+\frac{\tau^2}{t^2}\right).
 \end{equation*}
In particular, the horseshoe density has a pole at zero and Cauchy tails (see also Figure \ref{pic:hs} below). One may also consider different priors on the scales $\la_i$, as in \cite{vss16}.

\section{Posterior convergence: sparse sequence model}

\subsubsection*{Spike and slab posterior distribution with fixed $\al$}

We consider the following choices 
\[ \Gamma = 
\begin{cases}
& \text{Lap}(1) \\
or\\
 & \text{Cauchy}(1)
\end{cases} 
  \]
where Lap$(\la)$ denotes the Laplace (double exponential) distribution with parameter $\la$ and Cauchy$(1)$ the standard Cauchy distribution.
Different choices of parameters and prior distributions are possible but for clarity of exposition we stick to these common distributions. In the sequel $\ga$ denotes the density of $\Gamma$ with respect to the Lebesgue measure. It is necessary to have slab tails at least as heavy as Laplace: a Gaussian slab for instance would shrink too much towards $0$ (check for instance that a standard normal slab gives a posterior mean equal to $X_i/2$ on coordinate $i$, which is grossly suboptimal if $\te_{0,i}$ is large).
\\

By Bayes' formula the posterior distribution under \eqref{seqmod}  with fixed $\al\in[0,1]$ is 
\begin{equation} \label{post}
\Pi_\alpha[\cdot\given X] 
\sim \bigotimes_{i=1}^n\, (1-a(X_i))\delta_0 + a(X_i) G_{X_i}(\cdot),
\end{equation}
where, denoting by $\phi$ the standard normal density and 
$g(x)=\phi*\Gamma(x)=\int \phi(x-u)d\Gamma(u)$ the convolution of $\phi$ and $\Gamma$ at point $x\in\RR$,  the posterior weight $a(X_i)$ is given by, for any $i$,
\begin{equation} \label{postwei}
 a(X_i) = a_\al(X_i)=\frac{\al g(X_i)}{(1-\alpha)\phi(X_i) + \alpha g(X_i)}. 
\end{equation} 
The distribution $G_{X_i}$ has density, with respect to Lebesgue measure on $\RR$,
\begin{equation} \label{postdc}
\ga_{X_i}(\cdot) := \frac{\phi(X_i-\cdot) \ga(\cdot)}{g(X_i)}
\end{equation}
The behaviour of the posterior distribution $\Pi_\al[\cdot\given X]$ heavily depends on the choice of the sparsity parameter $\al$, and also (somewhat) on the chosen $\ga$.  \\

{\em Posterior median and threshold $t(\al)$.} The posterior median $\hat\te^{med}_\al(X_i)$ of the $i$th coordinate has a thresholding property \cite{js04}:  there exists $t(\al)>0$ such that $\hat\te^{med}_\al(X_i)=0$ if and only if $|X_i|\le t(\al)$. One can check that $t(\al)$ is {\em roughly} of order $\sqrt{2\log(1/\al)}$ for small $\al$.  
As before a default choice is $\al=1/n$; one can check that this leads to a posterior median behaving similarly as a hard thresholding estimator with threshold $\sqrt{2\log n}$. \\

{\em Posterior convergence for fixed $\al$.} The following Lemma from \cite{cm18} helps understanding the role of $\al$ for spike-and-slab priors.

\begin{lem} \label{thm:fixedal}
In the sparse sequence model, let us consider a spike--and--slab (SAS) prior with fixed $\al\in(0,1)$. Let $\Ga$ be either the standard Laplace or Cauchy distribution. Then there exist $\al_0>0$, $N_0\ge 2$ and $C>0$ such that for any 
$\al\le \al_0$ and $n\ge N_0$, 
\[ \sup_{\te_0\in\ell_0[s_n]} E_{\te_0} \int \|\te-\te_0\|^2 d\Pi_\al(X) 
\le Cn\al\sqrt{\log(1/\al)} + Cs_n(1+\log(1/\al)).
\]
As a particular case, one obtains, for $n$ large enough,
\[
\sup_{\te_0\in\ell_0[s_n]} E_{\te_0} \int \|\te-\te_0\|^2 d\Pi_\al(X) 
\le 
\begin{cases}
Cs_n\log{n} & \qquad \text{if } \al=1/n,\\
Cs_n\log(n/s_n) & \qquad \text{if } \al=s_n/n.
\end{cases}
\]
\end{lem}

For $\al=1/n$, or more generally $\al=n^{-b}$ with $b\ge 1$, the rate is near-minimax: it differs from  $Cs_n\log(n/s_n)$ only when $\log(n/s_n)=o(\log{n})$ which happens in the almost dense case where $s_n$ is $o(n)$ but $n/s_n$ grows very slowly, e.g. logarithmically.\\

If one takes the `oracle' choice $\al=s_n/n$, which means that one knows beforehand $s_n$, the obtained rate is the optimal one $Cs_n\log(n/s_n)$ up to a constant. \\

If $\al$ goes significantly above $s_n/n$ in that $\al\gg (s_n/n)\sqrt{\log{n}}$, the obtained rate is of larger order than $s_n\log(n/s_n)$ (the above give only an upper-bound; in this case it can be checked that the actual rate is indeed suboptimal). \\

Lemma \ref{thm:fixedal} can be proved by a direct analysis of the posterior distribution via the explicit expressions \eqref{post}--\eqref{postwei}. Yet we see that unless one known $s_n$, the obtained rate can be suboptimal. Although one only looses in the constant in most cases much by taking $\al=1/n$, in practice this choice is often too conservative in that signals below $\sqrt{2\log{n}}$ often get thresholded. This is a problem of adaptation to $\al$ and motivates the next discussion.

\subsubsection{Spike--and--slab: data-driven choice of $\al$}

As we have seen in Chapter \ref{chap:ada1}, there are two natural options to choose a prior hyperparameters: empirical Bayes and hierarchical Bayes. \\

{\em Empirical Bayes.} 
Recall that the marginal maximum likelihood empirical Bayes approach (MMLE) consists in forming a likelihood in terms of the parameter of interest (here $\al$) by integrating out the parameter $\te$. In model \eqref{seqmod}, for a spike--and--slab prior and fixing the distribution $\Gamma$ with density $\ga$, we have $\te\given \al\sim\Pi_\al$ and $X_i\given\te\sim P_\te$ independent normals. Then the Bayesian distribution of $X$ given $\al$ has density  (check it as an exercise)
\[ \int \prod_{i=1}^n p_\te(X_i) d\Pi_\al(\te).\]
 The maximisation of the corresponding `likelihood' leads to, with $g=\ga * \phi$, 
\[ \hat{\al} = \underset{\alpha}{\text{argmax}}\, \prod_{i=1}^n \left( (1-\al)\phi(X_i) + \al g(X_i) \right). \]
The plug-in posterior $\Pi_{\hat \al}[\cdot\given X]$ has been advocated  in  George and Foster \cite{georgefoster} and Johnstone and Silverman \cite{js04}. In \cite{js04}, the authors prove that the posterior coordinate-wise median of $\Pi_{\hat \al}[\cdot\given X]$  converges at the optimal rate $Cs_n\log(n/s_n)$ when $\ga$ is a Laplace or Cauchy slab. Yet, this is not necessarily the same as convergence for the complete plug-in  posterior $\Pi_{\hat\al}[\cdot\given X]$.\\

{\em Surprises with empirical Bayes.} The behaviour of the complete plug-in posterior is investigated in \cite{cm18}. Therein the following surprising result is obtained: the EB posterior $\Pi_{\hat\al}[\cdot\given X]$ converges at the optimal minimax rate if the slab distribution $\Ga$ is Cauchy, but is suboptimal for $\Ga$ Laplace. It turns out that heavy tails help here for the empirical Bayes approach; a Laplace slab can be accommodated, but this needs to `penalise' more: this can be achieved by taking rather a hierarchical Bayes approach with an appropriate prior on $\al$.\\

{\em Spike and slab and hierarchical Bayes.} In \cite{cv12}, the hierarchical prior
\begin{align*}
\te\given \al & \sim \Pi_\al \\
\al & \sim \text{Beta}(1,n+1)
\end{align*}
is considered, which leads to a beta-binomial prior on the dimension $|S|$ of a draw from the prior distribution. It is shown in \cite{cv12} that the corresponding posterior converges towards $\te_0$ at optimal rate $Cs\log(n/s)$ uniformly over $\ell_0[s]$. We prove a slightly weaker version of this result below in the case of a Laplace slab.

\subsubsection{A generic posterior convergence result} \label{sec:gene}

Let us work with the subset-selection prior \eqref{subsprior}. 
\begin{equation} \label{subsprior}
k \sim \pi_n, \qquad  S\given k  \sim \text{Unif}(\cS_k), \qquad 
\te\given S \sim \bigotimes_{i\in S} \Gamma \,\otimes\, \bigotimes_{i\notin S} \delta_0.
\end{equation}

Let us denote, for $\te\in \RR^n$, by $S_\te$ its support
\[ 
S_\te=\{i: \te_i\neq 0\}, \] 
that is the indices of its nonzero coordinates. We denote $S_0=S_{\te_0}$.

\begin{definition} \label{def:dim}
We say that a prior on dimension $\pi_n$ as in \eqref{subsprior} has {\em exponential decrease} if there exists a constant $D<1$ such that, for any $k\ge 1$,
\[ \pi_n(k)\le D\pi_n(k-1).\]
\end{definition}

{\em Examples.} The prior $\pi_n(k)\propto e^{-k}$ and the prior $\pi_n(k)\propto e^{-ak\log(bn/k)}$ for $a>0, b>1+e$, both have exponential decrease [Exercise]. For binomial priors on dimension, $\pi_n=\text{Bin}(n,\al)$, it can be checked that they also verify the exponential decrease if $n\al\leqa s_n$, which is verified for the choices $\al=1/n$ and $\al=s_n/n$. It can also be verified, see \cite{cv12}, that the beta-binomial prior 
\begin{align*} 
k\given \al & \sim \text{Bin}(n,\al)\\ 
\al & \sim \text{Beta}(1,n+1) 
\end{align*} 
verifies the exponential decrease property. 
 
\begin{thm} \label{thm:generic}
Take a prior $\Pi$ as in \eqref{subsprior} with $\Ga=\text{Lap}(1)$ and suppose, for $S_0=S_{\te_0}$, for some constant $d>0$, 
\[ \Pi(S_0)\ge e^{-ds_n\log{n}}. \]
Assume that the prior on dimension $\pi_n$ verifies the exponential decrease as in Definition \ref{def:dim}. Then for $M$ large enough, as $n\to\infty$,
\[ \sup_{\te_0\in\ell_0[s_n]} E_{\te_0} \Pi[\|\te-\te_0\|^2>Ms_n\log{n}\given X]=o(1).\]
\end{thm} 

{\em Corollary.} It is not hard to check that all the examples of priors mentioned just above the statement of Theorem \ref{thm:generic} verify, combined with the uniform prior on subsets as in \eqref{subsprior}, the condition on $\Pi(S_0)$ from the Theorem. So they all lead to a posterior convergence rate at least of the order $s_n\log{n}$. By a more precise argument, one can in fact prove (see \cite{cv12}) that both  priors on dimension $\pi_n(k)\propto e^{-k}$ and the Beta-binomial prior as above both lead to a posterior convergence rate of $Ms_n\log(n/s_n)$ for large $M$ (so, with the precise logarithmic) factor. This shows that both priors achieve the optimal rate in an adaptive way. Note that we also `recover' (although the results are not exactly equivalent as noted above), the posterior contraction rate obtained in Theorem \ref{thm:fixedal} for the binomial priors on dimension -- that lead to the SAS prior construction and vice-versa -- with $\al=1/n$ or $\al=s_n/n$ (up to the form of the log factor for the latter).  \\

\begin{proof}[Proof of Theorem \ref{thm:generic}] 
We follow ideas similar to that of the proof of the GGV theorem, but do not do the testing/entropy part explicitly. For any $C\subset\RR^n$ measurable, Bayes' formula can be written
\[ \Pi[C\given X]
=\frac{\int_C \exp\{ -\|X-\te\|^2/2\}d\Pi(\te)}{\int \exp\{ -\|X-\te\|^2/2\}d\Pi(\te)}
=\frac{\int_C \exp\{ \Delta(\te,X) \}d\Pi(\te)}{\int \exp\{\Delta(\te,X) \}d\Pi(\te)},
\]
where we have set $\Delta(\te,X)=-\|\te-\te_0\|^2/2+\psg\te-\te_0,X-\te_0\psd$.
[One may remark that the parameter set of sparse vectors is unbounded, so one cannot hope for a uniform bound from below of the denominator independent of how `large' the coordinates of $\te_0$ are. In fact, the bound below features an $L^1$ norm $\|\te_0\|_1$; fortunately, this can be compensated from a similar term appearing in the bound for the numerator as seen below]\\

Now let $D:=\int \exp\{\Delta(\te,X) \}d\Pi(\te)$ be the denominator on the last term of the display on Bayes' formula above, and $N=\int_C \exp\{\Delta(\te,X) \}d\Pi(\te)$ the numerator, with $C$ to be chosen below. 

{\em Bound on the denominator $D$.} Let us set, for $r_n\to\infty$ to be chosen, 
\[ B=\{\te:\ \|\te-\te_0\|\le r_n\}. \]
By restricting the integral on the denominator $D$ to the set $B$, 
\[ D\ge \Pi(B) \int e^{\Delta(\te,X)}d\bar{\Pi}(\te),\] 
where $\bar{\Pi}$ is the probability distribution $\bar{\Pi}(\cdot)=\Pi(\cdot\cap B)/\Pi(B)$. Let us now apply Jensen's inequality with the exponential map to obtain
\[ \int_B e^{\Delta(\te,X)}d\bar{\Pi}(\te)
\ge \exp\left\{ \int_B \Delta(\te,X) d\Pi(\te) \right\}.
\]
Noting that, on $B$, we have $-\|\te-\te_0\|^2/2\ge r_n^2/2$, and setting 
 \[ Z:=\psg \int_B(\te-\te_0)d\bar\Pi(\te),X-\te_0\psd=
 \int_B\psg \te-\te_0,X-\te_0\psd d\bar\Pi(\te)
  \] by linearity, one gets 
\[ D \ge \Pi(B) e^{-\frac{r_n^2}{2}+Z}.\]
Now under $P_{\te_0}$, we have $X-\te_0=\veps\sim\cN(0,I_n)$, with $I_n$ the identity matrix in dimension $n$. In particular, $ \psg \te-\te_0,X-\te_0\psd  \sim \cN(0,\|\te-\te_0\|^2)$. 
So,  $E_{\te_0}Z=0$ (Fubini), and by Jensen's inequality again this time applied with $x\to x^2$, and Fubini,
\begin{align*}
 E_{\te_0} Z^2 &\le \int_B E_{\te_0} \psg \te-\te_0,X-\te_0\psd^2 d\bar\Pi(\te)    = \int_B \|\te-\te_0\|^2 d\bar\Pi(\te)\le r_n^2,
\end{align*} 
using the definition of $B$. By Markov's inequality, 
$P_{\te_0}[Z>r_n^2]\le r_n^{-4} E_{\te_0}Z^2\le r_n^{-2}$. One deduces that on an event $\cA$ of $P_{\te_0}$--probability at least 
 $1-(1/r_n^2)$, we have
\[ D\ge \Pi(B) e^{-3r_n^2/2}.\]
Let us now focus on $\Pi(B)$. Note that an equivalent way of writing the prior $\Pi$ is as follows
\[ \Pi = \sum_{S} Q(S) \Pi_S, \quad
\Pi_S:=\bigotimes_{i=1}^n \Pi_{S,i}, \]
where $\Pi_{S,i}=\text{Lap}(1)$ if $i\in S$ and $\Pi_{S,i}=\delta_0$
 otherwise. Denoting  by $S_0:=S_{\te_0}$ the support of $\te_0$, and $\te_S=(\te_i, i\in S)$ for $S\subset \{1,\ldots,n\}$, and for $B$ as above, with $\|u\|_1=\sum_{i=1}^n |u_i|$,
\begin{align*}
 \Pi(B) & \ge Q(S_0) \Pi_{S_0}[\|\te-\te_0\|\le r_n]
  \ge Q(S_0)\Pi_{S_0}[\|\te-\te_0\|_1\le r_n] \\
& \ge Q(S_0)\int_{\|\te-\te_0\|_1\le r_n} \prod_{i\in S_{0}} \frac{1}{2} e^{-|\te_i|} d\te_i  \ge Q(S_0) 2^{-s_n}\int_{\|\te_{S_0}-\te_0\|_1\le r_n} e^{-\|\te_{S_0}-\te_0\|_1-\|\te_0\|_1} d\te_{S_0} \\
& \ge Q(S_0) 2^{-s_n}e^{-\|\te_0\|_1} \int_{\|\te_{S_0}\|_1\le r_n} e^{-\|\te_{S_0}\|_1} d\te_{S_0}, 
\end{align*}
where we have used that $\|u\|_2\le \|u\|_1$ for $u\in\RR^d, d\ge 1$ and
 invariance by translation of Lebesgue's measure [note: one can also use a lower bound involving volumes of balls]. Next, as 
$\{|\te_i|\le r_n/s_n,\ i\in S_0\}\subset \{\|\te_{S_0}\|_1\le r_n\}$, with $s_0=|S_0|\le s_n$,
\[  \int_{\|\te_{S_0}\|_1\le r_n} e^{-\|\te_{S_0}\|_1}d\te_{S_0}\ge 
\left(\int_{-r_n/s_n}^{r_n/s_n} e^{-|u|}du\right)^{s_0}\geqa (Ce^{-r_n/s_n})^{s_0}\geqa e^{-r_n}.
\]
From the previous bounds one deduces that, on the event $\cA$,
\[ D \geqa Q(S_0)e^{-3r_n^2/2 - r_n-Cs_n} e^{-\|\te_0\|_1}. \]

{\em Bound on the numerator $N$.} 
Recall that $N:=\int_C \exp\{\Delta(\te,X) \}d\Pi(\te)$, and we define $C$ now as  $C:=C_1\cap C_2$, with, for $K,M>0$, with $|A|$ the cardinality of $A$,
\[ C_1:= \{\te:\ |S_\te|\le Ks_n\},\quad C_2:=\{\te:\ \|\te-\te_0\|>Mr_n\}.\]
Let us also define the event 
\[ \cA_N=\left\{ \max_{1\le i\le n}|\veps_i|\le \sqrt{2\log{n}} \right\}. \]
A union bound gives, using the standard bound $\bar\Phi(x)\le \phi(x)/x$, for $x>0$,
 \[ P(\cA_N^c)\le 2n\bar\Phi(\sqrt{2\log{n}})=o(1)\] 
Also, noting that if $\te\in C_1$ we have $|S_{\te-\te_0}|\le (K+1)s_n$, for any such $\te$ on the event $\cA_N$,
\begin{align*}
 |\psg\te-\te_0,X-\te_0\psd| & \le \|\te-\te_0\|_1\max_{1\le i\le n}|X_i-\te_{0,i}|\\\
 & \le |S_{\te-\te_0}|^{1/2} \|\te-\te_0\| \max_{1\le i\le n}|\veps_i| 
  \le \sqrt{2(K+1)s_n\log{n}}\|\te-\te_0\|,
\end{align*}
where the second line uses Cauchy-Schwarz inequality.  
Now using the inequality $2ab\le \delta^{-1}a^2+\delta b^2$, one finds that, on $\cA_N$ and for large enough $c_2>0$, 
\[ |\psg\te-\te_0,X-\te_0\psd|  \le c_2 s_n\log{n}+ \frac{\|\te-\te_0\|^2}{4}. \]
Inserting this into the definition of $N$ leads to, on the event $\cA_N$,
\begin{align}
 N & \le \int_C e^{-\|\te-\te_0\|^2/4+c_2s_n\log{n}} d\Pi(\te)  \le e^{c_2s_n\log{n}}\sum_{S} Q(S) \int_{C_S}  e^{-\|\te_S-\te_0\|^2/4} 
 \frac{e^{-\|\te_S\|_1}}{2^{|S|}} d\te_S, \label{tr:N}
\end{align} 
where $C_S=\{\te:\ S_\te=S\}\cap C$, using that by definition the prior on the selected subset is product Laplace.  By the triangle inequality, 
\begin{equation}\label{trickl1}
 -\|\te_S\|_1\le -\|\te_{0,S}\|_1+\|\te_S-\te_{0,S}\|_1.
\end{equation} 
One bounds each term on the right hand side. Cauchy-Schwarz implies, on $C$, that
$\|\te_S-\te_{0,S}\|_1 \le \sqrt{Ks_n} \|\te_S-\te_0\|$. 
On the other hand, we have
\[ -\|\te_{0,S}\|_1 \le -\|\te_0\|_1 + \|\te_{0,S_0\cap S^c}\|_1\le -\|\te_0\|_1 + \sqrt{s_n} \|\te_S - \te_0\|, \]
as indeed, using again Cauchy-Schwarz,
\begin{align*}
\|\te_{0,S_0\cap S^c}\|_1
 & \le   |S_0\cap S^c|^{1/2}\|\te_{0,S_0\cap S^c}\|  \le  \sqrt{s_n} \|\te_S - \te_0\|.
\end{align*} 
Inserting back the obtained bounds in \eqref{trickl1} one gets
\[-\|\te_S\|_1\le -\|\te_0\|_1 + (\sqrt{K}+1)\sqrt{s_n}\|\te_S-\te_0\|. \]
The last term  is bounded by $K's_n+\|\te_S-\te_0\|^2/8$ for large enough $K'$. Now the integral in \eqref{tr:N} is bounded from above by
\[ e^{-\|\te_0\|_1-M^2r_n^2/16+K's_n} \int_{C_S} e^{- \|\te_S-\te_0\|^2/16}d\te_S. \]
By definition, $C_S$ contains only vectors of support at most $Ks_n$, so that
\[ \int_{C_S} e^{- \|\te_S-\te_0\|^2/16}d\te_S\le \int_{\RR^{|S|}} e^{- \|\te_S\|^2/16}d\te_S \le C^{Ks_n}.\]
This gives the following bound on the numerator, on the event $\cA_N$,
\[ N \le e^{c_2s_n\log{n}-M^2r_n^2/16+C's_n} e^{-\|\te_0\|_1}.\]
 
 {\em Putting the bounds together.} Gathering the previous bounds gives, on $\cA\cap\cA_N$,
 \[ \frac{N}{D} \le e^{3r_n^2/2 + r_n+c_3s_n\log{n}-M^2r_n^2/16}, \]
 where one uses the assumption on $Q(S_0)$. 
 By choosing $r_n^2=s_n\log{n}$ and taking $M$ large enough, one deduces, on  $\cA\cap\cA_N$, that 
 \[ N/D\le e^{-M^2r_n^2/32}=o(1),\]
  which implies that $E_{\te_0}\Pi[C_1\cap \{\te:\ \|\te-\te_0\|>Mr_n\} \given X]=o(1)$. Combining this with Lemma  \ref{lem:dim} for $K$ large enough, one obtains the desired contraction $E_{\te_0}\Pi[\|\te-\te_0\|>Mr_n \given X]=o(1)$ (a direct argument for a specific prior on dimension is also given in the Remark below).
\end{proof}
\vp

\noi {\em Remark.} If the prior on dimension $\pi_n$ in \eqref{subsprior} is chosen as $\pi_n(k)\propto e^{-Dk\log{n}}$, for some constant $D\ge 2$ (and so verifies slightly more than just exponential decrease), then one can use a direct argument without appealing to Lemma \ref{lem:dim} below. To do so, one shows $E_{\te_0}\Pi[C_1^c\cap \{\te:\ \|\te-\te_0\|>Mr_n\} \given X]=o(1)$ by following a similar argument as in the end of the proof of the above Theorem, but bounding  the prior mass $Q(S)$ in \eqref{tr:N} from above by $C\exp(-D|S|\log{n})$; details are left to the reader.

\vp

\begin{lem} \label{lem:dim}
Suppose the prior on dimension $\pi_n$ in \eqref{subsprior} verifies the exponential decrease property from Definition \ref{def:dim} and that $\Ga$ has a centered density with finite second moment. Then, for $S_\te=\{i:\, \te_i\neq 0\}$ 
the support of the vector $\te$, for $K$ large enough, as $n\to\infty$, 
\[  \sup_{\te_0\in\ell_0[s_n]} E_{\te_0} \Pi[\te:\, |S_\te|>Ks_n\given X] = o(1).\]
\end{lem}
\begin{proof}
See Proposition 4.1 and Lemma 4.1 in \cite{cv12}.  The proof is based on the fact that a prior with exponential decrease puts little mass on models of size larger that $Ks_n$, for $K$ a sufficiently large constant. 
\end{proof}

\subsection*{Sequence model: further results and applications}

\begin{enumerate}
\item {\em Confidence sets for $\te$ or for coordinates of $\te$.} One may try to use regions that get high posterior probability (so--called credible sets) to get confidence sets having specific coverage and smallest possible diameter. This is a problem of construction of adaptive confidence sets, and a solution here is possible only to a certain extent, and one has to either make assumptions on the sparsity parameter, or to allow for larger regions, to be able to solve the problem \cite{nv13}. Such questions are considered, for instance, in \cite{vsv17} for the horseshoe, and \cite{cs20} for spike and slab priors. 
\item {\em Variable selection and Multiple testing.} In applications such as genomics, one very important practical question is to select a subset of coordinates that contain `signal', with a certain control of the number of false positive (and possibly also false negative). In the sparse sequence model, this can be done through the famous Benjamini-Hochberg procedure (BH procedure), but also in a Bayesian way via, for instance, an empirical Bayes posterior, as suggested by Efron \cite{efron07}, and investigated from the frequentist perspective in \cite{cr20}. We discuss this is Chapter \ref{chap:mtc}.
\end{enumerate}

\section{Posterior convergence: high dimensional linear regression}

An important sparse model arising in applications and generalising the sequence model \eqref{seqmod} is 
\begin{equation}
Y = X\te + \veps, 
\end{equation}
where one observes the pair $(Y,X)$, with  $X$ a $n\times p$ real-valued matrix, $\te\in\RR^p$ an unknown vector to be estimated,  and $\veps$ a column noise vector of size $n$, for two integers $n,p$. For simplicity we take the components $\veps_i$ of $\veps$ to be iid $\cN(0,1)$.

The sparse high-dimensional regime corresponds to $n,p$ both going to infinity, where $p$ can be much larger than $n$. Without further assumptions on the matrix $X$, the model would be (heavily) non-identifiable (in the sense that $X\te=X\te'$ for many different $\te, \te'$ so there is no hope to recover the original $\te$). There has been much work in the 2000-2010 on the type of assumptions that one should use in this setting. We refer to the St-Flour notes \cite{saravdg-stf} for more background and discussion. The main idea is that, whereas it is not possible to assume that $X$ is invertible if $p>n$, it is conceivable that the matrix may act `nearly as an isometry' on subspaces that have only $s=o(n)$ coordinates that are non-zero. In particular when $X$ is a random matrix, this type of conditions is often satisfied with high probability. \\

Posterior convergence at the (near) minimax rate $s\log{p}$ was derived for spike and slab-type priors and Laplace slabs in \cite{spahd}. A general result holding in a variety of high-dimensional settings (albeit with a more complicated prior) with sharp dependence of constants in terms of the restricted eigenvalue-type conditions is obtained in \cite{gvz20}. Below we propose a relatively simple argument based on tempered posteriors, under a relatively mild boundedness condition. For a result without this condition for Laplace slabs in the spirit of the one of Theorem \ref{thm:generic} above, see \cite{spahd}. \\

{\em Notation.} We let $X_{.,i}$ be the $i$th column of the design matrix $X$, and 
\begin{align}
\label{DefNormX}
  \|X\| 
=\max_{i=1,\ldots,p}\|X_{.,i}\|_2 =\max_{i=1,\ldots,p} (X^tX)_{i,i}^{1/2}.
\end{align}
Asymptotics below are in $n,p\to\infty$. 

\begin{definition} \label{defsmallsv}
The smallest scaled singular value of dimension $s$ is defined as
\begin{align}
\omega(s) :=\inf \Big\{\frac{\|X\be\|_2}{\|X\|\, \|\be\|_2}:\quad  0\neq |S_\be|\leq s\Big\}.
  \label{smallsv}
\end{align}
\end{definition} 
 
A $\text{Cauchy}(\la)$ is the distribution of $\la Z$ with  $Z\sim \text{Cauchy}(1)$ standard Cauchy.
 
\begin{thm} \label{thm:hdreg}
Consider a spike--and--slab--type prior $\Pi$ as in \eqref{subsprior} with $\Ga=\text{Cauchy}(\la)$ with $\la=1/\|X\|$ and prior on dimension, for $k=0,1,\ldots,p$,
\begin{equation} \label{hdpriord}
\pi_p(k) \propto e^{- k \log{p}}.
\end{equation}
Suppose that there exist constants $A, B>0$ such that
 \begin{equation} \label{condsig}
 \max_{1\le i\le p} |\te_{0,i}| \le Ap^B/\|X\|^2.
 \end{equation}
Fix $\rho\in(0,1)$. Then for a large enough constant $M$, with $c=B+5$, large enough $s_n , p$,
\[ \sup_{\te_0\in\ell_0[s_n]} E_{\te_0} \Pi_\rho\left[\|\te-\te_0\|^2>\left\{\frac{M}{\|X\|^2 \omega(cs_n)^2}\right\} s_n\log{p}
\,\given\, X\right]\leqa e^{-s_n\log{p}}. \]
\end{thm} 
Again the choice of a Cauchy prior is to fix ideas, other densities with polynomial tails gives similar results.  A main part of the proof consists in proving Lemma \ref{lemhddim} below, which shows that the posterior does not overshoot the true dimension too much, as well as the auxiliary Lemma \ref{hdlemdenom} which bounds the denominator in the $\rho$--posterior's definition. \\

\begin{proof}
By definition of the $\rho$--posterior, for $B\subset \RR^p$ a measurable set, 
\begin{equation} \label{rhoposthd}
\Pi_\rho[B\given Y] = \frac{\int_B \exp\{-\rho \|Y-X\te\|^2/2 \} d\Pi(\te) }{\int \exp\{-\rho \|Y-X\te\|^2/2\} d\Pi(\te)}
= \frac{\int_B \La_{\te,\te_0}(Y) d\Pi(\te) }{\int \La_{\te,\te_0}(Y) d\Pi(\te)}, 
\end{equation}
where $\La_{\te,\te_0}(Y)=\La_{\te,\te_0}(Y,X,\rho)=\exp\{ -\rho \|X(\te-\te_0)\|^2+\rho (Y-X\te_0)'X(\te-\te_0) \}$ is the ratio of the likelihoods at points $\te$ and $\te_0$ respectively.  Denote $E_0:=E_{\te_0}$ and $s_0=|S_{\te_0}|$. 

Since $\La_{\te,\te_0}(Y)=(p_{\te}(Y)/p_{\te_0}(Y))^\rho$, we have  
\[ E_0 \La_{\te,\te_0}(Y) =\int p_\te(y)^\rho p_{\te_0}(y)^{1-\rho} d\mu(y)=e^{-(1-\rho)D_\rho(P_\te,P_{\te_0})}. \] 
One now bounds the denominator in \eqref{rhoposthd} using Lemma \ref{hdlemdenom}: denoting by $\Delta_n$ the deterministic lower bound from the Lemma and using Fubini's theorem, for any measurable $B$,
\begin{equation} \label{interhd}
E_0\Pi_\rho[B\given Y] \le \Delta_{n}^{-1}\int_B E_0 \La_{\te,\te_0}(Y) d\Pi(\te)
\le  \Delta_{n}^{-1}\int_B e^{-(1-\rho)D_\rho(P_\te,P_{\te_0})} d\Pi(\te).
\end{equation} 
Set $B:=\{\te:\ D_\rho(P_\te,P_{\te_0})>(M/2)\rho s_0\log{p} \}$ for a constant $M$ suitably large to be chosen. By Lemma \ref{lem:renyig}, we have $D_\rho(P_\te,P_{\te_0})=\rho\|X(\te-\te_0)\|^2/2$.  Deduce, with $C_1=M\rho/2$,
\[ E_0\Pi_\rho[B\given Y] \le e^{-C_1 s_0\log{p}} \Pi[B]/\Delta_n \leqa  e^{-C_1 s_0\log{p}}e^{(B+3)s_0\log{p}}, \]
 where $\Delta_n$ is bounded from below as in the proof of Lemma \ref{lemhddim} below. Deduce that 
 \[  E_0\Pi_\rho[\te:\ \|X(\te-\te_0)\|^2>Ms_0\log{p} \given Y] 
 \le e^{-(M/2)s_0\log{p}},
  \]
provided $M$ is chosen large enough.   

Now set $D:= \{\te: \|\te-\te_0\|^2>Ks_0\log{p}\}$ with $K=M\|X\|^{-2}\omega(Cs_0)^{-2}$, and $E:=\{\te:\ |S_\te|\le (B+4)s_0\}$. By Lemma \ref{lemhddim},
\[ E_0\Pi_\rho[D \given Y]\leqa e^{-s_0\log{p}} + E_0\Pi_\rho[\te\in D\cap E \given Y].   \]
Since for $\te\in E$, we have $|S_{\te-\te_0}|\le |S_{\te}|+|S_{\te_0}|\le (B+5)s_0$, for any $\te\in E$ it holds
\[  \| X(\te-\te_0)\|^2 \ge \omega((B+5)s_0) \|X\|^2\|\te-\te_0\|^2. \] 
Putting the previous bounds together leads, for any $\te\in D\cap E$, to $\|X(\te-\te_0)\|^2>Ms_0\log{p}$ which given the convergence rate for $X\te$ obtained above implies the result.

\end{proof} 
 
\begin{lem} \label{lemhddim}
Let $\Pi$ be the prior as in Theorem \ref{thm:hdreg} and fix $\rho\in(0,1)$. Then uniformly over $\te_0\in \ell_0[s_0]$, for $B$ the constant in \eqref{condsig},
\[ E_0 \Pi_\rho\left[\te:\ |S_\te|> (B+4) s_0 \given Y \right] \leqa e^{-s_0\log{p}}. \]
\end{lem}
\begin{proof}
Since the R\'enyi divergence is nonnegative, it follows from \eqref{interhd} that $E_0\Pi_\rho[B\given Y]\le \Pi[B]/\Delta_n$ for any measurable $B$. Applying this for $B:=\{\te:\ |S_\te|> (4+B) s_0\}$ as well as noting that
\[\Pi[B]=\sum_{k=(4+B)s_0}^p \pi_p(k)\le \sum_{k>(4+B)s_0} p^{-k}\le 2p^{-(4+B)s_0}, \]
using that the normalising constant for $\pi_p$ is at least $1$, leads to, combining with Lemma \ref{hdlemdenom} for which the term $M_n(\te_0)$ is bounded with \eqref{condsig},
\[ E_0 \Pi_\rho\left[B\given Y\right]\leqa e^{s_0\log(2+2Ap^B)+ 2s_0\log{p}+s_0\log(7e)}\Pi[B]. \]
Since $2+2Ap^B\le p^{B+1}$ for large enough $p$, the result follows.
\end{proof}

\begin{lem} \label{hdlemdenom}
Let $\Pi$ be the prior as in Theorem \ref{thm:hdreg} and fix $\rho\in(0,1)$. Set
\[ M_n(\te_0)=\sum_{j\in S_0} \log(2+2\la^2\te_{0,j}^2). \] 
Then the denominator in \eqref{rhoposthd} satisfies, for any (finite) data vector $Y$,
\[ D_\rho:=\int \La_{\te,\te_0}(Y) d\Pi(\te) \ge 
\exp\{-M_n(\te_0)-\rho/2 - \log{2} -2s_0\log{p} -s_0\log(7e)\}. \]
\end{lem}
\begin{proof} 
By restricting the denominator to the set of vectors of support $S_0$ and setting $v_{S_0}=\te_{S_0}-\te_{0,S_0}$,
\begin{align*}
 D_\rho & \ge \frac{\pi_p(s_0)}{\binom{p}{s_0}} \int \La_{\te,\te_0}(Y) \ga_{s_0}(\te_{S_0})d\te_{S_0} \\
& =  \frac{\pi_p(s_0)}{\binom{p}{s_0}} \int \exp\left\{-\rho\|X v_{S_0}\|^2/2+\rho(Y-X\te_0)'Xv_{S_0} \right\} \ga_{s_0}(v_{S_0}+\te_{0,S_0})dv_{S_0},
\end{align*} 
where $\ga_{S_0}$ denotes the product density of $|S_0|$ variables of law $\Ga$. 
A direct computation shows that for $\ga$  the Cauchy$(\la)$ density, for any reals $x,t$, the following inequality holds
\[ \frac{\ga(x)}{\ga(x+t)} \le 2(1+\la^2t^2). \]
Taking the inverse and applying it to the product of density ratios $\ga_{s_0}(v_{S_0}+\te_{0,S_0})/\ga_{s_0}(v_{S_0})$ leads to 
\begin{align*}  
 D_\rho & \ge \frac{\pi_p(s_0)}{\binom{p}{s_0}}\exp\{-M_n(\te_0)\} 
 \int \exp\left\{-\rho\|X v_{S_0}\|^2/2+\rho(Y-X\te_0)'Xv_{S_0} \right\} \ga_{s_0}(v_{S_0}) dv_{S_0}.
\end{align*} 
Jensen's inequality with the exponential function and the probability measure with density proportional to 
$\exp\{-\rho\|X v_{S_0}\|^2/2\}\ga_{s_0}(v_{S_0})$ on $\RR^{S_0}$ implies, noting that this distribution is symmetric (and thus mean--zero), that the integral in the last display is bounded below by 
\[ I_0 := \int \exp\left\{-\rho\|X v_{S_0}\|^2/2\right\}\ga_{s_0}(v_{S_0}) dv_{S_0}. \]
Since $\|Xv\|^2\le \|\sum_{j=1}^p v_jX_{\cdot,j}\|^2 \le \|v\|_1^2 \|X\|^2$, restricting the last integral to $\{\|v_{S_0}\|_1\|X\|\le 1\}$ gives
\[ I_0 \ge e^{-\rho/2} \int_{\|v_{S_0}\|_1\|X\|\le 1} \ga_{s_0}(v_{S_0}) dv_{S_0}.  \]
To bound the last quantity from below, using a bound of the $L^1$--norm in terms of the maximum,
\[ \Ga^{\otimes s_0}\Big[\sum_{i\in S_0} |\te_i|\le 1/\|X\|\Big]\ge \Ga^{\otimes s_0}\left[\max_{i\in S_0}|\te_i|\le (\|X\|s_0)^{-1}\right]=\left[ \Ga[\|X\||\te|\le 1/s_0]\right]^{s_0}. \]
For $s_0\ge 1$, since the standard Cauchy density is at least $1/(2\pi)\ge 1/7$ on $[-1,1]$,   the last display  is at least $(7s_0)^{-s_0}$. Using the standard bound $\binom{p}{s_0}\le (pe/s_0)^{s_0}$, and that the normalising constant from the prior \eqref{hdpriord} is at most $2$ (we assume $p\ge 2$ so that the prior is well defined), the result follows.
\end{proof} 

We conclude this section by mentioning that more precise results can be obtained for the posterior in this setting. For instance, if coefficients of $\te_0$ are all large enough, one can recover the true support $S_{\te_0}$; it it also possible for certain priors to derive results on the limiting distribution as mixtures of normals, see e.g. \cite{spahd}. 

\section{Deep neural networks}

 A neural network is a structure with an input layer, a number of hidden layers and an output layer. Each hidden layer has a number of units called {\em neurons}. An input is taken by the network in the form of a vector $x=(x_1,\ldots,x_d)'$ in $\RR^d$ for some $d\ge 1$. The input layer just contains $d$ units: each one passes the input coordinate $x_i$ unchanged to the neurons of the first hidden layer. A network with one hidden layer is depicted in Figure \ref{fig:NN}. 
Then $x$ is modified at the level of the first hidden layer in a way described below. 
Each layer is linked to the next one by arrows between neurons (or between units of input or output layer and a neuron). 

\begin{figure}[H]
\centering
\includegraphics{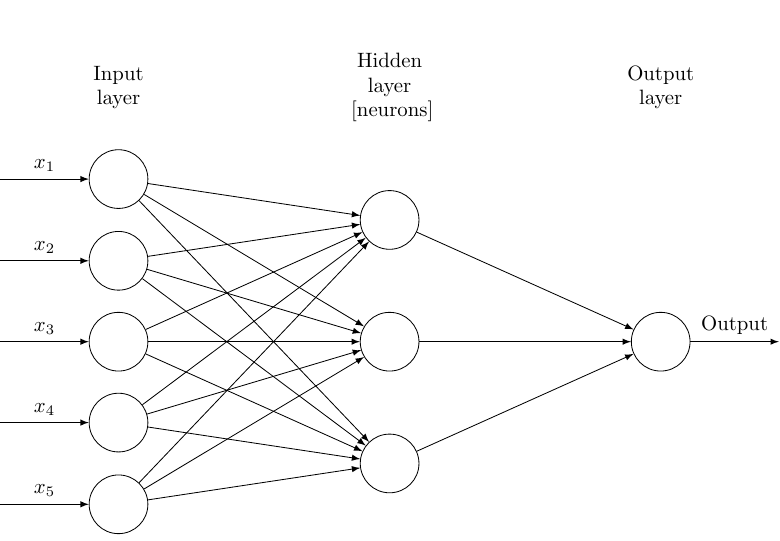}
\caption{Structure of neural network with one hidden layer}
\label{fig:NN}
\end{figure}

{\em Activation function.} The action of a given neuron is specified by an \sbl{activation function}, a function $\sigma:\RR\to\RR$. In the sequel, we consider the \sbl{ReLU activation} (ReLU stands for Rectified Linear Unit) given by
\begin{equation} \label{def:relu}
 \sigma(x) = x \vee 0 = x_+. 
\end{equation} 
To encode the action of all neurons in a given layer with input dimension $r\ge 1$,  we define the multidimensional shifted activation function as, given $v=(v_1, \ldots, v_r)'$ a vector of \sbl{shifts},
\begin{equation} \label{def:sv}
 \sigma_v y = \sigma_v(y) :=
\begin{bmatrix}
           \sigma(y_1-v_{1}) \\
           \sigma(y_2-v_{2})  \\
           \vdots \\
           \sigma(y_r-v_{r}) 
         \end{bmatrix}
\end{equation}

{\em Action of one layer.} Each arrow is given a {\em weight} that multiplies the input of the arrow. Let us consider a given layer with $q$ neurons and a vector of biases $v=(v_1, \ldots, v_q)'$. At the level of each neuron of the layer the incoming numbers from each arrow are summed, a bias relative to the neuron is applied and the result is finally passed through the activation function $\sigma$. If all the weights are aggregated in a matrix, say $W$, with dimension $p\times q$, where $p$ is the input dimension and $q$ the number of neurons of the considered layer, this means that the vector output of the considered layer is $\sigma_v Wx$. The action of the first neuron of the first layer is illustrated in Figure \ref{fig:neuron}. 

\begin{figure}[H]
\centering
\includegraphics{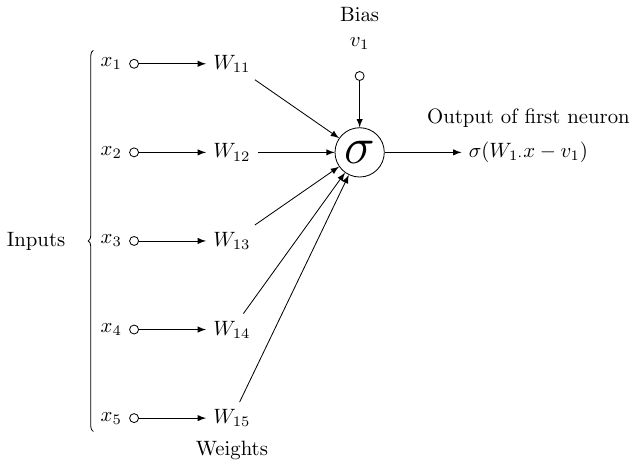}
\caption{Operations at the level of the first neuron and output}
\label{fig:neuron} 
\end{figure}

{\em Depth of the neural network.} The previous operations are repeated over successive layers, each new layer taking as input the output of the previous one (this of course requires that dimensions match). The total number of layers is denoted by $L$ and called network \sbl{depth}. A network with two layers is depicted in Figure \ref{fig:nn2}: the layers that are inbetween  input and output layers are called `hidden' layers.\\

{\em Width vector.} The number of neurons of a layer is called \sbl{width}. The width vector $\bp=(p_0,p_1,\ldots,p_L,p_{L+1})$ with $p_0=d$ (input dimension) and $p_{L+1}=1$ collects all widths. \\

{\em Network architecture.} The pair $(L,\bp)$ defines a network architecture. 
The \sbl{network parameters} are the entries of the matrices $(W_j)_{0\le j\le L}$ and shifts vectors $(v_j)_{1\le j\le L}$ of the successive layers. The total number of parameters is 
\begin{equation}\label{def:tot}
T= \sum_{l=0}^L p_l p_{l+1} + \sum_{l=1}^L p_l.
\end{equation}

\begin{figure}[H]
\centering
\includegraphics{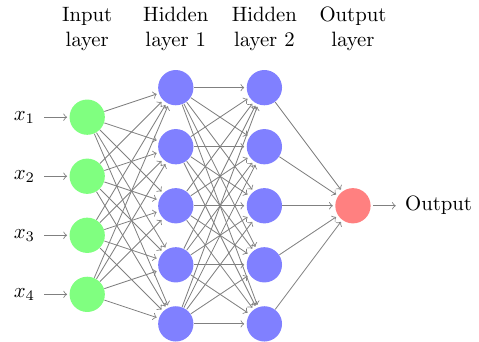}
\caption{Neural network with $2$ hidden layers}
\label{fig:nn2} 
\end{figure}

For a matrix $A=(a_{ij})$, let $\|A\|_\infty$ denote the (entrywise) maximum of the $|a_{ij}|$s. Similarly $\|v\|_\infty$ denotes, for a vector $v$, the maximum of $|v_i|$s. \\

\begin{definition}\sbl{[Global network]}$\ $
Denoting by $(W_j)_{0\le j\le L}$ and $(v_j)_{1\le j\le L}$ the successive weight matrices and bias vectors, the global network operates as a map $x\to f(x)$ from $\RR^d$ to $\RR$, where
\begin{equation} \tag{$\square$}\label{eq:dnn}
f(x)=W_L\sigma_{v_L}W_{L-1}\sigma_{v_{L-1}}\cdots\sigma_{v_1}W_0x,
\end{equation}
where $\sigma_v$ is defined in \eqref{def:sv} and the activation $\sigma$ is the ReLU function \eqref{def:relu}.
\end{definition}
For a vector or matrix $B$, let $\|B\|_0$ denote the  number of nonzero coefficients of $B$.  In the sequel, DNN stands for `Deep Neural Network', where {\em deep} means that the depth is not `small' (e.g. 1 or 2) but rather a number possibly allowed to go to infinity (slowly, perhaps) with the number of observations. A network with just a few layers (e.g. $1$ or $2$) is sometimes referred to as {\em shallow} network.\\

\begin{definition}\sbl{[DNN classes]}$\ $
Given a network architecture $(L,\bp)$, we denote, setting $v_0=0$,
\begin{align*} 
\cF(L,\bp) = \big\{f\ & \text{ as in } \eqref{eq:dnn},  
 \ \text{ for some } (W_j)_{0\le j\le L}, (v_j)_{1\le j\le L} , \quad \max_{0\le j\le L}\left( \|W_j\|_\infty \vee \|v_j\|_\infty\right) \le 1 \big\},
\end{align*}
and we also set, for $s>0$ a sparsity parameter,
 \begin{align*} 
\cF(L,\bp,s) = \cF(L,\bp,s,F) = \Big\{f\in \cF(L,\bp),\quad  
\sum_{j=0}^L (\|W_j\|_0 + \| v_j \|_0) \le s,\quad \| f \|_\infty\le F
\Big\}.
\end{align*}
\end{definition}
We assume here that the network parameters are bounded in absolute value by $1$. Another positive constant could be used. The rationale behind this choice is that in practice most of the time network parameters are initialized using bounded parameters. It is also typically observed that even after training the parameters of the network remain quite close in range to initial parameters. From the theoretical point of view, some approximation results are quite easily reachable if network parameters are allowed to be very large. But in order to be closer to practical applications where parameters typically remain bounded, we restrict ourselves to that case, and 
we will see below that this does not prevent us to obtain good inference properties.\\

The regularity class of functions we consider in these notes is the classical class of H\"older functions in dimension $d$ over, say, the unit cube: 
\[ \cC^\be_d([0,1]^d,K)=\Big\{f:[0,1]^d\to \RR:\ \ 
\sum_{\ba :\ \sum_{i=1}^r \al_i < \be}
\| \partial^{\ba} f \|_\infty + 
 \sum_{\ba:\ \|\ba\|_1=\lfloor \be\rfloor} \sup_{x,y\in D,\ x\neq y} 
\frac{|\partial^{\ba}f(x)-\partial^{\ba}f(y)|}{\|x-y\|_\infty^{\be-\lfloor \be\rfloor}} 
\le K\Big\}.
 \]
Classes of composite functions could be considered as well as, we comment on that again below.

\subsubsection*{Prior distributions on DNNs and posterior concentration}

Let us define an $n$--dependent prior distribution $\Pi$ as follows; the choices of parameters are motivated by the theoretical properties described in the next section. Set
\begin{equation} \label{def:ppar}
 L = \log^2{n},\qquad p_1=p_2=\cdots=p_L=n.
\end{equation}
The total number of parameters $T$ in a network of $\cF(L,\bp)$ is then, using \eqref{def:tot},  of order $Ln^2$ (taking $n$ large enough so that $d\le n$). Let 
$\cT=\RR^T=\{(\te_t)_{1\le t\le T}\}$ denote the parameter set, a parameter collecting elements of vectors $(v_j)_{1\le j\le L}$ and matrices $(W_j)_{0\le j\le L}$, written in some chosen (fixed throughout) given order; the chosen order does not matter. 
 
\begin{definition} \sbl{A sparse prior on DNNs}  \label{def:priordnn}
\begin{enumerate}
\item Draw a sparsity parameter $s$ in $\{0,1,\ldots,T\}$ according to a distribution $\pi_s$ defined by
\[ \pi_s(k) \propto e^{-k\log k}. \]
\item Given $s$, draw a subset $S\subset \{0,1,\ldots,T\}$ of cardinality $s$ uniformly at random among all possible such subsets.
\item Given $S$, set, for uniform variables independent across coefficients,
\[ 
\te_t \sim 
\begin{cases}
\delta_0 \quad \text{if } t\notin S\\
\text{Unif}[-1,1]\quad \text{if } t\in S
\end{cases}.
\]
\end{enumerate}
\end{definition}
By definition, the prior $\Pi$   puts mass on
 the  set, for $T$ the total number of possibly active parameters, and with no restriction on $F$ (so taking $F=+\infty$),
 \[  \Pi\left[\cF(L,\bp,T,+\infty)\right]=1. \]
 
\subsubsection*{Posterior contraction}

To make for the simplest statement, we consider the Gaussian white noise model setting, with a prior on $f$ defined as above. Analogous statements can easily be written in density estimation or nonparametric fixed design regression, similar to what we did in Chapter \ref{chap:rate1}. 

\begin{thm} \label{thm:dnn} \re{[Posterior convergence rate for H\"older functions]}$\ $
Let $f_0\in \cC_d^\be([0,1]^d,K)$ a function of regularity $\be>0$. Define the prior $\Pi$ on the function $f$ from the white noise model as in Definition \ref{def:priordnn}. Then there exists $M>0$ large enough such that
\[ E_{f_0}\Pi[ \|f-f_0\|_2 >M\veps_n \given X] =o(1),
 \]
as $n\to \infty$, where $\veps_n^2=(\log{n})^3 n^{-2\be/(2\be+d)}$. 
\end{thm}

Theorem \ref{thm:dnn} shows that DNNs with logarithmic depth and polynomial width (both in terms of $n$) are good models for approximating $\be$--smooth functions, for any $\be>0$. The output \eqref{eq:dnn} of a network with ReLU activation function is necessarily a piecewise affine function: 
 it is interesting to note that while such a network is not highly smooth (the `highest' regularity one can hope for is Lipschitz, but not $\cC^2$ for instance, except in the trivial case where the output is a linear function), it still enables one to recover optimal minimax rate, up to log factors, for any arbitrarily high smoothness level $\be>0$. \\
 
It can be shown that DNNs also adapt nearly optimally to other regularity structures, for instance in case the true $f_0$ is a composite function $f_0=g_k\circ g_{k-1}\circ\cdots\circ g_1\circ g_0$, DNNs also yield near--optimal rates: we refer to \cite{jsh20}, from which some of the key lemmas below are borrowed. This setting is discussed further in Section \ref{sec:dgp}.

\subsubsection*{Proof of posterior concentration}

\begin{proof}[Proof of Theorem \ref{thm:dnn}]
We check the three conditions of the GGV theorem successively, with here $d$ equal to the $\|\cdot\|_2$--distance on $[0,1]$, for which appropriate tests exists in the white noise model.\\ 

{\em Sieve.} For $\ols$ an integer to be chosen below, let $\cF_n := \cF(L,\bp,\ols)$. 
By definition of the prior,
\[ \Pi[\cF_n^c]\le \pi_s(s>\ols)\leqa e^{-\ols\log\ols}. \]
The sieve condition of the GGV Theorem is then fulfilled if $\ols \log\ols \geqa n \uli\veps_n^2$.\\

{\em Entropy.} Using the entropy Lemma \ref{lem:dnn_ent}, with $\|\cdot\|_2\le
\|\cdot\|_\infty$, for any $\delta>0$,
\[ \log N(\delta,\cF_n,\|\cdot\|_2) \le
\log N(\delta,\cF(L,\bp,\ols),\|\cdot\|_\infty) \le (\ols+1)\log\left(\frac{2(L+1)V^2}{\delta}\right),
\]
where $V\le (n+1)^{L+2}\le e^{C\log^3(n)}$ for a large enough constant $C$, so that for any vanishing sequence $(\oli\veps_n)$ with $\oli\veps_n \ge 1/\rn$,
\[\log N(\oli\veps_n,\cF_n,\|\cdot\|_\infty) \leqa \ols (\log{n})^3. \] 
For the entropy condition of the GGV Theorem to be fulfilled, we need $\ols (\log{n})^3\leqa n\oli\veps_n^2$. \\

{\em Prior mass condition.} For $\tilde{f}_0$ a suitable DNN--approximation of $f_0$ (to be defined below), 
\[ \Pi[\|f-f_0\|_2\le \uli\vepn] \ge  \Pi[\|f_0-\tilde{f}_0\|_2 + \|f-\tilde{f}_0\|_2 \le \uli\vepn]
 \ge  \Pi[\|f_0-\tilde{f}_0\|_\infty + \|f-\tilde{f}_0\|_\infty \le \uli\vepn].
\]
Let us apply Theorem \ref{thm:dnn_approx} to $f_0$ and the choices, for $A_0$ large enough constant (depending on $d, K, \be$),
\[ N=n^{\frac{d}{2\be+d}},\quad m=A_0\log{n}. \]
 There exists $\tilde{f}_0$ in $\cF(L',(d,\Lambda,\ldots,\Lambda,1),s_0)$ that approximates $f_0$ as in Theorem \ref{thm:dnn_approx}, with depth $L'\leqa m\leqa  \log{n}$ and sparsity $s_0'\leqa (\log{n})N$. Up to adding at the end of the network a subnetwork that equals the identity over $L/L'\approx \log{n}$ layers (note $x=(x\vee 0)-(-x)\vee 0$ which uses one layer, two units and $6$ parameters, out of which $4$ are non--zero; this is called `synchronisation'), one can suppose that $\tilde{f}_0\in \cF(L,(d,\Lambda,\ldots,\Lambda,1),s_0)\subset \cF(L,(d,n,\ldots,n,1),s_0)$ (this is called `enlarging'), where $s_0\leqa s_0'$ (we add of the order $\log{n}$ non--zero parameters). \\
 
 Let $\te_0=(\te_0(\tilde{f}_0))_{1\le t\le T}$ denote the collection of parameters of a  network as above encoding $f_0'$. By definition, only $s_0$ are non--zero. Let $S_0$ denote the subset of $\{1,\ldots,T\}$ corresponding to these non--zero parameters. Similarly, for a draw $\te=(\te_t)$ from the prior let $S(\te)$ denote the set of indices of its non--zero parameters. \\
 
By Theorem \ref{thm:dnn_approx}, we know that, for large enough $n$ (depending on $\be, K, d$), 
\[ \|f_0-\tilde{f}_0\|_\infty \leqa  \frac{N}{2^m} + N^{-\frac{\beta}{d}}
\leqa n^{-\Delta}+n^{-\frac{\be}{2\be+d}}, \]
 where $\Delta$ can be made arbitrarily large for large $A_0$, so the last display is smaller than $\uli\veps_n/2$ as long as $n^{-\frac{\be}{2\be+d}}\leqa \uli\veps_n$. If this holds we then have
 \[ \Pi[\|f-f_0\|_2\le \uli\vepn] \ge
 \Pi[\|f-\tilde{f}_0\|_\infty \le \uli\vepn/2] \ge\Pi[\|f-\tilde{f}_0\|_\infty \le \uli\vepn/2\given S(\te)=S_0]\Pi[S(\te)=S_0]
 \]
 On the event that $S(\te)=S_0$, the corresponding networks encoding $f$ and $\tilde{f}_0$ have same index sets for their non--zero coefficients, so we may now use the `error propagation' Lemma \ref{lem:dnn_propa} to obtain
 \begin{align*}
  \Pi[\|f-\tilde{f}_0\|_\infty \le \uli\vepn/2\given S(\te)=S_0]
 & \ge \Pi\left[\forall\, t\in S_0,\ \ |\te_t(f)-\te_t(\tilde{f}_0)|\le \frac{\uli\veps_n}{2V(L+1)} \right] \\
 & \ge \prod_{t\in S_0} \frac{\uli\veps_n}{V(L+1)} = \left(\frac{\uli\veps_n}{2V(L+1)}\right)^{s_0}\geq e^{-Cs_0\log^3{n}},
 \end{align*}
for a large enough  $C>0$, using that $\log{V}\leqa \log^3{n}$. On the other hand, again for large $C>0$,
\[\Pi[S(\te)=S_0] = \frac{1}{ \binom{T}{s_0} }e^{-s_0\log(s_0)}\ge
e^{-s_0\log(Te/s_0)-s_0\log{s_0}}\ge e^{-Cs_0\log{n}}. \]
Deduce that $\Pi[\|f-f_0\|_2\le \uli\vepn] \ge e^{-n\uli\veps_n^2}$ provided one chooses
\[ n\uli\veps_n^2\geqa s_0\log^3{n}.\]

{\em Conclusion.} Let us choose, for suitably large $A_2>0$,
\[\uli\vepn^2=A_2\max(n^{-\frac{2\be}{2\be+d}}, s_0\log^3{n}/n)=A_2(\log{n})^4N/n
= A_2(\log{n})^4n^{-\frac{\be}{2\be+d}}. \]
Then the conditions  on $\veps_n$ for the prior mass are satisfied and the condition for the sieve is $n\uli\veps_n^2\leqa \ols\log\ols$ which holds for the choice 
 $\ols=A_3s_0\log^2n$ for large enough $A_3>0$. Finally, from the sieve condition one gets the condition
 \[\oli\veps_n^2\geqa \frac{\ols \log^3{n}}{n}\geqa (\log{n})^6 n^{-\frac{2\be}{2\be+d}}. \]
The proof of the Theorem is complete since $\veps_n=\oli\veps_n\vee \uli\veps_n=\oli\veps_n$ as desired.
\end{proof}

\section{Complement: generic properties of DNNs}

\subsubsection*{Error propagation in a neural network and entropy}

Considering the class of functions $\cF(L,\bp,s)$, let us denote
\begin{equation} \label{def:v}
V:=\prod_{l=0}^{L+1}(p_l+1).
\end{equation}

\begin{lem} \label{lem:dnn_propa}
Let $f,f^*$ be two functions in $\cF(L,\bp,s)$ with matrix parameters $W_k, W_k^*$ and shift vectors $v_k, v_k^*$ for $k=0,1,\ldots,L+1$. Suppose that every individual parameter of $f$ (i.e. elements of matrices $W_k$ or bias vectors $v_k$) is at most $\veps>0$ away from the corresponding parameter of $f^*$. Then for $V$ as in \eqref{def:v},
\[ \| f - f^* \|_\infty \le \veps V(L+1). \]
\end{lem}

\begin{lem} \label{lem:dnn_ent}
For $V$ as in \eqref{def:v} and any $\delta>0$,
\[ \log N(\delta,\cF(L,\bp,s),\|\cdot\|_\infty) \le (s+1)\log\left(\frac{2(L+1)V^2}{\delta}\right).\]
In particular if $L\leqa \log{n}$ and $p_l\le n$ for all $l$, we have 
$\log N(\delta,\cF(L,\bp,s),\|\cdot\|_\infty) \leqa s\log^2(n)\log(1/\delta)$.
\end{lem}
The proof of these lemmas can be found in \cite{jsh20}.

\subsubsection*{Approximation properties for H\"older functions}

The next Theorem and following lemmas are borrowed from \cite{jsh20}. We try to give the ideas of the key steps below following \cite{jsh20} (the proof of Lemma \ref{dnn:sq} is original).

\begin{thm} \label{thm:dnn_approx}
\re{[Approximation of smooth functions by DNNs]} $\ $ 
Let $f\in \cC_d^\be([0,1]^d,K)$ a function of regularity $\be>0$. Let $m, N\ge 1$ be two integers. There exists a network, with $\La:=6(d+\lceil \be\rceil)N$,
\[ \tilde{f}\in \cF(L,(d,\Lambda,\ldots,\Lambda,1),s,\infty)\]
 with depth and sparsity verifying, for $C_0=1+\log_2 \lceil (d\wedge \be)\rceil$,  $c_0=141(d+\be+1)^{3+d}$, 
\[ L = 8+C_0(m+5),\qquad s\le c(m+6)N,\]
such that, for $c_1=(2K+1)(1+d^2+\be^2)6^d$ and $c_2=K3^\be$, and $N\ge
 (\be+1)^d \vee (K+1)e^d$,
\[ \| \tilde{f} - f \|_\infty \le c_1 \frac{N}{2^m} + c_2N^{-\frac{\beta}{d}}. \]
\end{thm}

[Sketch of proof] The general idea is as follows: there are two main steps. The first is not specific to DNNs and is that any $\be$--H\"older function can be well--approximated locally, using Taylor expansions, by a polynomial of order $\lfloor \be\rfloor$: one can approximate $f_0$ by a piecewise polynomial function, with a quality of approximation that depends on $\be$. The second idea, where the choice of activation function $\sigma$ comes in, is that it is possible to approximate quickly, in one dimension, the monomial $x\to x^2$ using a ReLU network. From there one then shows that ReLU networks suitably approximate $x\to x^p$ for $p\ge 2$; one can also check that the argument extends to dimensions $d\ge 2$ for approximating general monomials. From monomials one can easily approximate polynomials by combining networks, and now one can connect to the first part of the argument, by constructing a network that approximates the piecewise polynomial function mentioned above, that itself approximates $f_0$.\\

\begin{lem} \label{dnn:sq}
\sbl{[Approximating $x(1-x)$ with piecewise affine functions]}\\
Let $T^1:[0,1]\to [0,1/4]$ and more generally $T^k:[0,2^{-2(k-1)}]\to [0,2^{-2k}], k\ge 1$, be the maps 
\[ T^1(x)=\frac{x}{2}\wedge\left(\frac12-\frac{x}{2}\right),\quad
T^k(x)=\frac{x}{2}\wedge\left(\frac{1}{2^{2k-1}}-\frac{x}{2}\right).\]
Let us set $R^k:=T^k\circ T^{k-1}\circ\cdots\circ T^1$, for $k\ge 1$. Then for any $m\ge 1$,
\[ \left| x(1-x) - \sum_{k=1}^m R^k(x)\right| \le 4^{-m}.\]
\end{lem}
\begin{proof} Let $C(x)=x(1-x)$. The key is to observe the `fractal'-like property
\[ C(x) = T^1(x) +\frac14 C(4T^1(x)). \]
which can be checked e.g. algebraically (or can be seen on Figure \ref{pic:approx}  by noting that the picture repeats itself at a $1/4$ scale if one takes the first orange cord as a new $x$-axis). Next note that by definition $T^2(y)=T^{1}(4y)/4$ and more generally $T^{k+1}(y)=T^{1}(4^k y)/4^k$.  
By recursion one immediately obtains
\[ C(x) = T^1(x)+ T^2\circ T^1(x) + \cdots + T^k\circ\cdots\circ T^1(x) +  \frac1{4^k} C(4^k T^k\circ\cdots\circ T^1(x)). \]
The result follows by applying this with $k=m$ and noting that $C(\cdot)$ is bounded by $1/4$ on $[0,1]$.
\end{proof}

\begin{figure}[h]  \label{pic:approx}
\centering
\includegraphics[scale=1]{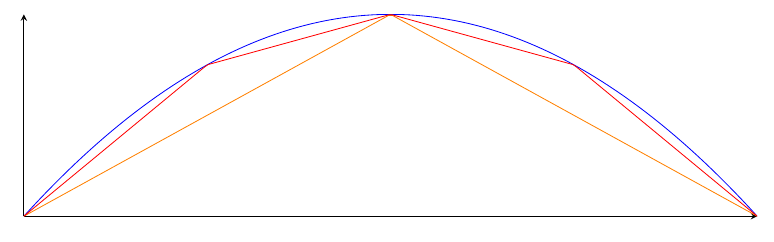}  
\caption{Approximating $x(1-x)$ via Lemma \ref{dnn:sq}: the error decreases exponentially with $m$}
\end{figure}

\begin{lem}
\sbl{[Approximating $(x,y)\to xy$ by a DNN]}
 Let $m\ge 1$. There exist a DNN that we denote $\text{Mult}_m(x,y)$ with
 \[ \text{Mult}_m \in \cF(m+4, (2,6,\cdots,6,2,2,2,1)),\] 
 such that for any $x,y\in[0,1]$ it holds $\text{Mult}_m(x,y)\in[0,1]$, $\text{Mult}_m(0,y)=\text{Mult}_m(x,0)=0$ and 
 \[ \left|\text{Mult}_m(x,y)-xy\right|\le 4^{-m}.\]
\end{lem}
\vspace{1cm}

In order to approximate a function $f\in \cC_d^\be([0,1]^d,K)$, we define a grid of $[0,1]^d$ as
\[ D(M)=\left\{ x_l = \left(\frac{l_j}{M}\right)_{j=1,\ldots,d},\quad l=(l_1,\ldots,l_d)\in\{0,1,\ldots,M\}^d
\right\}.
\] 
Around a given point $\pmb{a}\in[0,1]^d$, the function $f$ can be approximated by its Taylor polynomial: in dimension $d$ its expression is, for $\pmb{a}=(a_1,\ldots,a_d)$,
\[ P_{\pmb{a}}^\be(x) := \sum_{0\le |\al|< \be}  (\partial^{\al}f)(\pmb{a})\frac{(x-\pmb{a})^\al}{\al!}.\]
Taylor's expansion with Lagrange remainder gives, for any $f\in \cC_d^\be([0,1]^d,K)$,
\begin{equation} \label{dnn_taylor}
 |f(x) - P_{\pmb{a}}^\be(x)| \le K \|x-\pmb{a}\|_\infty^\be.
\end{equation} 
Define, again for any $f\in \cC_d^\be([0,1]^d,K)$ and $x=(x_1,\ldots,x_d)$,
\begin{equation} \label{def:localpol}
 P^\be f(x):= \sum_{x_l\in D(M)} (P_{x_l}^\be f)(x) 
\prod_{j=1}^d \displaystyle{(1-M|x_j-x_{l,j}|)_+}.
\end{equation}
Inside the hypercubes defined by consecutive gridpoints, $P^\be f(x)$ is a polynomial, so the overall function $P^\be f$ is piecewise--polynomial.\\
\begin{lem}
\sbl{[Approximation of $f$ by a piecewise--polynomial function]}
For any $f\in \cC_d^\be([0,1]^d,K)$, define $P^\be f$ as in \eqref{def:localpol}. Then
\[ \| f - P^\be f\|_\infty \le K M^{-\be}. \]
\end{lem}
{\em Proof.} One notes that the terms of the sum in the definition \eqref{def:localpol} are nonzero only at a given $x$ for $x_l$ such that $\|x-x_l\|_\infty\le 1/M$, otherwise the product in \eqref{def:localpol} is zero. Combine this with 
\[ \sum_{x_l=(l_1/M,\ldots,l_d/M)} \prod_{j=1}^d \displaystyle{(1-M|x_j-x_{l,j}|)_+} =
\prod_{j=1}^d \sum_{l=0}^M (1-M|x_j-l/M|)_+ = 1\]
(these functions form a `partition of unity') and Taylor's approximation \eqref{dnn_taylor} to obtain the result.

\section{Deep Gaussian process priors and adaptation to structure} \label{sec:dgp}

\subsubsection*{Motivation: compositional structures}

Here we state results in the so-called random design regression model (but we could state analog results in Gaussian white noise as in the previous section).

Consider observing i.i.d. pairs $Z_1=(X_1,Y_1),\ldots,Z_n=(X_n,Y_n)$ with
\begin{equation} \label{mod:reg}
Y_i=f_0(X_i)+\veps_i,\qquad 1\le i\le n,
\end{equation}
where $X_i$ are $[0,1]^d$--valued random variables (also called {\em design points}) and $\veps_i$ are independent standard normal $\cN(0,1)$ independent of the $X_i$'s, and $f_0:[0,1]^d\to \mathbb{R}$ an unknown function.  \\

Typical statistical goals in this setting are
\begin{itemize}
\item estimating the unknown regression function $f_0$ from the observations
\item finding estimates that behave (near--)``optimally" with respect to some criterion (e.g. minimax) over natural classes of parameters.
\end{itemize}

Let $\hat{f}(\cdot)=\hat{f}_n(Z_1,\ldots,Z_n)(\cdot)$ be an estimator of $f$.

The {\em prediction} risk in the setting of model \eqref{mod:reg} is defined as follows. Let $T$ be a `synthetic' data point, that is a variable independent of the $X_i$'s and generated from the distribution of $X_1$. Let
\begin{equation}\label{predrisk}
 R(\hat{f},f_0)= E\left[ \left(\hat{f}(T)-f_0(T)\right)^2 \right]=
E\left[ \left(\hat{f}(Z_1,\ldots,Z_n)(T)-f_0(T)\right)^2 \right].
\end{equation}

{\em Discovering a hidden `structure'.}  The `raw' regression data collected by the statistician takes the form, in the setting model \eqref{mod:reg}, of  $n$ vectors of size $d+1$: the $n$ pairs $(X_i^T,Y_i)$ with $X_i\in[0,1]^d$ and $Y_i$ a real, with the dimension $d$ possibly large (think for instance of e.g. $d=10$ or $20$). The unknown regression function $f_0(x_1,\ldots,x_d)$ depends on $d$ of variables, and we have seen that if $d$ is larger than a few units this may lead to a slow uniform convergence rate of the form $n^{-2\be/(2\be+d)}$ for the prediction risk. It is often the case though that the problem is effectively of smaller dimension than $d$. We give a number of frequently encountered examples
\begin{enumerate}
\item  $f_0$ in fact depends on just one variable (but we do not know it a priori), for instance
\[ f_0(x_1,\ldots,x_d)=g(x_1), \]
for some $g:[0,1]\to \RR$. In this case it seems reasonable to expect a rate $n^{-2\be/(2\be+1)}$, since the $f_0$ effectively depends on $1$ variable only.  
More generally, $f_0$ may depend on a small number $t\le d$ of variables, although we do not know a priori which ones, e.g.
\[ f_0(x_1,\ldots,x_d)=g(x_2,x_3,x_d), \]
in which case the effective dimension should be $3$, so we expect a rate $n^{-2\be/(2\be+3)}$.
\item In the preceding example, the function effectively depends on a small number of the original variables $x_i$, but it could depend on few variables only after transformation of the variables, for instance 
\[f_0(x_1,\ldots,x_d)=g(x_1+x_2+\cdots+x_d). \]
In this case $f_0(x_1,\ldots,x_d)=g(x')$ only depends on `one' variable $x'=x_1+\cdots+x_d$, so one expect a rate $n^{-2\be/(2\be+1)}$.
\item {\em Additive models.} It may be possible to write $f_0$ in an additive form
\[ f_0(x_1,\ldots,x_d) = \sum_{i=1}^d f_i(x_i), \]
for some functions $f_1, \ldots, f_d$ depending on one variable only. If all functions $f_i$ are at least $\be$--H\"older, one expects a rate $d \cdot n^{-2\be/(2\be+1)}$ that is $n^{-2\be/(2\be+1)}$ if $d$ is a fixed constant.
\item {\em Generalised additive models.} It may be possible to write $f_0$ in the  form
\[ f_0(x_1,\ldots,x_d) = h\left(\sum_{i=1}^d f_i(x_i)\right), \]
for some real-valued functions $f_1, \ldots, f_d$ (that are, as before, say all $\be$--H\"older) and an unknown real `link' function $h$ that is $\ga$--H\"older. One expects the rate to depend on $\be, \ga$, but not (too much) on the dimension $d$.
\end{enumerate}

{\em Class of compositions.} In all the settings of the previous paragraph, one may note that the original function $f_0$ can be written as a composition of functions
\[ f_0 = g_q\circ \cdots \circ g_1 \circ g_0,\]
for some integer $q\ge 1$. To fix ideas, for $q=1$ (composition of two functions) and if the input dimension is $d=10$, an example is
\[ f_0(x_1,\ldots,x_{10})= g_1(h_{01}(x_1,x_3),h_{02}(x_1,x_{10})),\]
with here $g_1$ a real function of $3$ variables, $h_{01}, h_{02}$ two real bivariate functions and $g_0$ the map $(x_1,\ldots,x_{10})\to (h_{01}(x_1,x_3),h_{02}(x_1,x_{10}))$. In this example, note that although the ambient dimension is $10$, all functions depend in fact of (at most) two variables, so the expected `effective' dimension of the problem is $2$, which should lead, if all mappings involved are $\beta$--smooth, to a rate of order $n^{-\be/(2\be+2)}$, much faster than the minimax rate $n^{-\be/(2\be+10)}$ in terms of the input dimension. 
 Similarly, for each of the examples in the list of the previous paragraph, {\em if one knew beforehand} that $f_0$ is in one class of the other, one could certainly develop a specific estimation method using the special structure at hand. In practice, however, it would be desirable to have a method that is able to automatically `learn the structure'. We are going to see that this is achieved by deep ReLU estimators.\\

Let us introduce the class, for $D=(d_0,\ldots,d_{q+1}), t=(t_0,\ldots,t_q), \beta=(\be_0,\ldots,\be_q)$, $d_0=d$,
\begin{align}
\mathcal{G}(q,D,t,\beta,K)=\Big\{\, f=g_q\circ\cdots& \circ g_0: \quad 
g_i= (g_{ij})_j:[a_i,b_i]^{d_i}\to [a_{i+1},b_{i+1}]^{d_{i+1}},\ \notag \\
& g_{ij}\in \mathcal{C}_{t_i}^{\be_i}([a_i,b_i]^{t_i},K),\quad  |a_i|,|b_i|\le K
\, \Big\}, \label{defcomp}
\end{align}
where we denoted $\mathcal{C}_{t_i}^{\be_i}$ for the H\"older ball over $t_i$ variables to insist on the fact that these functions depend on $t_i$ variables only (at most). 
The coefficients $t_i$ can be interpreted as the maximal number of variables each function $g_{ij}$ is allowed to depend on. In particular, this number is always at most $d_i$, but may be much smaller. Let us note that the decomposition of $f_0$ as a composition is typically not unique, but this is not of concern here because we are interested in estimation of $f_0$ itself only. \\

Compositional classes are quite rich and contain many interesting functions having a low dimensional ``effective dimensionality". There are quite popular for the analysis of deep learning algorithms. In particular, \cite{jsh20} shows that deep ReLU neural networks can get near optimal rates over such classes (one still assumes that some parameters of the classes are known). We see below that deep Gaussian processes possess analog properties (and are even fully adaptive to smoothness and structure). Let us first give an example and state what the optimal minimax rate over these classes is.\\

{\em Example.} In ambient dimension $5$, consider the function 
\[ f(x_1,x_2,x_3,x_4,x_5) = h_1(h_{01}(x_1,x_3,x_4),h_{02}(x_1,x_4,x_5),h_{03}(x_5)). \]
Then $h_0$ takes as input $5$ coordinates (so $d_0=5$) and takes its values in $\RR^3$ (hence $d_1=3$) so has three coordinate functions $h_{01}, h_{0,2}, h_{0,3}$, which themselves depend on only (at most) $3$ variables, so that here $d_0=3$. Since $h_1$ has three coordinates and (in general) depends on each of these, we have $d_1=t_1=3$. Finally, the final output of the regression is always a real number in this chapter so $d_2=1$. \\

Note that for $f_0=g_1\circ g_0$ with $d_1=d_0=t_1=t_0=1$ and $\be_0, \be_1\le 1$, it follows from the definition of the H\"older class that $f_0$ has regularity $\be_0\be_1$, so that one expects a convergence rate of order $n^{-\frac{\be_0\be_1}{1+2\be_0\be_1}}$. It turns out that the actual (or `effective') regularity depends on whether $\be_i\le 1$ or not. Let us define the following new `regularity' parameter
\begin{equation}\label{defbet}
\beta_i^* = \be_i \prod_{\ell=i+1}^q (\beta_{\ell}\wedge 1).
\end{equation}

{\em Convergence result for compositions.} Given $d, t, \beta$ as before, let us define the rate
\begin{equation}\label{vepset}
\veps_n^* = \max_{0\le i\le q}\, \left\{ n^{-\frac{\be_i^*}{2\be_i^*+t_i}} \right\}.
\end{equation}                                               
{\em Example.} For $d_0=d_1=t_0=t_1=q=1$ and $f=g_1\circ g_0$ with $\be_1,\be_0\le 1$, we have $\be_0^*=\be_0(\be_1\wedge 1)=\be_0\be_1$ and $\be_1^*=\be_1$, and the rate $\veps_n^*$ equals, since $\be_0\be_1\le \be_1$,                             
\[ \max \left ( n^{-\frac{\be_1}{2\be_1+1}} ,  n^{-\frac{\be_0\be_1}{2\be_0\be_1+1}}
\right) = n^{-\frac{\be_0\be_1}{2\be_0\be_1+1}}, \]                                                                 
which gives the rate announced above for this example. One may check that the formula \eqref{vepset} also gives the expected rate in the other examples above.

\begin{thm} \cite{jsh20} [Minimax optimality for compositions] \label{thm:lbcompo}
Consider the regression model \eqref{mod:reg}, where the $X_i$s are drawn from a distribution with density on $[0,1]^d$ which is bounded from above and below by positive constants. 
For arbitrary $\beta>0$, integer $q$ and vector of integers $D,t$, suppose $t_i\le \min(d_0,\ldots,d_{i-1})$ for all $i$. Then for large enough $K$,
   \[   \inf_{\hat{f}} \sup_{f_0\,\in\,  \mathcal{G}(q,D,t, \beta,K)} R(\hat{f},f_0) \ge c  {\veps_n^{*}}^2, \]
where the infimum is taken over all possible estimators $\hat f$ of $f$ in model \eqref{mod:reg}.
\end{thm}

\subsubsection*{Deep GPs: definition}

Deep Gaussian Processes (DGPs for short) were introduced in \cite{damianou13} and have witnessed increasing popularity in the machine learning community. Seminal results on posterior contraction have been obtained in \cite{fsh23} using a model selection (graph-)type prior to activate certain variables along the composition defining the DGP. Here we present recent results of \cite{cr23} for tempered posteriors using a `soft' selection of variables. 

\begin{definition} \sbl{[Deep Gaussian process].}  \label{def:dgp}
A \sbl{deep Gaussian process} (deep GP or  DGP) is a composition of Gaussian processes: for some integer $q\ge 1$, it is a stochastic process defined as 
\[ Z(t) = W_q \circ \cdots \circ W_0(t), \qquad t\in[0,1]^d,\]
where $W_i: \RR^{d_i}\to \RR^{d_{i+1}}$, with $(d_i)$ some integers and $d_0=d$, $d_{q+1}= 1$. 
\end{definition} 
{\em Remark.} Often, one restricts the range of the GPs in the composition defining a deep GP so that the successive GPs take values in a same compact subset, e.g. one sets $W_i' = (-M)\vee (W_i\wedge M)$ for some given $M>0$ and 
\[ Z' = W_q' \circ \cdots \circ W_0'. \]
The idea to take a deep GP as a prior is to make the prior more flexible (by adding `more randomness' compared to a single GP): it seems then likely that such a prior will approximate well compositions of functions -- which enable to approximate quite complex objects, as seen above -- (in the other direction, it can be shown that a single, non-deep GP, is {\em not} able to reach optimal rates relative to a compositional structure \cite{grs22}). 

As such a deep GP as in Definition \ref{def:dgp} is not yet flexible enough to adapt well to arbitrary compositional structure and smoothness. 
First, it seems natural to draw the `depth' $q$ randomly in the prior, but also, 
in order not to `overfit', to select randomly at each level which variables the process $W_i$ (mostly) depends on, in particular if one believes that there is a low dimensional compositional structure to which the true function $f_0$ we are trying to recover belongs. This motivates the following more general definition.

\begin{definition} \sbl{[Hierarchical DGP \cite{cr23}].}  \label{def:hdgp}
A hierarchical \sbl{deep Gaussian process} (HdGP) is defined as
 \begin{align*}
  q&&\sim\ &\Pi_q \\
 d_1,\dots,d_q\ &|\ q&\sim\ &\Pi_d[\cdot|q]\\
  \ma{(A_{ij})_k}\ &|\ q, d_1,\dots,d_q& \overset{\text{ind.}}\sim&\ma{\pi_{\tau}}^{\otimes d_i}\\
 \sbl{g_{ij}}\ &|\ q, d_1,\dots,d_q, \ma{A_{ij}} &
 \overset{\text{ind.}}{\sim}& W^{\ma{A_{ij}}} \\
 f\ &|\ q, d_1,\dots,d_q,\sbl{g_{ij}} &=\ &\Psi(\sbl{g_q})\circ \dots\circ \Psi(\sbl{g_0}),\\
 \end{align*}
 where the $(g_{ij})_j$ are the coordinate functions of $g_i$ (which takes values in $\RR^{d_{i+1}}$) and
\begin{itemize} 
\item $\Pi_q$ and $\Pi_d[\cdot\given q]$ are priors on integers,
\item $\ma{\pi_\ta}$ is a prior on scale parameters: for each function $\sbl{g_{ij}}$ in the composition, a vector $(A_{ij})$ of iid such scale parameters is drawn independently of the others, 
\item for a vector $\cA$ of dimension $d_{\cA}$, one denotes $W^{\cA}(u) = W(\cA_1u_1,\ldots,\cA_{d_{\cA}} u_{d_{\cA}})$, where given $d_\cA$, the law of $W$ is that of a given GP in dimension $d_\cA$,
 \item $\Psi(x)=(-M)\vee (x\wedge M)$ for some $M>0$.
\end{itemize}
\end{definition} 

\sbl{Deep horseshoe GP.} Let us consider the following prior choices: for the dimension $q$, one takes a prior with exponential decrease $\Pi_q(k)\propto e^{-q}$ and similarly for  $(d_i)$, one takes an exponentially decreasing prior for each $d_i$ independently. The coordinates functions $g_{ij}$ of the  function $g_i$ in the composition are given GP priors (given $A_{ij}$): they are taken to be centered GPs with squared--exponential covariance function, i.e. $E[W_xW_y]=\exp(-\|x-y\|^2)$, with $\|\cdot\|$ the euclidian norm on $\RR^{d_i}$. It now suffices to specify the prior on scale parameters. We take them independent with a {\em half-horseshoe distribution} with parameter $\ta>0$ fixed (e.g. $\ta=1$), see Definition \ref{def:hs} and Figure \ref{pic:hs}.\\

\begin{figure}[h]  \label{pic:hs}
\centering
\includegraphics[scale=0.6]{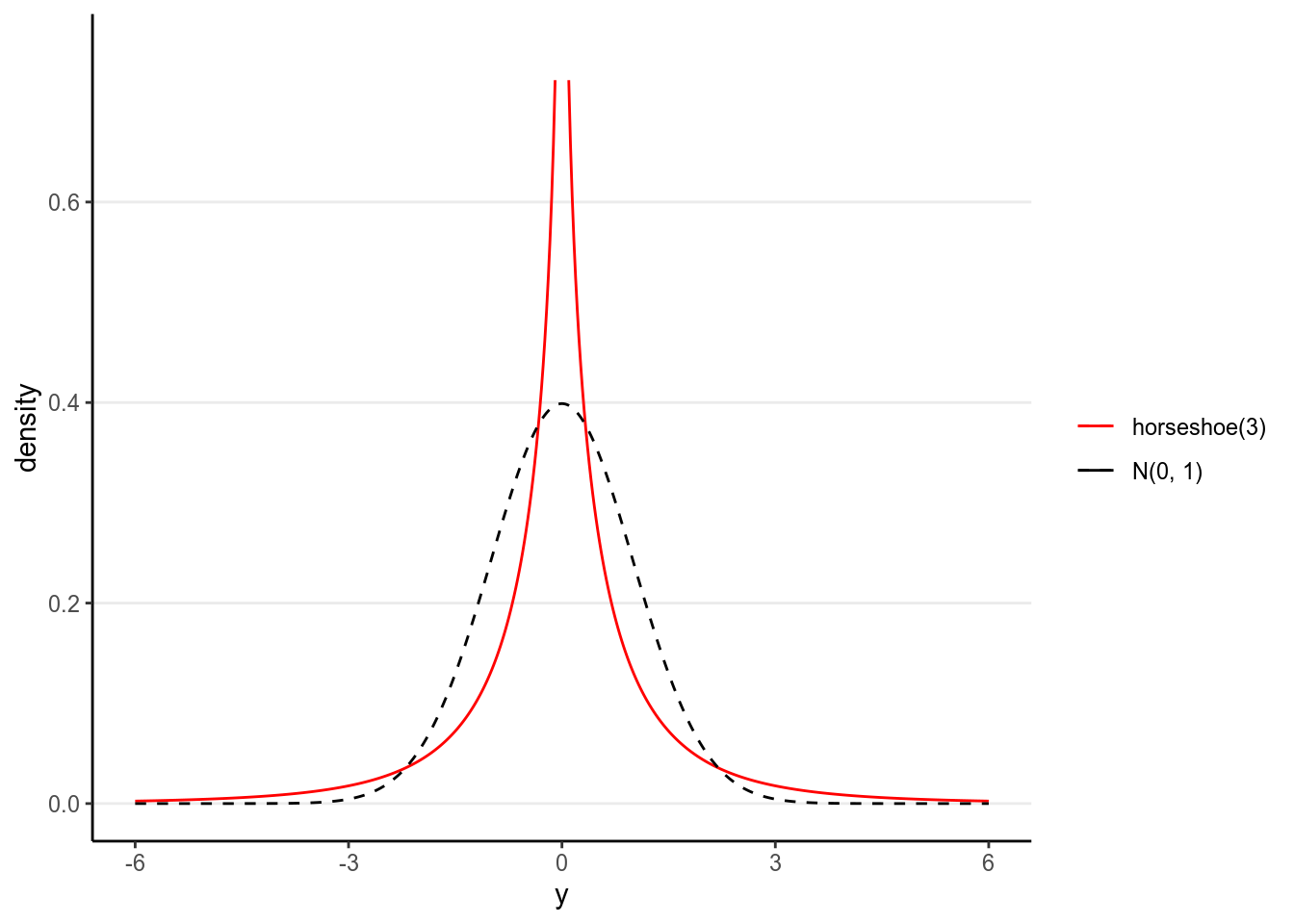}  
\caption{Horseshoe prior density with parameter $\tau=3$}
\end{figure} 

The idea of the above choice of Horseshoe deep GP prior is as follows: the density $\pi_\ta$ puts quite a lot of mass near zero and in the tails (i.e. the probability of drawing large values is quite high due to the Cauchy tails). Very small values of $A_{ij}$ allow to ``freeze" the corresponding coordinate: it is as if the prior is almost constant on this coordinate, which enables adaptation to the compositional `structure'. On the other hand, large values of $A_{ij}$ enable one to ``unsmooth" the very smooth paths of squared--exponential GPs and thus to adapt to smoothness as well (as already used in Chapter \ref{chap:ada1}).

\subsubsection*{Statement for deep GPs}

\begin{thm} \sbl{[Deep Horseshoe GP after \cite{cr23} (informal statement)]}
In the random design regression model, consider a deep horseshoe GP prior on $f$ (as defined in the previous section) and suppose $f_0$ belongs to $\cG(q,D,t,\be,K)$ as in \eqref{defcomp}. Then the corresponding $\alpha$--posterior distribution, for $\alpha\in(0,1)$, contracts at the optimal rate (up to logarithmic factors) given by \eqref{vepset}, adaptively both to smoothness and structure parameters.

Up to slowly varying factors, this continues to hold with same expression of the rates in a `growing dimension' setting where one allows the input dimension $d$ to grow polynomially in $n$, with actual true dimension $t_0\leqa (\log{n})^{1/2-\delta}, \delta>0$, all other dimensions $d_i,t_i$, $i\in\{1,\ldots,q\}$ being kept fixed, 
provided one chooses $\ta=(nd)^{-3/2}$. 
\end{thm}

We will not give a proof here, but just give the idea: thanks to Theorem \ref{alpost}, since one works with the $\alpha$--posterior, it is enough to verify the prior mass condition. The proof uses similar tools as for a simple GP, but one main new step consists in proving that coordinates with sufficiently small scale parameters can be ignored; it is also more involved due to the successive steps involved in the compositions; again, the concentration functions of the successive GPs play an important role. A key idea in the deep horseshoe prior is that the scale parameters serve both the purpose of {\em smoothness adaptation} if the variable should be included (case of large $A_{ij}$ which make the GP path `appropriately wiggly') and that of {\em structure adaptation} if the variable should be discarded (case of vanishing $A_{ij}$ which `freezes' the GP path along that direction). 
Finally, one relates the $\alpha$--Rényi divergence to the target quadratic distance for the regression model (which is quite easy for Gaussian noise as assumed here, using Lemma \ref{lem:renyig} in Appendix \ref{app-dist}).

\chapter{Bernstein--von Mises I: functionals} \label{chap:bvm1}

As a motivation for this chapter, let us mention two settings in the area of survival analysis in medical statistics (e.g. in the nonparametric survival model in Appendix \ref{app:models}) 
where 
credible sets are used in practice to quantify uncertainty: they are related to the survival function, which, for a (often censored) variable $X$ gives for each time $t$ the probability $P[X>t]=1-F(t)$ for $F$ the cumulative distribution function. The first is the problem of {\em inference on the median survival time}, in other words the quantile of level $1/2$ of $F$, the other that of {\em inference} on survival function itself.

In these situations one is not only interested in estimation --e.g.  producing an estimate of $F^{-1}(1/2)$ and $F(\cdot)$-- but also in uncertainty quantification. Figure \ref{fig:survmed} depicts empirical histograms from a sample of the posterior distribution; the picture shows that the (sample from the) posterior distribution of $F^{-1}(1/2)$  looks asymptotically normal. Taking the empirical quantiles on the picture then gives a {\em credible interval}. Can one prove that for large $n$ it is a {\em confidence interval}? Similarly, Figure \ref{fig:survb} shows a shaded area: it is the interior of the region formed by adding to the posterior mean survival $\bar{S}$ plus and minus the estimated posterior quantile of the supremum norm distance from a posterior sample to the posterior mean (plotted in solid blue line). Are those bands {\em confidence bands} for large $n$? 

We give  positive answers to these questions by establishing limiting shape results for posterior distributions in this Chapter and the next.

\begin{figure}[!h]
\label{fig:survmed}
\includegraphics[width=15cm]{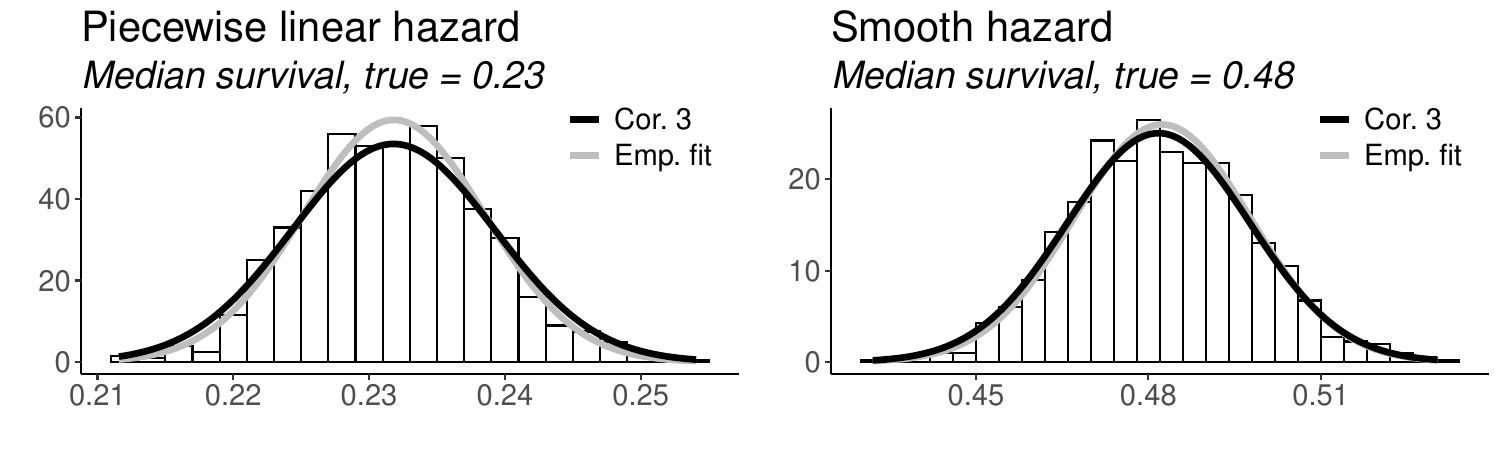}
\caption{Median survival in the nonparametric model with right-censoring \cite{cv21}}
\end{figure}
\begin{figure}[!h]
\label{fig:survb}
\includegraphics[width=15cm]{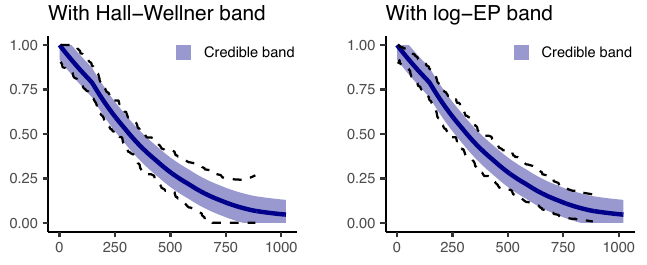}
\caption{\cite{cv21} Posterior mean of the survival (solid) with credible band (shaded area), compared to (dashed): the Hall-Wellner band, and the log-transformed equal precision (log-EP) band.}
\end{figure}

\section{Bernstein-von Mises: a limiting shape result}

In 1774, Laplace observed  and proved \cite{laplace1774} the following: consider data $X$ from, given $\te$, a binomial model $\{P_{\te}=\text{Bin}(n,\te),\, \te\in(0,1)\}$ and endow $\te$ with a uniform prior $\Pi$ on $(0,1)$. Then if the data was actually generated from a $P_{\te_0}$ distribution, with $\te_0$ fixed in $(0,1)$, the posterior distribution, i.e. the law of $\te$ given $X$, (which in this example is a Beta distribution) asymptotically looks like (what was later called a) Gaussian $\cN(X/n, \te_0(1-\te_0)/n)$ distribution. This can be seen as a Bayesian central limit theorem, where the object of study is not an empirical mean but the posterior distribution. 

In the past century, many contributions (including, but not limited to, those of Bernstein and von Mises) have much broadened the scope of this result: under quite mild `regularity' assumptions on the statistical model and prior, such a result indeed holds for the posterior, and is known under the name of {\em Bernstein--von Mises} (BvM) theorem. A result for parametric models which holds under fairly minimal conditions has been obtained by van der Vaart (see \cite{aad98}, Chapter 10). We give a version of it  now in a slightly simplified setting.

Consider a dominated statistical model $\cP=\{P_\te^{\otimes n},\ \te\in\Theta \}$ with $\Theta$ an open subset of $\RR^p$ with $p$ {\em fixed}. Suppose one observes $X=(X_1,\ldots,X_n)$ with law $P_{\te_0}^{\otimes n}$  for some `true' $\te_0\in\Theta$.  Set $\ell_\te=\log{p_\te}$ and suppose to fix ideas that $\te \to \ell_\te(x)$ is differentiable over $\Theta$ for any possible $x$ (this is stronger than what is actually needed, see \cite{aad98}, Chapter 7, for weaker conditions allowing for `differentiability in quadratic mean'), and assuming it exists, set $I_\te=E_\te[\dot{\ell_\te}\dot{\ell_\te}^T]$, called the {\em Fisher information} matrix, and where $\dot{\ell_\te}=\partial\ell_\te/\partial\te$ is called the {\em score function} and where we denoted by $T$ the transposition (instead of prime to avoid confusion with derivatives).  

The model $\cP$ is {\em locally asymptotically normal (LAN)} at $\te_0\in\Theta$ if, for any $h\in\RR$, as $n\to\infty$,
\[
\log \prod_{i=1}^n \frac{p_{\te_0+h/\sqrt{n}}}{p_{\te_0}}(X) 
= \frac{1}{\sqrt{n}} h^T\sum_{i=1}^n  \dot{\ell}_{\te_0}(X_i)- \frac12 h^TI_{\te_0}h+o_{P_0}(1), 
 \]
where $P_0=P_{\te_0}$ and $o_P(1)$ means a sequence going to $0$ in $P$--probability. This LAN property can be interpreted as a suitable expansion in terms of $\te$ of the $\log$--likelihood, so as a smoothness--type condition. Interestingly it holds under almost no further conditions (sufficient conditions in the previous setting are differentiability of $\te\to\sqrt{p_\te}$ and continuity of $\te\to I_\te$, see \cite{aad98}, Lemma 7.6). 

We also assume that a certain testing condition holds -- again, it holds very broadly, we refer to  \cite{aad98} for more discussion --, namely that for any $\veps>0$, there exist tests $\vphi_n$ such that, as $n\to\infty$,
\[ (T_P)\qquad  P_{\te_0}\vphi_n =o(1), \qquad  \sup_{\te:\, \|\te-\te_0\|>\veps} P_\te(1-\vphi_n) = o(1). \]
Set $\Delta_{n,\te_0}=\sum_{i=1}^n I_{\te_0}^{-1}\dot{\ell}_{\te_0}(X_i)$ (note this converges in distribution to a $\cN(0,I_{\te_0}^{-1})$ variable) and 
\[T_n = \te_0+ \Delta_{n,\te_0}/\sqrt{n}. \]
Given a prior $\Pi$ on $\Theta$, one forms the posterior distribution $\Pi[\cdot\given X]$ using Bayes' formula. 
Let $\ta_{T_n}$ be the map $\ta_{T_n}:x\to \sqrt{n}(x-T_n)$. The push-forward of the posterior through, $\Pi[\cdot\given X]\circ \ta_{T_n}^{-1}$, defines a `shifted and rescaled' posterior. The centering $T_n$ in the next result  can be replaced by another quantity $\hat\te$ such that $\hat\te=T_n+o_P(1/\rn)$. It turns out that such $\hat\te$ is then what is called a (linear) efficient estimator. 

\begin{thm}[Parametric Bernstein--von Mises theorem] \label{bvmpara}
Suppose the model is LAN at $\te_0\in\Theta$ with $I_{\te_0}>0$ and that the testing condition $(T_P)$ holds. Let the prior $\Pi$ on $\te$  be absolutely continuous with respect to the Lebesgue measure in a neighborhood of $\te_0$ with a continuous and positive density at $\te_0$. Then
\[ \| \Pi[\cdot\given X]\circ \ta_{T_n}^{-1} - \cN(0,I_{\te_0}^{-1})(\cdot) \|_{TV} = o_{P_0}(1). \]
\end{thm}

The BvM theorem can be interpreted as a limiting shape result: it states that the posterior distribution asymptotically ressembles a Gaussian distribution. It turns out that the variance $I_{\te_0}^{-1}$ as well as the centering, are `optimal' in the sense of the theory of efficiency for parametric models; we refer again to \cite{aad98} for more details.  An important consequence of the result is as follows.\\

{\em Application (credible/confidence) sets.} Suppose for simplicity that we are in dimension $1$ and that the posterior distribution has a continuous strictly increasing distribution function (so that quantiles are easiest to define): $\Theta\subset\RR$ and define $a_n(X), b_n(X)$ to be the quantiles at level $\al/2$ and $1-\al/2$ of $\Pi[\cdot\given X]$. By definition, the interval $[a_n(X),b_n(X)]$ is a credible set of level $1-\al$, that is 
\[ \Pi[\te\in [a_n(X),b_n(X)]\given X]=1-\al.\]
If BvM holds, we then automatically have, as $n\to\infty$, (see the exercises)
\begin{equation} \label{credconf}
P_{\te_0}[\te_0\in [a_n(X),b_n(X)] ] \to 1-\al.
\end{equation} 
Equation \eqref{credconf} states that it is also a confidence set asymptotically, of level $1-\al$. In parametric  models, this gives an automatic way of constructing confidence sets, which are also automatically of smallest possible length asymptotically. \\

It is natural to ask whether BvM--type results also hold in more general settings, in particular in semi- or non-parametric models. This Chapter considers the question for semiparametric models, while some nonparametric results are considered in Chapter \ref{chap:bvm2}.

\section{BvM: a result for semiparametric functionals}

Consider a generic dominated model $\cP=\{P_{\eta},\ \eta\in S\}$ with $dP_\eta=p_\eta d\mu$ for all $\eta$ in the parameter set $S$. As above, the log-likelihood is denoted by $\ell_n(\eta) = \log p_\eta(X)$. For example, $\cP$ can be  nonparametric (e.g. density estimation in which case $\eta=f$ a density) or it can be that the parameter $\eta$ has two parts $\eta=(\te,f)$ (these models are called {\em separated} semiparametric models).  It is often the case statisticians are interested in estimating a finite-dimensional parameter or aspect of the model. In density estimation one may want to  estimate a linear functional $\psi(f)=\int_0^1 a f$ of the unknown density $f$, where $a$ is a given square-integrable function (e.g. the indicator of an interval); in {\em  separated} semiparametric models a typical of interest is just $\psi(\eta)=\te$ itself. 

Consider a functional $\psi : S \rightarrow \R$. Suppose $\Pi$ is a prior distribution on $S$, with associated posterior $\Pi[\cdot\given X]$. We wish to study the properties of the marginal posterior distribution of $\psi(\eta)$, i.e the push-forward measure $\Pi[\cdot\given X]\circ \psi^{-1}$. We first consider a fairly general setting and introduce sufficient conditions for the posterior distribution to verify a BvM theorem. Afterwards, we apply this general result to the Gaussian white noise model and density estimation.
Let $P_0=P_{\eta_0}^{(n)}$.

We say that a distribution $Q_X$ on $\mathbb{R}$, depending on the data $X$, \textit{converges weakly in $P_0$-probability to a Gaussian distribution $\cN(0,V)$} if, as $n\to\infty$,
\begin{equation}\label{cvl}
 \be_\RR\left(Q_X, \cN(0,V) \right) \to^{P_0} 0, 
\end{equation} 
where $\be_\RR$ is the bounded Lipschitz distance between probability distributions on $\mathbb{R}$, 
see Appendix \ref{app-dist}.  Given a {\em rate}  $v_n$ and a {\em centering} $\mu=\mu(X)$, consider the map $\tau_\psi:\eta \to v_n(\psi(\eta)-\mu)$. We say that the posterior distribution of $v_n(\psi(\eta)-\mu)$ converges weakly in $P_0$--probability to a $\cN(0,V)$ distribution if \eqref{cvl} holds for 
\[ Q_X = \Pi[\cdot\given X] \circ \tau_\psi^{-1}.\]
In view of the parametric results above, it seems natural in the present more general setting to center the posterior again at an `efficient' estimator. Indeed a theory of efficiency paralleling the one for parametric models exists in this more general semiparametric setting; we refer to the Saint-Flour notes \cite{aadstflour} (or also \cite{aad98} Chapter 25) for more on this. 

When $v_n = \sqrt{n}$ and $\mu = \hat\psi$ is an efficient estimator of $\psi(\eta)$, writing $\mathcal{L}( \sqrt{n}(\psi(\eta)-\hat\psi) |X)$ for the marginal posterior distribution of $\sqrt{n}(\psi(\eta)-\hat\psi)$, the above says that, as $n\to\infty$,
\[ \mathcal{L}( \sqrt{n}(\psi(\eta)-\hat\psi) \given X) \approx \cN(0,V) \]
 Such a result, known as a \textit{semiparametric BvM theorem}, says that the above marginal  posterior distribution asymptotically converges to a Gaussian distribution. Alternatively, one can wrtite this distributional approximation as $\mathcal{L}( \psi(\eta) \given X) \approx \cN(\hat\psi, V/n)$. 
 \subsubsection{A generic LAN setting} 
The following setting formalises a generic semiparametric framework as in \cite{cr15} (see also \cite{c12} and \cite{gvbook}, where similar settings are considered in order to derive BvM theorems). For simplicity in the next condition we implicitly assume that $\eta-\eta_0$ can be embedded into the considered Hilbert space (otherwise the arguments can be adapted). 

\begin{assumption}[LAN framework]
\label{assum:lan} 
Let $(\mathcal{H}, \langle\cdot, \cdot\rangle_L)$ be a Hilbert space with associated norm $\|\cdot \|_L$. In the following, $R_n$ and $r$ are remainder terms which are controlled through the last part of the assumption.

{\em LAN expansion.} Suppose $\ell_n$ around $\eta_0$ can be written, for suitable $\eta$'s to be specified below, as
$$
\ell_n(\eta) - \ell_n(\eta_0) = -\frac{n}{2}\|\eta - \eta_0\|_L^2  + \sqrt{n}W_n(\eta - \eta_0) + R_n(\eta, \eta_0),
$$
where $W_n : h \mapsto W_n(h)$ is $P_0$--a.s. a linear map and $W_n(h)$ converges weakly to $\cN(0, \| h\|_L^2 )$ as $n \rightarrow \infty$.\\

{\em Functional expansion.} Suppose  the functional $\psi$ around $\eta_0$ can be written, for some $\psi_0\in\mathcal{H}$, as 
$$
\psi(\eta) - \psi(\eta_0) = \langle \psi_0, \eta - \eta_0 \rangle_L + r(\eta, \eta_0).
$$
Define, for any fixed (possibly small enough) $t\in\R$, a {\em path} through $\eta$ as, assuming $\eta_t\in S$ for small $t$,
\begin{equation}
\eta_t = \eta - \frac{t\psi_0}{\sqrt{n}}.\label{def:eta_perturbed}
\end{equation}

{\em Remainder terms control.}  
Suppose that there exists a sequence of measurable sets $A_n$ satisfying
\[ \Pi[A_n | X] = 1 + o_{P_0}(1), \]
	such that $\eta - \eta_0 \in \mathcal{H}$ for all $\eta \in A_n$ and $n$ sufficiently large, and for any fixed $t \in \mathbb{R}$,
$$  
\sup_{\eta \in A_n}|t\sqrt{n} r(\eta, \eta_0) + R_n(\eta, \eta_0) - R_n(\eta_t, \eta_0 )| = o_{P_0}(1).
$$
\end{assumption}
\vspace{.2cm}

For $\psi_0$ and $W_n$ as in Assumption \ref{assum:lan}, further define,
\begin{align}
\hat{\psi} &= \psi(\eta_0) + \frac{W_n(\psi_0)}{\sqrt{n}}, \hspace{5mm} V_{0} = \left|\left| \psi_0 \right| \right|^2_L.	\label{def:psi_hat}
\end{align}

The term $V_0$ is the {\em efficiency bound} for estimating $\psi(\eta_0)$; an estimator $\tilde{\psi}=\tilde{\psi}(X)$ is said to be {\em linear efficient} for estimating $\psi(\eta_0)$ if it can be expanded as $\tilde{\psi}=\psi(\eta_0) + W_n(\psi_0)/\sqrt{n}+o_P(1/\sqrt{n})$ or equivalently if $\sqrt{n}(\tilde{\psi}-\hat\psi)=o_P(1)$. For such an estimator,  $\sqrt{n}(\tilde{\psi}-\psi(\eta_0))$ converges in distribution to a $\cN(0,V_0)$ variable. Note that $\hat\psi$ is itself not an estimator as it depends on unknown quantities. But in all the following, this quantity can be replaced by any linear efficient estimator $\tilde\psi$ since $\tilde{\psi}=\hat\psi+o_P(1/\sqrt{n})$.\\

\noi {\em Remark.} In \eqref{def:eta_perturbed}, it is assumed that $\eta_t$ belongs to $S$. Some non-linear paths may be required in situations in which the parameter set is `constrained', such as in density estimation for which the density $f$ must satisfy the conditions $f\ge 0$ and $\int_0^1 f = 1$. In these situations the above often needs a slight adaptation, see below for the example of density estimation.

\subsubsection{A semiparametric BvM theorem}

\begin{thm}[Semiparametric BvM]\label{thm:genbvm}
Let $\Pi$ be a prior distribution on $\eta$ and suppose that the LAN framework in Assumption \ref{assum:lan} holds true with sets $A_n$. If for any $t\in\R$,
\begin{equation} \label{chvar}
\frac{\int_{A_n} e^{\ell_n(\eta_t)} d\Pi(\eta)}{\int e^{\ell_n(\eta)}d\Pi(\eta)} = 1+o_{P_0}(1),
\end{equation}
then the posterior distribution of $\sqrt{n}(\psi(\eta) - \hat{\psi})$ converges weakly in $P_0-$probability to a Gaussian distribution with mean 0 and variance $V_{0}$.
\end{thm} 

The last display of Theorem \ref{thm:genbvm} is a ``change-of-measure"--type condition. It is satisfied if a slight additive perturbation of the prior (replacing $\eta$ by $\eta_t$ or vice-versa) has little effect on computing the integrals on the display. It can often be checked by doing a change of measure in the prior.
 In the proof of Theorem \ref{thm:genbvm} below, one shows that under the assumptions, for any $t\in\R$,
\[ E\left(e^{t\sqrt{n}(\psi(\eta) - \hat{\psi})} 1_{A_n}\given X\right) = e^{o_P(1) + t^2V_0/2} \cdot \frac{\int_{A_n} e^{\ell_n(\eta_t)}d\Pi(\eta)}{\int e^{\ell_n(\eta)} d\Pi(\eta)}. \]
So \eqref{thm:genbvm} arises in order to check that the Laplace transforms converge as desired.

Similarly as for the parametric Theorem \ref{bvmpara}, it follows that quantile credible sets of the posterior for $\psi(f)$ are asymptotic confidence sets.\\ 

\begin{proof}[{\em Proof of Theorem \ref{thm:genbvm}}] 
To show that $\sqrt{n}(\psi(\eta)- \hat{\psi})$ converges in distribution (in $P_0$--probability) to a  $\cN(0,V_0)$ law, it suffices to do so for $\sqrt{n}(\psi(\eta)- \hat{\psi})1_{A_n}(\eta)$. Indeed, 
\[ \sqrt{n}(\psi(\eta)- \hat{\psi})=\sqrt{n}(\psi(\eta)- \hat{\psi})1_{A_n}(\eta)+\sqrt{n}(\psi(\eta)- \hat{\psi})1_{A_n^c}(\eta),\]
and since by assumption $\Pi[A_n^c\given X]=o_{P_0}(1)$, for $\eta\sim \Pi[\cdot\given X]$ the variable $1_{A_n^c}(\eta)$  goes to $0$ in probability, and so does $\sqrt{n}(\psi(\eta)- \hat{\psi})1_{A_n^c}(\eta)$ (the probability that it is non--zero is $\Pi[A_n^c\given X]$).  

Since convergence in distribution is implied by convergence of Laplace transforms (this is also true in--probability, see Lemma 1 of the supplement of \cite{cr15}), it is enough to show, for any real $t$, that $E[e^{t\sqrt{n}(\psi(\eta)- \hat{\psi})1_{A_n}}\given X]$ goes to $e^{t^2V_0/2}$ in $P_0$--probability. 
Since $e^{t\sqrt{n}(\psi(\eta)- \hat{\psi})1_{A_n}}=e^{t\sqrt{n}(\psi(\eta)- \hat{\psi})}1_{A_n}+1_{A_n^c}$, using again that $\Pi[A_n^c\given X]=o_{P_0}(1)$, it is enough to show that 
\begin{align*}
 E(e^{ t\sqrt{n} (\psi(\eta) - \hat{\psi})} | X, A_n) & := \frac{\int_{A_n} e^{ t\sqrt{n} (\psi(\eta) - \hat{\psi})} e^{\ell_n(\eta) -  \ell_n(\eta_t)} e^{\ell_n(\eta_t)} d\Pi(\eta)}{\int_{A_n} e^{\ell_n(\eta)}d\Pi(\eta)} \\
 & = \frac{\int_{A_n} e^{ t\sqrt{n} (\psi(\eta) - \hat{\psi})} e^{\ell_n(\eta) -  \ell_n(\eta_t)} e^{\ell_n(\eta_t)} d\Pi(\eta)}{\int  e^{\ell_n(\eta)}d\Pi(\eta)} \Pi(A_n\given X)^{-1} 
\end{align*} 
goes to $e^{t^2V_0/2}$ in $P_0$--probability, where $\eta_t = \eta - t\psi_0/\sqrt{n}$ the path as in \eqref{def:eta_perturbed}. 

 Using the LAN expansion in Assumption \ref{assum:lan} and the linearity of $W_n$,
\begin{align*}
\ell_n (\eta) - \ell_n(\eta_t) &= -\frac{n}{2}\|\eta - \eta_0\|_L^2 + \frac{n}{2} \|\eta_t-\eta_0\|_L^2 + \sqrt{n}W_n(\eta - \eta_t) + R_n(\eta,\eta_0) - R_n(\eta_t,\eta_0) \\
&= - t\sqrt{n} \langle \psi_0 , \eta - \eta_0 \rangle_L + \frac{t^2}{2} \|\psi_0\|_L^2 + t W_n(\psi_0) + R_n(\eta,\eta_0) - R_n(\eta_t,\eta_0),
\end{align*}
recalling that $\|\cdot\|_L$ is a norm induced by a Hilbert space. Using the definition \eqref{def:psi_hat} of $\hat{\psi}$ and the functional expansion in Assumption \ref{assum:lan},
\begin{align*}
t\sqrt{n} (\psi(\eta) - \hat{\psi}) = t \sqrt{n} \langle \psi_0 , \eta - \eta_0 \rangle_L - t W_n(\psi_0) + t \sqrt{n} r(\eta,\eta_0).
\end{align*}
Combining the last two displays thus gives
\begin{align*}
	t\sqrt{n} (\psi(\eta) - \hat{\psi}) &+ \ell_n(\eta) -  \ell_n(\eta_t) 
	=   \frac{t^2 \|\psi_0\|_L^2 }{2} + \underbrace{t\sqrt{n}r(\eta, \eta_0) + (R_n(\eta ,\eta_0) - R_n(\eta_t ,\eta_0))}_{\textrm{Rem}(\eta,\eta_0)},
\end{align*}
where $\sup_{\eta \in A_n}|\textrm{Rem}(\eta, \eta_0)| = o_{P_0}(1)$ by assumption.
Substituting this into the above display  gives
\[ E(e^{ t\sqrt{n} (\psi(\eta) - \hat{\psi})} \given X, A_n) = e^{o_{P_0}(1) + t^2\left|\left| \psi_0 \right| \right|^2_L/2} \cdot \frac{\int_{A_n} e^{\ell_n(\eta_t) }d\Pi(\eta)}{\int e^{\ell_n(\eta)}d\Pi(\eta)}.\]
Since the last ratio equals $1 + o_{P_0}(1)$ by assumption,
the last display goes to  $e^{t^2 V_0/2}$ in $P_0$--probability, which concludes the proof of Theorem \ref{thm:genbvm}. 
\end{proof}

\section{BvM: examples and applications}

We now apply the general result to  two prototypical nonparametric models (white noise regression and density estimation) and then briefly discuss the case of separated semiparametric models. 

\subsubsection{Gaussian white noise and sequence model}

Let us recall that projecting the Gaussian white noise model onto a given orthonormal basis $(\phi_k)$ of $L^2[0,1]$, one obtains a sequence model $X_k=f_k+\veps_k/\sqrt{n}$ with $f_k=\psg f,\phi_k\psd$. We will work directly in sequence space and write $X=(X_1,X_2,\ldots)$ and $f=(f_1,f_2,\ldots)$. Here $\cP=\{\otimes_{k\ge 1} \cN(f_k,1/n), \ (f_k)\in\ell^2 \}$ and  the parameter is just $\eta=f$. We have the exact LAN expansion 
\[ \ell_n(f) - \ell_n(f_0) = -\frac{n}{2} \|f - f_0\|_2^2 + \sqrt{n} W_n(f-f_0), \]	 
	where, for $g = \sum_{k=1}^\infty g_k \phi_k$, we set $W_n(g) = \sum_{k=1}^\infty g_k\varepsilon_k$. In particular, the LAN norm $\|\cdot\|_L = \|\cdot\|_2$, the remainder $R_n(f, f_0) = 0$. 
Now considering a functional $\psi(f)$ that verifies the expansion $\psi(f)=\psi(f_0)+\psg \psi_0,f-f_0\psd_2 + r(f,f_0)$ for some $\psi_0\in\ell^2$ for instance a linear functional $\psi(f)=\int_0^1 a(u)f(u)du$ for which $r(f,f_0)=0$, the efficient centering sequence is
	\[ \hat{\psi} = \psi(f_0) + \frac{W_n(\psi_0)}{\sqrt{n}} = \sum_{k=1}^\infty \psi_{0,k}X_k, \] 
with corresponding efficiency bound $V_0 = \|\psi_0\|_L^2 = \|\psi_0\|_2^2$. 
Theorem \ref{thm:genbvm} now implies that if the prior verifies the change--of--variable condition and concentrates on $A_n$ 
then the semiparametric BvM holds in the sense that
\[ \cL(\sqrt{n}(\psi(f)-\hat\psi)\given X) \to \cN(0,V_0),\]
 as $n\to\infty$, in probability under $P_0$. 

\subsubsection{Density estimation}

In density estimation $\cP=\{P_f^{\otimes n},\ f\in\cF\}$, where $P_f$ is the law of density $f$ on $[0,1]$ and $\cF$ a subset of positive densities on $[0,1]$. We assume that the true density $f_0$ is bounded away from $0$ and infinity on $[0,1]$. By definition,  the log-likelihood is then $\ell_n(f)=\sum_{i=1}^n \log{f}(X_i)$. Let us use the shorthand notation $\eta = \log f$, although below we write the LAN expansion in terms of $f$ (so our natural parameter is still $f$; it is also possible to use $\eta$ as main parameter with only minor differences in the statement, see \cite{cr15}, Theorem 4.1). 
 We have the LAN expansion:
\begin{align*}
\ell_n(f) - \ell_n(f_0) &=\sum_{i=1}^n\{\eta(X_i) - \eta_0(X_i)\}
= -\frac{n}{2}\|\eta - \eta_0\|_L^2 + \sqrt{n}W_n(\eta - \eta_0) + R_n(f, f_0),
\end{align*}
where, for $g \in L^2(f_0)=\{g: \int_0^1 g^2 f_0<\infty\}$ and  $F(g) := \int_0^1 gf$ for $f,g\in L^2[0,1]$, so
that $F_0 g=\int gf_0$,
\[ 
	\|g\|_L^2 = \int (g - F_0(g))^2f_0,\qquad W_n(g) = \frac{1}{\sqrt{n}}\sum_{i=1}^n [g(X_i) - F_0(g)],
\]
and
$
R_n(f, f_0) = \sqrt{n} P_{0}h + \frac{1}{2}\|h\|_L^2
$
for $h = \sqrt{n}(\eta - \eta_0)$.
For the functional expansion, we assume there exists a bounded measurable function $\effinf : [0,1]\to\mathbb{R}$ such that
\begin{equation}\label{expansion_psi}
\psi(f) - \psi(f_0) = \int_0^1 \effinf (f-f_0) + \tilde{r}(f,f_0)\qquad\text{and}\qquad \int_0^1 \effinf f_0 = 0.
\end{equation}
 In this case, as follows from the definition above $\psg f,g \psd_L=\int (f-F_0f)(g-F_0g)f_0$, and
\begin{align*}
\psi(f) - \psi(f_0) =  
&= \langle\,\frac{f-f_0}{f_0}\,,\,\effinf\, \rangle_L + \tilde{r}(f,f_0) =\langle \eta - \eta_0, \effinf \rangle_L + r(f,f_0),
\end{align*}
with $r(f,f_0)=\mathcal{B}(f,f_0) + \tilde{r}(f,f_0)$ and
$$
\mathcal{B}(f,f_0) = -\int \left[\eta - \eta_0 - \frac{f-f_0}{f_0}\right]\effinf f_0.
$$
The last steps  account for the fact that the functional expansion should hold in terms of $\eta = \log f$ rather than $f$ itself.
This gives  
\[ \hat{\psi} = \psi(f_0) + W_n(\effinf)/\sqrt{n} = \psi(f_0) + \sum_{i=1}^n \effinf(X_i)/n, \] and optimal (efficient) limiting variance $\|\effinf\|_L^2 = \int \effinf^2 f_0$. \\

{\em Paths along the model.} In density estimation, given the constraints on density functions, the path are slightly adapted. In terms of the parameter $f$, we set 
\[ f_t = fe^{-t\effinf/\sqrt{n}}/F(e^{-t\effinf/\sqrt{n}}), \]
which is a density by construction for any real $t$; equivalently, in terms of $\eta=\log{f}$ this becomes
\begin{equation} \label{path:density}
\eta_t = \eta - t\frac{\effinf}{\sqrt{n}} - \log F(e^{-t\effinf/\sqrt{n}}),
\end{equation}
which resembles \eqref{def:eta_perturbed} up to the logarithmic term. 
The latter is a constant that may depend on $f$ but not on the space variable defining the $\eta$'s. Therefore, it drops out from the expressions of the LAN norm and the term $W_n(h)$ above, as the various terms are always recentered by their expectations (in the $\eta$--parametrisation), so any additive constant cancels out.
The following adapts Theorem \ref{thm:genbvm} with slightly easier-to-verify conditions. Its proof follows the same lines and is given below.
\begin{thm}\sbl{[density estimation BvM]}\label{thm:bvmdens}$\ $
	Let $f\to \psi(f)$ be a functional on probability densities on $[0,1]$. 
Suppose that for some  $\eps_n =o(1)$ and sets $A_n \subset \{f: \|f-f_0\|_1 \leq \eps_n\}$, for  $\tilde{r}$ as in \eqref{expansion_psi},
\begin{align}
	& \Pi[A_n \given X] = 1 + o_{P_0}(1), \label{GTDE_1}\\
	&\sup_{f \in A_n}\tilde{r}(f, f_0) = o(1/\sqrt{n}). \label{GTDE_2} 
\end{align}
Denote
$f_t = fe^{-t\effinf/\sqrt{n}}/F(e^{-t\effinf/\sqrt{n}})$ and for $A_n$ as above,  assume that
\begin{equation} \label{GTDE_3}
\frac{\int_{A_n}e^{\ell_n(f_t)}d\Pi(f)}{\int e^{\ell_n(f)}d\Pi(f)} = 1 + o_{P_0}(1). 
\end{equation}
Then for $\hat{\psi} = \psi(f_0) + \sum_{i=1}^n \effinf(X_i)/n$, the posterior distribution of $\sqrt{n}(\psi(f) - \hat{\psi})$ converges weakly in $P_0-$probability to a Gaussian distribution with mean 0 and variance $\int  \effinf^2 f_0$.
\end{thm}

\begin{proof}[Proof of Theorem \ref{thm:bvmdens}]
Let us  verify that Assumption \ref{assum:lan} holds (recall that here we work with the parameter $f$ instead of the generic `$\eta$' in the previous section): 
$$
R_n(f, f_0) - R_n(f_t, f_0) = t\sqrt{n} \langle \eta - \eta_0, \effinf \rangle_L - \frac{t^2}{2}\|\effinf\|_L^2 + n\log F(e^{-t\effinf/\sqrt{n}}).
$$
Expanding the last term, we have for $f \in A_n \subset 
\{\|f-f_0\|_1 \leq \varepsilon_n \}$,
\begin{align*}
    n\log F (e^{-t\effinf/\sqrt{n}}) &=n\log\Big\{1 - \frac{t}{\sqrt{n}}\int_0^1 f \effinf + \frac{t^2}{2n}\int_0^1 f \effinf^2 + o\Big(\int_0^1 f\left(t^2 \effinf^2/n \right) \Big) \Big\} \\
    &= n\log\Big\{ 1 - \frac{t}{\sqrt{n}}\langle \eta - \eta_0, \effinf \rangle_L - \frac{t}{\sqrt{n}} \mathcal{B}(f, f_0) + \\
    & \quad + \frac{t^2}{2n}\|\effinf\|_L^2 +\frac{t^2}{2n}(F - F_0)(\effinf^2) + o(n^{-1}) \Big\} \\
    &= -t\sqrt{n} \langle \eta - \eta_0, \effinf \rangle_L - t\sqrt{n}\mathcal{B}(f, f_0) + \frac{t^2}{2}\|\effinf\|_L^2 + o(1),
\end{align*}
since $(F - F_0)(\effinf^2) \leq \|\effinf\|^2_\infty \|f - f_0\|_1 \lesssim \eps_n$ on $A_n$. Hence we have
$$
R_n(f, f_0) - R_n(f_t, f_0) = -t\sqrt{n}\mathcal{B}(f, f_0) + o(1),
$$
and the condition on remainder terms in Assumption \ref{assum:lan} reduces to, since 
$\mathcal{B}(f, f_0)=r(f,f_0)-\tilde{r}(f,f_0)$,
\begin{align*}
\sup_{f \in A_n}|\sqrt{n}r(f,f_0)| = o_P(1),
\end{align*}
which is satisfied by assumption. One can then follows the same arguments as in the proof of Theorem \ref{thm:genbvm} to conclude that the BvM theorem holds. 
\end{proof}

\subsubsection{Separated semiparametric models}

Here the model is indexed by $\eta=(\te,f)$ with (say) $(\te,f)\in \RR\times \cF$, for some set of functions $f$. Then $\psi(\eta)=\te$ is often the parameter of interest. The Hilbert space $\cH$ in the LAN expansion is a product
  $\cH=\mathbb{R}\times\cG_{\eta_0}$. Let $\overline{\cF}$ be the closure in $\cH$ of the linear span 
of all elements of the type $(0,f-f_0)$, where $f$ belongs to $\cF$.

Let us define the element $(0,\gamma(\cdot))\in\overline{\cF}$ as the orthogonal projection of the vector $(1,0)$ onto the closed subspace $\overline{\cF}$. 
The element $\gamma$ is called {\em least favorable direction}.
For any $(s,g)\in\cH$, 
\begin{equation} \label{lan-split}
\|\,(s,g)\,\|_L^2  =  (\|\,(1,0)\,\|_L^2 - \|\,(0,\gamma)\,\|_L^2) s^2 + \|\,(0,g+s\gamma)\,\|_L^2,
\end{equation}
decomposing $(s,g)=s(1,-\ga)+(0,g+s\ga)$ and noting that the two vectors in the sum are orthogonal. 
One denotes $\tilde{I}_{\eta_0} := \|\,(1,0)\,\|_L^2 - \|(0,\gamma)\|_L^2$ called {\em efficient information}. The LAN inner product is 
\[ \psg\, (s_1,g_1)\,,\,(s_2,g_2)\, \psd_L = \tilde{I}_{\eta_0}s_1s_2+ 
\psg\, (0,g_1+s_1\ga)\,,\,(0,g_2+s_2\ga)\, \psd_L \]
and from this one sees that, with $\eta=(\te,f)$ and $\eta_0=(\te_0,f_0)$,
\[ 
\psg\, \eta-\eta_0\,,\,(1,-\ga)\, \psd_L/\tilde I_{\eta_0}
=\te-\te_0 + \psg\, (0,f-f_0+\ga)\,,\,(0,-\ga+\ga) \,\psd_L/\tilde{I}_{\eta_0} = \te-\te_0
.\]
Deduce that $\psi$ has a linear expansion around $\eta_0$ in terms of the LAN inner product, with  representer
\[ \psi_0 = (1, -\ga)/\tilde{I}_{\eta_0}.\]
If $\ga=0$, ones says there is {\em no loss of information} and $\tilde{I}_{\eta_0}$ 
is the information in the model where $f$ would be known (that is, the standard Fisher information). 
If $\tilde{I}_{\eta_0}$  itself is nonzero,
 let us also denote $$\Delta_{n,\eta_0} = \tilde{I}_{\eta_0}^{-1} W_n(1,-\gamma). $$ 
An estimator $\hat{\te}_n$ of $\te_0$ is said asymptotically linear and efficient if 
 $\rn(\hat{\te}_n-\te_0)=\Delta_{n,\eta_0} + o_{P_0}(1)$.
 
 If $\gamma = 0$, we have 
 $\eta_t = (\te + t\tilde I_{\eta_0}^{-1}/\sqrt{n}, f)$
 and the change of variable condition \eqref{chvar} is not difficult to verify: 
 it holds if $\pi = \pi_\te \otimes\pi_f$ with $\pi_\te$ positive and continuous at $\te_0$.\\
 
 If $\ga\neq 0$, the change of variable condition is discussed below in case the nuisance parameter $f$ is endowed with a Gaussian process prior.\\
 
{\em Further intuition about the efficient information.} When the LAN condition holds, the model asymptotically looks like a Gaussian shift experiment with inner-product $\psg\cdot,\cdot\psd_L$. How much {\em information} is available for estimating a given parameter is completely encoded in the inner-product. Observe that 
from \eqref{lan-split}, one deduces $\|s,g\|_L^2 \ge (\|1,0\|_L^2 - \|0,\gamma\|_L^2) s^2 = \infoe s^2$ with equality when $g=-s\ga$. The quantity $\infoe$ represents the `smallest  curvature' of paths approaching $\eta_0$.

\subsubsection{Application: Gaussian process priors}

We now proceed to apply these results to concrete priors. We focus on Gaussian process priors and functionals either in white noise or density estimation, setting $\eta=f_0$ and $\eta=\log{f_0}$ respectively (in white noise, one could consider obtaining  these results also by a direct approach, and also, estimators for linear functionals are easy to obtain by other means, but the goal here is to illustrate the systematic nature of the methods we introduce; in more complex models, where a direct method cannot be used, they can be particularly useful). 
Recall from Chapter \ref{chap:rate1} that for a Gaussian process prior with concentration function $\varphi_{\eta}$, a contraction rate $\veps_n$ (in terms of certain distances) is given by a solution to
\begin{equation}\label{eq:conc_eqn}
\varphi_{\eta_0}(\eps_n) \leq n \eps_n^2.
\end{equation}
For the next result, we focus for simplicity on the white noise model, but a result of the same flavour holds much more broadly. We further comment on that below.

\begin{thm}
\label{thm:spgp}
Consider the Gaussian white noise model and   let  $\Pi$ be a mean-zero Gaussian prior in $L^2[0,1]$ on $f$ with associated RKHS $\mathbb{H}$. Suppose that $\eps_n \to 0$ satisfies \eqref{eq:conc_eqn} with $\eta_0 = f_0 \in L^2[0,1]$, and that 
Assumption \ref{assum:lan} holds for $\psi(f) = \psi(f_0) + \langle \psi_0, f - f_0 \rangle_2 + r(f,f_0)$ and $A_n \subset \{f: \|f-f_0\|_2 \leq \eps_n\}$.
Suppose that there exist sequences $\psi_n \in \mathbb{H}$ and $\zeta_n\rightarrow0$ such that
\begin{equation}\label{eq:RKHS_conditions}
\| \psi_n - \psi_0 \|_2 \leqa \zeta_n, \hspace{5mm}
\|\psi_n\|_\mathbb{H} \leqa \sqrt{n}\zeta_n, \hspace{5mm}
\sqrt{n}\eps_n \zeta_n \rightarrow 0.
\end{equation}
Then the posterior distribution of $\sqrt{n}(\psi(f) - \hat{\psi})$ converges weakly in $P_0-$probability to a Gaussian distribution with mean 0 and variance $\|\psi_0\|^2_2.$
\end{thm}

The sequence $\psi_n$ allows one to approximate the Riesz representer, $\psi_0$, of the functional by elements of the RKHS $\mathbb{H}$. This is helpful since for elements of the RKHS, one can directly deal with the change of measure condition \eqref{chvar} using the Cameron-Martin theorem. Note that if $\psi_0 \in \mathbb{H}$, one may immediately take $\psi_n = \psi_0$ and $\zeta_n = 0$, so condition \eqref{eq:RKHS_conditions} is automatically verified.\\

\begin{proof}[Sketch of proof, see \cite{c12, ltcr23} for details]
It is enough to verify the change--of--measure condition \eqref{chvar}, where $\eta_t=\eta-t\psi_0/\sqrt{n}$. \\

In case $\psi_0\in\mh$, one can use directly Cameron--Martin's change of variable formula (Appendix) and set $g=\eta-t\psi_0/\sqrt{n}$: one is left with controlling the change of measure term $\exp(-tU\psi_0/\sqrt{n}-t^2\|\psi_0\|_\mh^2/(2n))$, which can be done relatively easily by controlling $U\psi_0$ with high probability.\\

In case $\psi_0\notin \mh$, one changes variables  setting $g_n=\eta-t\psi_n/\sqrt{n}$ instead. Before doing this, one  writes $\eta_t=g_n+r_n$, with $r_n=t(\psi_n-\psi_0)/\rn$, and expands
\begin{align*}
\lefteqn{-n\|\eta_t-\eta_0\|_L^2/2 + \sqrt{n} W_n(\eta_t-\eta_0)
 = -n\|g_n-\eta_0\|_L^2/2+ \sqrt{n} W_n(g_n-\eta_0) }  && \\
& \qquad \qquad \qquad\qquad\qquad  \qquad\qquad -n\|r_n\|_L^2/2 - n\psg g_n-\eta_0, r_n \psd_L + \sqrt{n} W_n(r_n).
\end{align*} 
The first two terms on the right hand-side of the last display enable to recover, up to the LAN-remainder term, a term $\ell_n(g_n)-\ell_n(g_0)$, which matches the denominator in \eqref{chvar} as needed (once the change of variables is made and one substracts $\ell_n(g_0)$ up and down within the exponentials). One now shows  that the last three terms in the last display are small.

First, $\sqrt{n} W_n(r_n)=tW_n(\psi_n-\psi_0)$. In the white noise model, this follows a $\cN(0,\| \psi_n-\psi_0\|^2)$ distribution hence is $O_{P_0}(\zeta_n)=o_{P_0}(1)$. Similarly $n\|r_n\|_L^2=t^2\|\psi_n-\psi_0\|_L^2=o(1)$. Also,
\[ n|\psg g_n-\eta_0, r_n \psd_L| \leqa n\|g_n-\eta_0\|_L \|r_n\|_L  \]
using Cauchy-Schwarz inequality. As $ \|r_n\|_L\leqa \zeta_n/\rn$, and if $\|\eta_t-\eta_0\|_L\leqa \veps_n$ (which here follows from using that $\|f-f_0\|_2\le \veps_n$ on $A_n$), the last display goes to $0$ using the last condition in \eqref{eq:RKHS_conditions}.

Finally, one changes variables and needs to control the change of measure term $\exp(-tU\psi_n/\rn-\|\psi_n\|_\mh^2/(2n))$. The deterministic term is bounded using $\|\psi_n\|_{\mh}=O(\rn\zeta_n)$ by \eqref{eq:RKHS_conditions}, while the stochastic term is bounded on a set of large probability using this condition again.\\
\end{proof}

As is apparent from the last sketch of proof, for a general model, the key ingredients for estimation of a functional with a GP are: approximation of the functional in terms of the LAN norm by RKHS elements, with the corresponding first two conditions (and rate $\zeta_n$) and a condition that makes appear both the estimation rate of $f_0$ ({\em with respect to the LAN norm} -- note that the latter is not necessarily a `testing' distance, which may require more work to derive and/or a slightly slower rate) and $\zeta_n$. The latter is sometimes called no-bias condition --a somewhat similar condition appears e.g. for frequentist semiparametric estimators -. Since it arises from applying Cauchy-Schwarz, it is not necessarily sharp; in fact, determining a sharp conditions can be delicate in general, as fine details of the prior may matter \cite{ic12b}. 

For density estimation for instance and exponentiated and normalised GP priors, an analogous statement to Theorem \ref{thm:spgp} can be found in \cite{cr15}. For survival models and  the Cox model see \cite{cv21, nc23}.

We now turn to specific examples of Gaussian priors. Given an orthonormal basis of $L^2$,  
 define the Sobolev scales in terms of the $(\phi_k)$ basis:
\begin{equation}\label{eq:sobolev}
\mathcal{H}^\beta(R) :=  \left\{ f \in L^2[0,1]: \sum_{k= 1}^\infty k^{2\beta} |\langle f,\phi_k\rangle_2|^2 \leq R^2 \right\}.
\end{equation}
If $(\phi_k)$ is the Fourier basis for instance, then $\mathcal{H}^\beta$ coincides with the usual notion of Sobolev smoothness of periodic functions on $(0,1]$. 

Recall that a `$\ga$-smooth' infinite series GP prior is defined by, for $\gamma>0$
\begin{equation}
    W(x) = \sum_{k=1}^\infty k^{-\gamma-1/2} Z_k \phi_k(x) , \qquad \qquad Z_k \sim^{iid} \mathcal{N}(0,1).\label{def:infinite_series_prior}
\end{equation}
The infinite series prior \eqref{def:infinite_series_prior} models an almost $\gamma$-smooth function in the sense that it assigns probability one to $\mathcal{H}^s$ for any $s<\gamma$. Let us consider also a {\em rescaled} squared-exponential process on $\R$ with parameter $\gamma>0$: it is the mean-zero stationary Gaussian process with covariance kernel
\begin{equation} \label{rescsq}
K(s,t) = K(s-t) = \exp\left(-\frac{1}{k_n^2}(x-y)^2\right),
\end{equation}
where $k_n = \left(n/\log^2{n}\right)^{-\frac{1}{1+2\gamma}}$ is the length scale. The sample paths of this process are analytic, and so are typically too smooth to effectively model a function of finite smoothness in the sense that they yield suboptimal contraction rates \cite{vvvz11}. Rescaling the covariance kernel using the decaying lengthscale $k_n$ as above allows one to overcome this and model a $\gamma$-smooth function  \cite{vvvz07}.

For simplicity, we state the following result for linear functionals, but it can be extended to certain non-linear functionals modulo appropriate control of the functional's expansion.

\begin{corollary}\label{cor:gp_examples}
Let $W$ be a mean-zero Gaussian process taken as prior on $f$ in Gaussian white noise. Let $\psi(f) = \int_0^1 f a$ be a linear functional and consider the two cases:
\begin{itemize}
\item[(i)] $W$ is an infinite Gaussian series  with parameter $\gamma$, $f_0 \in \mathcal{H}^\beta$ and $a \in \mathcal{H}^\mu$;
\item[(ii)] $W$ is a rescaled squared-exponential process with parameter $\gamma$, and $\eta_0 \in \cC^\beta$ and $a \in \cC^\mu$.
\end{itemize} 
Then the posterior distribution of $\sqrt{n}(\psi(\eta) - \hat{\psi})$ converges weakly in $P_0-$probability to a Gaussian distribution with mean 0 and variance  $\|a\|_2^2$ if
\[  \gamma \wedge \beta > \frac{1}{2} +(\gamma - \mu) \vee 0.\]
\end{corollary} 
The interpretation of the condition is as follows: it allows for a prior that undersmoothes (with respect to the regularity of $f_0$) i.e. $\ga\le \be$, as soon as $\mu>1/2$. An oversmoothing prior $\ga>\be$ can be fine too, but only if 
 the smoothness $\mu$ is large enough ($\mu>\ga-\be+1/2$ suffices). 
 
Other `$\ga$--smooth' GP priors, such as Riemann--Liouville processes (that can be seen as Brownian motion integrated $\ga-1/2$ times) or Mat\'ern processes (stationary processes of covariance $K(s,t) = K(s-t) = \int_\R e^{-i(s - t)\lambda} (1+|\lambda|^2)^{-\gamma -1/2} d\lambda$) can be used within the same conditions. Again, using the same techniques, similar conditions arise when using these priors in more complex models. \\

\subsubsection{Application: nonlinear functionals}

One can obtain results for appropriately smooth non-linear functionals by linearisation. Consider to fix ideas the density estimation model: given a functional $\psi(f)$, for an appropriate  $a(\cdot)$ with $\int a f_0=0$ and suitably small $\tilde{r}$, one may be able to write
\begin{equation*}
\psi(f)= \psi(f_0) + \psg \frac{f-f_0}{f_0} , a \psd_L + \tilde{r}(f,f_0).
\end{equation*}
For example, for the quadratic functional $\psi(f) = \int_0^1 f^2(u)du$, one can take $a(u)=2f_0(u)-2\int_0^1 f_0^2(u) du$ and $\tilde{r}(f,f_0)=\int_0^1 (f-f_0)^2$. Then, provided one can control this remainder term uniformly over a set of high posterior probability, Theorem \ref{thm:bvmdens} applies. 
If one is able to establish a posterior convergence rate $\delta_n$ for $\|f-f_0\|_2$ (in density estimation, this requires a little more work compared to the derivation of an $\|\cdot\|_1$--rate) then, Condition \eqref{GTDE_2} holds if $\rn\delta_n^2=o(1)$, which itself can be obtained if $\be>1/2$. We refer to \cite{cr15} for details and other examples. 

Another approach to estimating possibly nonlinear functionals is via the nonparametric Bernstein--von Mises results presented in the next Chapter. In the context of survival analysis (see Appendix \ref{app:models}), it is of interest to estimate the median survival time $F^{-1}(1/2)$. A BvM result for this functional can be deduced from a `nonparametric' BvM theorem for $F(\cdot)$.

\section{Further results}
 
\subsubsection*{Tempered posteriors}

Using the same techniques as for the standard posterior distribution, one can derive Bernstein--von Mises theorems for the tempered posteriors $\Pi_{\al_n}[\cdot\given X]$ introduced in Chapter \ref{chap:rate1}. We briefly discuss the conditions required:  suppose the LAN framework holds in the same way as in Assumption \ref{assum:lan}, except that the path is now $\eta_t=\eta-t\psi_0/\sqrt{n\al_n}$ and that the remainder terms condition is, with $A_n$ satisfying
 \[ \Pi_\an[A_n | X] = 1 + o_{P_0}(1), \]
	such that $\eta - \eta_0 \in \mathcal{H}$ for all $\eta \in A_n$ and $n$ sufficiently large, and for any fixed $t \in \mathbb{R}$,
\[
\sup_{\eta \in A_n}|t\sqrt{n\an} r(\eta, \eta_0) + \an(R_n(\eta, \eta_0) - R_n(\eta_t, \eta_0 ))| = o_{P_0}(1).
\] 
All the rest remaining the same, the change of variable condition becomes,
 for any $t\in\R$,
\[ \frac{\int_{A_n} e^{{\alpha_n}\ell_n(\eta_t)}d\Pi(\eta)}{\int e^{{\alpha_n}\ell_n(\eta)}d\Pi(\eta)} = 1+o_{P_0}(1).\]
Under these conditions, a BvM--type result holds for the $\al_n$--posterior: for $\tau_n : \eta \rightarrow \sqrt{n\alpha_n}(\psi(\eta) - \hat{\psi})$,
\begin{align}\label{conv_BvM_alpha}
		\Pi_{\alpha_n}[\cdot \given X] \circ \tau_n^{-1} \overset{\cL}{\to} \mathcal{N}(0,V_0),
	\end{align}
 where $V_0$ is the efficiency bound for estimating $\psi(\eta)$ and 	$\hat{\psi}$ is linear efficient in that  
	\begin{align}\label{conv_hat_psi}
		\sqrt{n}(\hat{\psi} - \psi(\eta_0)) \xrightarrow{\mathcal{L}} \mathcal{N}(0, V_0).
	\end{align}

Note that for $0<\alpha_n<1$, the scaling in the BvM result is $\sqrt{n\al_n}$ instead of $\sqrt{n}$: the length of the resulting credible sets will then overshoot the optimal length given by the semiparametric efficiency bound. We now investigate how this can be remedied. For simplicity, we focus on the case where $\psi(\eta)$ is one dimensional.  

For  $0<\delta <1$, let $a_{n, \delta}^X$ denote the $\delta$--quantile of the $\alpha_n$--posterior distribution of $\psi(\eta)$ and consider the quantile region
\[  \mathcal{I}_{\al_n}= \mathcal{I}(\delta,\al_n,X):=(a_{n, \frac{\delta}{2}}^X\, ,\, a_{n, 1-\frac{\delta}{2}}^X].\]
 By definition, $\Pi_{\al_n}\left[ \psi(\eta)\in  \mathcal{I}_{\al_n} \given X\right] =1-\delta$, 
that is,  $\mathcal{I}_{\al_n}$ is a $(1-\delta)$--credible set (assuming the $\al_n$--posterior CDF is continuous, otherwise one takes generalised quantiles).  
	In order to obtain an efficient confidence interval from the $\alpha_n$--posterior distribution of $\psi(\eta)$, we consider a modification of the quantile region. Let $\bar{\psi}$ be an estimator of $\psi(\eta_0)$ built from the $\alpha_n$--posterior distribution of $\psi(\eta)$ (e.g. the posterior median or mean) and set 
	\begin{align}\label{def::eff_region}
		\mathcal{J}_{\al_n}:=\left(\sqrt\alpha_n(a_{n, \frac{\delta}{2}}^X - \bar{\psi}) +\bar{\psi} \,,\, \sqrt\alpha_n(a_{n, 1-\frac{\delta}{2}}^X - \bar{\psi}) +\bar{\psi}\right].
	\end{align}
We call this a {\em shift--and--rescale} version of the quantile set (or sometimes corrected set): this new interval is obtained by recentering $\mathcal{I}_{\al_n}$ at $\bar\psi$ and applying a {\em shrinking} factor $\sqrt{\al_n}$.

	\begin{thm}\label{prop_corrected_region}
		Suppose \eqref{conv_BvM_alpha}--\eqref{conv_hat_psi} hold for some $0<\al_n<1$, and that the estimator $\bar{\psi}$ satisfies
		\begin{align}\label{assum::estimator_bar_psi}
		\bar{\psi} = \hat{\psi} +o_{P_0}(1/\sqrt{n}).
		\end{align}
	Then $\mathcal{J}_{\al_n}$ in \eqref{def::eff_region}  is an asymptotically efficient confidence interval of level $1-\delta$ for $\psi(\eta_0)$, i.e.
	\begin{equation} \label{coveralpost}
	 P_0\left[\psi(\eta_0) \in \mathcal{J}_{\al_n}\right] \to 1-\delta 
	\end{equation} 
as $n\to\infty$.	
In the case that $\alpha_n = \alpha \in (0,1]$ is {\em fixed}, if one takes $\bar{\psi}$ to be the $\alpha$--posterior median, then \eqref{assum::estimator_bar_psi} holds. In particular, in that case the region \eqref{def::eff_region} is an asymptotically efficient confidence interval of level $1-\delta$ for estimating $\psi(\eta_0)$.     
\end{thm}

Theorem \ref{prop_corrected_region} states that if the re--centering is close enough to the efficient estimator $\hat\psi$, then the shift--and--rescale modification leads to a confidence set of optimal size (in terms of efficiency) from an information-theoretic perspective, and this is always possible for fixed $\al$ if one centers at the posterior median. When $\al_n$ can possibly go to zero, the situation is more delicate. Indeed, although by definition \eqref{conv_BvM_alpha} is centered around an efficient estimator at the scale $1/\sqrt{n\al_n}$,  it is not clear in general how to deduce from this a similar result at the smaller scale $1/\sqrt{n}$, and the coverage property \eqref{coveralpost} may fail if $\al_n$ goes to zero too quickly.\\

\textit{Modified credible sets in Gaussian white noise.}
 In the Gaussian white noise model interpreted as sequence model $X_k=f_k+\veps_k/\sqrt{n}$, let $f_0=(f_{0,k})$ and place a prior $f_k\sim \cN(0,\la_k)$ independently on each coordinate. 
 
We consider the problem of estimating the linear functional $\psi(f) = \int_0^1 a(t) f(t) dt = \sum_{k=1}^\infty a_k f_{k}$. 
By conjugacy arguments, the \ap distribution of $\psi(f)$ is  $\cN(a_{n,1/2}^X,\bar\sigma^2)$, with 
\[ a_{n,1/2}^X=\sum_{k=1}^\infty \frac{n\an \lambda_k}{1+n\an\lambda_k}a_k X_k,\qquad \bar\sigma^2=\sum_{k=1}^\infty \frac{\lambda_k}{1+n\an\lambda_k}a_k^2.\] Suppose the smoothness of the true function $f_0$, the representer $a$ and the prior are specified through the magnitude of their basis coefficients as follows, for $\beta, \mu, \gamma > 0$,
\begin{align} \label{params}
	f_{0,k} = k^{-\frac{1}{2} - \beta}, \hspace{5mm} a_k = k^{-\frac{1}{2} - \mu}, \hspace{5mm}
	\lambda_k = k^{-1-2\gamma}.
\end{align}
Setting $\bar{\psi}=a_{n,1/2}^X$ the posterior mean/median,  the shift--and--rescale set is, with $z_\delta$ the standard Gaussian quantiles,
$$
\mathcal{J}_{\al_n}=\left(\bar{\psi} +\sqrt{\an} z_{\delta/2}\bar{\sigma} \,,\, \bar{\psi} +\sqrt{\an} z_{1-\delta/2}\bar{\sigma}\right].
$$
The following result describes the behaviour of the shift--and--rescale sets. 

\begin{prop}\label{prop:corrected_gwn}
In the Gaussian white noise model with Gaussian prior as above, suppose that \eqref{params} holds. Let $(\al_n)$ be a sequence in $(0,1]$. The sets $\mathcal{J}_{\al_n}$ verify the converage property \eqref{coveralpost}
	\begin{enumerate}
		\item when $\beta + \mu > 1+ 2\gamma$:  if and only if $\sqrt{n} \an \rightarrow \infty$.
		\item when $\beta + \mu = 1+ 2\gamma$:  if and only if  $\sqrt{n}\an/\log{n}  \rightarrow \infty$.
		\item when $1/2 + \gamma < \beta + \mu < 1+ 2\gamma$:  if and only if  $n^{1-\frac{1+2\gamma}{2(\beta + \mu)}}\an \rightarrow \infty.$
	\end{enumerate}
\end{prop}  

One assumes $\gamma + 1/2 < \beta + \mu$, which corresponds to the case where  BvM holds for the standard posterior ($\an \equiv 1$)  of $\psi(\eta)$ holds, see  Theorem 5.4 in \cite{kvv11}.
  In agreement with this, we see by setting $\al_n=1$ in Proposition \ref{prop:corrected_gwn} that in all cases standard credible sets $\mathcal{J}_1$ are efficient confidence sets. 
 The point here is to investigate to what extent shift--and--rescale sets $\mathcal{J}_{\al_n}$ centered at the posterior median remain efficient confidence sets when $\al_n$ goes to $0$.  In Cases 1 and 2 of the result, the condition is very mild and any $(\al_n)$ essentially slower than $1/\sqrt{n}$ works (as in a simple Gaussian location model). When $\beta+\mu$ approaches $1/2+\gamma$ (Case 3),  $\alpha_n$ is only allowed to decrease quite slowly to $0$ to preserve efficiency. An interpretation is that the problem becomes more `nonparametric' and the $\al_n$--posterior median does not necessarily concentrate fast enough in order for  \eqref{assum::estimator_bar_psi} to be satisfied.

\subsubsection*{Further examples and discussion}

The methods of this chapter extend virtually to any model verifying a LAN-type expansion as above. Among the examples that have been considered so far using this approach let us mention
\begin{itemize} 
\item separated semiparametric models: the estimation of the translation parameter of a symmetric signal \cite{c12}, or of the error standard deviation in Gaussian regression \cite{djvz13}, are examples of the case with no loss of information; the problem of alignment of curves \cite{c12} and the semiparametric Cox model  \cite{c12, nc23} are examples of settings with loss of information;
\item Survival analysis and the Cox model: the nonparametric right-censoring model is considered in \cite{cv21}, where BvM for linear functionals of the hazard rate are derived, as well as for other regular functionals of the hazard, such as the median survival time and the full survival function (to be discussed in more details in the next Chapter) both often considered in medical applications; in \cite{nc23}, inference on the proportional hazards' parameter is considered as well as linear functionals of the hazard, both in the multivariate case;
\item Diffusions models: BvM results for functionals of drift vector fields of
              multi-dimensional diffusions are obtained in \cite{nr20};
\item Inverse problems: for linear inverse problems, BvM theorems are derived for Gaussian process priors in \cite{kvv11}; linear and nonlinear inverse problems in nonconjugate settings are considered in e.g. \cite{nicklschroedinger, nicklsoehl19, an19, mnp21}; we refer to \cite{nickl23} for an overview.
\end{itemize}
While many relevant results can already be obtained using the above approach, one possible difficulty for some applications is that bias may arise, depending on the functional and chosen prior (which can be the  case for instance if the change of variables condition above does not hold true). Also, another setting that is less explored so far is that of high-dimensional models. We mention only a few references here for illustration 
\begin{itemize}
\item in a causality setting, a central quantity is the {\em average treatment effect}: combining the previous approach with an appropriate propensity score-dependent prior,  Ray and van der Vaart \cite{rv20} show that efficient inference can be obtained under strictly weaker conditions than for Gaussian process priors (the regularity conditions under which bias does not arise are weaker);
\item in a high-dimensional model where one observes a high-dimensional Gaussian vector with invertible  covariance matrix $\Sigma$, \cite{gz16} derive BvM theorems for functionals of $\Sigma$ and of its inverse.
\end{itemize}

\section*{{\em Exercises}}
\addcontentsline{toc}{section}{{\em Exercises}} 
 
\begin{enumerate} 
\item Consider the model $\{\cN(\te,1)^{\otimes n},\ \te\in\RR\}$ with a $\cN(0,1)$ prior on $\te$. 
Prove that BvM holds in total variation with limiting distribution $\cN(\bar{X},1/n)$. You may, for instance, use the inequality $\|P-Q\|_1^2\le K(P,Q)$ and then use the explicit expression of the KL divergence between Gaussian distributions.
\item Show that if the Bernstein--von Mises theorem holds, then quantile credible sets of level $1-\al$ are asymptotically confidence sets of level $1-\al$. 
\item {\em Tempered posteriors}
\begin{enumerate}
\item In the conjugate normal location model with a Gaussian prior, check that \eqref{assum::estimator_bar_psi} holds if and only if $\sqrt{n}\al_n\to\infty$.
\item By using the explicit expression of the posterior distribution, prove Proposition \ref{prop:corrected_gwn}.
\end{enumerate}
\end{enumerate}

\chapter{Bernstein--von Mises II: multiscale and applications} \label{chap:bvm2}

\section{Towards a nonparametric BvM}

In view of the  semiparametric BvM results of Chapter \ref{chap:bvm1}, it is natural to ask whether a {\em nonparametric} BvM theorem can be formulated. If so is it still possible to deduce applications to confidence sets? \\ 
  
A natural place to start is the Gaussian white noise model \eqref{mod.gwn}. Even for such a simple model, Cox (1993) \cite{C93} and Freedman (1999) \cite{Free99} have shown the impossibility of a nonparametric BvM result in a strict $L^2$-setting. Leahu  \cite{L11} shows that a BvM-result in white noise, in the total variation sense and for Gaussian conjugate priors, can only hold for priors inducing a heavy undersmoothing (such priors do not in fact induce $L^2$ random functions). Though these results are nice mathematically, as a consequence of the roughness the induced credible sets are typically too large for most  nonparametric applications.\\
 
\ti{An enlarged space} The idea proposed in \cite{bvmnp, bvmnp2} consists of two parts 1) enlarge the space in which results are formulated and 2) change the notion of convergence in the result. In view of the negative results mentioned above, that rule out the existence of BvM in a pure $L^2$ setting, part 1) is quite natural: one defines a space common to all components of the white noise model. But this space enlargement has another effect: since the space is larger, the norm becomes weaker. Hence tightness at a `fast', parametric, $1/\rn$ rate becomes possible and enables part 2), that is weak convergence of probability measures in the enlarged space for broad classes of prior distributions. Finally, to get back to more `usual' spaces, one `projects back' from the enlarged space to a space such as $L^2$ or $L^\infty$. Depending on the applications one has in mind, the choice of enlarged space may vary: we present two examples in the next sections.
\\

\section{Nonparametric BvMs in weighted Sobolev spaces}

\ti{An enlarged $L^2$--type space} To stay in a Hilbert space setting while enlarging $L^2$, a natural class comes to mind, that of {\em negative}-order Sobolev spaces $\{H^r_2\}_{r<0}$ on $[0,1]$, defined similarly as usual Sobolev spaces, but with negative orders.  
 To obtain sharp results we  need `logarithmic' Sobolev spaces, for a real $s$ and $\delta>0$,  writing 
 $f_{lk}=\langle \psi_{lk}, f \rangle$,
\begin{equation} \label{def-nsob}
H^{s,\delta}_2 := \Bigg\{f :\quad \|f\|^2_{s,2,\delta} := \sum_{l\ge 0} l^{-2\delta} 2^{2ls} \sum_{k=0}^{2^l-1} 
  f_{lk}^2 < \infty \Bigg\}, 
\end{equation}  
  where $\{\psi_{lk}\}$ is a wavelet basis on $[0,1]$ (a Fourier-type basis is also possible here, but wavelets will be crucially needed in the next Section, so for easy reference we write already in terms of wavelet notation). 
  The space should be large enough so that the Gaussian experiment in \eqref{mod.gwn}  can be realised as a tight random element in that space. The critical value for this to be the case turns out to be $s=-1/2$. So, define the collection of Hilbert spaces
\begin{equation} \label{spaces}
H :=  H^{-1/2,\delta}_2,\quad \|\cdot\|_H := \|\cdot\|_{-1/2,2,\delta},~~\delta>1/2.
\end{equation}  
If we denote by $\dob$ the centered Gaussian Borel random variable on $H$ with identity covariance, then the Gaussian white noise model \eqref{mod.gwn} can be written as
\begin{equation} \label{shift}
\ixn = f + \dob/\sqrt{n}, 
\end{equation} 
a natural Gaussian shift experiment in the Hilbert space $H$. 
For any $\delta>1/2$, a simple calculation reveals that the $\|\dob\|_{-1/2,2,\delta}$-norm converges almost surely (by computing the expectation) and $\dob$ is then seen to be tight in $H^{-1/2,\delta}_2$. 
   
We denote by $\mathcal N$ the law of $\dob$ a standard, or canonical, Gaussian probability measure on the Hilbert space $H$. Now we are ready to define the notion of convergence we consider: it is simply weak convergence, but in the space $H$. \\

\ti{Notion of weak BvM in the space $H$} Let $\Pi$ be a prior on $L^2$ and $\Pi[\cdot\given X]$ the corresponding posterior distribution on $L^2$ in the Gaussian white noise model (viewed as sequence models). The latter naturally induces a posterior on $H$ by the injection $L^2\to H$. Let 
\[ \Pi_n=\Pi(\cdot \given X)=\Pi(\cdot \given \ixn)\] denote the corresponding posterior distribution on $H$ given the data from the white noise model \eqref{mod.gwn}, or equivalently, from (\ref{shift}). 
On $H$ and for $z \in H$, define the  transformation 
\begin{equation} \label{def-tau}
\tau_z: f \mapsto \sqrt n (f-z).
\end{equation}
 Let $\Pi_n \circ \tau_{\ixn}^{-1}$ be the image of the posterior law under $\tau_{\ixn}$. The shape of $\Pi_n \circ \tau_{\ixn}^{-1}$ reveals how the posterior concentrates on $1/\sqrt n$-$H$-neighborhoods of the efficient estimator $\ixn$. 
\begin{definition} \label{bvmp}
Consider the white noise model \eqref{mod.gwn} viewed in $H$ as \eqref{shift}. Under a fixed function $f_0$, denote by $P_0=P_{f_0}^{(n)}$ the distribution of $\ixn$. Let $\beta$ be the bounded Lipschitz metric for weak convergence of probability measures on $H$. We say that a prior $\Pi$ satisfies the weak Bernstein - von Mises phenomenon in $H$ if, as $n \to \infty$, \[ \beta(\Pi_n \circ \tau_{\ixn}^{-1},\mathcal N) \to^{P_0} 0.\]
\end{definition}
Thus when the weak Bernstein-von Mises phenomenon holds the posterior necessarily has the approximate shape of an infinite-dimensional Gaussian distribution. Moreover, we require this  Gaussian distribution to equal $\cN$ -- the canonical choice in view of efficiency considerations. The covariance of $\cN$ is the Cram\'er-Rao bound for estimating  $f$ in the Gaussian shift experiment \eqref{shift} in $H$-loss, and this can be seen to carry over to sufficiently regular real-valued functionals.

At this point one may wonder whether the notion of  weak convergence in $H$ is strong enough to include interesting applications. \\
 
\ti{Uniformity classes for weak convergence}  
Since we have a statement in terms of weak convergence (as opposed e.g. to  total variation) we cannot infer that $\Pi_n \circ \tau_{\ixn}^{-1}$ and $\mathcal N$ are approximately the same for every Borel set in $H$, but only for sets $B$ that are continuity sets for the probability measure $\mathcal N$. For statistical applications of the BvM phenomenon one typically needs some uniformity in $B$. Weak convergence in $H$ implies that $\Pi_n \circ \tau_{\ixn}^{-1}$ is close to $\mathcal N$ uniformly in classes of subsets of $H$ whose boundaries are sufficiently regular relative to the measure $\mathcal N$. By a result on such `uniformity classes' recalled in Appendix \ref{chap:app}, it is enough that $\mathcal N$ does not charge the `boundary' of the elements of the class. By general properties of Gaussian measures on separable Banach spaces, the collection of all centered balls (or rather, in a $L^2$-perspective, ellipsoids) for the $\|\cdot\|_H$-norm verifies \eqref{nunif} and thus forms a $\cN$-uniformity class.
 \\
 
\ti{Application: Weighted $L^2$-credible ellipsoids} Recall that $H$ stands for the space $H(\delta)$ from (\ref{spaces}) for some arbitrary choice of $\delta>1/2$.
 Denote by $B(g,r)= \{f \in H: \|f-g\|_H \le r\}$. In terms of the wavelet basis $\{\psi_{lk}\}$ of $L^2$, this corresponds to $L^2$-ellipsoids  of radius $r$
 \[ \Big\{\{c_{lk}\}: \quad \sum_{l,k} l^{-2\delta} 2^{-l}   |c_{lk}-\langle g, \psi_{lk} \rangle|^2 \le r^2 \Big\}. \] 
To find the appropriate radius, one may use the quantiles of the posterior distribution. Given $\alpha>0$, one solves for $R_n\equiv R ( \ixn, \alpha)$ such that 
\begin{equation} \label{qut}
\Pi(f: \| f - T_n \|_{H} \le R_n/\sqrt n \given \ixn) = 1-\alpha,
\end{equation}
where $T_n=\ixn$.  
 By its mere definition, a  $\|\cdot\|_{H}$-ball centred at $T_n$ of radius $R_n$ constitutes a level $(1-\alpha)$-{\em credible} set for the posterior distribution. The weak Bernstein-von Mises phenomenon in $H$ implies that this credible ball asymptotically coincides with the exact $(1-\al)$-confidence set built using the efficient estimator $\ixn$ for $f$, by the following result.
 
\begin{thm} \label{sct}
Suppose the weak Bernstein-von Mises phenomenon in the sense of Definition \ref{bvmp} holds. Given $0<\alpha<1$ consider the credible set
\begin{equation} \label{cred}
C_n=\left\{f :\ \|f-\ixn \|_{H} \le R_n/\sqrt n \right\}
\end{equation}
where $R_n\equiv R ( \ixn, \alpha)$ is such that $\Pi(C_n| \ixn) = 1-\alpha$. Then, as $n \to \infty$, 
\[ P_0 (f_0 \in C_n) \to 1-\alpha\quad\text{and}\quad R_n=O_P(1). \]
\end{thm}

\begin{proof}[Proof of Theorem \ref{sct}]
By the result in Appendix \ref{chap:app}, the balls $\{B(0, t)\}_{0 \le t <\infty}$ form a $\mathcal N$-uniformity class, and we can thus conclude from Definition \ref{bvmp} that $$\sup_{0\le  t < \infty}\left|\Pi(f: \|f-\ixn\|_{H(\delta)} \le t/\sqrt n |\ixn) - \mathcal N(B(0,t))\right| \to 0$$ in $P_0$-probability, as $n \to \infty$. This combined with (\ref{qut}) gives $$\mathcal N(B(0,R_n)) = \mathcal N(B(0,R_n)) - \Pi(f: \|f-\ixn\|_{H(\delta)} \le R_n/\sqrt n |\ixn) + 1-\alpha,$$ which goes to $1-\alpha$ as $n \to \infty$ in $P_0$-probability, and thus, by the continuous mapping theorem, $R_n$ converges to $\Phi^{-1}(1-\alpha)$ in probability. This implies
\begin{align*}
P_0(f_0 \in C_n) &= P_0 (f_0 \in B(\ixn, R_n/\sqrt n)) = P_0 (0 \in B(\dob, R_n)) \\
&= P_0 (0 \in B(\dob,  \Phi^{-1}(1-\alpha))) +o(1)  = \mathcal N(B(0, \Phi^{-1}(1-\alpha)) +o(1) \\
&= \Phi(\Phi^{-1}(1-\alpha))+o(1) = 1-\alpha+o(1).
\end{align*}
\end{proof}

Although the confidence set $C_n$ in \eqref{cred} has small radius with respect to the $H$--norm, ideally one would like to be back in $L^2$. We can do so by using a  helpful `interpolation' idea.\\
  
\ti{Interpolation idea} 
Let $B_{\al,\|\cdot\|_\infty}(g,R)$ denote a ball in the H\"older space of functions of order $\al>0$ on $[0,1]$, and  $B_{\|\cdot\|_2}(g,R)$ a ball for the standard $L^2$-norm on $[0,1]$. 

Let $c_1,c_2>0$ be given and let $g\in L^2$. Then there exists $c_3>0$ such that for $n\ge 1$,
\begin{equation} \label{interpolation}
 B_{\|\cdot\|_H}\left(g,\frac{c_1}{\rn} \right)\ \cap \ B_{\|\cdot\|_{\alpha,\infty}}\left(0,c_2\right)\ \subset\
B_{\|\cdot\|_2}\left(g, \frac{c_3 l_n}{n^{\frac{\al}{2\al+1}}}\right), 
\end{equation} 
where $l_n=\log^{2\delta}{n}$. So, provided it is possible to further intersect the previous credible set with a $\alpha$-H\"older ball of fixed radius, the resulting set is automatically included in a $L^2$-ball of radius precisely the standard minimax nonparametric rate for $\alpha$-regular functions, up to a logarithmic term (one may note that the log-term comes from the logarithmic correction to the space $H$). \\
 
\ti{Example of nonparametric confident credible set}
For the sake of simplicity, consider first a uniform wavelet prior $\Pi$ on $L^2$ arising from the law of the random wavelet series, for $\al>0$,
\begin{equation} \label{ugm}
 U_{\al,M} = \sum_{l} \sum_{k} 2^{-l(\al+1/2)}u_{lk} \psi_{lk}(\cdot), 
\end{equation}
where the $u_{lk}$ are i.i.d.~uniform on $[-M,M]$ for some $M>0$ and indexes $l,k$ vary as usual. Such priors model functions that lie in a fixed H\"older ball of $\|\cdot\|_{\al, \infty}$-radius $M$, with posteriors $\Pi(\cdot\given \ixn)$ contracting about $f_0$ at the $L^2$-minimax rate within logarithmic factors if $\|f_0\|_{\al, \infty} \le M$, see \cite{gn11}. Of course in practice using such a prior means that an upper-bound on the $\alpha$-H\"older norm is known, which may not always be the case (but the method may then be adapted).
 
In this situation it is natural to intersect the credible set $C_n$ with a H\"older ball 
\begin{equation} \label{credunif}
\cC_n =\left\{f: \|f\|_{\al, \infty} \le M,\quad \|f-\bar f_n\|_{H} \le R_n/\sqrt n \right\},
\end{equation}
where $R_n$ is as in (\ref{qut}) with $T_n= \bar f_n$.  By definition of the prior $\Pi$ induced by $U_{\al,M}$ above, we have $\|f\|_{\al, \infty} \le M$, $\Pi$-almost surely, so also $\Pi_n$-almost surely.  In particular $\Pi(\cC_n|\ixn)=1-\alpha$, so $\cC_n$ is a credible set of level $1-\alpha$. Theorem \ref{sct} implies the following result.

\begin{corollary} \label{unifset}
Consider observations  from  \eqref{shift} under a function $f_0 \in C^\al$ with $\|f_0\|_{\al, \infty}<M$. Let $\Pi$ be the law of $U_{\al, M}$, let $\Pi(\cdot|\ixn)$ the associated posterior and let $\cC_n$ be as in (\ref{credunif}). Then 
\[P^n_{f_0} (f_0 \in  \cC_n) \to 1-\alpha\] as $n \to \infty$ and the $L^2$-diameter $|\cC_n|_2$ of $\cC_n$ satisfies, for some $\kappa>0$, 
\[ |\cC_n|_2 = O_P(n^{-\al/(2\al +1)} (\log n) ^\kappa).\]
\end{corollary} 

More generally, to avoid the use of a fixed bound $M$ in \eqref{ugm}, one may use more general priors  that have infinite support on each coordinate. Then instead of intersecting with an $\alpha$-H\"older ball of given radius $M$, it is enough to intersect it with a ball whose radius is a type of `posterior quantile' for a $\alpha$-order norm. The results then parallel those for the uniform priors, but without the boundedness constraint. A precise statement is given in \cite{bvmnp}, Corollary 2. \\

\ti{Conclusion so far on credible sets} The meaning of the results we have presented so far is as follows: provided one can show the weak nonparametric BvM for standard nonparametric priors modelling $\alpha$--smooth functions - this will be the object of the next paragraphs - it is possible to deduce {\em confident credible sets} which have diameter equal to the nonparametric minimax rate of convergence for such problems, up to a slowly varying factor (this extra log-term can in fact be replaced by an {\em arbitrary} factor $M_n\to\infty$, up to a slightly different definition of the space $H$). It is important to note that for the argument to go through, $\alpha$--smooth priors should be allowed, as opposed to priors inducing a too severe undersmoothing.

 We also note that such confidence sets are for {\em fixed regularity} $\alpha$ (i.e.~one should know $\alpha$, or a lower bound on it, to at least `undersmooth'). Construction of {\em adaptive} confident credible sets is an interesting further problem, but is qualitatively somewhat different: in particular, rates will typically change, unless something more is assumed on the considered functions. 
 \\

\subsubsection*{A BvM-theorem in $H(\delta)$}

Let $\Pi_n=\Pi(\cdot|\ixn)$ be the posterior distribution on $L^2$. Under the following Conditions \ref{lolk}  (which depends on $\delta'>0$ to be specified in the sequel) and \ref{finidim},  we prove a weak BvM in $H(\delta)$ for any $\delta>1/2$. For the product priors as above we then verify the Conditions.

\begin{condition} \label{lolk}
Suppose for every $\veps>0$ there exists a constant $0<M \equiv M(\varepsilon)<\infty$ independent of $n$ such that, for any $n \ge 1$, some $\delta'>1/2$,
\begin{equation} \label{bdclt}
E_0 \Pi\left[ \left\{ f: \|f-f_0\|^2_{H(\delta')} > \frac{M}{ n} \right\} |\, \ixn\right]  \le \veps.
\end{equation}
\end{condition}
Condition \ref{lolk} is a tightness-type condition. We now also need convergence of finite-dimensional distributions.

Let $V$ be any  finite-dimensional projection subspace of $L^2$ defined as the linear span of a finite number of $\psi_{lk}$ basis elements, equipped with the $L^2$-norm, and denote by $\pi_V$ the orthogonal projection onto $V$. For $z \in H(\delta)$ define the transformation 
\[ T_{z,V}:  f \mapsto \sqrt n ~ \pi_{V}(f-z)\] from $H(\delta)$ to $V$, and consider the image measure $\Pi_n \circ T_{z,V}^{-1}$. The finite-dimensional space $V$ carries a natural Lebesgue product measure on it.

\begin{condition} \label{finidim}
Suppose that for any finite-dimensional projection subspace $V$, spanned by a finite arbitrary collection of basis elements $\psi_{lk}$, for $\cN(0, I)$ the dim$(V)$--dimensional standard Gaussian law,
\[ \beta_V(\Pi_n \circ T_{V,\ixn}^{-1},\cN(0, I)) \to^{P_0} 0, \]
where $\beta_V$ denotes the bounded--Lipschitz metric on $V$.
\end{condition} 

Condition \ref{finidim} asks that a parametric BvM result holds for the (rescaled) projected posterior onto finite-dimensional subspaces $V$. In the case of product priors in white noise, verifying this is not difficult, see below. 


On $H(\delta)$ and for $z \in H(\delta)$, define the measurable map $$\tau_z: f \mapsto \sqrt n (f-z).$$ 
Let $\mathcal N$ be the Gaussian measure on $H(\delta)$ as above.

\begin{thm} \label{fbvm}
Fix $\delta>\delta'>1/2$ and assume Condition \ref{lolk} for such $\delta'$, as well as Condition \ref{finidim}. If $\beta$ is the bounded Lipschitz metric for weak convergence of probability measures on $H(\delta)$ then, as $n \to \infty$, $\beta(\Pi_n \circ \tau_{\ixn}^{-1},\mathcal N) \to 0$ in $P_0$-probability.
\end{thm}
\begin{proof} 
It suffices to show that for any $\varepsilon>0$ there exists $N=N(\varepsilon)$ large enough such that for all $n \ge N$,
\begin{equation*} 
P_0 \left(\beta(\Pi_n \circ \tau_{\ixn}^{-1},\mathcal N)>4\varepsilon \right) < 4\varepsilon,
\end{equation*}
Fix $\veps>0$ and let $V_J$ be the finite-dimensional subspace of $L^2$ spanned by $\{\psi_{lk}: k \in \mathcal Z_l, l \in \mathcal L, |l|\le J\}$, for any integer $J\ge 1$.  Writing $\tilde \Pi_n$ for $\Pi_n \circ \tau_{\ixn}^{-1}$ we see from the triangle inequality
$$ \beta(\tilde \Pi_n, \mathcal N)  \le \beta (\tilde \Pi_n, \tilde \Pi_n \circ \pi_{V_J}^{-1}) + \beta(\tilde \Pi_n \circ \pi_{V_J}^{-1}, \mathcal N \circ \pi_{V_J}^{-1}) + \beta(\mathcal N \circ \pi_{V_J}^{-1}, \mathcal N).$$
The middle term converges to zero in $P_0$-probability for every $V_J$, by convergence of the finite-dimensional distributions (Condition \ref{finidim}). Next we handle the first term.
Set $Q=M = M(\veps^2/4)$ and define a random subset $D$ of $H(\delta')$  as
$$D = \{ g:\, \|g+\dob\|_{H(\delta')}^2 \le Q \}.$$
Under $P_0$we have $\tilde \Pi_n(D)=\Pi_n(D_n)$, where 
$$D_n=\{f:\ \|f-f_0\|_{H(\delta')}^2 \le Q/n\}$$ is the complement of the 
set appearing in \eqref{bdclt}. In particular, using Condition \ref{lolk} and Markov's inequality yields
$P_0( \tilde \Pi_n(D^c)>\veps/4)\le \veps^2/\veps=\veps$. 

If $Y_n\sim \tilde \Pi_n$ (conditional on $\ixn$), then $\pi_{V_J}(Y_n) \sim  \tilde \Pi_n \circ \pi_{V_J}^{-1}$. For $F$ any bounded function on $H(\delta)$ of Lipschitz-norm less than one
\begin{align*}
& \left|\int_{H(\delta)} F d\tilde \Pi_n - \int_{H(\delta)} F d(\tilde \Pi_n \circ \pi_{V_J}^{-1}) \right| = \left|E_{\tilde \Pi_n} \left[F(Y_n)-F(\pi_{V_J}(Y_n))\right] \right| \\
& \le E_{\tilde \Pi_n} \left[ \|Y_n-\pi_{V_J}(Y_n)\|_{H(\delta)} 1_{D}(Y_n)\right]  + 
 2{\tilde \Pi_n}(D^c),
\end{align*}
where $E_{\tilde \Pi_n}$ denotes expectation under $\tilde \Pi_n$ (given $\ixn$). With $y_{lk}=\langle Y_n, \psi_{lk} \rangle$,
\begin{align*}
E_{\tilde \Pi_n} &\left[ \|Y_n-\pi_{V_J}(Y_n)\|_{H(\delta)}^2 1_{D}(Y_n) \right]
 =E_{\tilde \Pi_n} \left[ \sum_{l > J} a_l^{-1} (\log a_l)^{-2\delta} \sum_k |y_{lk}|^2 1_{D}(Y_n) \right]\\
& = E_{\tilde \Pi_n} \left[ \sum_{l > J} a_l^{-1} (\log a_l)^{2\delta'-2\delta-2\delta'} \sum_k |y_{lk}|^2 1_{D}(Y_n)  \right]\\
& \le (\log a_J)^{2\delta'-2\delta}  E_{\tilde \Pi_n} \left[ \|Y_n\|_{H(\delta')}^2 1_{D}(Y_n) \right]  
 \le 2 (\log a_J)^{2\delta'-2\delta} \left[ Q + \|\dob \|_{H(\delta')}^2 \right].
\end{align*}
From the definition of $\be$ one deduces
\[ \be( \tilde \Pi_n, \tilde \Pi_n \circ \pi_{V_J}^{-1} ) \le 
 2{\tilde \Pi_n}(D^c) + \sqrt{2}(\log a_J)^{\delta'-\delta}\sqrt{Q+ \|\dob \|^2_{H(\delta')}}.
\]
Since $a_J \to \infty$ as $J \to \infty$, deduce $P_0(\be( \tilde \Pi_n, \tilde \Pi_n \circ \pi_{V_J}^{-1} ) >\veps)<2 \veps$ for $J$ large enough, combining the previous deviation bound for $\tilde \Pi_n(D^c)$ and that $\|\dob\|_{H(\delta')}$ is bounded in probability (since it has bounded expectation).  
 A similar (though simpler) argument leads to
$ P_0(\beta(\mathcal N \circ \pi_{V_J}^{-1}, \mathcal N) >\veps) <\veps,$ using again that any variable with law $\mathcal N$ has square integrable Hilbert-norm on $H(\delta')$. This concludes the proof.
\end{proof}

\subsubsection*{The BvM-theorem for Product Priors}

 Let us  consider priors of the form $\Pi=\otimes_{lk} \pi_{lk}$ defined on the coordinates of the orthonormal basis $\{\psi_{lk}\}$, where $\pi_{lk}$ are probability distributions with Lebesgue density $\vphi_{lk}$ on the real line, with the following assumptions.  For some fixed density $\vphi$ on the real line and admissible indexes $k,l$,  
\[ \vphi_{lk}(\cdot) = \frac{1}{\sil} \vphi\left(\frac{\cdot}{\sil}\right)~~\forall k, \qquad
\text{with}~ \sil>0. \] 
Suppose the coefficients $\psg f_0,\psi_{lk}\psd=: f_{0,lk}$ of the true function $f_0$ satisfy, for some $\al,R>0$ 
\begin{equation} \label{tolkf}
 \sup_{l\ge 0,\, 0 \leq k\leq 2^l-1} 2^{l(\frac12+\al)} |f_{0,lk}|  \le R, \qquad \al>0. 
\end{equation}

\begin{condition} \label{tolk}
Suppose that for a finite constant $M>0$, 
\[ \hspace{-2cm} {\bf  (P1) } \qquad\qquad \sup_{l,k} \frac{|f_{0,lk}|}{\sil} \le M.\]
Suppose also that for some $\ta>M$ and $0<c_\vphi\le C_\vphi<\infty$  \[
 \hspace{-1cm} {\bf (P2) } \qquad
\quad  \vphi(x)\le C_\vphi\ \ \forall x\in\RR,\quad
\vphi(x)\ge c_\vphi\ \ \forall x\in(-\ta,\ta),\quad \text{and}\quad \int_\mathbb R x^2\varphi(x)dx < \infty.\]
\end{condition}
This allows for a rich variety of base priors $\vphi$, such as Gaussian, sub-Gaussian, Laplace, most Student laws, or more generally any law with positive continuous density and finite second moment, but also uniform priors with large enough support.  The full prior on $f$ considered here is thus a sum of independent terms over the basis $\{\psi_{lk}\}$, including many, especially non-Gaussian, processes. 
One may also consider Gaussian processes such as Brownian motion, even if their
 Karhunen-Lo\`eve expansion is not a (localised) wavelet basis, as long as it is smooth enough, see \cite{bvmnp} for details. \\

\ti{Finite-dimensional distributions}  Since the posterior is a product prior, it is enough to verify this when $V$ is spanned by one coordinate only. One can view this as a semiparametric problem with no loss of information, and use the corresponding results from Chapter \ref{chap:bvm1}: by Theorem \ref{thm:genbvm} therein and the paragraphs on semiparametric models after it, it suffices to check that the posterior is  
consistent in $L^2$ and the prior density is continuous and positive around the true $\te_{0,lk}$ (one just assumes it e.g. on the whole real line).\\

\ti{Tightness at rate $1/\sqrt{n}$ for product priors}

\begin{thm} \label{thm-mins}
Consider data from the white noise model \eqref{shift} under a fixed function $f_0$ with coefficients $(f_{0,lk})$ over the basis $\{\psi_{lk}\}$. Then if the product prior $\Pi$ and $f_0$ satisfy Condition \ref{tolk}, we have, as $n\to\infty$,
\[ E^{(n)}_{f_0} \int \|f-f_0\|_{H}^2 d\Pi(f \given \ixn)  = O\left(\frac{1}{n}\right) .\]
\end{thm}
Interestingly, even if the result of Theorem \ref{thm-mins} is in terms of a `weak' norm (the rate is `fast' though!), it does not seem possible to derive this convergence rate using a testing approach as in the general rate Theorem \ref{thm:ggvt}. Instead, we use a type of {\em multiscale} approach, which we later formalise for more complex norms as in Section \ref{sec:sn}. \\

\begin{proof}[Proof of Theorem \ref{thm-mins}]
We decompose the indexing set $\mathcal L$ into $\cJ_n:=\{l\in\mathcal L,\ \rn\sigma_l \ge S_0 \}$ and its complement, where $S_0$ is a fixed positive constant. The quantity we wish to bound equals, by definition of the $H$-norm and Fubini's theorem
$$\sum_{l,k} a_l^{-1} (\log a_l)^{-2\delta} E_0 \int (\tlk-\tolk)^2 d\Pi(\tlk\given \ixn).$$
Define further
$B_{lk}(\ixn):= \int (\tlk-\tolk)^2 d\Pi(\tlk\given \ixn)$ whose $P_0$-expectation we now bound. We write $\ix=\ixn$ and $E=E_{0}$ throughout the proof to ease notation. 

Using the independence structure of the prior $\Pi(\tlk\given \ix)=\pi_{lk}(\tlk\given \ix_{lk})$, and under $P_0$,
\begin{align*}
  B_{lk}(\ix) & =    \frac{\int (\tlk - \tolk)^2 
  e^{-\frac{n}{2} (\tlk - \tolk)^2 +\rn\elk  (\tlk - \tolk)} \vphi_{lk}(\tlk)d\tlk}
  {\int e^{-\frac{n}{2} (\tlk - \tolk)^2 +\rn\elk  (\tlk - \tolk)} \vphi_{lk}(\tlk)d\tlk} \\
   & = \frac{1}{n}
   \frac{\int v^2 
  e^{-\frac{v^2}{2} +\elk v} \frac{1}{\rn\sil}\vphi\left(\frac{\tolk+\nm v}{\sil}\right)dv}
  {\int e^{-\frac{v^2}{2} +\elk v}\frac{1}{\rn\sil} \vphi\left(\frac{\tolk+\nm v}{\sil}\right)dv}
  =:\frac{1}{n}\frac{N_{lk}}{D_{lk}}(\elk). 
\end{align*}
Taking a smaller integrating set on the denominator makes the integral smaller
\[ D_{kl}(\veps_{kl}) \ge 
\int_{-\sqrt{n} \sil}^{\sqrt{n} \sil} e^{-\frac{v^2}{2} +\elk v}\frac{1}{\rn\sil} \vphi\left(\frac{\tolk+\nm v}{\sil}\right) dv. \]
To simplify the notation we suppose that $\ta>M+1$. If this is not the case, one multiplies the bounds of the integral in the last display by a small enough constant. 
For indices in $\cJ_n^c$, the argument of the function $\vphi$ in the previous display stays in  $[-M+1,M+1]$ under {\bf (P1)}. Under assumption {\bf (P2)} this implies  that the value of $\vphi$ in the last expression is bounded from below by  $c_\vphi$. 
Next applying Jensen's inequality with the logarithm function
\begin{eqnarray*}  \log D_{kl}(\veps_{kl}) &\ge& \log(2c_\vphi)-\int_{-\sqrt{n} \sil}^{\sqrt{n} \sil} \frac{v^2}{ 2} \frac{dv}{2\sqrt{n}\sil} + \elk\int_{-\sqrt{n} \sil}^{\sqrt{n} \sil}  v \frac{dv}{2\sqrt{n}\sil} \\
&=&  \log(2c_\vphi)- (\sqrt{n}\sil)^2/6.
\end{eqnarray*}
Thus, $D_{kl}(\veps_{kl}) \ge 2c_\vphi e^{- (\sqrt{n}\sil)^2/6}$, which is bounded away from zero for indices in $\cJ_n^c$. 
For the numerator, let us split the integral defining $N_{kl}$ into two parts 
$\{v:\ |v|\le \rn \sil\}$ and $\{v:\ |v| > \rn \sil\}$. That is $N_{kl}(\veps_{kl})=(I)+(II)$. Taking the expectation of the first term and using Fubini's theorem,
\begin{align*}
E (I) & =  \int_{-\sqrt{n} \sil}^{\sqrt{n} \sil} v^2 
  e^{-\frac{v^2}{2}} E[ e^{\elk v}] \frac{1}{\rn\sil}\vphi\left(\frac{\tolk+\nm v}{\sil}\right)dv \le 2 n\sil^2C_\vphi/3.
\end{align*}
The expectation of the second term is bounded by first applying Fubini's theorem as before and then changing variables back
\begin{align*}
E (II) & =   \int_{|v|> \sqrt{n} \sil} v^2 
  e^{-\frac{v^2}{2}} E[ e^{\elk v}] \frac{1}{\rn\sil}\vphi\left(\frac{\tolk+\nm v}{\sil}\right)dv \\
  & = \int_{\frac{\tolk}{\sil}+1}^{\pli} \left(\rn\sil u - \rn \sil\frac{\tolk}{\sil} \right)^2\vphi(u)du
   +  \int_{-\infty}^{\frac{\tolk}{\sil}-1} \left(\rn\sil u - \rn \sil\frac{\tolk}{\sil} \right)^2\vphi(u)du \\
 & \le 2 n\sil^2\left[ \frac{\tolk^2}{\sil^2} + \int_{-\infty}^{\pli} u^2\vphi(u)du \right].
\end{align*}
Thus, using {\bf (P1)} again, $E (I)+E(II)$ is bounded on ${\cJ^c_n}$ by a fixed constant times $n\sigma_l^2$. In particular, there exists a fixed constant independent of $n, k, l$ such that $E(nB_{lk}(X))$ is bounded from above by a constant on ${\cJ^c_n}$.

Now about the indices in $\cJ_n$. For such $l,k$, using {\bf (P1)-(P2)} one can find $L_0>0$ depending only on $S_0, M, \ta$ such that, for any $v$ in $(-L_0,L_0)$, $\vphi((\tolk+\nm v)/\sil) \ge c_\vphi.$ Thus the denominator $D_{lk}(\elk)$ can be bounded from below by
\[ D_{lk}(\elk) \ge c_\vphi
\int_{-L_0}^{L_0} e^{-\frac{v^2}{2} +\elk v}\frac{1}{\rn\sil} dv.\]
On the other hand, the numerator can be bounded above by
\begin{align*}
 N_{lk}(\elk) 
 & \le C_\vphi \int v^2 e^{-\frac{v^2}{2} +\elk v}\frac{1}{\rn\sil} dv,
\end{align*}
Putting these two bounds together leads to
\[ B_{lk}(\elk) \le \frac{1}{n} \frac{ C_\vphi }{ c_\vphi }
\frac{ \int v^2 e^{-\frac{v^2}{2} +\elk v}dv}{\int_{-L_0}^{L_0} 
e^{-\frac{v^2}{2} +\elk v} dv}. \]
The last quantity has a distribution independent of $l,k$. Let us thus show that
 \[ Q(L_0)= E\left[ \frac{ \int v^2 e^{-\frac{1}{2}(v-\veps)^2}dv}{\int_{-L_0}^{L_0} 
e^{-\frac{1}{2}(v-\veps)^2} dv} \right] \]
is finite for every $L_0>0$, where $\veps \sim N(0,1)$. 
In the numerator we substitute $u=v-\veps$. Using the inequality $(u+\elk)^2\le 2v^2+2\elk^2$, the second  moment of a standard normal variable appears, so
\[Q(L_0) \le C E\left[ \frac{1+\veps^2}{ \int_{-L_0}^{L_0} 
e^{-\frac{1}{2}(v-\veps)^2} dv} \right] \]
for some finite constant $C>0$. Denote by $g$ the density of a standard normal variable, by $\Phi$ its distribution function and $\bar{\Phi}=1-\Phi$. It is enough to prove that the following quantity is finite
\[ q(L_0) := \int_{-\infty}^{+\infty} \frac{(1+u^2)g(u)}{\bfi(u-L_0)-\bfi(u+L_0)}du
= 2 \int_{0}^{+\infty} \frac{(1+u^2)g(u)}{\bfi(u-L_0)-\bfi(u+L_0)}du,  \]
since the integrand is an even function. Using the standard inequalities
\[  \frac{1}{\sqrt{2\pi}} \frac{u^2}{1+u^2}\frac{1}{u} e^{-u^2/2}
\le \bfi(u) \le  \frac{1}{\sqrt{2\pi}} \frac{1}{u} e^{-u^2/2},\qquad u\ge 1,\]
it follows that for any $\delta>0$, one can find $M_\delta>0$ such that, for any $u\ge M_\delta$, 
\[ (1-\delta) \frac{1}{u} e^{-u^2/2} \le \sqrt{2\pi} \bfi(u) \le
\frac{1}{u} e^{-u^2/2},\qquad u\ge M_\delta.\]
Set $A_\delta=2L_0 \vee M_\delta$. Then for $\delta <1-e^{-2L_0}$ we deduce
\begin{align*}
q(L_0) & \le 2 \int_0^{A_\delta}\frac{(1+u^2)g(u)}{\bfi(A_\delta-L_0)-\bfi(A_\delta+L_0)}du
+2\sqrt{2\pi}\int_{A_\delta}^{\pli}(u-L_0)(1+u^2) 
\frac{e^{\frac{1}{2}(u-L_0)^2}g(u)}{1-\delta-e^{-2L_0}}du\\
& \le C(A_\delta,L_0) + \frac{2e^{-L_0^2/2}}{1-\delta-e^{-2L_0}}\int_{A_\delta}^{\pli} u(1+u^2)e^{-L_0 u}du <\infty.
\end{align*}
Conclude that $\sup_{l,k}E_0|B_{lk}(\ix)| = O(1/n)$. Since $\sum_{l,k}a_l^{-1} (\log a_l)^{-2\delta} <\infty$ the result follows. 
\end{proof}
\vp

\begin{thm}\label{thm-l2}
With the notation of Theorem \ref{thm-mins}, suppose the product prior $\Pi$ and $f_0$ satisfy Condition \ref{tolk}. 
Then for any real numbers $\gamma,\delta$,  \[ E^{n}_{f_0} \int \|f-f_0\|_{\gamma,2,\delta}^2 d\Pi(f \given \ixn) = O\left(\sum_{l, k}a_l^{2\gamma} (\log a_l)^{-2\delta} (\sigma_l^2 \wedge n^{-1} ) \right).\]
\end{thm}
For the above theorem note that $\gamma=\delta=0$ gives $\|\cdot\|_{0,2,0}=\|\cdot\|_2$.\\
 
\begin{proof}
We only prove $\gamma=\delta=0$, the general case is the same. Using Fubini's theorem,
\[  E_0 \int \|f-f_0\|_{2}^2d\Pi(f \given \ixn)
= \sum_{l,k} E_0 \int (\tlk-\tolk)^2 d\Pi(\tlk\given \ix) 
=  \sum_{l,k} E_0 B_{lk}(\ix).
\]
In the proof of Theorem \ref{thm-mins} we have shown, with the notation $ \cJ_n:=\{l\in\mathcal L,\ \rn\sigma_l \ge S_0 \}$,
$$\sup_{l\in \cJ_n,\, k} E_0 B_{lk}(\ix) =  O( n^{-1})\ , \ \sup_{l\notin \cJ_n, 1\le k\le 2^l} E_0 B_{lk}(\ix)  =  O( \sigma_l^2).$$
For any $l\in\cJ_n^c$, by definition of $\cJ_n$ it holds $\sil^2<S_0^2n^{-1}$,  thus $\sil^2\le(1\vee S_0^2)( \sigma_l^2 \wedge n^{-1})$. Similarly, if $l\in\cJ_n$ we have $n^{-1}\le (1\vee S_0^{-2})( \sigma_l^2 \wedge n^{-1})$. 
\end{proof}

\vp

Combining Theorems \ref{fbvm}, \ref{thm-mins}, and convergence of finite dimensional distributions implies that for product priors the weak BvM theorem in the sense of Definition \ref{bvmp} holds. The following results can be seen to be uniform (`honest') in all $f_0$ that satisfy Condition \ref{tolkf} with fixed constant $M$.

\begin{thm} \label{pfbvm}
Suppose the assumptions of Theorem \ref{thm-mins} are satisfied and that $\vphi$ is continuous near $\{\te_{0,lk}\}$ for every $k\in \mathcal Z_l, l \in \mathcal L$. Let $\delta>1/2$. Then for $\beta$ the bounded Lipschitz metric for weak convergence of probability measures on $H(\delta)$ we have, as $n \to \infty$, $\beta(\Pi_n \circ \tau_{\ixn}^{-1},\mathcal N) \to^{P_0} 0.$ 
\end{thm}
\begin{proof}
We only need to verify Condition \ref{lolk} with some $1/2<\delta'<\delta$ so that we can apply Theorem \ref{fbvm}. From Theorem \ref{thm-mins} with any such $\delta'$ in place of $\delta$, we see that
\begin{equation} \label{mkv}
 nE_0\int \|f-f_0\|^2_{H(\delta')}d\Pi(f|\ixn) = O(1),  
\end{equation}
which verifies Condition \ref{lolk} for some $M$ large enough using Markov's inequality.
Verification of Condition \ref{finidim} has already been done above.
\end{proof}

\vp

\ti{Further applications and uniform semiparametrics} Another important set of applications of the weak nonparametric BvM theorem is related to continuous functionals. Indeed, by the continuous mapping theorem, it immediately follows that the weak convergence result in Definition \ref{bvmp} implies weak convergence of the image measures through any {\em continuous} mapping $\psi:H\to \mathcal{Y}$, for some given space $\mathcal{Y}$. Applications include semiparametric BvM results for
linear and smooth nonlinear functionals, credible bands for selfconvolutions, etc.,  
see \cite{bvmnp}, Section 2. In this perspective, one may see the weak nonparametric BvM as a semiparametric BvM `uniform in many functionals'.

This also leads to a natural question: is the choice of space $H$ canonical? What if the goal is to obtain credible sets in different norms than $\|\cdot\|_2$, such as $\|\cdot\|_\infty$? We consider this next.
 
\section{Nonparametric BvMs in multiscale spaces} \label{sec:bvmmult}

In the previous section we have considered the white noise model and confidence-sets results linked to the $\|\cdot\|_2$-norm. In addition to the question of obtaining results in terms of different norms, it is natural to consider the nonparametric BvM question for other statistical models as well. 

In this section and Section \ref{sec:donsker}, we focus for simplicity on density estimation, following \cite{bvmnp2}. We note that the construction below can also be followed in the white noise model as an alternative to the construction in the previous section. The ideas can be applied to more complex models as well, see Section \ref{sec:sn}. 
\\

We define H\"{o}lder-type spaces $C^s$ of continuous functions on $[0,1]$, for $f_{lk}=\langle f, \psi_{lk} \psd$,
\begin{equation} \label{hold}
C^s([0,1]) =  \left\{f \in C([0,1]):  \|f\|_{s, \infty} := \sup_{l, k} 2^{l(s+1/2)}  |f_{lk}| < \infty \right\}.
\end{equation}

\vp

\ti{Density model, limiting distribution}
Consider the density model \eqref{mod.dens} where we observe $X=(X_1, \dots, X_n)$ i.i.d.~from law $P$ with density $f$ on $[0,1]$. 

The first step is to identify the limiting distribution for the BvM result. In white noise, the identification of the limit 
was somewhat straightforward,  via  the model equation $\ixn = f +n^{-1/2}\dob$, where $\ixn$ can be seen as an estimator of $f$. In the density model, let us take an intermediate step via projections onto the wavelet basis $\{\psi_{lk}\}$, which we assume to be a localised basis such as the Haar or CDV bases used in Chapter \ref{chap:rate1}.  Let $P_n$ denote the empirical measure associated to observations $X_1,\ldots,X_n$. We also denote $Pf=\int fdP$ and $L^2(P)=\{f:[0,1]\to \mathbb R: \int_0^1 f^2dP<\infty\}$. \\

\ti{The $P$-white bridge process} 
A  natural estimate of the wavelet coefficiencts $f_{lk}=\langle f, \psi_{lk} \rangle$ is 
\[ P_n\psi_{lk}\equiv \langle P_n,\psi_{lk}\rangle = \frac{1}{n} \sum_{i=1}^n \psi_{lk}(X_i). \] 
By the central limit theorem, for $k,l$ fixed and as $n \to \infty$, the random variable $\rn(P_n - P)(\psi_{lk})$ converges in distribution to
\begin{equation} \label{GPwhite}
\mathbb G_{P}(\psi_{lk}) \sim \cN(0,\sigma_{lk}^2),
\end{equation}
with $\sigma_{lk}^2=\text{Var}_P(\psi_{lk}(X_1))=\int_0^1 (\psi_{lk}-P\psi_{lk})^2 dP$. 
In analogy to the white noise process $\dob$, the process $\mathbb G_P$ arising from (\ref{GPwhite}) can be defined as the centered Gaussian process indexed by $L^2(P)$
  with covariance function
\[ \mathbb{E}\left[\mathbb G_{P}(g) \mathbb G_{P}(h)\right]=\int_0^1 (g-Pg)(h-Ph)dP. \]
We call $\mathbb G_P$ the {\em $P$-white bridge process}.  Now we turn to the definition of spaces similar to the large space $H$ of the previous section. \\

\ti{Multiscale spaces $\cM(w)$ and $\cM_0(w)$} 
For monotone increasing weighting sequences $w=(w_l: l \ge J_0-1), w_l \ge 1,$ we define multi-scale sequence spaces
\begin{equation} \label{def-multi}
\mathcal M \equiv \mathcal M(w) \equiv \left\{x=\{x_{lk}\}: \ \ \|x\|_{\mathcal M(w)} \equiv \sup_{l}\frac{\max_{k}|x_{lk}|}{w_l} <\infty \right\}.
\end{equation}
The space $\mathcal M(w)$ is a non-separable Banach space (it is isomorphic to $\ell_\infty$). However, the weighted sequences in $\mathcal M(w)$ that vanish at infinity form a separable closed subspace for the same norm, which leads us to define
\begin{equation} \label{em0}
\mathcal M_0=\mathcal M_0(w)= \left\{x \in \mathcal M(w): \ \ 
\lim_{l \to \infty} \max_k \frac{|x_{lk}|}{w_l}=0\right\}.
\end{equation}
Furthermore, we call a sequence $(w_l)$ {\em admissible} 
if $w_l/\sqrt l \uparrow \infty$ as $l \to \infty$.\\

\ti{$P$-white bridge as tight measure on $\M_0(w)$} 
The idea behind the definition of the enlarged space $\M_0(w)$ is, as for $H$, to find a `smallest' (separable) large space the limit $\mathbb G_P$ belongs to. The next proposition also applies to white noise $\dob$. 

\begin{prop} \label{CKRP} Let  $\mathbb G_P$ be a $P$-white bridge. For $\omega=(\omega_l)=\sqrt l$ we have 
$E\|\mathbb G_P\|_{\mathcal M(\omega)} < \infty.$ If $w=(w_l)$ is admissible then $\mathbb G_P$ defines a tight Gaussian Borel probability measure 
in  $\mathcal M_0(w)$. 
\end{prop} 
The idea is similar to the one for $\dob$ in the $H$ space: cince there are $2^l$ i.i.d.~standard Gaussians $g_{lk}=\langle \psi_{lk}, dW \rangle$ at the $l$-th level, we have from a standard bound $E \max_{k}|g_{lk}| \leq C \sqrt l$ for some universal constant $C$. The Borell-Sudakov-Tsirelson inequality applied to the maximum at the $l$-th level gives, for any $M$ large enough,
\begin{align*}
 \Pr\left(\sup_l l^{-1/2} \max_k |g_{lk}| > M \right) 
& \le \sum_l \Pr \left( \max_k|g_{lk}| - E \max_k |g_{lk}| > \sqrt l M - E \max_k |g_{lk}| \right) \\
& \leq 2 \sum_l \exp \left\{- c(M-C)^2 l \right\}.
\end{align*}
Now using $E[X] \le K + \int_K^\infty \Pr[X\ge t]dt$ for any real-valued  random variable $X$ and any $K\ge 0$, one obtains that  $\|\mathbb W\|_{\mathcal M(\omega)}$ has finite expectation.

Now that a candidate limit process has been identified, we discuss briefly possible centerings. Contrary to the Gaussian white noise case, where one could center at $\ix$ itself, here the natural analogue $\mathbb{P}_n$ does not converge in an appropriate sense, but one can use a truncated version thereof. The next paragraphs on truncations are not {\em per se} needed to establish BvM, but rather give examples of possible `efficient' centerings in the BvM result.  \\

\ti{Truncated empirical measure, convergence}
Any $P$ with bounded density $f$ has coefficients $\langle f, \psi_{lk} \rangle \in \ell_2 \subset \mathcal M_0(w)$. We would like to formulate a statement such as $$\sqrt n (P_n -P) \to^d \mathbb G_P \text{ in } \mathcal M_0,$$ as $n \to \infty$, paralleling (\ref{shift}) in the Gaussian white noise setting. The fluctuations of $\sqrt n (P_n-P)(\psi_{lk})/\sqrt{l}$ along $k$ are stochastically bounded for $l$ such that $2^l \le n$, but are unbounded for high frequencies. Thus the empirical process $\sqrt n (P_n-P)$ will not define an element of $\mathcal M_0$ for every admissible sequence $w$. In our nonparametric setting we can restrict to frequencies at levels $l, 2^l \le n$, and introduce an appropriate `projection' $P_n(j)$ of the empirical measure $P_n$ onto $V_j$ via
\begin{equation} \label{def-pn}
 \langle P_n(j), \psi_{lk}\rangle=
\begin{cases}
 \ \psg P_n, \psi_{lk} \psd & \text{ if }\ l\le j \\
 \  0 & \text{ if }\ l> j, 
\end{cases}
\end{equation}
which defines a tight random variable in $\mathcal M_0$. The following theorem shows that $P_n(j)$ estimates $P$ efficiently in $\mathcal M_0$ if $j$ is chosen appropriately. Note that the natural choice $j=L_n$ such that $$2^{L_n} \sim N^{1/(2\gamma+1)},$$ where $N=n$ (if $\gamma>0$) or $N=n/\log n$ (if $\gamma \ge 0$), is possible.

\begin{thm} \label{CLT}
Let $w=(w_l)$ be admissible. Suppose $P$ has density $f$ in $C^\gamma([0,1])$ for some $\gamma \ge 0$. Let $j_n$ be such that $$\sqrt n 2^{-j_n(\gamma+1/2)}w^{-1}_{j_n} = o(1), ~~\frac{2^{j_n} j_n}{n} =O(1).$$ Then we have, as $n \to \infty$, $$\sqrt n (P_n(j_n) - P) \to^d \mathbb G_P\text{ in } \mathcal M_0(w).$$ 
\end{thm}
The proof is based on controlling maxima of  variables $\langle \sqrt n (P_n-P), \psi_{lk}\rangle$ for any $k$ and $l\le j_n$, which one does using Bernstein's inequality (see \cite{bvmnp2}, proof of Thm. 1 for details).\\
 
\ti{Weak nonparametric BvM in $M_0(w)$}
As in the previous section,  we metrise weak convergence of laws in $\mathcal M_0(w)$ via $\beta_{\mathcal M_0(w)}$ (see \eqref{def-weakcv}), 
and view the prior $\Pi$ on the functional parameter $f\in L^2$ as a prior on sequence space $\ell_2$ under the wavelet isometry $L^2 \cong \ell_2$.

\begin{definition} \label{weakbvmp}
Let $w$ be admissible, let $\Pi$ be a prior and $\Pi(\cdot\given X)$ the corresponding posterior distribution on $\ell_2 \subset \mathcal M_0=\mathcal M_0(w)$, obtained from observations $X$ in the density model. Let $\tilde \Pi_n$ be the image measure 
of $\Pi(\cdot \given X)$ under the mapping 
\[\tau: f \mapsto \sqrt n (f-T_n)\]
 where $T_n=T_n(X)$ is an estimator of $f$ in $\mathcal M_0$. Then we say that $\Pi$ satisfies the \sbl{weak Bernstein von Mises} phenomenon \sbl{in $\mathcal M_0$} with centering $T_n$ if, for $X\sim P_{f_0}^{\otimes n}$ and fixed $f_0$, as $n \to \infty$,
\[\beta_{\mathcal M_0}(\tilde \Pi_n, \mathcal N) \to^{P_0} 0,\] 
where $\mathcal N$ is the law in $\mathcal M_0$ of  
$\,\mathbb G_{P_{0}}$ and $f_0 \in L^\infty$. 
\end{definition}
Although we do not explain this in details here, it is possible to give a notion of {\em efficiency} of estimators of $f$ in $\cM_0$, see e.g. \cite{aadstflour, AadWell} and $P_n(j_n)$ from Theorem \ref{CLT} is efficient; this is the centering that will be taken for the BvM results in $\cM_0$ to follow.\\
  
\ti{Sufficient conditions}  
Similar to  results for the space $H$, to prove a weak BvM in $\cM_0(w)$ it suffices to prove tightness in an appropriate space and convergence of finite-dimensional distributions. The proof of the next Proposition is similar in spirit to that of Theorem \ref{fbvm} and is thus omitted (we refer to \cite{bvmnp2} for details).\\

Let $\pi_{V_J}, J \in \mathbb N,$ be the projection operator onto the finite-dimensional space spanned by the $\psi_{lk}$'s with scales up to $l \le J$. Let $f \sim \Pi(\cdot|X)$, $T_n=T_n(X)$, and let $\tilde \Pi_n$ denote the laws of $\sqrt n (f- T_n)$. For $\mathcal N$ equal to the Gaussian probability measure on $\mathcal M_0(w)$ given by 
$\mathbb G_P$ for $P$ with bounded density. The following conditions are sufficient to prove the weak BvM in $\cM_0(w)$ for $\Pi[\cdot\given X]$

\begin{prop} \label{prop:bvmw}
Suppose, using the notation from the above paragraph, that
\begin{enumerate}  
\item   for some sequence $(\bar w)$ with $\bar w_l/\sqrt l \ge 1$, 
for some finite constant $C$,
\begin{equation} \label{inexp}
E\left[\|f-T_n\|_{\cM_0(\bar{w})}\given X\right] = E\left[\sup_l \bar w_l^{-1} \max_{k}|\langle \sqrt n (f - T_n), \psi_{lk} \rangle| \given X\right]=O_{P_0}(1/\rn).
\end{equation}
\item  finite-dimensional distributions converge: i.e. for any finite-dimensional subspace $V_J$,
\begin{equation} \label{fidij}
\beta_{V_J}\left(\tilde \Pi_n \circ \pi_{V_J}^{-1}, \mathcal N \circ \pi_{V_J}^{-1}\right) \to^{P_0} 0.
\end{equation}
\end{enumerate}
Then,  for any $w$ such that $w_l/\bar w_l \uparrow \infty$ as $l \to \infty$,  
\[ \beta_{\mathcal M_0(w)}(\tilde \Pi_n, \mathcal N) \to^{P_0} 0. \] 
\end{prop}

\ti{Weak nonparametric BvM, density model}
We define multi-scale priors $\Pi$ on some space $\mathcal F$ of probability density functions $f$ giving rise to absolutely continuous probability measures. Suppose the true density $f_0$ is bounded away from $0$ and $\infty$.

We choose as cut-off parameter $L_n$ given by, for $\al>0$, (the closest integer solution to)
\begin{equation} \label{def-ln}
 2^{L_n} = n^{\frac{1}{2\al +1}}.
\end{equation}

{\it {\bf (S)} Priors on log-densities.} 
Given a multi-scale wavelet basis $\{\psi_{lk}\}$ as above, consider the prior $\Pi$
induced by, for any $x\in[0,1]$ and $L_n$ as in \eqref{def-ln},
\begin{align}
T(x) & = \sum_{l\le L_n} \sum_{k=0}^{2^l -1} \sil \al_{lk} \psi_{lk}(x)  \label{def-T} \\
f(x)  & = \exp\left\{  T(x) - c(T) \right\},\quad c(T)=\log \int_0^1 e^{T(x)} dx, \label{def-prior-log}
\end{align}
where $\al_{lk}$ are i.i.d. random variables of continuous probability density $\varphi: \mathbb R \to [0,\infty)$ equal to,  for a given $0\le \ta <1$ and 
$x\in \RR$, and $c_\ta$ a normalising constant, 
\begin{equation}
 \varphi_{H,\tau}(x) = c_\ta \exp\{-(1+|x|)^{1-\tau}\}.  \label{phi-ht} 
\end{equation} 
Suppose the prior parameters $\sil$ satisfy, for $\al>1/2$,
\begin{equation}\label{cond-sil}
 \sil =2^{-l(\al+1/2)}.
\end{equation}

{\bf (H)} {\it Random histograms density priors}.
Consider the regular dyadic partition of $[0,1]$ at level $L\ge 1$:  
$I_0^L=[0,2^{-L}]$ and $I_k^L=(k2^{-L},(k+1)2^{-L}]$ for  $k=1,\dots, 2^L-1$.  Consider
\[ \cH_L^1 :=\left\{h:\ h(x) 
= 2^L \sum_{k=0}^{2^L-1} \omega_k  \1_{\ikl}(x),\quad (\omega_k) \in\cS_L \right\},\]
where $\mathcal S_L$ denotes the unit simplex in $\mathbb R^{2^L}$, the set of 
histograms that are densities on $[0,1]$ with $L$ equally spaced dyadic knots.

Let us set $L=L_n$ as defined in \eqref{def-ln} 
and for some fixed constants $a, c_1,c_2>0$  let
\begin{equation} \label{prk}
L = L_n,\qquad \omega_L \sim \cD(\alpha_0, \dots , \alpha_{2^L-1}), \quad
c_1 2^{-La}\le \alpha_k \le c_2,
\end{equation}
for any admissible index $k$, where $\cD$ denotes the discrete Dirichlet distribution on $\mathcal S_L$. \\

The priors {\bf(S), (H)}  are `multiscale' priors where high frequencies are ignored 
-- corresponding to truncated series priors considered frequently in the nonparametric   Bayes literature. The resulting posterior distributions $\Pi(\cdot\given X)$ attain minimax   optimal contraction rates up to logarithmic terms in Hellinger and $L^2$-distance (\cite{cr15}). \\  

The following theorem shows that the above priors satisfy a weak BvM theorem in $\mathcal M_0$ in the sense of Definition \ref{weakbvmp}, with efficient centring $P_n(L_n)$ (cf.~Theorem \ref{CLT}). Denote the law $\mathcal L(\mathbb G_{P_0})$ of $\mathbb G_{P_0}$ from Proposition \ref{CKRP} by $\mathcal N$. 

\begin{thm}\label{thmdens1}
Let $\mathcal M_0=\mathcal M_0(w)$ for any admissible $w=(w_l)$. Let  $X=(X_1, \dots, X_n)$ i.i.d.~from law $P_0$ with density $f_0\in \cF_0$. Let $\Pi$  be a prior on the set of probability densities $\mathcal F$ that is
\begin{enumerate}
\item either of type {\bf (S)}, in which case one assumes $\log f_0 \in C^\alpha$
 for some $\alpha > 1$,
\item or of type {\bf (H)}, and one assumes $f_0 \in C^\alpha$ for some $1/2<\al\le1$. 
\end{enumerate}
Suppose the prior parameters satisfy \eqref{def-ln}, \eqref{cond-sil} and \eqref{prk}. Let $\Pi(\cdot \given X)$ be the induced posterior distribution on $\mathcal M_0$.   
Then, as $n \to \infty$,
\begin{equation} \label{iidbvm}
 \be_{\mathcal M_0}(\Pi(\cdot\given X)\circ\tau_{P_n(L_n)}^{-1}, \mathcal N)\to^{P_0} 0.
\end{equation}
\end{thm}

\ti{Idea of proof} We refer to \cite{bvmnp2} for a detailed proof, we briefly sketch the main ideas now.

To verify convergence of finite-dimensional distributions, it is enough to show asymptotic normality (with correct centering and variance) for any finite collection of induced posteriors on $\psg f,\psi_{lk} \psd_2$, for $l\le J$ and arbitrary $k$, for some given $J\ge 1$. This exactly the purpose of the BvM theorems in Chapter \ref{chap:bvm1} which we can use here. 

In order to verify tightness, one needs to derive posterior contraction at rate $1/\sqrt{n}$ in a $\cM_0(w)$--norm. There are two main difficulties: first, this is not a testing-type distance in general and second, the parametric-type rate makes it fall outside the scope of the generic nonparametric results from the first Chapters, where the rate has to be slower than $1/\sqrt{n}$. The main ideas here will be to
\begin{enumerate}
  \item  prove sub-Gaussianity of the posterior induced on individual coefficients $f_{lk}=\psg f,\psi_{lk}\psd$ (with high probability and in a non-asymptotic fashion): this can be done by controlling the Laplace transform of the posterior on $f_{lk}$, in a similar fashion as for the BvM theorem in Chapter \ref{chap:bvm1}, but keeping the finite sample estimates. This will in particular imply 
 tightness at the level of individual $f_{lk}$'s; 
 \item reconstruct the $\cM_0(w)$--norm from the $f_{lk}$'s, which amounts to controlling a maximum of sub-Gaussian (by the definition of the norm and point 1.) variables, which grows as the square-root-logarithm of the number of variables: this slowly varying term can be compensated using a slowly increasing sequence $(w_l)$.
\end{enumerate}
We will deploy similar arguments  with an explicit proof in  Section \ref{sec:sn} below on posterior convergence in the supremum norm, where we discuss also other models.\\
  
\ti{Application to credible bands} As the multiscale spaces $\cM_0(w)$ are defined via maxima of collections of wavelet coefficients, they are particularly well-suited to the study of the supremum norm. Indeed, by analogy to the study of the credible sets in the previous section, which were shown to have (nearly) optimal diameter in $L^2$, it is possible up  to minor adaptations to carry out the same method to obtain credible sets which are confidence {\em bands} having an optimal diameter (up to an arbitrary undersmoothing factor $M_n\to\infty$) in $L^\infty[0,1]$. We refer to \cite{bvmnp2}, Section 4.2 for explicit statement. Instead in the next Section  we present in some detail a different application.  \\
  
\ti{Adaptive nonparametric BvM} It is of interest to obtain results as above for a prior that is adaptive to the smoothness of $f$: indeed, one can then, under conditions such as self--similarity, deduce confidence sets (in $L^2$ or $L^\infty$ depending on the prior) using the interpolation idea explained above. This was pioneered in the work by Kolyan Ray \cite{ray17} in white noise (see also \cite{cr21}, and \cite{cm21, cr22} in density estimation) using multiscale spike-and-slab priors with `flat initialisation' (i.e. that put no spike on the first levels).

\section{Application: Donsker's theorem} \label{sec:donsker}

 \ti{Bayesian Donsker's theorem}
 Whenever a prior on $f$ satisfies the weak Bernstein-von Mises phenomenon in the sense of Definition \ref{weakbvmp}, we can deduce from the continuous mapping theorem many BvMs via continuous functionals from $\cM_0(w)$ to arbitrary spaces.
 
One may do so for integral functionals $L_g(f) = \int_0^1 g(x) f(x)dx$ \textit{simultaneously} for many $g$'s satisfying bounds on the decay of their wavelet coefficients. More precisely a bound $\sum_k |\langle g, \psi_{lk} \rangle| \le c_l$ for all $l$ combined with a weak BvM for $(w_l)$ such that $\sum c_l w_l<\infty$ is sufficient. Let us illustrate this in a key example $g_t = 1_{[0,t]}, t \in [0,1],$ where we can derive results paralleling the classical Donsker theorem for distribution functions and its BvM version for the Dirichlet process proved by Albert Lo \cite{L83} using conjugacy. For simplicity we restrict to situations where the posterior $f\given X^{(n)}$ is supported in $L^2,$ and where the centering $T_n$ in Definition \ref{weakbvmp} is contained in $L^2$. In that case the primitives  
\[ F(t) = \int_0^tf(x)dx,\quad \mathbb T_n(t)=\int_0^t T_n(x)dx\] 
define random variables in the separable space $C([0,1])$ of continuous functions on $[0,1]$, and we can formulate a BvM-result in that space. Different centerings, such as the empirical distribution function, are discussed below.

\begin{thm} \label{kolmsmir} Let $\Pi$ be a prior supported in $L^2([0,1])$ and suppose the weak Bernstein - von Mises phenomenon in the sense of Definition \ref{weakbvmp} holds true in $\mathcal M_0(w)$ for some sequence $(w_l)$ such that $\sum_l w_l 2^{-l/2}<\infty$, and with centering  $T_n \in L^2$. Define the  cumulative distribution function 
\begin{equation} \label{primit}
F(t) = \int_0^t f(x)dx,\ t \in [0,1].
\end{equation}
Let $G$ be a $P_0$-Brownian bridge.
If $X \sim P_{f_0}^{\otimes n}$ for some fixed $f_0$ then as $n \to \infty$, 
\begin{equation} \label{funlim}
\be_{C([0,1])}(\mathcal L(\sqrt n (F-\mathbb T_n)\given X), \mathcal L(G))
\to^{P_0} 0,
\end{equation} 
\begin{equation} \label{kslim}
\beta_\mathbb R (\mathcal L(\sqrt n \|F-\mathbb T_n\|_\infty\given X),  
\mathcal L(\|G\|_\infty)) \to^{P_0}0.
\end{equation}
\end{thm}
Let us now apply this result to the case of priors  {\bf(S)} or {\bf{(H)}}, for which the weak BvM has been obtained above with centering the truncated empirical measure $P_n(L_n)$.  Theorem \ref{kolmsmir} leads to a result with centering the primitive of $P_n(L_n)$. One can check that this can be replaced by 
\[ F_n(t) = \frac 1n \sum_{i=1}^n 1_{[0,t]}(X_i),\quad t \in [0,1],\] the empirical distribution function based on a sample $X_1,\ldots,X_n$.
\begin{corollary} \label{classd}
Let $\Pi$ be a prior of type {\bf(S)} or {\bf{(H)}} and suppose the conditions of Theorem \ref{thmdens1} are satisfied.  Then, as $n \to \infty$, 
\[ \beta_{L^\infty([0,1])}(\mathcal L(\sqrt n (F-F_n)\given X), \mathcal L(G_{P_{0}})) \to^{P_0} 0, \]
\[ \beta_\mathbb R (\mathcal L(\sqrt n \|F-F_n\|_\infty\given X), \mathcal L(\|G_{P_{0}}\|_\infty)) \to^{P_0} 0. \]
\end{corollary}
This result is the Bayesian analogue of Donsker's theorem for the empirical distribution function $F_n$. To our knowledge, most results of this kind in a Bayesian context  have been obtained under some form of, at least partial, conjugacy of the model and prior. Note that here the results are obtained from general principles.  \\
  
\begin{proof}[Proof of Theorem \ref{kolmsmir}] 
Consider the mapping $L$ from $\cM_0(w)$ to $L^\infty[0,1]$ --it can be checked that its range is in fact included in $\cC^0([0,1])$-- defined as 
\begin{equation} \label{defl}
\{h_{lk}\} \mapsto L_t(\{h_{lk}\}):=\sum_{l,k} h_{lk} \int_0^t \psi_{lk}(x)dx, ~t \in [0,1].
\end{equation}
This is a well-defined and continuous map since, for $0<c<C<\infty$,
\begin{align*}
\left|\sum_{l,k} h_{lk} \int_0^t  \psi_{lk}(x)dx \right| &\le \sum_{l,k} |h_{lk}| |\langle 1_{[0,t]}, \psi_{lk} \rangle| \\
& \le c\sup_{l,k} w_l^{-1}  |h_{lk}| \sum_{l} w_l 2^{-l/2}  \le C\|h\|_{\mathcal M_0},
\end{align*}
where we have used $\sup_{t \in [0,1]} \sum_k |\langle 1_{[0,t]}, \psi_{lk} \rangle| \le c2^{-l/2}$, using the localisation properties of the wavelet basis (see e.g. Lemma 3 in \cite{ginenickl09}). Also, $L$ coincides with the primitive map on any function $h$ with wavelet coefficients $\{h_{lk}\}\in \ell_2$, since then $$L_t(\{h_{lk}\}) = \sum_{l,k} h_{lk} \langle 1_{[0,t]}, \psi_{lk} \rangle = \psg h,1_{[0,t]} \psd = \int_0^t h(x)dx$$ in view of Parseval's identity. Moreover, if $\mathbb G$ is a tight Gaussian random variable in $\cM_0$ then the linear transformation $L(\mathbb G)$ is a tight Gaussian random variable in $C([0,1])$, equal in law to  a $P$-Brownian bridge for our choice $\mathbb G = \mathbb G_P$, after checking the identity of the corresponding reproducing kernel Hilbert spaces (and using again that $L$ coincides with the primitive map on $L^2$). The result now follows from the continuous mapping theorem. 
\end{proof}
\vp

\ti{Application: confidence bands for the CDF $F$} 
A natural credible band for $F$ is to take $C_n, R_n$ such that, with $L$ the map defined in \eqref{defl},
\begin{equation} \label{band0}
C_n = \left\{F: \|F-\mathbb T_n\|_\infty 
\le R_n/\sqrt n \right\}, ~~~\Pi\circ L^{-1}(C_n|X)=1-\alpha.
\end{equation}
The proof of the following result implies in particular that $C_n$ asymptotically coincides with the usual Kolmogorov-Smirnov confidence band. The result is true also with centring $\mathbb T_n = F_n$ (in which case the proof requires minor modifications).
\begin{corollary}
Under the conditions of Theorem \ref{kolmsmir}, let $X \sim P_{f_0}, F_0= \int_0^\cdot f_0(t)dt$ and $C_n$ as in (\ref{band0}). Then we have, as $n \to \infty$, $$P_{f_0}(F_0 \in C_n) \to 1-\alpha,~~\text{and}~~R_n \to^{P_{f_0}} const.$$  
\end{corollary}
\begin{proof}
The proof is similar to that of Theorem \ref{sct}, replacing $H(\delta)$ by $C([0,1])$ (a separable Banach space): The function $\Phi$ in that proof is strictly increasing: to check this it suffices to note that any shell $\{g\in C([0,1]):\ s<\|g\|_\infty<t\}$, $0\le s<t$, contains an element of  the RKHS  of the $P_0$-Brownian bridge. Using also Theorem \ref{CLT} in the sampling model case all arguments from the proof of Theorem \ref{sct} go through. 
\end{proof}

\section{Posterior rates in the supremum norm: examples} \label{sec:sn}

Let us now discuss  the case of the supremum-norm, which is not in general a testing distance. 
One generic way of proceeding is to use a `multiscale' approach as outlined above for $\cM_0$--norms. One bounds the $\|\cdot\|_\infty$--norm from above in terms of wavelet coefficients $f_{lk}=\psg f,\psi_{lk}\psd_2$, and then control these individually and in an uniform way. In principle this method can be applied as soon as one can control (the Laplace transform of) posteriors over linear functionals, which we have studied in Chapter \ref{chap:bvm1}. 

Let us give now a very simple example and proof where the posterior Laplace transform can be evaluated explicitly, but this can be replaced by the tools from Chapter \ref{chap:bvm1}. Below we discuss other existing results. 

We assume that $f_0$ is $\al$--H\"older-regular in the sense of \eqref{tolkf}, with here $R,\al$ assumed {\it known for simplicity}.   
Define a prior $\Pi$ on $f$ via an independent product prior on its coordinates $f_{lk}$ onto the considered basis. The  component  $f_{lk}$ is assumed
to be sampled from a prior with density $\sil^{-1}\varphi(\cdot/\sil)$ with respect to Lebesgue measure on $[0,1]$, where, for $\al, R$ as in \eqref{tolkf}, $x\in\RR$ and a given $B>R$
\begin{equation} \label{prior-unif}
 \vphi(x) = \frac{1}{2B} \ind{[-B,B]}(x),\qquad \sil=2^{-l(\frac12+\al)}.
\end{equation} 
This gives a simple example of a random $f$ with bounded $\alpha$-H\"{o}lder norm (see also \cite{gn11}, Section 2.2)
\begin{prop} \label{prop-unif} 
Consider observations $X$ from the Gaussian white noise model.
 Let  $f_0$ and $\alpha$ satisfy Condition \ref{tolkf}
 and let the prior be chosen according to \eqref{prior-unif}. Then there exists $M>0$ such that for $\veps_{n,\al}^*=(\log{n}/n)^{\frac{\al}{2\al+1}}$,
\[ E_0 \int \|f-f_0\|_{\infty} d\Pi(f\given X^{(n)}) = O\left( \veps_{n,\al}^* \right). \]
\end{prop}
Uniform wavelet priors thus lead to the minimax rate of convergence in sup-norm. \\
 
 \begin{proof}
Let $L_n$ be defined by $2^{L_n}=(n/\log{n})^{\frac{1}{2\al+1}}$. Denote by $f^{L_n}$ the orthogonal projection  of $f$ 
 in $L^2[0,1]$ onto $\text{Vect}\{\psi_{lk},\ l\le L_n,\ 0\le k<2^l\}$, and $f^{L_n^c}$ the projection of $f$ onto 
$\text{Vect}\{\psi_{lk},\ l> L_n,\ 0\le k<2^l\}$. Then
\[ f-f_0 = f^{L_n} - \hat{f}^{L_n} + \hat{f}^{L_n} - f_0^{L_n} + f^{L_n^c} -f_0^{L_n^c}, \]
where $\hat{f}^{L_n}$ is the projection estimator onto the basis $\{\psi_{lk}\}$ with cut-off $L_n$. Note that the previous equality as such is an equality in $L^2$. However, if the wavelet series of $f$ into the basis $\{\psi_{lk}\}$ is absolutely convergent $\Pi$-almost surely (which is the case for all priors considered in this paper),  we  also have  $f(x)=f^{L_n}(x)+f^{L_n^c}(x)$ pointwise for Lebesgue-almost every $x$, $\Pi$-almost surely, and similarly for $f_0$. Now, 
\begin{align*}
& E^\Pi[ \|f-f_0\|_{\infty} \given X^{(n)} ]  = \int \|f-f_0\|_{\infty} d\Pi(f\given X^{(n)}) &\\ 
& \le  \underbrace{\int \|f^{L_n} -\hat{f}^{L_n} \|_{\infty} 
d\Pi(f\given X^{(n)})}_{(i)} + \underbrace{\int  \|f^{L_n^c}\|_{\infty} d\Pi(f\given X^{(n)})}_{(ii)} 
+  \underbrace{\|\hat{f}^{L_n}-f_0\|_{\infty}}_{(iii)}. & 
\end{align*}
We have (iii)$\ \le \|f_0^{L_n^c}\|_{\infty}+ \|\hat{f}^{L_n} - f_0^{L_n}\|_{\infty}$. Using \eqref{tolkf} and the localisation property of the wavelet basis $\| \sum_k |\psi_{lk}|\|_\infty \leqa 2^{l/2}$, see below, one obtains
\[ \|f_0^{L_n^c}\|_{\infty} 
\le   \sum_{l>L_n} \left[ \max_{k}|f_{0,lk}|    \| \sum_{k} |\psi_{lk}| \|_{\infty} \right]
\leqa h_n^{\al} \leqa \veps_{n,\al}^*,
  \]
where $\leqa$ means less or equal to up to some universal constant. The term $\|\hat{f}^{L_n} - f_0^{L_n}\|_{\infty}$  depends on the randomness of the observations only,
\[  \|\hat{f}^{L_n} - f_0^{L_n}\|_{\infty}  
= \frac{1}{\rn} \| \sum_{l\le L_n,\, k} \veps_{lk} \psi_{lk}(\cdot)\|_{\infty}.  \]
Using that the expectation of the maximum of order $2^{L_n}$ standard Gaussian variables is at most $C\sqrt{L_n}$ and that $\|\psi_{lk}\|_\infty\leqa 2^{l/2}$, the last display  is bounded under $E_0$ by a constant times $\veps_{n,\al}^*$.

{\it Term (i).} By definition $\hat f^{L_n}$ has coordinates $\hat f_{lk}$ in the basis $\{\psi_{lk}\}$,  so using the localisation property of the wavelet basis as above, one obtains
\[ \|f^{L_n} - \hat{f}^{L_n} \|_{\infty} \leqa \frac{1}{\rn} \sum_{l\le L_n} 
2^{l/2}\left[ \max_{0\le k < 2^l} \rn|f_{lk} - x_{lk}| \right].  \]
For $t>0$, via Jensen's inequality and bounding the maximum by the sum, using $\Pi_n$ as a shorthand notation for the posterior $\Pi[\cdot\given X^{(n)}]$,
\[ t E_0 E^{\Pi_n}[\max_{0\le k < 2^l} \rn|f_{lk} - x_{lk}|]
\le \log\sum_{k=0}^{2^l-1}  E_0 E^{\Pi_n}\left[e^{t \rn (f_{lk} - x_{lk})} +
e^{-t \rn (f_{lk} - x_{lk})} \right], \]
for any $l\ge 0$.  

Simple computations similar to the proof of Theorem \ref{thm-mins} (indices in $\cJ_n$) 
 yield a sub-Gaussian behaviour for the  Laplace transform of $\rn(f_{lk}-x_{lk})$ under the posterior distribution, which is bounded above by $C e^{t^2/2}$ for a constant $C$ independent of $l\le L_n$ and $k$. 
From this deduce, for any $t>0$ and $l\le L_n$,
\[ E_0 E^{\Pi}[\max_{0\le k< 2^l} \rn|f_{lk} - x_{lk}|\given X^{(n)}]
\leqa \frac{\log(C2^l)}{t} + \frac t 2. \]
The choice $t=\sqrt{2\log(C2^l)}$ leads us to the bound  
\begin{align*}
E_0 (i) & \leqa \frac 1 \rn \sum_{l\le L_n} \sqrt{l} 2^{l/2}
 \leqa \sqrt{L_n/(nh_n)} \leqa \veps_{n,\al}^*.
\end{align*}
{\it Term (ii).}  Under the considered  prior, the wavelet coefficients of $f$ are bounded by $\sil$, so using again the localisation property of the wavelet basis, 
\begin{align*}
 E_0 (ii) 
 &  \leqa \sum_{l >L_n} 2^{l/2} E_0 E^{\Pi}\left[\ \max_{k} |f_{lk}| \  \given X \right]
  \leqa   \sum_{l >L_n} 2^{l/2}\sigma_l \leqa h_n^{\al} = \veps_{n,\al}^*.
\end{align*} 
\end{proof}
\vp

{\em Further results.} Convergence in the supremum norm can be obtained using the previous techniques for a variety of models -- so far for histogram or wavelet series priors -- including regression, density estimation, survival analysis, the Cox model and diffusion models.\\

We note also the existence of results for specific models and/or priors, mostly using (partial) conjugacy or specific properties of the prior or model:
\begin{itemize}
\item spike--and--slab priors achieve supremum-norm rates adaptive to the regularity in regression models \cite{hrs15} and density estimation by exponentiation and normalisation \cite{zn22};
\item tree--based priors in the spirit of Bayesian CART achieve adaptive supremum-norm rates in regression models \cite{hrs15} up to a logarithmic factor;
\item it is possible to construct counterparts of the previous results in density estimation (at least for H\"older regularities between $0$ and $1$), without going through renormalisation, by considering a P\'olya tree structure \cite{cm21, cr22};
\end{itemize}

\chapter{Classification and multiple testing} \label{chap:mtc}

In this chapter we consider the simple sparse sequence model
\begin{equation} \label{mod.spa7}
 X_i = \theta_i + \veps_i,\quad i=1,\ldots,n,
\end{equation}
with $\veps_i$ iid $\cN(0,1)$ and an unknown vector of means $\te=(\te_1,\ldots,\te_n)$.  

One assumes that the true $\theta_0$ is {\em sparse} in that it belongs to the class of {\em nearly black vectors}
\begin{equation}
 \ell_0[s_n] = \left\{ \theta\in\RR^n:\ \text{Card}\{i: \te_i  \neq 0 \} \le s_n \right\},
\end{equation}
where the sparsity parameter $s_n$ is typically assumed to verify  $s_n=o(n)$ and $s_n\to\infty$ as $n\ra\infty$. It can be seen as a `signal+noise' model with a sparsity constraint on $\te$.\\

Classical well-studied problems for inference over the unknown $\te$ in this model are those of {\em estimation} (where one wishes to find estimators $T=T(X)$ making the quadratic risk $E_{\te}\| T-\te\|^2$ small), {\em signal detection} (where one wishes to test the presence of non-zero signal, i.e. $H_0:\, \te=0$ against $H_1:\, \te$ non-zero with some structure), {\em confidence sets} among others, see the survey paper \cite{bcg21} and references therein.\\

Here we consider the `multiple testing' and `classification' problems, to be defined just below. Roughly speaking the risks for these problems measure how well a procedure does in terms of number or proportion of errors, that is e.g. how many coordinates corresponding to zero signals are incorrectly labelled as `signal' (this is called `type I' error), and vice-versa how many coordinates corresponding to signals are labelled as `noise' (called the `type II' error). Although we focus on standard Gaussian noise for simplicity, the results below can be adapted much beyond that, the main condition being that the tail of the noise should be lighter than exponential (i.e. excluding Laplace noise for example). This chapter is based on a series of recent results on Bayesian multiple testing, mainly \cite{acr23} (see also \cite{cr20, acr22, acg22}).

\section{General principles: sequence model and decision theory}

A multiple testing or classification procedure is a $\vphi=(\vphi_i(X))_{1\le i\le n}\in\{0,1\}^n$ measurable with
\begin{center}
$\vphi_i=1\ $ if and only if the null hypothesis $H_{0i}$ is rejected.\\
\end{center}

{\em The classification problem.} The classification loss is defined as 
\begin{align*}
 L_C(\te,\vphi) & =  \sum_{i=1}^n \ind{\te_{i}=0,\vphi_i=1} 
+  \sum_{i=1}^n \ind{\te_{i}\neq 0,\vphi_i=0} \\
& =: \qquad \ \ \sbl{N_{FP}(\te,\vphi) \qquad+\qquad N_{FN}(\te,\vphi)},
\end{align*}
where $N_{FP}(\te,\vphi)$ (resp. $N_{FN}(\te,\vphi)$) is the number of false positives (resp. false negatives) of a procedure $\vphi$ at $\te$. The corresponding {\em classification risk} is
\[ R_C(\te,\vphi) = E_\te \sbl{N_{FP}(\te,\vphi)} + E_\te \sbl{N_{FN}(\te,\vphi)}.
\]

\noindent {\em The multiple testing problem.} Suppose one wants to simultaneously test 
\[ H_{0i}:\ \te_i=0\qquad \text{against}\qquad H_{1i}:\ \te_i\neq 0,\qquad \text{for all }
1\leq i \leq n. \]
The \sbl{False Discovery Rate (FDR)} of $\sbl{\vphi}$ at vector $\te$ is defined as
\[ \FDR(\te,\vphi) = E_{\te}\left[ \frac{\sum_{i=1}^n \ind{\te_{i}=0,\vphi_i=1}}{1\vee  \sum_{i=1}^n \ind{\vphi_i= 1}} \right] 
\sbl{=E_{\te}\left[\frac{\text{nb. false discoveries}}{\text{nb. discoveries}}\right]}\]
The \sbl{False Negative Rate (FNR)} of $\sbl{\vphi}$ at $\te$ is defined as
\[ 
\FNR(\te,\vphi) 
= E_{\te}\left[\frac{\sum_{i=1}^n \ind{\te_{i}\neq 0} (1-\vphi_i(X))}{1\vee \sum_{i=1}^n \ind{\te_{i}\neq 0}}\right] \sbl{=E_{\te}\left[\frac{\text{nb. false negatives}}{\text{nb. non--zeros}}\right]}
\]
The overall {\em multiple testing risk}, denoted by $\fr$, is 
\[ \fr(\te,\vphi) :=  \FDR(\te,\vphi)+  \FNR(\te,\vphi). \]

If, as we do below for simplicity, one restricts to signals in $\ell_0[s_n]$ that have exactly $s_n$ non-zero coordinates, the type II error for the classification risk is just $s_n$ times $E_\te \sbl{N_{FN}(\te,\vphi)}$. For the type I error, the comparison is less clear as, if the number of discoveries is of order e.g. $2s_n$, the ratio inside the expectation defining the FDR can differ from $N_{FP}(\te,\vphi)/s_n$ by a factor $2$. \\

{\em Goal(s).} From the multiple testing point of view, one would like to understand for which signals $\te$ the $\cR$ risk can be made either to vanish (i.e. that there exists $\vphi$ such that $\fr(\te,\vphi)$ goes to zero for any of these signals), or to be kept under a certain, typically small, level $t$, say e.g. $t=0.1$ or $t=0.05$. Similarly, for classification, one would like to make the corresponding $R_C$ risk be a $o(s_n)$, that is find classes of signals and procedures for which one missclassifies at most a small number (relative to $s_n$) of the $n$ coordinates of $\te$.

\subsubsection{Natural questions}

To make the goals a little more precise, we formulate some associated questions. \\

{\em Question 1 (under beta-min conditions).} Suppose the non-zero coordinates of $\te_0$ (the ``signals") are all above in absolute value above a certain threshold $M$. What is the `phase-transition' in terms of $M$ from impossibility to possibility of multiple testing or classification?\\

{\em Question 2 (arbitrary sparse signals).} In fact one would like to say more, that is, for arbitrary values of signals, what is the intrinsic difficulty of the multiple problem, in terms for instance of the best possible $\cR$--risk achievable? For instance, are the different signals on Figure \ref{fig:mtconfig} of similar difficulty? (the answer is no except for very special values)\\
\begin{figure}[h!] \label{fig:mtconfig}
\begin{center}
\includegraphics[scale=.7]{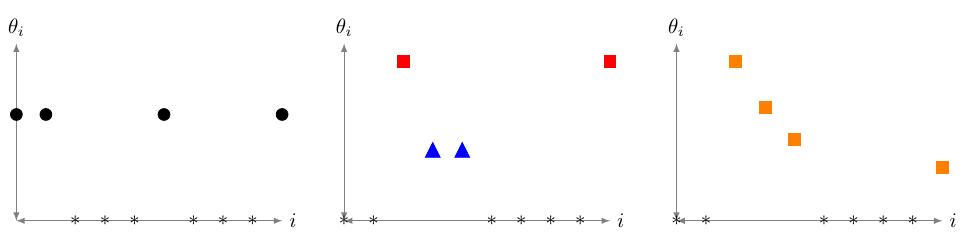}
\end{center}
\caption{Different signal configurations}
\end{figure}

{\em Question 3 (trade--off between type I and type II error?)} Since it is customary to allow for a small testing error, say $t$ (small), can one compensate errors of type I by errors of type II and vice-versa? \\

{\em Question 4 (large signals)} If all (or most) of the signals are very large, the $\fr$ risk will go to zero; what is the optimal rate in this case, and can it be achieved?

\subsubsection{A simple Bayesian procedure}

Putting a prior distribution $\Pi$ on $\te$, such as a spike--and--slab prior as we did in this context in Chapter \ref{chap:ada2}, one can form a posterior distribution $\Pi[\cdot\given X]$ in the sequence model $\cP=\{\otimes_{i=1}^n \cN(\te_i,1),\, \te_i\in\RR\}$. 

\begin{definition} Given a prior  $\Pi$ on $\theta$ and data $X$ from the model, define the \sbl{$\ell$-value} as
 \[ \ell_i(X) = \Pi[ \te_i = 0 \given X]. \]
Let $t\in(0,1)$ be  given. The \sbl{$\ell$--value procedure} at level $t$ is 
\[ \vphi^{\ell}_i = (\vphi^{\ell}_t)_i =\1\{\ell_i(X) \le t\}. \]
\end{definition}

This procedure is natural: it rejects the null if the posterior probability of being zero is small(-er than $t$). If $\Pi=\Pi_\al$ is the spike and slab prior that draws coordinates independently with law $(1-\al)\delta_0+\al \Ga$ for a weight $\al\in[0,1]$ and slab distribution $\Ga$, we have the explicit expression
\[ \ell_i(x) =  \frac{(1-\al)\phi(x)}{(1-\al)\phi(x) + \al g(x)}, \]
where $g$ is the convolution $g=\ga*\phi$, for $\ga$ density of $\Ga$.

\subsubsection{Basic decision--theoretic facts}

\begin{prop} \label{prop:bayes}
Let $\Pi$ be a fixed given prior on $\te$. 

The Bayesian classification risk over procedures $\vphi$
\[ R_{C,B}(\Pi,\vphi) := \int R_C(\te,\vphi)d\Pi(\te) \]
is minimal for the $\ell$--value procedure at level $1/2$.

Moreover, the $\ell$--value procedure $\vphi_t$ at level $t\in(0,1)$ satisfies
\begin{align*}
\FDR_B(\Pi,\vphi_t):= \int \FDR(\te,\vphi_t)d\Pi(\te)\le t. 
\end{align*}
That is, it controls the Bayesian FDR at level $t$
\end{prop} 
\begin{proof}
Let us note, for $E$ the expectation in the Bayesian model, using the chain rule on 
 \[ R_{C,B}(\Pi,\vphi) = EE[ L_C(\te,\vphi)\given \te] = E[ L_C(\te,\vphi)] =E E[L_C(\te,\vphi)\given X].   \]
Now $E[L_C(\te,\vphi)\given X]$ can be written as
\[ E[L_C(\te,\vphi)\given X] = 
 \sum_{i=1}^n \Big\{ \ell_i(X) \ind{\vphi_i=1} +  (1-\ell_i(X)) \ind{\vphi_i=0} \Big\}
 = C_X + \sum_{i=1}^n  (2\ell_i(X)-1) \ind{\vphi_i=1},
 \]
where $C_X$ is independent of $\vphi$.  The last display is minimal for $\vphi_i=\ind{2\ell_i(X)-1\le 0}$, that is for the $\ell$--value procedure at level $1/2$. 
 
On the other hand, using again the chain rule on conditional expectations, and writing $\vphi=\vphi^t$ 
\begin{align*}
\FDR_B(\Pi,\vphi) & 
= E \left[  E \bigg[ \, \frac{\sum_{i=1}^n \1_{\te_i= 0} \1_{\vphi_i(X)=1}}{1\vee \sum_{i=1}^n  \1_{\vphi_i(X)=1}}\,\given\, X\, \bigg] \right]
 =E_{X} \left[\frac{\sum_{i=1}^n \ell_i(X) \1_{\{\ell_i(X)\leq t\}}}{1\vee \sum_{i=1}^n  \1_{\{\ell_i(X)\leq t\}}}\right] \\
 & \leq t \:P(\exists i\::\: \ell_i(X)\leq t) \le t.
\end{align*}
\end{proof}

The above basic facts say that in an ideal setting where the prior were true,  $\vphi^\ell_{1/2}$ would be optimal from the point of view of the Bayesian classification risk. The latter depends on $\Pi$ and is an integrated risk. For the frequentist perspective, one would like to control the usual {\em pointwise} risk (say $R_C(\te,\vphi)$ in classification) at a given sparse $\te$. Somewhat surprisingly perhaps, the previous simple procedure gives an (asymptotically) optimal answer also in the frequentist sense for spike--and--slab priors, provided a proper choice of the weight $\al$ is made: here we consider a data--dependent choice, i.e. an empirical Bayes approach, which we already introduced in Chapter \ref{chap:ada1}.\\

{\em Remark.} It is also possible to build procedures that control the Bayesian FDR very close to $t$ \cite{acr22} (not only below $t$), such as the $q$ value--procedure  $q_i(X')=\Pi[\te_i=0\given |X|\ge |X'|]$ (in fact, the famous Benjamini--Hochberg multiple testing procedure is an empirical Bayes $q$--value procedure \cite{storey03}).

\subsubsection{A data-driven Bayesian procedure}

Let us recall the empirical Bayes spike and slab procedure: one simply choses the spike-and-slab posterior with $\al$ replaced by $\hat\al$ chosen by maximising the marginal likelihood.

 \bfr
\sbl{Bayesian multiple testing procedure $\vphi^\ell$ $\ $ \pgr{``EB $\ell$-value procedure"}} 
\begin{enumerate}
\item Spike and slab prior: $\al\in[0,1],\ \ \Ga$ law of density $\ga$ on $\RR$ 
    \begin{center}
 $ \Pi_\al \, \sim \, \displaystyle\bigotimes_{i=1}^n\, (1-\al)\delta_0 + \al \Ga$\\
  \end{center} 
\item `Estimate' $\al$ by $\hat{\al}$ the MMLE, where $g=\ga*\phi$,
 \[ \hat \al = \argmax{\al\in [1/n,1]}\ \,
 \prod_{i=1}^n \ (1-\al)\phi(X_i) + \al g(X_i)  \qquad \]
\item Define $\ell_i(X) = \Pi_{\hat \al}[ \te_i = 0 \given X]$ and set for fixed $t$ (e.g. $t=1/2$),
\[ \vphi^\ell_i(X) = \ind{\ell_i(X)\le t}. \]
\end{enumerate}
\efr
 
Empirical Bayes procedures  have been advocated for multiple testing among others by Bradley Efron \cite{efron07, efron08} (Efron even suggests taking a data--driven slab $\hat\Ga$, which we do not do here), without a theoretical analysis though. One can prove that the EB $\ell$--value procedure controls the (pointwise, i.e. frequentist) FDR at any sparse $\te_0$ (\cite{cr20}, just assuming polynomial sparsity $s_n=O(n^{1-\be})$ for some $\be\in(0,1)$).

\section{Sharp classification and multiple testing in sequence model}

In order to address Question 1 above, let us introduce the set, for a real $b$,
\[ \cL_0[s_n;b]= \bigg\{ \theta\in \ell_0[s_n] \::\: |\theta_{i}| \geq \: \sqrt{2\log \frac{n}{s_n}}  + b\ \text{ for all } i\in S_\te,\ |S_\te|=s_n \bigg\} \]
We wish to investigate optimality for the multiple testing risk $\mathfrak{R}(\te,\vphi) = \FDR(\te,\vphi) + \FNR(\te,\vphi)$. 
\vspace{.05cm}

\begin{thm}[minimax rate]  \label{thm:minmax}
Let $\pgr{b}\in[-\infty,+\infty]$. As $n, s_n,n/s_n\to\infty$,   
\vp
\[ \rho_{\pgr{b}} := \inf_{\vphi} \sup_{\te\in \cL_0[s_n;b]} 
 \mathfrak{R}(\te,\vphi)  = \bar\Phi(\pgr{b})+o(1). \]
\begin{itemize}
\item[(i)] if $\pgr{b}=b_n\to -\infty$, then $\rho_{\pgr{b}}\to 1$. Multiple testing is impossible for any procedure.
\item[(ii)] if $\pgr{b}=b\in\RR$ fixed, then $\rho_{\pgr{b}}\to \bar\Phi(\pgr{b})$. Multiple testing is partly possible.
\item[(iii)] if $\pgr{b}=b_n\to +\infty$, then $\rho_{\pgr{b}}\to 0$. Multiple testing is possible with $o(1)$ error.
\end{itemize}
The same result holds for the normalised classification risk $R_C(\te,\vphi)/s_n$.
\end{thm}
This result addresses Question 1 if the parameters $b,s_n$ are given (the `non-adaptive' case). To prove it, one shows first that the minimax risk is at most $\bar\Phi(\pgr{b})+o(1)$ by exhibiting a procedure that achieves this bound, namely the {\em oracle} procedure 
\[ \vphi_i^*=\ind{\,|X_i|>\sqrt{2\log(n/s_n)}\,}. \] 
Note that $\vphi^*$ depends on $s_n$, which in general is unknown to the statistician. 
The lower bound is based on bounding the minimax risk by a well-chosen Bayes risk, see Section \ref{sec:proofmt} for proofs. 

The following result, again for the $\ell$--value procedure $\vphi^\ell$, addresses adaptation: it shows that the EB $\ell$--value procedure above is asymptotically minimax adaptive, i.e. it achieves the bound without using knowledge of the class parameters $s_n,b$. 

\begin{thm}[Adaptation]
\begin{itemize}
\item The EB \sbl{$\ell$--value procedure}  achieves boundary \sbl{(iii)} i.e. for arbitrary $\pgr{b}\to +\infty$
\[  \sup_{\te\in \cL_0[s_n;b]} 
 \mathfrak{R}(\te,\vphi^\ell) = o(1). \]
\item The EB \sbl{$\ell$--value  procedure} is sharp minimax adaptive \sbl{(ii)}: for $t\in(0,1)$, $\pgr{b}\in\RR$,
\[ \sup_{\te\in \cL_0[s_n;b]} 
 \mathfrak{R}(\te,\vphi^\ell)  \to \bar\Phi(\pgr{b}). \]
\end{itemize}
\end{thm}

In earlier work, \cite{butucea18} using a Lespki--type method showed for classification that boundary \sbl{(iii)} can be achieved adaptively if $b\geqa \log\log^{1/2}(n/s_n)$. The next result covers both boundaries without condition on $b$. 

\begin{thm}
For any real $b$ (or $b=b_n\to+\infty$, with $\bar\Phi(b)+o(1)$ replaced by $o(1)$ below),
\[ \sup_{\te\in \cL_0[s_n;b]} E_\te\left[L_C(\te,\vphi^{\ell})/s_n\right] \le \bar\Phi(b)+o(1). \]
So the $\ell$-value procedure $\vphi^\ell$ achieves boundaries \sbl{(ii)} and \sbl{(iii)} sharply.
\end{thm}

We now turn to Question 3 and determine whether there can be a trade--off between the two types of risks  in terms of the optimal minimax constant.

\begin{definition} \label{def:spap}
A procedure $\vphi=\vphi(X)\in\{0,1\}^n$ is said to be \pgr{sparsity-preserving} up to a multiplicative factor $A_n$ if
\[ \sup_{\te\in \cL_0[s_n;b]} P_{\te}\left[\sum_{i=1}^n \vphi_i(X) > A_n s_n \right]=o(1)\]
\end{definition}

Many procedures can be shown to be sparsity-preserving for $A_n\uparrow\infty$ slowly or even $A_n=A$ fixed, including the $\vphi^\ell$ value procedure for any fixed $t$, but also most estimation procedures encountered in the literature and that produce sparse estimators (e.g. the LASSO with well--chosen parameter).

\begin{thm}[No trade--off]  \label{thm:notrade}
Let $\cC_{A_n}[s_n]$ denote the class of {\em sparsity preserving} $\vphi$'s up to a multiplicative factor $A_n$. Suppose $(\log A_n)\le (\log{n/s_n})^{1/4}$. Then for any real $b$, as $n\to\infty$, 
\[ \inf_{\vphi\in \cC_{A_n}[s_n]} \sup_{\te\in \cL_0[s_n;b]} \FNR(\te,\vphi) =   \bar\Phi(\pgr{b})+o(1). \]
\end{thm}

The interpretation of Theorem \ref{thm:notrade} is that for most procedures, one cannot trade--off a part of  type I error for less type II error: any sparsity preserving procedure as above must have its type II error rate (FNR) at least equal to $\bar\Phi(\pgr{b})$ over the class $\cL_0[s_n;b]$. 

As an important consequence, any sparsity-preserving procedure that controls the FDR at level say $\al'$ close to $\al$ (this is the case for the popular BH--procedure at fixed level $\al$, which has an FDR equal to $c_n\al$ for $c_n\in(0,1)$ bounded away from $0$) must have an $\cR$-risk of at least $\al'+\bar\Phi(\pgr{b})>\bar\Phi(\pgr{b})$ over the class $\cL_0[s_n;b]$, and hence be suboptimal in terms of the $\cR$-risk. Similar results also hold for the classification loss, we omit the detailed statement. \\

While we have now answered Questions 1 and 3 above, a somewhat unsatisfactory aspect of the formulation of the results for now is that they hold only under a quite strong condition on signals, i.e. all signals must be above 
$\sqrt{2\log(n/s_n)}+b$. Recalling the three types of signals from Figure \ref{fig:mtconfig},  intuitively results above should be sharp only for the signal on the left, and there should be room for improvement for the middle and right signals. We state now a result for {\em arbitrary} signals.\\

For the formulation of the result, we need a more precise class. Define, for $\bb=(b_1,\ldots, b_{s_n})$ a vector of real numbers such that $\sqrt{2 \log(n/s_n)} +b_j>0$ for all $j$, the class
\begin{align*} 
 \Theta_{\bb}  =
\Big\{ \theta\in \ell_0[s_n] \::\: 
 \exists\,  & 
 i_1, \ldots,  i_{s_n} \mbox{ all distinct, } \
 |\theta_{i_j}| \ge \sqrt{2 \log(n/s_n)} +b_j>0
 \Big\}.
\end{align*}
Note that the union of all possible classes $ \Theta_{\bb} $ when $\bb$ varies gives the set of sparse vectors with exactly $s_n$ non-zero coordinates. Further denote
\begin{equation} \label{lamdab}
\Lambda_n (\bb)=s_n^{-1} \sum_{j =1 }^{s_n} \bar\Phi\left(b_j \right).
\end{equation}

\begin{thm}[Sharp minimaxity and adaptation for arbitrary signals] \label{thm:mtarbitrary}
For any possibly $n$--dependent vector $\bb$, and $\Theta_{\bb}, \La_n(\bb)$ as above, as $n\to\infty$,
\[  \inf_\vphi \sup_{\te\in\Theta_{\bb}}  \mathfrak{R} (\theta,\vphi)= \La_n(\bb) + o(1). \]
This bound is achieved by the EB $\ell$--value procedure (under polynomial sparsity if $\limsup_{n}\La_n(\bb)=1$).
\end{thm}

\begin{figure}[h!]
\begin{center}
\begin{tabular}{cc}
\hspace{-5mm}\includegraphics[scale=0.65]{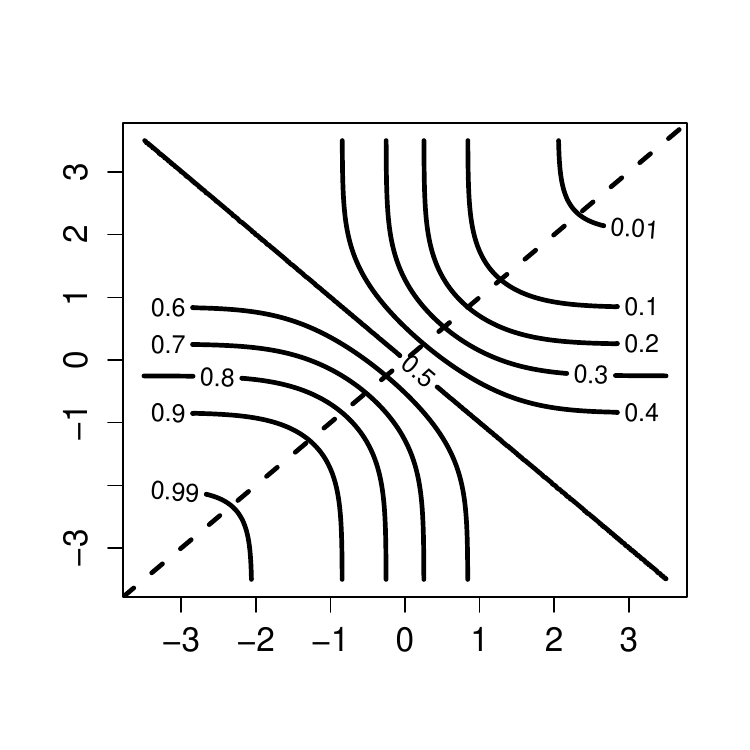}&\hspace{-1cm}\includegraphics[scale=0.65]{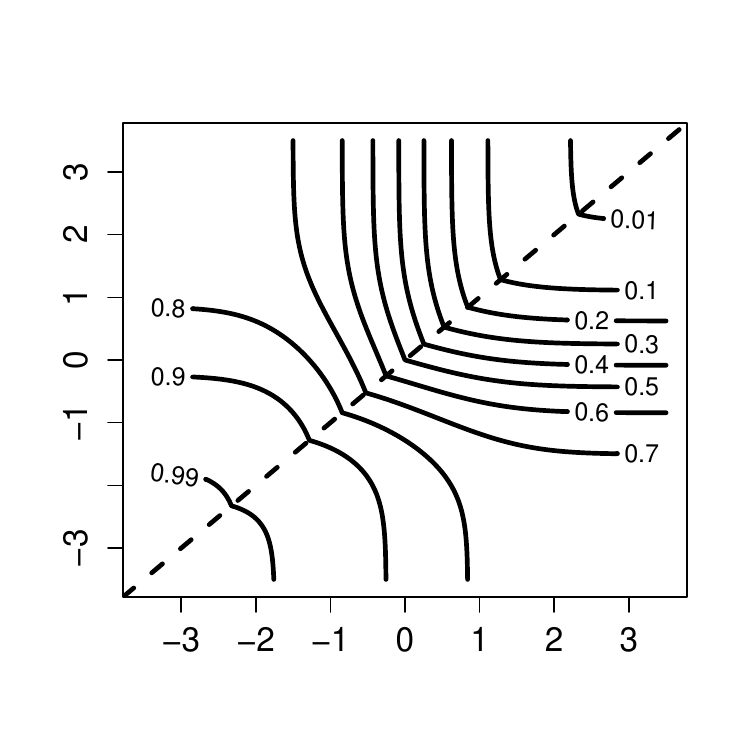}
\vspace{-1cm}
\end{tabular}
\caption{Level sets of $\Lambda_\infty$ in example of two signal strengths, with $\lfloor s_n q\rfloor$ of the $\te_i$'s equal to $\sqrt{2\log(n/s_n)}+\max(x,y)$ and others equal to $\sqrt{2\log(n/s_n)}+\min(x,y)$. Left: $q=1/2$; right: $q=1/4$. \label{fig:generalboundary}} 
\end{center}
\end{figure}

\begin{figure}[h!]
\begin{center} 
\hspace{-10mm}
\includegraphics[scale=0.45]{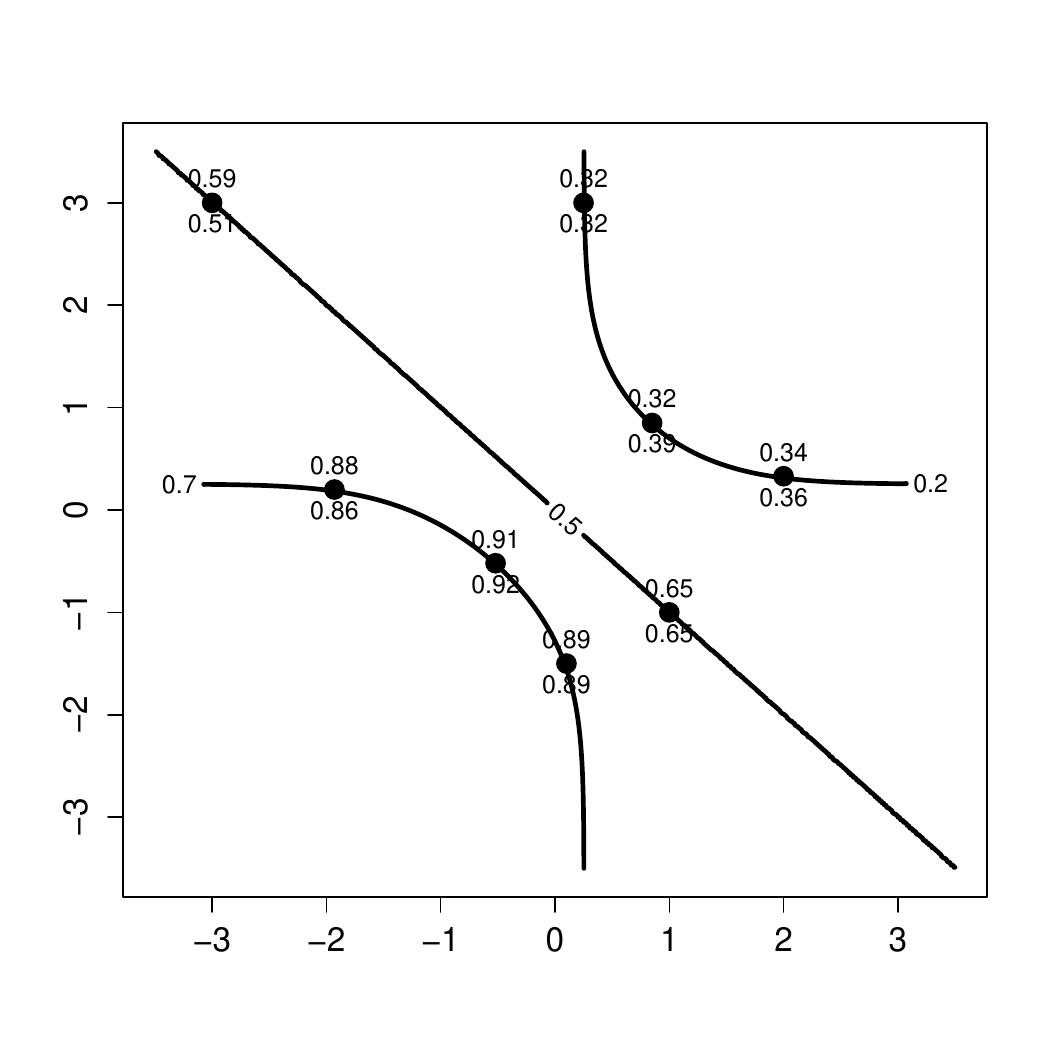}
\hspace{-10mm}
\vspace{-1cm}
\includegraphics[scale=0.45]{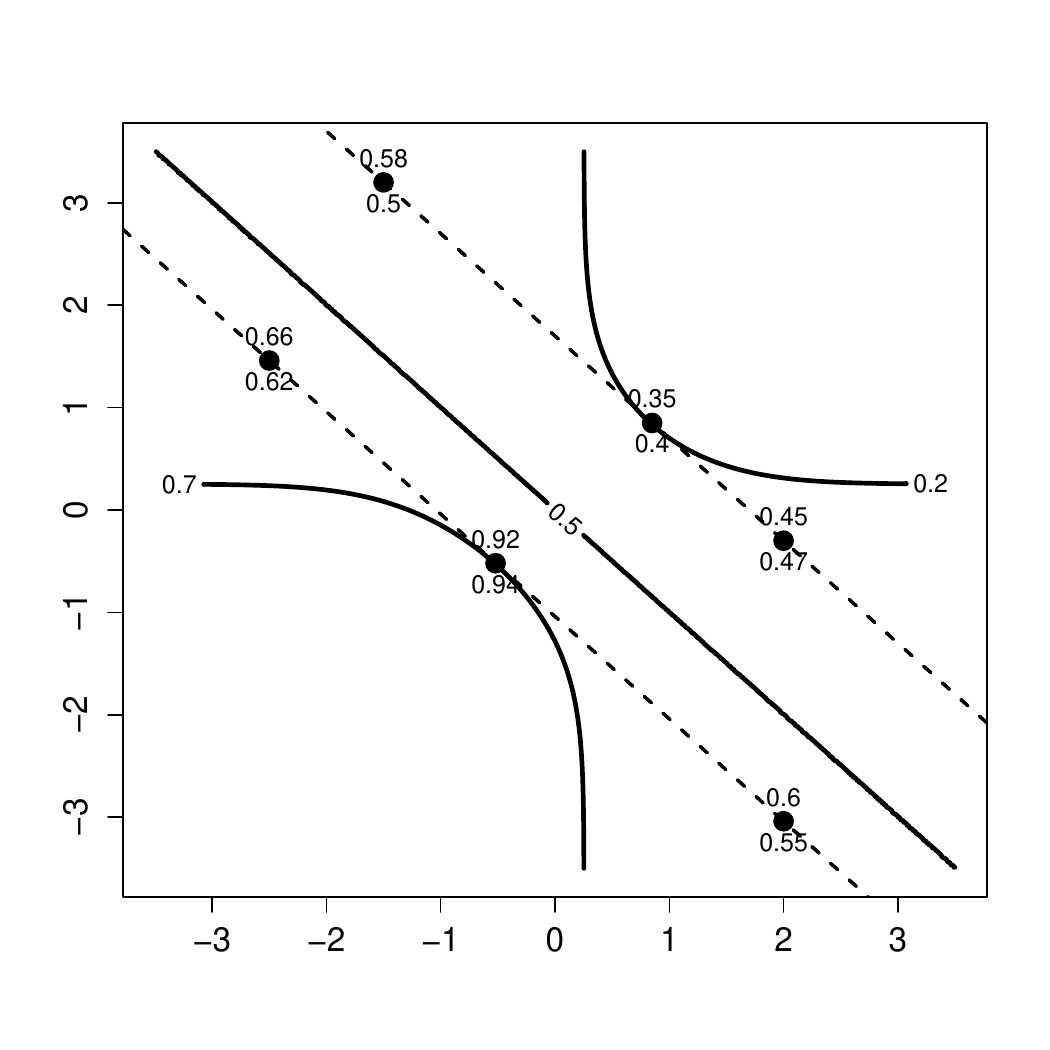}
\caption{Same as Figure~\ref{fig:generalboundary} (left, case $q=1/2$) with finite sample risk ($n=10^6$, $s_n=20$) at some particular configurations (displayed by black dots). The risk is computed via $100$ Monte-Carlo simulations. Two procedures are implemented:  the EB $\ell$-value procedure ($t=0.3$, risk displayed below each dot) and the BH procedure ($\alpha=0.1$, risk displayed above each dot). Left: dots are located on the $\Lambda_\infty$-level sets of values in $\{0.7,0.5,0.2\}$ of the boundary function. Right: dots are located on lines such that the average of  $x$ and $y$ is kept constant (not $\Lambda_\infty$-level sets). \label{fig:generalboundarywithpoints}}
\end{center}
\end{figure}

\noi Theorem \ref{thm:mtarbitrary} provides an answer to Question 2. We give three applications to specific signals.\\
 
\noindent \sbl{Example [Single signal strength]} Suppose $\bb=(b,\dots,b)$ then 
$\Theta_{\bb}=\cL_0[s_n;b]$ and  
\[ \Lambda_n(\bb)= \overline{\Phi}(b). \]

\noindent \sbl{Example [Two signals]}	
Let $x,y\in \RR^2$. 
 Define $\bb\in \RR^{s_n}$ by,  for   $q\in (0,1)$,
 \[ b_j=
 \begin{cases}
 M=x \vee y, & \quad 1\leq j\leq \lfloor s_n  q\rfloor, \\
 m= x \wedge y, & \quad \lfloor s_n q\rfloor < j\leq s_n. 
 \end{cases}
 \]
 Then as $n\to\infty$ we have
\begin{align*}
{\Lambda}_n(\bb) & = \frac{1}{s_n} \sum_{j=1}^{s_n} \overline{F}(b_j)
 = q \overline{\Phi}(M) + (1-q) \overline{\Phi}(m) + o(1).
\end{align*} 
Denoting $\La_\infty(x,y) = q \overline{\Phi}(x \vee y)+(1-q) \overline{\Phi}(x \wedge y)$, level lines of $\La_\infty$ are displayed in Figure \ref{fig:generalboundary}. \\

\noindent \sbl{Example [mixed signals]}
Suppose the nonzero entries of $\theta$ are given by $(Aj/s_n)\sqrt{2\log(n/s_n)}$ for $j=1,\ldots,s_n$ and $A\ge 1$ fixed. This is a special example of the third case in  Figure \ref{fig:mtconfig}. Then 
\[ \La_n(\theta) = A^{-1} + o(1). \]  
(to check this, one can e.g. separate signal coordinates $i$'s in three subsets delimited by $(s_n/A)(1\pm r_n)$ for $r_n=o(1)$ suitably slowly). In this setting, what contributes to $\La_n(\te)$  is the proportion of signals below the optimal threshold $a_n^*=\sqrt{2\log(n/s_n)}$.\\

As a consequence of Theorem \ref{thm:mtarbitrary}, one obtains a {\em testability} condition for an arbitrary sparse vector $\theta$ with $s_n$ non-zero coordinates. Namely, there exists a multiple testing procedure that has a $\fr$--risk of at most $\alpha$ asymptotically if and only if
\[ \varlimsup_{n} \ \frac1{s_n} \sum_{j\in S_\theta} \overline{\Phi}\left(|\theta_j|-\sqrt{2\log(n/s_n)}\right) \le \alpha \]

\vp 
 
The previous theory extends to other noise distributions, for instance to Subbotin noise -- density proportional to $\exp(-|x|^\zeta/\zeta)$--, as well as other forms of the model (e.g. scale instead of location). The following is a set of possible assumptions (see \cite{acr23}). The noise density is denoted by $f$, we set $f_a(\cdot)=f(\cdot-a)$ and $F, F_a$ are the corresponding distribution functions, with $\overline{F}(x)=1-F(x)$. 

\vp

{\em Other noise assumptions.} Consider a noise density $f$ that verifies $f_{-a}(-x)=f_a(x)$ for $a,x\in\RR$. Suppose also that there exists a constant $L$ such that  $F$ is $L$-Lipschitz and that $f_{a}(x)/f_0(x)$  is increasing in $x\in\RR$ for any $a>0$. Finally, suppose there exist sequences of positive numbers $a_n^*\to \infty$, $\delta_n\to 0$ such that
	\begin{align} \label{eqn:modelassumption2}
		(n/s_n)\overline{F}_0\brackets[\big]{a_n^*-\delta_n} \to \infty, \\
		\label{eqn:modelassumption3}
		(n/s_n)\overline{F}_0\brackets[\big]{a_n^*} \to 0.
	\end{align}
The standard Gaussian density is a special case, with $(\phi_a/\phi)(x)=e^{ax-a^2/2}$ and for which one can take, for instance, $a_n^*=\sqrt{2\log(n/s_n)}$ and $\delta_n=(\log(n/s_n))^{-1/4}$. More generally, for a Subbotin density with $\zeta>1$, one may take  $a_n^*= (\zeta \log(n/s_n))^{1/\zeta}$ and $\delta_n = (\log(n/s_n))^{-\kappa}$ for $\kappa\in (0,1-1/\zeta)$.

\section{Adaptation to large signals}

In this section we consider a specific situation when there is a lot of signal:  all non-zero signals are above a certain large value and the multiple testing risk $\fr$ goes to $0$. We then investigate the optimal rates of decrease of the $\fr$ risk. We then only state one further result illustrating that this case is perhaps harder than it looks at first in that there are no adaptive estimators in this setting.

We focus for simplicity on the case of polynomial sparsity: for some $\be\in(0,1)$, and $0<a<b$,
\begin{equation} \label{polspa}
an^{1-\be} \le s_n \le b n^{1-\be},
\end{equation} 
with $a,b$ fixed throughout (they can be known or unknown).
Define, some $r>0$, a number quantifying the  signal strength
\[  M(r)=\sqrt{2r\log{n}}. \]         
Let us define a class of large signals with parameters $r>\be$ and $\be\in(0,1)$,
\begin{align*}
  \Theta(r,\be) &= \bigcup_{s_n\in[an^{1-\be},bn^{1-\be}]}\left\{ \te:\ |S_\te|=s_n,\ \ \ |\te_i|\ge M(r)\ \ \forall\, i\in S_\te \right\}.             
\end{align*} 
The condition $r>\be$ corresponds to the fact that we consider here large signals above the `boundary' regime i.e. signals around level $\sqrt{2\log{n/s_n}}=\sqrt{2\be\log{n}}=M(\be)$. As an immediate consequence of the results in the previous section, the $\cR$ risk goes to $0$ uniformly over $\Theta(r,\be)$. Our goal here is to find the best (in the minimax sense) of decrease to $0$.

Define for $r>\be$
\begin{align}
\sqrt{\kappa(r,\be)} & =(\sqrt{r}-\be/\sqrt{r})/2,  \label{kapparb}\\
\ta(r,\be) & = \left(\sqrt{r}-\sqrt{\kappa(r,\be)}\right)\sqrt{2\log{n}}. \label{taurb}
\end{align}

\begin{thm}[minimax rate, large signals]  
For $\be\in (0,1)$ and $\be<r$, for $\kappa(r,\be)$ as in \eqref{kapparb},
\[ \inf_\vphi \sup_{\te\in\Theta(r,\be) }  \mathfrak{R} (\theta,\vphi) \asymp n^{-\kappa(r,\be)}\] 
This rate is achieved for the thresholding procedure with oracle threshold $\ta(r,\be)$ as in \eqref{taurb}.
\end{thm}

Recall $L_C(\te,\vphi)=N_{FP}(\te,\vphi)+N_{FN}(\te,\vphi)$ and $R_C(\te,\vphi)=E_\te L_C(\te,\vphi)$ is the classification risk. \\

\begin{thm}[minimax rate for classification, large signals]   
Again for $\kappa(r,\be)$ as in \eqref{kapparb},
\[ \inf_\vphi \sup_{\te\in\Theta(r,\be) }  R_C (\theta,\vphi) \asymp s_n n^{-\kappa(r,\be)} \asymp n^{1-\be-\kappa(r,\be)}. \] 
This rate is achieved for the thresholding procedure with oracle threshold $\ta(r,\be)$.
\end{thm}

There are two regimes 
\bi
\item $1-\be-\kappa>0$: an overwhelming portion of signals can be recovered\\
but not every signal: this is called the ``almost sure recovery" regime, 
\item $1-\be-\kappa<0$: the $R_C$ risk goes to zero; it can be shown that support can be recovered with $1-o(1)$ probability, called the ``exact recovery" regime. 
\ei

A natural question is then to find a  procedure $\vphi$ that achieves the rate $n^{-\ka(\be,r)}$ for the $\cR$ risk in an adaptive way with respect to $\be, r$.

\begin{thm} \label{mt:impos}
Adaptation is impossible over the whole range of parameters $r>\be$ and $\be\in(0,1)$: there is a polynomial in $n$ price to pay for not knowing $s_n,r$.
\end{thm}

The intuition behind Theorem \ref{mt:impos} is that although in principle an `easy' problem from the point of view of rates (i.e. the risk goes to $0$), it can be very hard for adaptation since only one error can ruin the rate, especially in the exact recovery regime, where all of the support of $\te$ can be recovered with high probability. We show in \cite{acr23} that adaptation is possible: either if one assumes that $1-\be-\kappa>0$, or for the partial adaptation question where either $s_n$ or $r$ is known.

\section{Bayesian multiple testing in structured settings}

Although conceptually a very simple and computationally attractive method (provided one can compute or at least approximate posterior inclusion probabilities of variables, i.e. the $\ell$--values), aside from the previous results, there is so far very little on frequentist justification of its use for (multiple) testing. It would be desirable to have a theory of (possibly multiple) testing paralleling the one presented earlier here for estimation. \\

Let us now briefly discuss a slightly different setting where the signal values under the alternative $H_1$ are drawn randomly according to a certain process (whose parameters are unknown). We say that they are `structured'. An example is as follows. 

Consider a nonparametric Hidden Markov Model (HMM) setting where one observes (only) $X=(X_1,\ldots,X_N)$ with
\begin{align*} 
\te=(\te_1,\ldots,\te_N) \:\: & \sim \: \text{Markov}_{K=2}(\pi,Q) \qquad \pgr{\te_i\in\{0,1\} \text{ latent state}}
\\
X_i \given \te_i \: \  \ &{\sim}\  \: f_{\te_i} \ \ \text{indep.} \qquad \qquad \quad \text{\pgr{emission densities}}
\end{align*} 
where $f_0, f_1$ are {\em emission} densities with respect to Lebesgue measure on $\RR$, and $\text{Markov}_{K=2}(\pi,Q)$ is the distribution (in the stationary regime say) of a Markov chain with $2$ states $0$ and $1$, initial distribution $\pi$ and $2*2$ transition matrix $Q$. One practical example: two people speak and an observer registers their voices; you would like to know who is speaking at each time-point between $1$ and $n$, which is a multiple testing problem.

This model can be interpreted as a Bayesian model where the latent variables get a prior distribution verifying the Markov property. Since the parameters, here $H=(\pi, Q, f_0, f_1)$, are unknown, one may estimate using estimates $\hat{H}=(\hat\pi, \hat{Q}, \hat{f}_0, \hat{f}_1)$ and then base testing on $\ell$--values 
\[ \ell_{\hat{H}}(X)=\Pi_{\hat H}[\te_i=0\given X], \]
where $\Pi_H$ is the probability distribution in the Bayesian model, for a set of parameters given by $H$. One can show that the cumulative $\ell$--values procedure has a number of optimality properties \cite{acg22}. Similar results have been obtained in other settings, e.g. \cite{sc09} for earlier results in the HMM parametric setting and  \cite{rrv22} on pairwise comparison with an underlying (unseen) stochastic block model structure.

\section{Proofs}  \label{sec:proofmt}

We give here the proof of Theorem \ref{thm:minmax} and refer to \cite{acr23} for a proof of the other results. A main difficulty here will be the lower bound result. A standard observation to start with is that the minimax risk is bounded from below by the Bayesian risk: for a given parameter set $\Theta$ and $\pi$ a prior on $\Theta$,
\begin{equation} \label{lbrisk}
 \inf_\vphi \sup_{\te\in\Theta} \fr(\te,\vphi) \ge \int_\Theta  \fr(\te,\vphi) d\pi(\te). 
\end{equation} 
To make the lower bound as large as possible, one tries to find an appropriate ``least-favorable" $\pi$.
 Below we denote by $P_\pi$ (resp. $E_\pi$) the probability (resp. expectation) under the Bayesian model $X\given \te\sim P_\te$ and $\te\sim\pi$. \\

\subsubsection*{Proof of Theorem  \ref{thm:minmax}}

\begin{proof}[Proof of Theorem \ref{thm:minmax}, upper bound]

We first start by proving that the minimax rate does not exceed $\bar\Phi(b)$ in each case.\\

Since the trivial test $\vphi=0$ has $\fr$--risk equal to $1$, the upper-bound for the case $b=b_n\to -\infty$ is immediate. 

Define, for $a_n^*=\sqrt{2\log(n/s_n)}$, the procedure
\[ \vphi_i(X) = \ind{|X_i|>a_n^*}.\]
Write $V=N_{FP}$ and $S=N_{TP}$ for the number of false discoveries and true discoveries, respectively, made by $\vphi$. Then $V$ is binomial with parameters $n-|S_\te|$ and $2\bar\Phi(a_n^*)= (s_n/n)\cdot o(1)$.

By Markov's inequality, for any $a_n$ the number of false discoveries satisfies
\[ P_\te (V>a_n s_n) \leq  E_\te[V]/(a_ns_n) \le 2n\bar\Phi(a_n^*)/(a_n s_n).\] This latter expression tends to zero  if $a_n$ tends to zero slowly enough, yielding that $V=o_P(s_n)$.

Similarly, using that $\Var_\te(S)\leq  E_\te[S]\leq s_n$ and applying Chebyshev's inequality, the number of true discoveries satisfies
\[ P_\te[|S-E_\te S|\ge s_n^{3/4} ]\le s_n^{-1/2}. \] 
Put together the last two display lead us to define $\mathcal{A}$, an event of probability tending to one on which $V$ and $S$ are suitably bounded.

Denoting $\phi_a(\cdot)=\phi(\cdot-a)$ (and similarly  $\bar\Phi_a(\cdot)=\int_\cdot^{+\infty} \phi_a$), we have by symmetry 
\[ P_{\theta_i=a}(\abs{X_i}>a_n^*) = \overline{\Phi}_a(a_n^*)+\Phi_a(-a_n^*)=\overline{\Phi}_a(a_n^*)+\overline{\Phi}_{-a}(a_n^*).\] 
Noting that $\overline{\Phi}_{-\abs{a}}(a_n^*)\leq \overline{\Phi}_0(a_n^*)\to 0$, 
so is a $o(1)$ uniformly in $a$, we thus calculate
\begin{align*}
	E_\theta [S] &= \sum_{i\in S_\theta} P_{\theta_i}(\abs{X_i}>a_n^*) = \sum_{i\in S_\theta} (\overline{\Phi}_{\abs{\theta_i}}(a_n^*) + o(1)) \\ 
	&\ge  s_n (\overline{\Phi}_{a_n^*+b}(a_n^*) + o(1)) = s_n [1-\overline{\Phi}(b) +o(1)].\
\end{align*}
The combined risk of $\vphi$ is then, for fixed $b$,
\begin{align*} \mathfrak{R}(\theta,\vphi) &= E_\theta \frac{V}{V + S} +E_\theta \frac{s_n-S}{s_n} \\ &\leq  P(\mathcal{A}^c) + \frac{ o(s_n)}{o(s_n)+ s_n(1-\overline{\Phi}(b)+o(1))-s_n^{3/4}} + \frac{s_n(\overline{\Phi}(b)+o(1))}{s_n} 
\end{align*} 
and the last display is $\overline{\Phi}(b) + o(1)$ as desired (and for $b=b_n\to\pli$, the last term is $o(1)$).
\end{proof}

\begin{proof}[Proof of Theorem \ref{thm:minmax}, lower bound] 
We apply the general lower bound of Theorem~\ref{thm:bayesLB}: for any  $\rho\in(0,1), \eta>0$, and any prior that puts mass $1$ on $\cL_0[s_n;b]$, if
\begin{align*}
M_\rho & = \sum_{i=1}^n P_\pi\left[\, \theta_i\neq 0\,,\, \ell_i(X)> \rho/(1+\rho)\,\right], \\
\la        & =  (1-\eta) M_\rho /s_n,
\end{align*}
where $\ell_i(X)=P_\pi(\te_i=0 \given X) = \pi(\te_i =0 \given X)$ for $1\leq i\leq n$ are the $\ell$--values, we have
\[	 \inf_\vphi \sup_{\te\in\cL_0[s_n;b]}  \mathfrak{R} (\theta,\vphi) \geq   \left(\lambda \wedge \frac{\rho \lambda}{1+\rho \lambda}\right)  (1-  e^{- c\eta^2 M_\rho}), \]
	for some universal constant $c>0$.	
We  apply this with $\rho=\rho_n$ and $\eta=\eta_n$ certain positive sequences converging slowly to infinity and $0$ respectively to be specified later on.

Let us define $\pi$ as a product prior over $s_n$ blocks of consecutive coordinates 
\[ Q_1=\{1,2,\ldots, q\}\,,\, Q_2=\{q+1,\ldots,2q\}\,,\,\ldots, Q_{s_n}=\{(s_n-1)q+1,\ldots,n'\},\]
 where $q= \floor{n/s_n}$ and $n'=qs_n$. We write $Q_\infty$ for the (possibly empty) set $\{n'+1,\dots,n\}$. 
Over each block $Q_j$, $1\leq j\leq s_n$, define $\pi$ as follows:  
first draw an integer $I_j$ from the  uniform distribution $\cU(Q_j)$ over the block $Q_j$ and next for each $i\in Q_j$ set $\theta_{i}=a_n^*+b$ if $i=I_j$ and $\theta_{i}=0$ otherwise.  
For $i\in Q_\infty$, set $\theta_{i}=0$. Let us also denote $a:=a_n:=a_n^*+b$ for short. By construction $\pi(\cL_0[s_n;b])=1$. 

For all $1\leq j\leq s_n$ and $i\in Q_j$, denoting $h(x,a):=(\phi_a/\phi)(x)$, noting that by the prior's definition $\te_i=0$ means $i\neq I_j$, the $\ell$--value can be written, using Bayes' formula,
\begin{align*}
	\ell_i(X)&=P_\pi(i\neq I_j\:|\: X)=1- w_i^j(X) \\
	w_i^j(X)&= \frac{(\phi_a/\phi)(X_i)}{\sum_{k\in Q_j} (\phi_a/\phi)(X_k)}= \frac{h(X_i,a)}{\sum_{k\in Q_j} h(X_k,a)},
\end{align*} 
In addition,
\begin{align}
	M_\rho&=  \sum_{i=1}^n P_\pi[ \theta_i\neq 0,\ell_i(X)> \rho/(1+\rho)]=\sum_{j=1}^{s_n} \sum_{i\in Q_j} P_\pi[ i=I_j\,,\,\ell_i(X)> \rho/(1+\rho)]\nonumber\\
	&=\sum_{j=1}^{s_n} \sum_{i\in Q_j} P_\pi\left[ i=I_j\,,\, (\rho+1) h(X_i,a) < \sum_{k\in Q_j} h(X_k,a) \right]\nonumber\\
	&\geq \sum_{j=1}^{s_n} \sum_{i\in Q_j} P_\pi\left[ i=I_j\, ,\, \#\{ k\in Q_j\backslash\{i\}\::\:  h(X_k,a)> h(X_i,a)  \}\geq \rho\right]\nonumber\\
	&=  s_n  P_{X_1\sim \phi_{a}}\left[ \#\{ k\in Q_1\backslash\{1\}\::  h(\varepsilon_k,a)> h(X_1,a)  \}\geq \rho\right],\label{equMrho}
	\end{align}
where we have used symmetry to see that all the probabilities on the third line are the same (hence one may take $i=j=1$ on the last line), and where  the $\varepsilon_k$'s are iid standard normal.     

Since $x\to(\phi_a/\phi)(x)$ is increasing for any $a>0$, 
we see that $\varepsilon_k> X_1$ implies $h(\varepsilon_k,a)> h(X_1,a)$, so that the last display can be further lower bounded by 
\begin{equation*}
	s_n  P_{X_1\sim \phi_{a}}\left[ \#\{ k\in Q_1\backslash\{1\}\ :  \varepsilon_k> X_1 \}\geq \rho\right]. 
	\end{equation*}
By Lemma~\ref{lem:ControlOfpn}, provided $\rho$ verifies \eqref{condrho}, we have
\[P_{X_1\sim \phi_a}\left[ \#\{ k\in Q_1\backslash\{1\}\ : \varepsilon_k\geq X_1 \}> \rho\right]=\Phi_a(a_n^*)+o(1)\] so that continuing the inequalities we have, since $a=a_n^*+b$ and $\Phi(-x)=\overline\Phi(x)$ for all $x$,
 \begin{align}
	M_\rho/s_n
	&\geq \Phi_a \left(a_n^* \right) +o(1)= \overline\Phi(b) +o(1) ,\label{equMrhoboundedA}
\end{align}
 This gives the lower bound 
\begin{align*}
\inf_\vphi \sup_{\theta\in  \cL_0[s_n;b]} \fr(\te,\vphi) \geq &  \left([ \overline\Phi(b) +o(1)](1-\eta) \wedge \frac{\rho [ \overline\Phi(b) +o(1)](1-\eta)}{1+\rho [ \overline\Phi(b) +o(1)](1-\eta)}\right) \\&\times (1-  e^{- c\eta^2 [ \overline\Phi(b) +o(1)](1-\eta)s_n}) .
\end{align*}
Now using that for all $x\in [0,1]$, $A,y>0$, we have 
$$[x\wedge (y x/(1+y x))](1-e^{-A x})\geq x+0\wedge(1-1/(1+yx)-x) -xe^{-Ax}\geq x- 1/y - 1/(Ae),$$
we deduce
\begin{align*}
\inf_\vphi \sup_{\theta\in  \cL_0[s_n;b]} \fr(\te,\vphi) \geq & [ \overline\Phi(b) +o(1)](1-\eta) - \rho^{-1}-1/(c\eta^2 s_n e).
\end{align*}
Choosing $\rho=\lfloor(n/(3s_n))\overline{\Phi}(a_n^*-\delta_n)\rfloor\to \infty$ and $\eta=s_n^{-1/4}$, one gets
$$
\liminf_n \inf_\vphi \sup_{\theta\in  \cL_0[s_n;b]} \fr(\te,\vphi)\geq  \overline\Phi(b)+o(1),
$$
which proves the lower bound.
\end{proof}

\subsubsection*{Generic lower bounds}

We present in this section general lower bounds that can be applied in any model where we observe $X\sim P_\theta$, $\theta\in \Theta$, for any parameter set $\Theta\subset \R^n$.

\begin{thm} \label{thm:bayesLB} For $\pi$ prior on $\R^n$, 
	let $ \ell_i(X)=P_\pi(\theta_i=0\:|\: X)$, $1\leq i\leq n$. For $\eta\in (0,1), \rho>0$, let 
\begin{align*}
M_\rho & = \sum_{i=1}^n P_\pi\left[\, \theta_i\neq 0\,,\, \ell_i(X)> \rho/(1+\rho)\,\right], \\
\la        & =  (1-\eta) M_\rho /s_n, \qquad (s_n\ge 1).
\end{align*}
Suppose $\pi$ is a prior on $\cL_0[s_n;b]$, i.e. $\pi(\cL_0[s_n;b])=1$. Then,   for a universal constant $c>0$,
\[ 	\inf_\vphi \sup_{\te\in \cL_0[s_n;b]}  \mathfrak{R} (\theta,\vphi)  \geq   \left(\lambda \wedge \frac{\rho \lambda}{1+\rho \lambda}\right)  (1-  e^{- c\eta^2 M_\rho}). \]
	\end{thm}
		
		The lower bound (i) says roughly that the Bayes risk for $ \mathfrak{R}$ is lower bounded by the type-two error of the Bayes procedure for the $\rho$-weighted classification risk problem. 	\\
			\begin{proof}
	For all $\theta\in \R^n$ and $\vphi$, let us write $D_n(X) = \sum_{i\leq n} \vphi_i(X)$ and let
	$$
	Q(\theta,\vphi,X)= \sum_{i=1}^n \left\{ \ind{\te_i=0} \frac{\vphi_i(X)}{1\vee D_n(X)}
	+ \ind{\te_i\neq 0} \frac{1-\vphi_i(X)}{\norm{\theta}_0 \vee 1} \right\},
	$$
	so that $ \mathfrak{R} (\theta,\vphi) =E_\theta Q(\theta,\vphi,X)$. 
	Denote  
	$$
	\fr_\pi:=E_\pi \mathfrak{R} (\theta,\vphi).  
		$$
By \eqref{lbrisk}, it suffices to bound $\fr_\pi$ from below. Let us introduce a weighted classification loss
	\begin{equation}\label{equLrho}
	L_\rho(\te,\varphi)= \sum_{i=1}^n \left\{ \ind{\te_i=0} \vphi_i(X)+ \rho \ind{\te_i\neq 0} (1-\vphi_i(X)) \right\}.
	\end{equation}
For any $\delta_n>0$, if $D_n(X)\leq s_n(1+\delta_n)$ and $\norm{\theta}_0\leq s_n$, we have that $Q(\theta,\vphi,X)$ is at least 
	\[
	\left(\frac{1}{1+\delta_n}\wedge \rho^{-1} \right)	\frac{L_\rho(\te,\varphi)}{s_n}.
	\]
	If $D_n(X)\geq s_n(1+\delta_n)$ (hence $D_n(X)\geq 1$ for $s_n\ge 1$) and $\norm{\theta}_0\leq s_n$, we have 
	\[ Q(\theta,\vphi,X) 
\geq  \frac{D_n(X)-\norm{\theta}_0}{ D_n(X)}\geq  \frac{D_n(X)-s_n}{   D_n(X)}\geq  \frac{\delta_n}{1+\delta_n}. \]
	Hence, since $\pi(\|\te\|_0\le s_n)=1$, we have for all $\delta_n>0$,
	\begin{align*} 
	&E_\pi[Q(\theta,\vphi,X)] = E_\pi[E_\pi[Q(\theta,\vphi,X)\:|\:X]] \\
	&\ge  E_\pi\left(\ind{D_n(X)\leq s_n(1+\delta_n)} 
	 \left(\frac{1}{1+\delta_n}\wedge \rho^{-1}\right)\frac{E_\pi[L_\rho(\te,\varphi) \given X]}{s_n}\right.\\
	&\qquad +\left.\ind{D_n(X)> s_n(1+\delta_n)} \frac{\delta_n}{1+\delta_n}\right).
	\end{align*}
	By taking the minimum of terms inside the expectation (and the inf over $\vphi$),
	\begin{align*}
	 \fr_\pi &\geq E_\pi\left( \frac{\delta_n}{1+\delta_n} \wedge \left\{ \left(\frac{1}{1+\delta_n}\wedge \rho^{-1}\right)\frac{\inf_\vphi E_\pi (L_\rho(\te,\vphi)|X)}{s_n}\right\}  \right).
	 \end{align*}
	 Now arguing as in Proposition \ref{prop:bayes}, the infimum in $\vphi$  for the (weighted) classification loss is attained for the  procedure $ \ind{\ell_i(X)> \rho/(1+\rho)}$, so that 
	 \begin{align}
	 \inf_\vphi E_\pi (L_\rho(\te,\vphi)|X)& = \sum_{i=1}^n \left\{ \ell_i(X) \ind{\ell_i(X)\leq \rho/(1+\rho)}
	+ \rho(1-\ell_i(X))  \ind{\ell_i(X)> \rho/(1+\rho)} \right\} \nonumber\\
	&\geq \rho  \sum_{i=1}^n  (1-\ell_i(X))  \ind{\ell_i(X)> \rho/(1+\rho)} 
	=: \rho L'_\rho. \label{equLprimerholb}
	 \end{align}
	  This entails, for $L'_\rho$ as in \eqref{equLprimerholb},
	 \begin{align}
	\fr_\pi &\geq E_\pi\left( \frac{\delta_n}{1+\delta_n} \wedge \left\{ \left(\frac{1}{1+\delta_n}\wedge \rho^{-1}\right)\frac{\rho L'_\rho}{s_n}\right\}  \right).	\label{equinterm}
	\end{align}
Let us note  that  $M_\rho=\sum_{i=1}^n E_\pi[(1-\ell_i(X))  \ind{\ell_i(X)> \rho/(1+\rho)}]=E_\pi L'_\rho$ by the chain rule on conditional expectations. Also, by Bernstein's inequality, for $c>0$ some constant,
	\[ 
P_\pi\left(L'_\rho < M_\rho (1-\eta)\right)\leq e^{- c \eta^2 M_\rho}.
	\]		
Noting that introducing the indicator $\ind{L'_\rho \ge M_\rho (1-\eta)}$ inside the expectation in \eqref{equinterm} only makes it smaller, deduce that the right-hand side of \eqref{equinterm}  is at least
	 \begin{align*}
	\left[ \frac{\delta_n}{1+\delta_n} \wedge \left\{ \left(\frac{\rho}{1+\delta_n}\wedge 1\right)\frac{M_\rho (1-\eta) }{s_n}\right\} \right]  (1-  e^{- c \eta^2 M_\rho}).
	\end{align*}
	The result 
	follows by letting $\delta_n=\rho M_\rho (1-\eta)/s_n$.
	\end{proof}

The following lower bound is similar to the one of Theorem~\ref{thm:bayesLB}, but somewhat more classical, because it is for the classification risk $E_\theta  \lc(\te,\vphi)/s_n$.

\begin{thm} \label{thm:bayesLB-classif}
	Using the same notation and conditions as in Theorem \ref{thm:bayesLB},
$$
	\inf_\vphi \sup_{\te\in\cL_0[s_n;b]}  E_\theta  \lc(\te,\vphi)/s_n \geq  M_1/s_n,
		$$
where $M_1=\sum_{i=1}^n P_\pi[ \theta_i\neq 0,\ell_i(X)> 1/2]$ (that is, $M_1$ is $M_\rho$ of Theorem~\ref{thm:bayesLB} for $\rho=1$).
	\end{thm}

\begin{proof}
Using \eqref{lbrisk} for the classification risk, for $\Theta=\cL_0[s_n;b]$ and $\Pi$ a prior on $\Theta$,
\[\sup_{\te\in\Theta }  E_\theta  \lc(\te,\vphi)/s_n \geq E_{\pi} \lc(\te,\vphi)/s_n.\] 
Taking the infimum over $\vphi$ and using Proposition \ref{prop:bayes}, one gets, for $\vphi_i^\ell=\ind{\ell_i(X)>1/2}$,
$$
\inf_\vphi E_{\pi} \lc(\te,\vphi) \ge E_{\pi} \lc(\te,\vphi^\ell)\ge M_1,
$$
by keeping only the part of the loss regarding false negatives. The result follows.
\end{proof}

\vp

\begin{lem}\label{lem:ControlOfpn} Let $a_n^*=\sqrt{2\log(n/s_n)}$, $\delta_n=(\log(n/s_n))^{-1/4}$. 
	For any integer sequence $\rho=\rho_n$ satisfying 
	\begin{equation} \label{condrho}
	1\leq \rho \leq (n/2s_n)\overline{\Phi}(a_n^*-\delta_n)-1
	\end{equation}
	 we have, for any $a$ (which may depend on $n$), 
	\[ P\left(\# \braces{2\leq i \leq n/s_n : \eps_i>X_1\sim \phi_a}\leq \rho-1\right) = \overline{\Phi}_{a}(a_n^*)+o(1).\]
	In particular, one may choose $\rho$ tending to infinity or $\rho=1$. The $o(1)$ term is uniform in $a$. 
\end{lem}
\begin{proof}
Write $A_n$ for the event $A_n = \braces[\big]{ \#\braces{2\leq i \leq n/s_n : \eps_i>X\sim \phi_a}\leq \rho-1}$. 
For $\delta_n$ as in the lemma, set $U_n = \#\braces{2 \leq i \leq n/s_n : \eps_i > a_n^* - \delta_n}$ and note that \[U_n \sim \operatorname{Bin}(\floor{n/s_n}-1, \overline{\Phi} (a_n^*-\delta_n)).\] Define $V_n$ correspondingly without the $\delta_n$. As $a_n^*, \delta_n$ verify  \eqref{eqn:modelassumption2}--\eqref{eqn:modelassumption3} with $F_0=\Phi$,  
we see that		\[ E[U_n] \to \infty, \qquad E[V_n] = (\floor{n/s_n}-1)\overline{\Phi}(a_n^*)\to 0.\]
For any $\rho=\rho_n\geq 1$ such that $\rho-1<E[U_n]/2$ (which is true under the specified condition on $\rho$), we may apply Chebyshev's inequality with the bound $\Var(U_n)\leq E[U_n]$ to obtain
\[ P(U_n \leq \rho-1) \leq  P( \abs{U_n-E[U_n]} \geq E[U_n]-(\rho-1)) \leq P\brackets[\Big]{\abs{U_n-E[U_n]}>\frac{E U_n}{2}} \leq \frac{4}{E U_n}\to 0.\]
Similarly, applying Markov's inequality, we have for any $\rho\geq 1$,
\[ P(V_n>\rho-1) = P(V_n\geq  \rho)\leq  \frac{E V_n}{\rho}\to 0.\]
We have thus shown that on an event $B_n$ of probability tending to 1, 
$U_n\geq \rho$ and $V_n\leq \rho-1$.  In words, on $B_n$, at most $\rho-1$ of the numbers $(\eps_i,~2\leq i\leq n/s_n)$ are larger than $a_n^*$ and at least $\rho$ of them are larger than $a_n^*-\delta_n$. It follows that, on $B_n$, the event $A_n$ holds if $X_1>a_n^*$ and fails if $X_1\leq a_n^*-\delta_n$. Thus, 
\begin{align*} P(A_n) \geq P_a(X_1 > a_n^*) - P(B_n^c),\\
	P(A_n^c) \geq P_a(X_1 \leq a_n^*-\delta_n) -P(B_n^c).
\end{align*}
Since $\Phi_a$ is Lipschitz we deduce the result.  
\end{proof}

\section*{{\em Exercises}}
\addcontentsline{toc}{section}{{\em Exercises}} 

\begin{enumerate}
\item Check that for Subbotin noise with $\zeta\in(1,2]$, for $\kappa\in (0,1-1/\zeta)$, the sequences 
\[ a_n^*= (\zeta \log(n/s_n))^{1/\zeta}, \ \delta_n = (\log(n/s_n))^{-\kappa} \]
verify \eqref{eqn:modelassumption2}--\eqref{eqn:modelassumption3}. Verify that the proofs above go through provided one replaces  $\sqrt{2\log(n/s_n)}$ in the definition of $\cL_0[s_n;b]$ by this new $a_n^*$ and the sharp testing constant $\overline{\Phi}(b)$ by $\overline{\Phi}^{\zeta}(b)$, where $\Phi^\zeta$ is the cumulative distribution function of the Subbotin density.
\end{enumerate}

\chapter{Variational approximations} \label{chap:vb}

In this chapter, we consider the use of variational approximations to posterior distributions. The three complementary works \cite{ar20, ypb20, zg20} provide generic results and conditions under which approximations of posterior distributions in certain variational classes converge at (at least) the same rate as the posterior distribution itself. The case of tempered posterior distributions is also considered. These results apply already to a variety of models and priors, including many non-parametric or latent variable models. Here we follow mostly the presentation of \cite{zg20}. The case of high-dimensional models needs a separate treatment, and is considered in \cite{rs22}: we present it briefly.

\section{General principles}

In variational methods, one wishes to find a best (or close to best) approximation of a given {\em target} distribution (in the framework of these lectures it will be the {\em posterior} distribution) within a given class of simple distributions. The approximation will be quantified in terms of a measure of distance (or divergence) between distributions. 

\subsection{Divergences}

The $\rho$--\sbl{Rényi divergence} between probability measures $P$ and $Q$ is defined as, for $\rho>0$ and $\rho\neq 1$, 
\[ D_\rho(P,Q) = \frac{1}{\rho-1} \log \int \left(\frac{dP}{dQ}\right)^{\rho-1} dP, \]
if $P$ is absolutely continuous with respect to $Q$, and $+\infty$ otherwise. In the first case, and if $P, Q$ have densities $p,q$ with respect to $\mu$, we have
\[ D_\rho(P,Q) = \frac{1}{\rho-1} \log \int p^\rho q^{1-\rho} d\mu. \]
If $\rho=1$, one similarly defines, for $P\ll Q$ (otherwise we set it to $+\infty$ as above),
\[ D_1=K(P,Q) = \int \log\left(\frac{dP}{dQ}\right) dP\]
  the Kullback--Leibler divergence between $P$ and $Q$. \\
  
The following facts are classical, see for example the review paper \cite{vanerven}: $\rho\to D_\rho(P,Q)$ is an increasing function; as $\rho\to 1$, $D_\rho \to D_1$. Also, $D_{1/2}, D_2$ are related respectively to the squared--Hellinger distance $h^2$ and the $\chi^2$ divergence in the sense that
\[ D_{1/2} = -2\log(1-h^2/2),\qquad D_2= \log(1+\chi^2). \] 

\subsection{Variational families and optimisation} 

\begin{definition}
Let $\cS$ be a family of distributions. The \sbl{variational posterior} with respect to the family $\cS$ is the miminiser of the KL-divergence between any element of $\cS$ and the posterior distribution. That is,
\begin{equation} \label{varpb}
 \hat{Q} = \underset{Q\in\cS}{\text{argmin}}\, K(Q,\Pi(\cdot\given X)). 
\end{equation} 
\end{definition}

Often, an exact solution to \eqref{varpb} is not available, but an approximation is; then the results that follow also hold for this approximation as long as the latter is close enough to an exact solution, if it exists, of \eqref{varpb}.\\

If the class $\cS$ is very large, it may even contain the true posterior in which case one would have $\hat{Q}=\Pi(\cdot\given X))$. Of course, the purpose is to choose a class $\cS$ sufficiently simple so that the optimisation problem \eqref{varpb} is simpler to solve compared to direct sampling from the posterior. For direct sampling from (an approximation of) $\Pi(\cdot\given X)$, unless the posterior is available in closed form (which is rarely the case), one generally resorts to a general method such as MCMC (Monte Carlo Markov Chain). However in high dimensions or in problems with latent variables the MCMC method may be slow to converge. In such cases, variational approximations of the posterior are very popular in practice. The idea is to choose a class both sufficiently rich to approach the true posterior reasonably well, but at the same time sufficiently simple so that \eqref{varpb} is fast to solve numerically.  In other words, there is a {\em trade--off} between good approximation properties and computability. \\

We will not focus much more here on this trade-off, but give two examples of popular classes below. Before this, we note that a nice property of the optimisation problem \eqref{varpb} is that the normalising constant in the expression of the posterior density from Bayes' formula (i.e. the denominator) vanishes when one optimises in $Q\in\cS$. Indeed, writing $dQ=qd\mu$ and, using Bayes' formula, $d\Pi(\te\given X) = p_\te(X)\pi(\te)/\int   p_\te(X)\pi(\te)d\mu$, and noting that $D_X=  \int   p_\te(X)\pi(\te)d\mu$ depends only on $X$ but not on $Q$ or $\te$,
\begin{align*} K(Q,\Pi(\cdot\given X)) & = \int \log\left(\frac{q(\te)}{p_\te(X)\pi(\te)/D_X}\right) q(\te)d\mu \\
& = \int \log\left(\frac{q(\te)}{p_\te(X)\pi(\te)}\right) q(\te)d\mu + \log D_X,
\end{align*}
and the last term is independent of $Q$ so it is enough to minimise the first term (note that $D_X$ does not vanish if one would consider $K(\Pi(\cdot\given X),Q)$ in \eqref{varpb}). In particular, when solving the variational problem, there is no need to compute $D_X$, which often can be delicate or at least time-consuming. 

\begin{definition}\sbl{[Mean--field classes]} \label{defmf}
Suppose the parameter $\te\in \Theta$ can be written $\te=(\te_1,\te_2,\ldots, \te_m)$ with $m$ an integer, or $m=+\infty$. 
 The mean--field variational class $\cS_{MF}$ is the class of distributions
\[ \cS_{MF} = \left\{Q:\ dQ(\te) = \prod_{j=1}^m dQ_j(\te_j) \right\}. \]
That is, $\cS_{ML}$ consists of product measures only. As a special subcase, one may consider specific distributions for the $Q_j$. Let $\cG=\{\cN(\mu,\sigma^2),\ \mu\in\RR, \sigma\ge 0\}$ be the set of $1$--dimensional Gaussian distributions. The \sbl{Gaussian mean field class} is 
\[ \cS_{GMF} = \left\{Q:\ dQ(\te) = \prod_{j=1}^m dQ_j(\te_j),\ \ \ Q_j\in \cG\ \ (\forall\, j) \right\}. \]
\end{definition}

The idea of the mean--field class is to ignore dependencies in the posterior distribution and to approximate it by a distribution of product form. Of course, some information is then typically lost in this process: for instance, for $\te\in\RR^2$, a Gaussian distribution $\cN(\te,\Sigma)$ with $\Sigma$ a {\em non-diagonal} $2\times 2$ covariance matrix cannot be perfectly approximated by a product of $1$--dimensional Gaussians. One may think though that the loss is `of the order of a multiplicative factor in the variance', so maybe not huge. 

\section{A generic result for variational posteriors}

\subsection{Statement}

Consider a statistical model $\cP=\{P_\te^{(n)},\ \te\in\Theta\}$ as before in the course, dominated by $\mu^{(n)}=\mu$, where $\Theta$ is a parameter set (e.g. space of functions).\\

Let $L(\cdot,\cdot)$ be a loss function between probability measures such that $L(P,Q)\ge 0$ for any such measures $P,Q$. Examples of losses include $L=n h^2$, i.e. $n$ times the Hellinger squared distance, or also, if $\te$ is a sequence, $L(P_{\te}^{(n)},P_{\te'}^{(n)})=n\|\te-\te'\|^2$. Note the specific normalisation chosen with a multiplicative factor $n$: this is related to the fact that a typical example is the one of product measures $P_{\te}^{(n)}=P_{\te}^{\otimes n}$ for which typical divergences such as the KL scale with $n$ (recall that $K(P_{\te}^{\otimes n},P_{\te'}^{\otimes n})=nK(P_{\te},P_{\te'})$). \\
 
{\em Generic conditions.} Consider the following conditions
\begin{enumerate}
\item[(\underline{T})] $\quad$ For any $\veps>\veps_n$, there exists $\Theta_n(\veps)$ measurable subsets of $\Theta$ and $\vphi_n$ test functions such that
\[ E_{\te_0}\vphi_n + \sup_{\te\in\Theta_n(\veps),\ L(P_{\te}^{(n)},P_{\te_0}^{(n)})>C_1n\veps^2} E_\te(1-\vphi_n) \le e^{-Cn\veps^2}. \]
\item[(\underline{S})] $\quad$ For  any $\veps>\veps_n$,
\[ \Pi(\Theta_n(\veps)^c) \le e^{-Cn\veps^2}. \]
\item[(\underline{P})] $\quad$ There exists $\rho>1$ such that
\[ \Pi\left( D_\rho(P_{\te_0}^{(n)},P_{\te}^{(n)}) \le C_3n\veps_n^2 \right)
\ge e^{-C_2n\veps_n^2}.
 \]
\end{enumerate}
These conditions are almost identical to the ones used before in the lectures. There are two differences. First, (\underline{T}) and (\underline{S}) are required to hold for any $\veps>\veps_n$. It is generally not too difficult to find a sequence of sets $\Theta_n(\veps)$ indexed by $\veps$ verifying (\underline{T}) for any $\veps>\veps_n$ (and not just for $\veps=\veps_n$). Second, the KL-neighborhood used before is replaced by a $D_\rho$--neighborhood where $\rho>1$. This is useful in that it enables one to obtain posterior masses of complements of neighborhoods that decrease exponentially fast to $0$ (instead of just polynomially -- recall we used simply Tchebychev's inequality in proving the GGV theorem -- here we get rather an exponential-type inequality). 

See Lemma \ref{lemga2}, where it is shown that under such slightly strengthened assumptions compared to Theorem \ref{thm:ggvt}, the original posterior converges at rate $\veps_n$ and with an exponential decrease to $0$.
\\

In what follows, for a given function $f$, we use the notation $Qf=\int f dQ$.

\begin{thm}\sbl{[Convergence rate for $\hat{Q}$, \cite{zg20}]} \label{thmvar}
Let $(\veps_n)$ be a sequence such that $n\veps_n^2\ge 1$. Let $\Pi$ be a prior distribution on $\Theta$. 
Consider $\hat{Q}$ the variational Bayes approximation \eqref{varpb} to the posterior distribution $\Pi[\cdot\given X]$ with variational class $\cS$ and set
\begin{equation}\label{defga}
\ga_n^2=\frac1n \inf_{Q\in\cS} E_{\te_0} K(Q,\Pi(\cdot\given X)).
\end{equation}
 Suppose the generic conditions (\underline{T}), (\underline{S}), (\underline{P}) are verified with rate $\veps_n$, loss function $L$,  positive constants $C_1,C_2,C_3$, $C>C_2+C_3+2$ and $\rho>1$. Then there exists $M=M(C_1,C,\rho)$ such that
 \[ E_{\te_0} \hat{Q} L(P_\te^{(n)},P_{\te_0}^{(n)}) \le Mn(\veps_n^2+\ga_n^2). \] 
\end{thm}

The interpretation is as follows: $\veps_n$ is the convergence rate of the original posterior distribution in terms of the loss $L$, while $\gamma_n$ is the contribution arising from considering the variational approximation. Note that $\gamma_n^2$ is defined in terms of the posterior and is this still implicit. In the next lines, we bound it from above by a more universal quantity.

\subsection{Sufficient conditions}

\begin{lem} \label{lemga}
The rate $\ga_n$ defined in \eqref{defga} verifies
\begin{equation} \label{ubga}
\ga_n^2 \le \frac1n \inf_{Q\in\cS} \left[K(Q,\Pi) + Q K(P_{\te_0}^{(n)},P_{\te}^{(n)}) \right].
\end{equation}
\end{lem}
\begin{proof}
Denote as shorthand $\Pi_X=\Pi(\cdot\given X)$ and $P_\Pi^{(n)}=\int P_\te^{(n)}d\Pi(\te)$ so that the denominator in Bayes' formula is $p_{\Pi}^{(n)}=\int p_\te^{(n)}d\Pi(\te)$ the marginal density of $X$ in the Bayesian setting.  Bayes' formula writes  $d\Pi_X=p_\te^{(n)}d\Pi/p_{\Pi}^{(n)}$. 
Then $K(Q,P_X)=\int \log(dQ/d\Pi_X)dQ$, and
\[ \log \frac{dQ}{d\Pi_X} = \log \frac{dQ}{d\Pi} + \log \frac{p_{\Pi}^{(n)}}{p_\te^{(n)}}. \]
Deduce that $K(Q,P_X)$ can be further written as
\begin{align*}
E_{\te_0} K(Q,\Pi_X)
& = \int \log \frac{dQ}{d\Pi}dQ + E_{\te_0}\int \log\frac{p_{\Pi}^{(n)}}{p_\te^{(n)}} dQ
\end{align*} 
and, using Fubini's theorem and $K(P,Q)\ge 0$,
\begin{align*}
\lefteqn{E_{\te_0}\int \log\frac{p_{\Pi}^{(n)}}{p_\te^{(n)}} dQ  =  Q \int \log\frac{dP_{\Pi}^{(n)}}{dP_\te^{(n)}}dP_{\te_0}^{(n)} }& \\ 
& = Q\left[ \int \log \frac{dP_{\te_0}^{(n)}}{dP_{\te}^{(n)}} dP_{\te_0}^{(n)}+
\int \log\frac{dP_{\Pi}^{(n)}}{dP_{\te_0}^{(n)}} dP_{\te_0}^{(n)} \right] \\
& = Q\left[ K(P_{\te_0}^{(n)}, P_\te^{(n)}) - K(P_{\te_0}^{(n)}, P_{\Pi}^{(n)})\right] \le  QK(P_{\te_0}^{(n)}, P_\te^{(n)}).
\end{align*}
The lemma follows by using the definition of $\ga_n$ and taking the infimum over $Q\in \cS$.
\end{proof}

By combining Lemma \ref{lemga} and Theorem \ref{thmvar}, we immediately obtain the following.\\

\begin{corollary}
Under the conditions of Theorem \ref{thmvar}, suppose further that, for 
\[ \cE = \left\{Q:\ \text{Supp}(Q)\subset \{\te:\ K(P_{\te_0}^{(n)},P_{\te}^{(n)})\le C_2 n\veps_n^2\} \right\}, \]
it holds
\begin{equation}
\inf_{Q\in \cS\cap \cE}\, K(Q,\Pi) \le C_1n\veps_n^2.
\end{equation}
Then for a large enough constant $M'$, 
\[ E_{\te_0} \hat{Q} L(P_\te^{(n)},P_{\te_0}^{(n)}) \le M'n\veps_n^2. \]
\end{corollary}

\subsection{Result for mean-field class}

\begin{thm}\sbl{[Convergence rate for $\hat{Q}$, mean-field case]} \label{thmmf}
Under the conditions of Theorem \ref{thmvar}, suppose that one can find a distribution $\tilde{Q}$ in the mean-field class $\cS_{MF}$ and a subset 
\[ \cA_m := \bigotimes_{j=1}^m \tilde{\Theta}_j \ \subset \Theta \]
of product form which verifies
\begin{align*} 
(A)\ \qquad \qquad  \cA_m
& \subset \left\{\te:\ K(P_{\te_0}^{(n)}, P_\te^{(n)})\le C_1n\veps_n^2,\
\log\left(\frac{d\tilde{Q}}{d\Pi}(\te)\right) \le C_2n\veps_n^2 
\right\}, \\
(B)\ \ \  \qquad  \tilde{Q}\left( \cA_m \right) 
& \ge e^{-C_3n\veps_n^2}.
\end{align*}
Then the term $\ga_n$ in \eqref{defga} with $\cS=\cS_{MF}$ verifies
\[ \ga_n^2 \le (C_1+C_2+C_3)\veps_n^2. \]
In particular, the conclusion of Theorem \ref{thmvar} holds with rate $\veps_n^2$. 
\end{thm}

Note that in case the prior itself belongs to $\cS_{MF}$, one can take $\tilde{Q}=\Pi$, which simplifies the conditions even further. Also, condition (B) in the Theorem can be interpreted as asking a prior mass condition which is ``coherent with the structure of the variational class".

\subsection{An example of application: the sequence model}

Consider the Gaussian sequence model $X_i=\te_i+\xi_i/\sqrt{n}$ with $\xi_i$ independent $\cN(0,1)$ variables. Suppose the true $\te_0=(\te_{0,1},\te_{0,2},\ldots,)$ belongs to a Sobolev ball 
\[ \te_0\in \{\te:\ \sum_{j\ge 1} j^{2\be}\te_{j}^2\le L^2\}.\] 
Also set $L(P_{\te}^{(n)},P_{\te'}^{(n)})=n\|\te-\te'\|^2$. \\

Consider a sieve prior $\Pi$ defined hierarchically: sample $k$ from a distribution $\pi$ on integers; then given $k$ sample $\te_1,\ldots,\te_k$ independently with density $f_j$ on coordinate $j$;  set $\te_j=0$ for all $j>k$. 

Let us consider the variational posterior $\hat{Q}$ using the mean-field class
\[ \hat{Q} = \underset{Q\in\cS_{MF}}{\text{argmin}}\, K(Q,\Pi(\cdot\given X)), \]
where $\cS_{MF}$ is as in Definition \ref{defmf}. The next result shows that the variational posterior $\hat{Q}$ is adaptive to smoothness and reaches an optimal contraction rate in $\ell^2$ up to a logarithmic term.

\begin{thm}
In the Gaussian sequence model, suppose $\te_0$ and $\Pi$ are as described above, for some $\be, L>0$. Take as $\pi$ the prior on integer such that $\pi(k)\propto e^{-\ta k}$ and take $f_j$ to be the standard normal density for any coordinate $j$ that is nonzero under the prior. Then
\[ E_{\te_0}\hat{Q}\|\te-\te_0\|^2 \leqa \left( \frac{\log{n}}{n}\right)^{\frac{2\be}{2\be+1}}. \]
\end{thm}

\section{Proof of the generic theorem}

\subsubsection*{Useful lemmas and their proofs}

The proofs of Theorems \ref{thmvar} and \ref{thmmf} are quite direct applications of the combination of the next Lemmas. Lemma \ref{lemga2} makes the conclusion of Theorem \ref{thm:ggvt} more precise by assuming slightly stronger conditions.

\begin{lem} \label{lemtr}
Let $f\ge 0$ and let $P,Q$ be two probability measures. Then
\[ \int f dQ \le K(Q,P) + \log \int e^{f(x)} dP(x). \]
\end{lem}
\begin{proof}
One writes, with the notation $\int e^f dP=P e^f$,
\begin{align*}
 K(Q,P) + \log \int e^{f(x)} dP(x)
& = \int \log\left(\frac{dQ}{dP}\cdot Pe^f\right)dQ \\
& = \int  \log\left(\frac{dQ}{ e^f dP}\cdot e^f \cdot Pe^f\right)dQ
= \int  \log\left(\frac{dQ}{ dP'} \right)dQ + \int f dQ,
\end{align*}
with $dP'=e^f dP/(Pe^f)$. The results follows by using $\int  \log\left(dQ/ dP' \right)dQ=K(Q,P')\ge 0$.
\end{proof}

\begin{lem} \label{lemun}
\[ E_{\te_0}\hat{Q}L(P_\te^{(n)},P_{\te_0}^{(n)})
\le \inf_{a>0} \frac1a 
\left[ \inf_{Q\in \cS} E_{\te_0}K(Q,\Pi(\cdot\given X))+
\log E_{\te_0}\int e^{aL(P_{\te}^{(n)},P_{\te_0}^{(n)})}d\Pi(\te\given X)
\right].
\]
\end{lem}

\begin{proof}
One applies Lemma \ref{lemtr} with $Q=\hat{Q}$, $P=\Pi[\cdot\given X]$ and $f(\te)=a L(P_\te^{(n)},P_{\te_0}^{(n)})$ for a given $a>0$. Deduce, writing $P_\te=P_\te^{(n)}$ as shorthand,
\[
a \hat{Q}L(P_\te,P_{\te_0}) \le K(\hat{Q},\Pi[\cdot\given X])
+ \log \left[ \Pi[\cdot\given X]\left\{ e^{a L(P_\te,P_{\te_0})} \right\} \right].
 \] 
First, one takes the expectation under $E_{\te_0}$ and uses Jensen's inequality with the logarithm. Upon noting that $K(\hat{Q},\Pi[\cdot\given X])\le K(Q,\Pi[\cdot\given X])$ for any $Q\in\cS$ by definition, the result follows by dividing by $a$, taking the infimum over such $Q$'s, followed by the infimum over $a>0$.
\end{proof}

\begin{lem} \label{lemga2}
Suppose conditions (\underline{T}), (\underline{S}), (\underline{P}) hold. Then for $\la=\rho-1$ and for any $\veps>\veps_n$,
\[ E_{\te_0} \Pi( L(P_{\te}^{(n)},P_{\te_0}^{(n)})>C_1n\veps^2\given X)
\le e^{-Cn\veps^2} + e^{-\la n\veps^2}+2e^{-n\veps^2}.
 \]
\end{lem}

\begin{proof}
Writing $P_\te=P_\te^{(n)}$ as shorthand, let us set, with $\rho:=1+\la$,
\[ U_n = \left\{\te: L(P_\te,P_{\te_0}) > C_1n\veps^2 \right\},\quad 
 K_n =  \left\{\te: D_\rho(P_{\te_0},P_\te) > C_3n\veps^2 \right\},
\]
and, for $\tilde{\Pi}=\Pi|_{K_n}=\Pi(\cdot\cap K_n)/\Pi(K_n)$,
\[ A_n =  \left\{ \int \frac{dP_\te}{dP_{\te_0}}(X) d\tilde{\Pi}(\te) > e^{-(C_3+1)n\veps^2} \right\}.
\] 
Using the tests $\vphi_n$ in $(\underline{T})$, the quantity at stake for the Lemma is bounded by
\[ 
E_{\te_0}\Pi(U_n\given X)\le E_{\te_0}\left[\Pi(U_n\given X)(1-\vphi_n)1_{A_n}\right]
+P_{\te_0}A_n^c+E_{\te_0}\vphi_n.
\]
The last term is bounded by $e^{-Cn\veps^2}$ by $(\underline{T})$. 
Using Markov's inequality,
\begin{align*}
P_{\te_0}A_n^c & \le  P_{\te_0}\left[ 
\left\{ \int \frac{dP_\te}{dP_{\te_0}}(X) d\tilde{\Pi}(\te)\right\}^{-\la}
 > e^{\la(C_3+1)n\veps^2} \right]\\
 & \le e^{-\la(C_3+1)n\veps^2} E_{\te_0}\left\{ \int \frac{dP_\te}{dP_{\te_0}}(X) d\tilde{\Pi}(\te)\right\}^{-\la}\\
& \le e^{-\la(C_3+1)n\veps^2}  E_{\te_0}\left\{ \int \left(\frac{dP_\te}{dP_{\te_0}}(X)\right)^\la d\tilde{\Pi}(\te)\right\},
\end{align*}
where the last line uses Jensen's inequality and convexity of $x\to x^{-\la}$ on $\RR^+$.  Fubini's theorem implies
\[ E_{\te_0}\left\{ \int \left(\frac{dP_\te}{dP_{\te_0}}(X)\right)^\la d\tilde{\Pi}(\te)\right\}
= \int \int \frac{(dP_{\te_0})^{\la+1}}{(dP_\te)^{\la}} d\tilde{\Pi}(\te)
= \int e^{\la D_{1+\la}(P_{\te_0},P_{\te})} d\tilde{\Pi}(\te)\le e^{\la C_3n\veps^2},
\]
using the definition of $K_n$, on which $\tilde{\Pi}$ is supported. 
Putting the previous inequalities together gives $P_{\te_0}A_n^c \le e^{-\la n\veps^2}$. 
It remains to bound $\Pi(U_n\given X)(1-\vphi_n)$ on the event $A_n$. By definition, on $A_n$,
\[ \int  \frac{dP_\te}{dP_{\te_0}}(X)d\Pi(\te)\ge \Pi[K_n] 
\int \frac{dP_\te}{dP_{\te_0}}(X)d\tilde\Pi(\te)\ge \Pi[K_n] e^{-(C_3+1)n\veps^2}. \]
We have $\Pi[K_n] \ge e^{-C_2n\veps_n^2}\ge e^{-C_2n\veps^2}$ for $\veps>\veps_n$ using $(\underline{P})$. The last display is thus bounded from below by $e^{-(C_2+C_3+1)n\veps^2}$.  Using this fact, one can bound the denominator of Bayes' formula (written with $dP_\te/dP_{\te_0}$) from below to get
\begin{align*}
E_{\te_0}\left[ \Pi(U_n\given X)(1-\vphi_n)1_{A_n}\right]
& \le e^{(C_2+C_3+1)n\veps^2} E_{\te_0}\left[ \int_{U_n} \frac{dP_\te}{dP_{\te_0}}(X) (1-\vphi_n) d\Pi(\te) \right] \\
& \le e^{(C_2+C_3+1)n\veps^2} \int_{U_n} P_{\te}(1-\vphi_n) d\Pi(\te).
\end{align*}
Let us further bound from above, using $(\underline{T}), (\underline{S})$, 
\begin{align*}
\int_{U_n} P_{\te}(1-\vphi_n) d\Pi(\te) & \le \Pi\left[ \Theta_n(\veps)^c \right]
+ \int_{U_n \cap \Theta_n(\veps)} P_{\te}(1-\vphi_n) d\Pi(\te)
\le e^{-Cn\veps^2} + e^{-Cn\veps^2}.
\end{align*}
Putting the previous inequalities together yields, provided $C>C_2+C_3+2$,
\[ E_{\te_0}\left[\Pi(U_n\given X)(1-\vphi_n)1_{A_n}\right]
 \le 2e^{(C_2+C_3+1)n\veps^2 -Cn\veps^2}\le 2e^{-Cn\veps^2}.\]
Combining all previous bounds gives the result.
\end{proof}
 
\begin{lem} \label{lemlap}
Suppose the random variable $X$ verifies 
\[ P(X\ge t) \le c_1e^{-c_2 t}\qquad \text{for any } t\ge t_0>0.\]
Then for any $a\in(0,c_2/2]$, 
\[ E e^{aX} \le e^{at_0}+c_1. \]
\end{lem}
\begin{proof}
Using the formula $EY\le M+\int_M^\infty P[Y\ge y]dy$ for $Y=e^{aX}$ and the assumption,
\[ E[e^{aX}] \le M+\frac{c_1}{c_2-a} a M^{1-(c_2/a)}. \]
Setting $M=e^{at_0}$ and using $a\le c_2/2$, the former is bounded by $M+c_1(a/a)e^{a-c_2}\le M+c_1$.
\end{proof}

\subsubsection*{Proof of the main results}

\begin{proof}[Proof of Theorem \ref{thmvar}]
From Lemma \ref{lemga2}, one deduces that for any $t\ge t_0=C_1n\veps_n^2$,
\[ E_{\te_0}\Pi[L(P_\te,P_{\te_0})>t \given X] \le C_1e^{-C_2t}. \]
One then uses Lemma \ref{lemlap} to deduce, for small $a$, that
\[ E_{\te_0}\left\{\Pi[\cdot\given X]\left[ e^{aL(P_\te,P_{\te_0})}\right]\right\}\le 4+e^{aC_1n\veps_n^2}\]
for small $a$. The result now follows from an application of Lemma \ref{lemun}.
\end{proof}

\begin{proof}[Proof of Theorem \ref{thmmf}]
Invoking Lemma \ref{lemga}, it is enough to find $Q\in \cS_{MF}$ such that
\[ QK(P_{\te_0},P_\te) + K(Q,\Pi) \le (C_1+C_2+C_3)\veps_n^2.  \]
Define $Q = \bigotimes_{j=1}^m Q_j$, with $Q_j = \tilde{Q}_j|_{\tilde{\Theta}_j}$ the restriction of $\tilde{Q}_j$ to $\tilde{\Theta}_j$, both defined in the statement of the lemma. By definition $Q\in \cS_{MF}$ and $\text{Supp}(Q)\subset \bigotimes_{j=1}^m \tilde{\Theta}_j$. 
\end{proof}

\section{High-dimensional regression}

Consider the high-dimensional regression model 
\[ Y = X\te + \veps,\]
where the notation is as in Chapter \ref{chap:ada2}. Here for simplicity we focus on one example of design matrix, namely we assume it has independent Gaussian entries
\begin{equation} \label{condx}
X_{ij} \sim \cN(0,1) \quad\text{iid}.
\end{equation} 
The results below also hold under conditions on $X$ similar in spirit to those stated in Chapter \ref{chap:ada2} (and related to those in \cite{spahd}). \\

Let us consider a subset-selection prior on $\te$ as in Section \ref{sec:gene} (with here $p$ instead of $n$) 
\begin{equation} \label{subsprior2}
k \sim \pi_p, \qquad  S\given k  \sim \text{Unif}(\cS_k), \qquad 
\te\given S \sim \bigotimes_{i\in S} \Gamma \,\otimes\, \bigotimes_{i\notin S} \delta_0, 
\end{equation}
with here $\Ga=\text{Lap}(\la)$ a Laplace distribution with parameter $\la$.
 Suppose the following slightly faster than exponential decrease: there exist constants $A_1, \ldots, A_4>0$ with  
\begin{equation} \label{condpvb}
 A_1p^{-A_3}\pi_p(s-1)\le \pi_p(s)\le A_2p^{-A_4}\pi_p(s-1),
\end{equation} 
for $s=1,\ldots,p$. This condition is satisfied for instance for the following hierarchical Bayes version of the spike and slab prior, for some fixed $u>1$ and $\la>0$,
\begin{align*}
\al & \sim \text{Beta}(1,p^u)\\
\te=(\te_i)_{1\le i\le p} \given \al & \sim \bigotimes_{i=1}^p\, (1-\al)\delta_0 + \al \text{Lap}(\la),
\end{align*}

As a variational class, let us consider the {\em mean-field} spike and slab class
\begin{equation} \label{mfhd}
\cP_{MF} = \left\{P_{\mu,\sigma,\ga}
=\bigotimes_{i=1}^p \, (1-\ga_i)\delta_0 + \ga_i \cN(\mu_i,\sigma_i^2),\quad  \mu_i\in\RR, \sigma_i\in\RR^+, \ga_i\in[0,1]  \right\}.
\end{equation}
Define the corresponding variational Bayes posterior distribution
\begin{equation} \label{tildepivb}
\tilde{\Pi}\, =\, \underset{P_{\mu,\sigma,\ga}\in\cP_{MF}}{\text{argmin}}\ \, K(P_{\mu,\sigma_\ga} ,  \Pi(\cdot\given Y) ).
\end{equation}
By taking the mean-field class \eqref{mfhd} in this context, one enforces substantial independence in the variational posterior distribution, with a much reduced complexity in terms of models: there are only $p$ inclusion variables in \eqref{mfhd} instead of $2^p$ models the posterior puts mass on. Note also that while one may choose a Gaussian distribution for slabs from the variational class, it is important to keep a Laplace slab in the prior itself (otherwise one may face over-shrinkage).

\begin{thm} \label{thmvbhd}
 Let the prior $\Pi$ be a subset-selection prior as in  \eqref{subsprior2}  that satisfies \eqref{condpvb} with slab $\Ga$ in \eqref{subsprior2} a standard Laplace 
  variable. Suppose also that $s_n=o(\sqrt{n/\log{n}})$. Then, on an event of overwhelming probability under the law of $X$ as in \eqref{condx}, for $M_n$ going to infinity arbitrarily slowly,
\[ \sup_{\te_0\in\ell_0[s_n]} E_{\te_0}\tilde{\Pi}
\left(\te: \|\te-\te_0\|_2 > M_n \sqrt{ \frac{s_n\log{p}}{n} } \right) = o(1). \]
\end{thm}

Theorem \ref{thmvbhd} is a special case of Theorem 1 in \cite{rs22}, obtained by using the conditions on $X$ assumed therein are verified for the design \eqref{condx} with overwhelming probability (see e.g. \cite{spahd}, Section 2.2 for a discussion and further references). On the other hand, under the same conditions on $X$ and for the same prior on $\Pi$, it follows from \cite{spahd}, Theorem 2, that the original posterior distribution $\Pi[\cdot\given Y]$ converges towards $\te_0$ in $\|\cdot\|_2$ norm at the same rate $\sqrt{s_n\log{p}/n}$ which can be shown to be (near)-optimal in this setting. 
\\

This shows that the variational Bayes approximation $\tilde\Pi[\cdot\given Y]$ converges at the same rate as the original posterior $\Pi[\cdot\given Y]$. An advantage here of the VB-posterior is that this approximation is quite fast to compute, while sampling from the original posterior typically requires the use of MCMC algorithms that scale significantly slower in terms of dimension.\\

For more details on the proposed variational algorithm (used to solve the optimisation problem, i.e. finding the best approximant in the considered mean-field class) we refer to \cite{rs22}; the method is implemented in the R package \textsf{sparsevb} \cite{csr21}, which covers both linear and logistic regression.

\addtocontents{toc}{\protect\vskip20pt}

\appendix

\chapter{Appendix} \label{chap:app}

\section{Distances between probability measures} \label{app-dist}

\begin{definition}\label{defhel}
Let $P,Q$ probability distributions dominated by a measure $\mu$, i.e. $dP=pd\mu$ and $dQ=qd\mu$. The \sbl{$L^1$--distance} is defined as 
\[ \|P-Q\|_1 = \int |p-q| d\mu \]
and the \sbl{Hellinger distance} as 
\[ h(P,Q) = \left( \int (\sqrt{p}-\sqrt{q})^2 d\mu \right)^{1/2}. \]
\end{definition}

These distances verify the following properties (left as an exercise)
\begin{itemize}
\item  $\|P-Q\|_1\le 2$ and $h(P,Q) \le \sqrt{2}$.
\item $\|P-Q\|_1\le 2 h(P,Q)$ [use Cauchy-Schwarz]
\item If $\max(p,q)\ge c_0>0$ then $h(P,Q)\le C\|P-Q\|_1$ for some $C>0$.
\item Defining the \sbl{total variation norm} (between measures defined on a common $\sigma$--field $\cA$) as $\|P-Q\|_{TV}=\sup_{A\in\cA}|P(A)-Q(A)|$,
\[ \| P-Q\|_1 = 2\|P-Q\|_{TV}. \]
\end{itemize}

\begin{lem}[\sbl{Total variation distance}] \label{tvl1}
Let $P, Q$ be two probability measures defined on a joint $\sigma$--field $\cA$ and dominated by $\mu$, that is $dP=pd\mu, dQ=qd\mu$. 
The total variation distance $\|P-Q\|_{TV}=\sup_{A\in\cA}|P(A)-Q(A)|$ verifies
\begin{align*}
 2\|P-Q\|_{TV} & = \|p-q\|_1.
\end{align*} 
Also, the supremum  defining the total variation distance is attained for  $A=\{x:\ q(x)<p(x)\}$.
\end{lem}
\begin{proof}
Let $A$ denote the set $A=\{x:\ q(x)<p(x)\}$ and $A^c$ its complement. Then
\begin{align*}
\|p-q\|_1
& = \int_A (p-q)d\mu + \int_{A^c} (q-p)d\mu -\int_{p=q}(q-p)d\mu \\
& = P(A)-Q(A) + Q(A^c)-P(A^c) =2(P(A)-Q(A))=2\int_{A} (p-q)d\mu.
\end{align*}
By symmetry, one also has $\|p-q\|_1=2\int_{q>p} (q-p)d\mu$.
On the other hand, for any $B\in\cA$, 
\[ P(B)-Q(B)=\int \1_B (p-q)d\mu\le \int_{p>q} \1_B(p-q)d\mu = \int_{A} (p-q)d\mu=\|p-q\|_1/2.\]
By symmetry, $Q(B)-P(B)\le \int_{q>p} (q-p)d\mu=\|p-q\|_1/2$. Combining all these facts gives the result.  
\end{proof}
\vp

For any $\al\in(0,1)$, the \sbl{$\al$--R\'enyi divergence} between distributions $P,Q$ having densities $p$ and $q$ with respect to $\mu$ is defined as 
\begin{align*}
		D_\alpha(P, Q) = -\frac{1}{1-\alpha}\log \left( \int p^\alpha q^{1-\alpha}d\mu \right).
	\end{align*}

\begin{lem}[R\'enyi divergence between Gaussians] \label{lem:renyig}
Fix $\rho\in(0,1)$ and let $\mu, \nu\in\RR^p$ and $\sigma, \ta>0$, then
\[ D_\rho\left( \cN(\mu,\sigma^2), \cN(\nu,\tau^2) \right) 
=\rho \frac{(\mu-\nu)^2}{2\sigma_\rho^2} + \frac{1}{1-\rho}\log\frac{\sigma_\rho}{\sigma^{1-\rho}\ta^\rho},
\]
where $\sigma_\rho^2=(1-\rho)\sigma^2+\rho\ta^2$ (see e.g. \cite{vanerven}, Eq. (10)).
\end{lem}

\sbl{Bounded--Lipschitz metric.} The space of bounded Lipschitz functions on the metric space $(H,d)$ is \[ BL_H(1)=\left\{f: H \to \mathbb R,\quad \sup_{s \in H}|f(s)| + \sup_{s \ne t, s,t \in H } |f(s)-f(t)|/d(s,t) \le 1 \right \}. \]
The bounded--Lipschitz metric on probability measures on $H$ is $\beta:=\beta_H$ defined as
\begin{equation} \label{def-weakcv}
\beta_H(\mu, \nu) = \sup_{u \in BL_H(1)}\left|\int_H u(s)(d\mu-d\nu)(s)\right|. 
\end{equation}
 The metric  $\beta$ metrises weak convergence of probability measures on $H$, see  e.g. \cite{D02}, Theorem 11.3.3. \\ 
 
\sbl{Uniformity classes for weak convergence. } 
Suppose one has shown $\be_H(\mu_n,\cN)\to0$ as $n\to\infty$ for a sequence of probability measures on $H$ and $\cN$ a fixed distribution on $H$. 
We call a family $\mathcal U$ of measurable real-valued functions defined on $H$ \sbl{a $\cN-$uniformity class for weak convergence} if for any sequence $\mu_n$ of Borel probability measures on $H$ that converges weakly to $\cN$ we also have 
\begin{equation} \label{cunif}
\sup_{u \in \mathcal U}\left|\int_H u(s)(d\mu_n-d\cN)(s)\right| \to 0
\end{equation}
as $n \to \infty$.  For any subset $A$ of $H$, define  the $\delta$-boundary of $A$ by $\partial_\delta A = \{x \in H: d(x,A)<\delta, d(x,A^c)<\delta\}.$ By Theorem 2 of Billingsley and Tops{\o}e \cite{BT67}, a family $\mathcal A$ of measurable subsets of $H$ is a $\cN$-uniformity class if and only if  
\begin{equation} \label{nunif}
\lim_{\delta \to 0} \sup_{A \in \mathcal A} \cN(\partial_\delta A)=0.
\end{equation}
This typically allows for enough uniformity to deal with a variety of concrete nonparametric statistical problems. 
By general properties of Gaussian measures on separable Banach spaces, the collection of all centered balls (or rather, in a $L^2$-perspective, ellipsoids) for the $\|\cdot\|_H$-norm verify \eqref{nunif} and thus form a $\cN$-uniformity class.

\section{Inequalities}

\begin{lem}[\sbl{Hoeffding's inequality}, see e.g. \cite{blmbook}] \label{hoeff}$\ $
Let $Z_i$ be independent random variables with $a_i\le Z_i\le b_i$ for reals $a_i,b_i$ and $1\le i\le n$. Then
\[ P\left[ \sum_{i=1}^n Z_i >t \right] \le \exp\left\{-\frac{2t^2}{\sum_{i=1}^n (b_i-a_i)^2}\right\}. \]
\end{lem}

\begin{lem} \label{lem:dist}
For any bounded functions $v,w$, if one denotes $p_v=e^v/\int_0^1 e^{v(u)} du$, 
\begin{align*}
h(p_v,p_w) & \le  \|v-w\|_\infty e^{\|v-w\|_\infty/2}\\
K(p_v,p_w) & \leqa \|v-w\|_\infty^2 (1+\|v-w\|_\infty)e^{\|v-w\|_\infty}\\
V(p_v,p_w) & \leqa \|v-w\|_\infty^2 (1+\|v-w\|_\infty)^2e^{\|v-w\|_\infty}.
\end{align*}
\end{lem}

\begin{proof}
We prove the first inequality. For the second and third, we refer to \cite{vvvz} (or to the book \cite{gvbook}). 
\begin{align*}
h(p_v,p_w) & = \| \frac{e^{v/2}}{\|e^{v/2}\|_2} - \frac{e^{w/2}}{\|e^{w/2}\|_2} \|_2 \\
& = \| \frac{e^{w/2}-e^{v/2}}{\|e^{w/2}\|_2} + e^{v/2}\left( \frac{1}{\|e^{w/2}\|_2}-  
\frac{1}{\|e^{v/2}\|_2}\right) \|_2 \\
& \le 2 \frac{\| e^{w/2}-e^{v/2}\|_2}{\|e^{w/2}\|_2}.
\end{align*}
One can also bound from above
\begin{align*}
| e^{v/2}-e^{w/2}| & = e^{w/2} |e^{v/2-w/2}-1| \\
& \le e^{w/2}\|\frac{v-w}{2}\|_\infty e^{\|v-w\|_\infty/2},
\end{align*}
where one uses the inequality $|e^x -1|\le |x|e^{|x|}$, valid for all $x>0$. Combining the previous bounds, 
\[ h(p_v,p_w)^2 \le \frac{\int e^{w+\|v-w\|_\infty} \|v-w\|_\infty^2}{\int e^w}, \]
which is no more than $\|v-w\|_\infty^2 e^{\|v-w\|_\infty}$, as requested.
\end{proof}

\section{Testing and Entropy} \label{app:testing}

\subsubsection*{Tests in density estimation}

\begin{proof}[Proof of Theorem \ref{thm:tests} for $d=\|\cdot\|_1$] Let $f_0, f_1$ two densities with $\|f_0-f_1\|_1 > \veps$. Let $\mathbb{P}_n=n^{-1}\sum_{i=1}^n \delta_{X_i}$ denote the empirical measure associated to $X_1,\ldots X_n$ and for a measurable set $B\subset [0,1]$, let
\[ \mathbb{P}_n(B) = \frac1n \sum_{i=1} \1_{X_i\in B}.\]
Let $A$ denote the set, abbreviated as $A= \{f_0<f_1\}$, 
\[ A =\{x:\ f_0(x)<f_1(x)\} \]
Let us define the test
\[ \vphi_n  = \1\left\{\mathbb{P}_n(A) > P_{f_0}(A) + \frac{\|f_0-f_1\|_1}{3}\right\}. \]
The term $E_{f_0}\vphi_n$, also called type I--error of the test, is bounded by
\begin{align*}
E_{f_0}\vphi_n & = 
 P_{f_0}\left[ \sum_{i=1}^n (\1_{X_i\in A}-P_{f_0}(A)) >  \|f_0-f_1\|_1/3\right] \\
 & \le \exp\{ -Cn\|f_0-f_1\|_1^2 \}\le e^{-Cn\veps^2},
\end{align*}
where we use Hoeffding's inequality Lemma \ref{hoeff} and $\| f_0-f_1\|_1>\veps$.\\

Let us now consider the term $E_{f}(1-\vphi_n)$, also called type II--error of the test, for $f$'s in the ball $\{f:\ \|f-f_1\|<a\veps\}$ with $a=1/5$.
\[ E_f(1-\vphi_n) = 
P_{f_0}\left[ \mathbb{P}_n(A)-P_f(A) \le P_{f_0}(A)-P_f(A)+\|f_0-f_1\|_1/3\right]
\]
We now claim that with $f$ chosen as above the last display, the term $P_{f_0}(A)-P_f(A)$ is at most $-D\|f_0-f_1\|_1$ for suitably large $D>0$.
\[ P_{f_0}(A)-P_f(A)
= P_{f_0}(A)-P_{f_1}(A) + P_{f_1}(A)-P_f(A)=:(i)+(ii).
\] 
The choice of $A$ ensures $(i)=-\|f_0-f_1\|_1/2$ by Lemma \ref{tvl1}, which also implies $|(ii)|\le \|P_{f_1}-P_f\|_{TV}=\|f_1-f\|_1/2\le a\veps/2=\veps/10\le \|f_0-f_1\|_1/10$. So $(i)+(ii)\le  -(2/5)\|f_1-f_0\|_1$. As $-2/5+1/3=-1/15$, one obtains
\[
E_f(1-\vphi_n) \le 
P_{f_0}\left[ \mathbb{P}_n(A)-P_f(A) \le - \|f_1-f_0\|_1/15 \right] 
\le e^{-c'n\|f_1-f_0\|_1^2}\le e^{-c'n\veps^2},
 \]
invoking Hoeffding's inequality (Lemma \ref{hoeff}) again with $c'>0$ a suitably small constant.  
\end{proof}

\subsubsection*{Tests in the Gaussian sequence model}

Let us now see an example of verification of the testing condition (T) in the non--iid setting: recall the Gaussian sequence model where $X=(X_k)_{k\ge 1}$ is observed and, for $\veps_k$ iid $\cN(0,1)$ and $\te_k$ a given sequence in $\ell^2$,
\[ X_k = \te_k +\frac{\veps_k}{\rn},\quad k\ge 1.\]
For sequences $a=(a_k)$ and $b=(b_k)$ let us write, provided the corresponding series converge 
\[ \psg a,b\psd = \sum_{k\ge 1} a_k b_k, \quad \|a\|^2 = \sum_{k\ge 1} a_k^2. \] 
\begin{lem} \label{lem:testseq}
Let $\te, \te_1, \te_0$ be squared--integrable sequences, and let for $r=\|\te_0-\te_1\|/4$,
\[ B(x_1,r)= \{\te:\ \|\te-\te_1\| \le r\}. \]
The test $\vphi_n=\1\{2 \psg \te_1-\te_0,X\psd > \|\te_1\|^2-\|\te_0\|^2\}$ verifies, for $\bar\Phi(u)=P(\cN(0,1)>u)$, 
\begin{align*}
E_{\te_0}\vphi_n & \le \bar\Phi(\rn\|\te_1-\te_0\|/2),\\
\sup_{\te\in B(\te_1,r)}\, E_\te(1-\vphi_n) & \le \bar\Phi(\rn\|\te_1-\te_0\|/4)
\end{align*}
In particular, condition (T) is verified.
\end{lem}
\begin{proof}
For the first inequality, one works under $E_{\te_0}$
\begin{align*}
P_{\te_0}[2\psg \te_1-\te_0,X\psd > \|\te_1\|^2-\|\te_0\|^2] & 
= P[2\psg \te_1-\te_0,\te_0\psd + 2\psg \te_1-\te_0,\veps/\rn\psd > \|\te_1\|^2-\|\te_0\|^2] \\
& = P[2\psg \te_1-\te_0,\veps \psd > \rn \|\te_0-\te_1\|^2] = P[\cN(0,1)
>\rn \|\te_0-\te_1\|/2],
\end{align*}
where the last line uses that $\psg \te_1-\te_0,\veps \psd$ is a random variable of distribution $\cN(0,\|\te_1-\te_0\|^2)$, which gives the first inequality. For the second, one works this time under $E_\te$, for $\te\in B(\te_1,r)$. 
 \begin{align*}
P_\te[2\psg \te_1-\te_0,X\psd \le \|\te_1\|^2-\|\te_0\|^2] & 
= P[2\psg \te_1-\te_0,\te\psd + 2\psg \te_1-\te_0,\veps/\rn\psd 
\le \|\te_1\|^2-\|\te_0\|^2].
\end{align*}
We now write $\psg \te_1-\te_0,\te\psd = 
\psg \te_1-\te_0,\te-\te_0\psd + \psg \te_1-\te_0,\te_0 \psd$.   By writing $\te=\te_1+rv$ with $\|v\|\le 1$, 
\[\psg \te_1-\te_0,\te-\te_0\psd = \|\te_1-\te_0\|^2
+r \psg \te_1-\te_0,v\psd\ge \|\te_1-\te_0\|^2-r\|\te_1-\te_0\|\ge \frac34 \|\te_1-\te_0\|^2,
 \]
 where the last line uses Cauchy--Schwarz' inequality and $\|v\|\le 1$. Rearranging the probability at stake,
 \begin{align*}
  E_\te[1-\vphi_n] &\le P\left[ 2\psg \te_1-\te_0,\veps/\rn\psd \le \|\te_1-\te_0\|^2-\frac32\|\te_1-\te_0\|^2\right]\\
& \le P[\cN(0,\|\te_1-\te_0\|^2)\le -\rn \|\te_1-\te_0\|^2/4]=\bar\Phi(\rn\|\te_1-\te_0\|/4)
 \end{align*}
as desired. Property (T) now immediately follows for $\|\te_1-\te_0\|>\veps$ by using the standard inequality $\bar\Phi(u)\le e^{-u^2/2}$ for $u>0$.  
\end{proof}

\subsubsection*{Tests in nonparametric regression}

In the Gaussian white noise model, a similar proof as in the sequence model above shows that the test, with $\|\cdot\|_2$ denoting the $L^2$--norm on functions,
\[ \vphi_n =\1\{2 \int_0^1 (f_1-f_0)(t)dX^{(n)}(t)
>  \|f_1\|^2-\|f_0\|^2\} \} \]
verifies the conclusions of Lemma \ref{lem:testseq}. Also, $B_n(f_0,\veps_n)=\{f:\ \|f-f_0\|_2\le \veps_n\}$ is again an $L^2$--ball.\\

In the nonparametric regression with fixed design, similar properties hold as in the sequence model, upon replacing $\psg\cdot,\cdot\psd$ and $\|\cdot\|$ by 
\[\psg f,g\psd = \frac1n \sum_{i=1}^n f(t_i)g(t_i),\quad
\|f\|_n^2 = \frac1n\sum_{i=1}^n f(t_i)^2. \]
\ma{Exercise.} Establish an analogue of Lemma \ref{lem:testseq} in this case.

\subsubsection*{Entropy of unit ball}

For $y\in\RR^k$, a radius $R\ge 0$, and for $\|x\|^2=\sum_{i=1}^k x_i^2$ the standard euclidian norm, let
 \[ B_{\RR^k}(y,R)=\{x\in\RR^k:\ \|x-y\|\le R\}\] 
 denote the euclidian ball of center $y$ and radius $R$.  

\begin{lem}[\sbl{Entropy of unit ball in $\RR^k$}] \label{lem:entub}
For any $\delta>0$, for any $M>0$, the covering number of $B_{\RR^k}(0,M)$ with respect to the euclidean norm verifies, with $a\vee b=\max(a,b)$,
\[ N(\delta,B_{\RR^k}(0,M),\|\cdot\|) \le \left(1 \vee \frac{3M}{\delta}\right)^k. \]
\end{lem}

\begin{proof}
If $\delta\ge M$, the result is clear: one ball suffices to cover, so one assumes $\delta<M$. By applying an homothecy of ratio $M$, which multiplies norms by $M$, we 
see that it is enough to consider the case $M=1$ with $\delta< 1$ (up to setting $\delta'=\delta/M$). Let 
\[ N:= N(\delta,B,\|\cdot\|),\quad\text{with } B:=B_{\RR^k}(0,1).\]
Let $N'=N_s(\delta)$ denote the {\em maximal} number of points of $B$ separated by at least $\delta$ for $\|\cdot\|$.  Consider a collection of such points $(x_i,\ 1\le i\le N')$ and note that the collection of balls $B(x_i,\delta)=B_{\RR^k}(x_i,\delta)$ must cover $B$, otherwise one could find a point $y$ separated from all the $x_i$'s by at least $\delta$, contradicting maximality.  On the other hand, since $x_i$'s are in $B$,
\[ B(x_1,\frac{\delta}{2})\bigcup \cdots\bigcup B(x_{N'},\frac{\delta}{2})\ \subset B(0,1+\frac{\delta}{2}).\]
Also, the balls $B(x_i,\frac{\delta}{2})$ are disjoint by definition of $\delta$--separation. Denoting by $\mathscr{V}(A)$ the volume of a measurable subset $A$ of $\RR^k$ with respect to Lebesgue measure, one deduces
\[ \sum_{i=1}^{N'} \mathscr{V}(B(x_i,\delta/2)) \le \mathscr{V}(B(0,1+\delta/2)).\]
A change of variables gives $\mathscr{V}(B(0,r))=r^k \mathscr{V}(B(0,1))$ for $r>0$. Since $\mathscr{V}(B(x_i,\delta/2)) =\mathscr{V}(B(0,\delta/2))$, one obtains
\[ N' (\delta/2)^k \mathscr{V}(B(0,1)) 
\le  \left(1+\frac{\delta}{2}\right)^k \mathscr{V}(B(0,1)).
\]  
One concludes that $N\le N'\le (\frac{2+\delta}{\delta})^k\le (3/\delta)^k$ as announced.
\end{proof}

\section{Gaussian processes} \label{app:gps}

\subsection{Definitions and examples}

We now define two notions: the one of \sbl{Gaussian process}, interpreted as a collection of normal random variables, and the one of (Banach--valued) Gaussian random variable. In standard settings such as within separable Banach spaces,  both notions are essentially equivalent, as we briefly discuss below. 

 \begin{definition} \label{def-gpp}
A \sbl{Gaussian process} $W=(W_t)_{t\in T}$ is a stochastic process (i.e. a collection of random variables) indexed by the set $T$ such that for any $t_1,\ldots,t_k\in T$ and any $k\ge 1$, the vector $(W_{t_1},\ldots,W_{t_k})$ is a Gaussian random vector.
\end{definition}
In the sequel to fix ideas we take $T=[0,1]$ as index set. 
Recall that a Gaussian vector is characterised by its mean and variance--covariance matrix. So, a Gaussian process (we also write GP) $W$, if it exists (we assume so), must be characterised by the quantities
\begin{align*}
\mu(t) & = E[W_t]\\
K(s,t) & = \text{Cov}(W_s,W_t)=E[(W_s-EW_s)(W_t-EW_t)].
\end{align*}
Those are called respectively {\em mean function} and {\em covariance} (operator).\\

In the sequel we restrict for simplicity to {\em centered} Gaussian processes, i.e. we   take $\mu(t)=0$ for all $t$. The map $(s,t)\to K(s,t)$ is symmetric and {\em definite--positive} in that for any finite collection $t_1,\ldots,t_k\in[0,1]$, the matrix $(K(t_i,t_j))$ is definite positive. This follows immediately from the expression of $K(\cdot,\cdot)$. \\

Let us insist again that the definition above gives the finite--dimensional distributions (also called FIDIs) $(W_{t_1},\ldots,W_{t_k})$ but we shall not construct a process $W(t)=W_t$ that verifies this (it is possible to do so).\\

The map $t\to W(t)$ is called \sbl{trajectory} (or realisation) of $W$. It can then happen that the process admits a \sbl{version} (that is, there exists $(Z_t)_{t\in T}$ with $P[Z_t=W_t]$ for any $t\in T$, which imples the FIDIs are the same) whose trajectories are continuous, that is $t\to Z_t(\omega)$ is continuous. Then the image of the map $\omega\to (Z_t(\omega))_t$ is included in $\cC^0[0,1]$. More generally, $W$ may have a version that has trajectories in a {\em separable Banach space} $\mb$, such as $(\cC^0[0,1],\|\cdot\|_\infty)$ above. 

\begin{definition} \label{def-gpv}
A \sbl{$\mb$--valued Gaussian random variable} is a map $W:\Omega\to \mb$, measurable for the Borel $\sigma$--field of $\mb$, such that for any $b^*$ in the dual space $\mb^*$, the real variable $b^*W$ is Gaussian.
\end{definition}

Note that if $Y$ is a Gaussian variable in $\mb$ according to this definition, then if $\cT\subset \mb^*$, the collection $(b^*Y,\ b^*\in\cT)$ is a Gaussian process indexed by $\cT$ (this is because: a random vector is Gaussian iff any finite linear combination of its coordinates is Gaussian, and: $\mb^*$ is a linear space).  For example, if $\mb=\cC^0[0,1]$ as above, and $\cT$ is the set of linear maps $b_t:f\to f(t)$ for $f\in\mb$ (they are continuous, therefore in $\mb^*$), then the collection of variables $(W(\omega)(t),\ t\in[0,1])$ is a Gaussian process on $[0,1]$ with trajectories in $\mb$.\\

Under a measurability condition, we have that if $(W_t)$ is a Gaussian process with trajectories in $\mb$, then it is also a Gaussian random variable in $\mb$. 

\begin{lem} 
Si $(W_t)_{t\in T}$ has a version that admits trajectories in a separable Banach space $\mb$, and if for any $w\in\mb$, the quantity $\|W-w\|_\mb$ is a random variable (i.e. is a measurable quantity as function of $\omega$), then $W$ is a Gaussian random variable in $\mb$.
\end{lem}
We omit the proof: it is based on the fact that when $\mb$ is separable the Borel $\sigma$--field is generated by balls. In the sequel, since we generally work with separable Banach spaces, we will use interchangeably both concepts.

\begin{definition} \label{smallb}
For $W$ taking values in the separable Banach space $(\mb,\|\cdot\|_\mb)$, we call \sbl{small ball probability} the quantity, for $\veps>0$,
\[ P[\|W\|_\mb<\veps]=: \exp(-\vphi_0(\veps)). \]
Sometimes $\vphi_0(\veps)=-\log P[\|W\|_\mb<\veps]$ is called small--ball term.
\end{definition}

\subsection{RKHS of a Gaussian process}

Let $(W_t)_{t\in T}$ be a Gaussian process. Consider the space
\[ \cC_W=\overline{\text{Vect}\{W(t),\ t\in T\}}^{L^2}
= \overline{\left\{ \sum_{i=1}^p\al_i W(t_i),\ t_i\in T, p\ge 1\right\}}^{L^2},\]
where $\overline{\cA}^{L^2}$ stands for the completion of given set $\cA$ of random variables (defined on a common probability space $\Omega$) in $L^2(\Omega)$. 
This space is called the {\em first order chaos} associated to $W$.

\begin{definition} \label{rkhs}
The RKHS of a centered Gaussian process $W=(W_t)_{t\in T}$ is the set 
\[ \mh = \left\{ g_H:T\to\RR,\ \ g_H(t)=E[W_t H], \ H\in\cC_W \right\}. \]
This set is equipped with the inner product $\psg \cdot,\cdot \psd_\mh$ given by, 
for $H_1, H_2\in \mh$,
\[ \psg g_{H_1} , g_{H_2} \psd_\mh = E[H_1 H_2]. \]
The space $\mh$ is a Hilbert space called Reproducing Kernel Hilbert Space (RKHS) of $W$.
\end{definition}
Note that $\mh$ is indeed a Hilbert space since the map $H\to g_H$, mapping $\cC_W$ into $\mh$, is by definition an isometry, and $\cC_W$ is a Hilbert space as closed sub--space of $L^2$.\\

We now list a number of useful properties
\begin{itemize}
\item If $H=W_s$, then $g_H(t)=E(W_tW_s)=K(s,t)$ so $g_H(\cdot)=K(s,\cdot)$ in this case.
\item If $H=\sum_{i=1}^p a_i W_{s_i}$, then similarly $g_H(\cdot)=\sum_{i=1}^p K(s_i,\cdot)$. 
\item For $H\in\cC_W$ arbitrary, we have
\[ g_H(t)=E[W_tH]=\psg K(t,\cdot),g_H(\cdot)\psd_\mh\]
This is sometimes called the \sbl{reproducing formula}, as it expresses the value of any element of the RKHS at a point as an inner product involving the function itself.
\item Paralleling the definition of $\cC_W$ and thanks to the isometry $H\to g_H$, we have 
\[ \overline{ \left\{ \sum_{i=1}^p\al_i K(s_i,\cdot),\ s_i\in T, p\ge 1\right\} }^{\mh}=\mh. \]
\item In case $W$ has trajectories in a separable Banach space $\mb$, it can be proved that $\mh$ identifies with a subspace of $\mb$. We will see this is indeed the case in the examples investigated below.
\end{itemize}

Let $\cS_W$ denote the space $\text{Vect}\{W(t),\ t\in T\}$, so that $ \cC_W=\overline{\cS_W}^{L^2}$. \\

{\em Example 1 (Gaussian vector in $\RR^k$).}  Let $W=(W_1,\ldots,W_k)^T$ be a (column) Gaussian vector. It is a Gaussian process indexed by $T=\{1,\ldots,k\}$. Here we will identify a vector in $\RR^k$ and a function $T=\{1,\ldots,k\}\to \RR$. \\

If $W\sim\cN(0,\Sigma)$ with an invertible covariance matrix $\Sigma$ then, for any $u,v\in\RR^k$,
\begin{equation} \label{rkhs:vec} 
(\mh,\|\cdot\|_\mh)=(\RR^k,\|\cdot\|_\mh),\qquad \psg u,v \psd_\mh
=u^T \Sigma^{-1}v.
\end{equation}
This means that the RKHS of a Gaussian vector in $\RR^k$ coincides with the ambient space, but with a twisted geometry that is given by the inner product as above. Let us check this: for $H=\sum_{j=1}^k a_j W_j$ in $\cS_W$ and $a=(a_1,\ldots,a_k)$, $W=(W_1,\ldots,W_k)$,
\[ g_H(i)=E[HW_i]=\sum_{j=1}^k a_j E[W_jW_i]=\sum_{j=1}^k a_j \Sigma_{i,j}=(\Sigma a)_i, \]
so that $g_H$ can be identified with the vector $\Sigma a$. In particular, as $\Sigma$ is invertible, any vector $v\in\RR^k$ is in $\mh$, as $v=\Sigma a_v$ with $a_v=\Sigma^{-1}v$. Now $v=E[WW^T]a_v=E[W W^Ta_v]$. For $u,v\in\RR^k$,
\[ \psg u,v \psd_\mh=E[W^T a_u W^T a_v]=a_u^T E[WW^T] a_v=a_u^T\Sigma a_v= u^T \Sigma^{-1}v, \]
using that for any real number  $x^T=x$ (here $x=W^T a_u$),  which proves \eqref{rkhs:vec}. \\

{\em Example 2 (Random series).} More generally, if one now considers
\[ W_t=\sum_{j=1}^{\infty} \sigma_j \zeta_j e_j(t),\qquad t\in[0,1],\]
with $(\sigma_j)\in\ell^2$, $\zeta_j$ independent $\cN(0,1)$ variables and 
$\{e_j(\cdot)\}$ an orthonormal basis of $L^2[0,1]$, it can be shown that
\[ \mh = \left\{ \sum_{j=1}^{\infty} \la_j  e_j(t),\quad
(\la_j) \text{ such that } \sum_{j=1}^{\infty} \sigma_j^{-2}\la_j^2 <\infty 
\right\},\]
equipped with the inner product
\[ \psg  \sum_{j=1}^{\infty} \la_j  e_j, \sum_{j=1}^{\infty} \mu_j  e_j\psd_\mh = 
\sum_{j=1}^{\infty} \sigma_j^{-2} \la_j  \mu_j.
\]

{\em Example 3 (Brownian motion).} Consider standard Brownian motion $(B_t)_{t\in[0,1]}$. 

\begin{lem}\label{rkhs:bm}
The RKHS of Brownian motion on $[0,1]$ is
\[ \mh=\left\{ \int_0^\cdot g(u)du,\quad g\in L^2[0,1]\right\}, \]
equipped with the inner product
\[ \psg \int_0^\cdot g_1(u)du, \int_0^\cdot g_2(u)du \psd_{\mh}
= \int_0^1 g_1(u) g_2(u)du.
\]
\end{lem}

The proof of Lemma \eqref{rkhs:bm}, left as an exercise, uses the density of step functions in $L^2$.

{\em Remark. } The last inner product can equivalently been written, observing that elements of $\mh$ are differentiable almost everywhere, $\psg h_1,h_2\psd_\mh= \int_0^1 h_1' h_2'$. Also, note that for any $h\in\mh$ as above, we have $h(0)=0$. One can `release' Brownian motion at zero and instead consider the process $Z_t=B_t+N$, for $N$ a standard normal variable independent of $(B_t)$. Using a similar proof as for Lemma \ref{rkhs:bm}, one can show that the RKHS $\mh_{Z}$ of this process is $\mh=\left\{ c+\int_0^\cdot g(u)du,\ \ g\in L^2[0,1], c\in\RR\, \right\}$ with inner--product $\psg f_1,f_2\psd_{\mh_Z}=f_1(0)f_2(0)+\int_0^1 f_1'f_2'$.\\

\subsection{Fundamental properties via RKHS and concentration function}

\subsubsection{Cameron--Martin formula} 

\begin{thm}\re{[Cameron--Martin]} $\ $ \label{thm:cameronm}
For a Gaussian random variable $W$ taking values in $\mb$ separable Banach space and $h\in\mb$, the distributions $P_{W+h}$ and $P_{W}$ of $W+h$ and $W$ are mutually absolutely continuous if and only if $h\in\mh$. Let us further define the map 
\begin{align*}
U:\mh & \to\cC_W \\
h=g_H & \to H.     
\end{align*}
Then for any $h\in\mh$,
\[ \frac{dP_{W+h}}{dP_W}(W)=\exp\left\{Uh-\frac{\|h\|_\mh^2}{2}\right\}.\]
\end{thm}

{\em Remark.} For Brownian motion, this in fact coincides with Girsanov's formula.

\subsubsection{Ball probabilities and concentration function}

Using Cameron--Martin formula above, it is possible to show that if $h$ belongs to $\mh$, then for any $\veps>0$,
\begin{equation}\label{usmallb}
 P[\|W-h\|_\mb<\veps]\ge e^{-\|h\|_\mh^2/2}P[\|W\|_\mb<\veps]. 
\end{equation}
If $h$ does not belong to $\mh$, some control of the probability on the left-hand side is possible via the concentration function that we define now.

\begin{definition}
Let $W$ be a Gaussian random variable taking its values in $\mb$ separable Banach space, with RKHS $\mh$. 
Let $w$ belong to $\overline{\mh}^\mb$, the closure in $\mb$ (with respect to the norm of $\mb$) of $\mh$. For any $\veps>0$, define
\begin{align*}
 \vphi_w(\veps)
& = \inf_{h\in\mh,\ \|h-w\|_\mb<\veps} \frac12\|h\|_\mh^2
- \log P[ \|W\|_\mb<\veps ] \\
& \ =: \vphi^A_{w}(\veps) + \vphi_0(\veps) 
\end{align*}
The function $\vphi_w(\cdot)$ is called the \sbl{concentration function} of the process $W$. 
\end{definition}

The concentration function is the sum of the small--ball term and of an approximation term, which measures the ability of elements of $\mh$ to approximate a given $w$. Note that for $w\in\overline{\mh}^\mb$ and a given $\veps>0$, the approximation term is always finite. If $w\in \overline{\mh}^\mb$ but $w\notin\mh$, then $\vphi_w^A(\veps)\to +\infty$ as $\veps\to 0$ (otherwise by extracting a subsequence the norm $\|h\|_\mh$ would be finite).

\begin{thm} \label{thm:usb}
Let $W$ be a Gaussian random variable taking its values in $\mb$ separable Banach space, with RKHS $\mh$. 
Suppose $w\in\overline{\mh}^\mb$. Then for any $\veps>0$,
\[ e^{-\vphi_w(\veps/2)}
\le P(\|W-w\|_\mb <\veps) 
\le e^{-\vphi_w(\veps)}.
 \]
\end{thm}

This result is particularly useful for Bayesian nonparametric arguments, as, if the $\|\cdot\|_\mb$--norm can be related to the KL--type divergence defining the KL--type neighborhood $B_{KL}(f_0,\veps)$, then the above result in particular can provide a lower bound on the prior mass term $\Pi[B_{KL}(f_0,\veps_n)]$ appearing in the third condition of the GGV theorem.

\subsubsection{Borell's inequality}

Let $\mb_1$ and $\mh_1$ respectively denote the unit ball of $\mb$ and of $\mh$. \\

By definition  $P[W\in\veps\mb_1]=P[\|W\|_\mb<\veps]=e^{-\vphi_0(\veps)}$. Borell's inequality generalises this result. In words, it says that for large $M$, slightly enlarging $M\mh_1$ by adding elements of norm (in $\mb$) of at most $\veps$, the resulting set captures most of the mass of $W$.

\begin{thm}\re{[Borell]} $\ $ \label{thm:borell} For $W$ a Gaussian random variables taking values in $\mb$ separable Banch space, for any $\veps>0$ and any $M>0$, for $\Phi(u)=P(\cN(0,1)\le u)$,
\[ P[W\in \veps\mb_1 + M\mh_1]\ge \Phi(\Phi^{-1}(e^{-\vphi_0(\veps)})+M). \]
\end{thm}

\section{More on specific statistical models} \label{app:models}

\ti{Gaussian white noise and Gaussian sequence models} (\cite{ginenicklbook}, Section 6.1.1) 
Recall the definition of the white noise model, for $t\in[0,1]$
\[ dX^{(n)}(t)=f(t)dt+dW(t)/\sqrt{n},\]
where $f$ is in $L^2[0,1]$. Denote by $P^X_f$ the distribution induced on $\cC^0[0,1]$ by $X^{(n)}(x)=\int_0^x f(t)dt + W(x)/\sqrt{n}$. Then $P^X_f$ is absolutely continuous with respect to $P^X_0$ (the law for $f$ equal to the function identically $0$), with likelihood ratio given by
\[ \frac{dP^X_f}{dP^X_0}(X) =\exp\left\{n \int_0^1f(t)dX(t) -n\|f\|_2^2/2\right\}.\] 
In particular, the model is dominated with dominating measure $P^X_0$.\\

A similar result holds for the Gaussian sequence model, for $k\ge 1$,
\[ X_k = \te_k + \veps_k/\rn, \]
where $\te$ in a sequence in $\ell^2$. Let $P^X_\te$ denote the law on $\RR^{\mathbb{N}^*}$ induced by a tensor product of $\cN(\te_k,1/\sqrt{n})$ variables. Then $P^X_f$ is absolutely continuous with respect to $P^X_0$ (the law for $\te$ equal to the null sequence), with likelihood ratio given by
\[ \frac{dP^X_f}{dP^X_0}(X) =\exp\left\{n \sum_{k\ge1} f_k X_k -n\|f\|_2^2/2\right\}.\] 
In particular, the model is dominated with dominating measure $P^X_0$.\\

In the sequence model, it holds, adapting the formula above for $\te_0=0$, 
\[ \frac{dP_\te^{(n)}}{dP_{\te_0}^{(n)}}(X)=
e^{n\psg X,\te-\te_0\psd -\frac{n}{2}\|\te\|^2+\frac{n}{2}\|\te_0\|^2}. \]
One deduces  $K(P_{\te_0}^{(n)},P_\te^{(n)}) = n\|\te_0-\te\|^2/2$ and 
$V(P_{\te_0}^{(n)},P_\te^{(n)})=n\|\te_0-\te\|^2$, so that \[ B_n(\te_0,\veps_n) = \left\{ \te\in\ell^2:\ \|\te-\te_0\|^2\le \veps_n^2\right\} \]
(both inclusions hold, despite the $1/2$ factor above). 
In this case, the KL--type neighborhood is just a ball for the $L^2$--norm.\\

\ti{Nonparametric survival model: definition} In this model one is interested in (not necessarily observed) i.i.d. survival times $T_1, \ldots, T_n$, whose observation is possibly interfered with by i.i.d. censoring times $C_1, \ldots, C_n$, which are independent of the survival times, so that we observe \[ X=X^n =((Y_1, \delta_1), \ldots, (Y_n, \delta_n))\]
 i.i.d. pairs, where $Y_i = T_i \wedge C_i$ the minimum between the $i$th time of interest and the i$th$ censoring time and $\delta_i = \1\{T_i \leq C_i\}$ a variable indicating whether censoring has occurred or not, for $i=1,\ldots,n$.\\

One main object of interest is the {\em survival function} $S(t) = P(T_1 > t)$.  
The {\em hazard rate} is  
\[ \lambda(t) = \lim_{h \downarrow 0} h^{-1}P(t \leq T_1 < t + h \mid T \geq t).\] Integrating the hazard yields the {\em cumulative hazard} $\Lambda(\cdot) = \int_0^\cdot \la(u)du$.\\

\ti{Cox proportional hazards model: definition} This model can be viewed as a semiparametric version of the survival model, allowing for the presence of covariates. 
The observations $X=X^n$ are $n$ independent identically distributed (i.i.d.) triplets given by 
\[ X = ((Y_1, \delta_1, Z_1), \dots, (Y_n, \delta_n, Z_n))\] 
where the observed $Y_i$ are censored versions of (not necessarily observed) survival times $T_i$,
with $\delta_i$ indicator variables informing on whether $T_i$ has been observed or not: that is, $Y_i = T_i \wedge C_i$ and $\delta_i = \1\{T_i \leq C_i\}$, where $C_i$'s are i.i.d. censoring times.  The  variables $Z_1, \dots, Z_n \in \mathbb{R}^p$, $p$ fixed, are called covariates. 
For a fixed covariate vector $z\in\mathbb{R}^p$ and $t>0$, define the {\it conditional hazard rate}  $\lambda(t\given z) = \lim_{h\to0} h^{-1} P(t \leq T \leq t+h \given T \geq t, Z = z)$. 

The Cox model assumes, for some unknown {\em parameter of interest} $\theta\in\mathbb{R}^p$ and denoting by $\theta^T z$ the standard inner product in $\mathbb{R}^p$, 
\[ \lambda(t \given z) = e^{\theta^T z} \lambda(t),\]
where $\lambda(t)$ is the {\it baseline hazard function}. The latter is generally not known, so that the parameters of the model are $(\te,\la(\cdot))$, making it a semiparametric model.

\addcontentsline{toc}{chapter}{Bibliography}

\bibliographystyle{abbrv}
\bibliography{bibstf}

\def\cprime{$'$} \def\cprime{$'$} \def\cprime{$'$}
\begin{thebibliography}{100}

\bibitem{gravi}
B.~P. Abbott et~al.
\newblock Observation of gravitational waves from a binary black hole merger.
\newblock {\em Phys. Rev. Lett.}, 116:061102, 2016.

\bibitem{acg22}
K.~Abraham, I.~Castillo, and E.~Gassiat.
\newblock Multiple testing in nonparametric hidden {M}arkov models: an
  empirical {B}ayes approach.
\newblock {\em J. Mach. Learn. Res.}, 23:Paper No. [94], 57, 2022.

\bibitem{acr23}
K.~Abraham, I.~Castillo, and E.~Roquain.
\newblock Sharp multiple testing boundary for sparse sequences.
\newblock 2021.
\newblock Arxiv eprint 2109.13601.

\bibitem{acr22}
K.~Abraham, I.~Castillo, and E.~Roquain.
\newblock Empirical {B}ayes cumulative {$\ell$}-value multiple testing
  procedure for sparse sequences.
\newblock {\em Electron. J. Stat.}, 16(1):2033--2081, 2022.

\bibitem{an19}
K.~Abraham and R.~Nickl.
\newblock On statistical {C}alder\'{o}n problems.
\newblock {\em Math. Stat. Learn.}, 2(2):165--216, 2019.

\bibitem{ac23}
S.~Agapiou and I.~Castillo.
\newblock Heavy-tailed {B}ayesian nonparametric adaptation.
\newblock 2023.
\newblock Arxiv eprint 2308.04916.

\bibitem{agapiouetal21}
S.~Agapiou, M.~Dashti, and T.~Helin.
\newblock Rates of contraction of posterior distributions based on
  {$p$}-exponential priors.
\newblock {\em Bernoulli}, 27(3):1616--1642, 2021.

\bibitem{as23}
S.~Agapiou and A.~Savva.
\newblock Adaptive inference over {B}esov spaces in the white noise model using
  $p$-exponential priors.
\newblock {\em arXiv preprint arXiv:2209.06045}, 2023.

\bibitem{alquier21}
P.~Alquier.
\newblock User-friendly introduction to {PAC}-{B}ayes bounds.
\newblock 2023.
\newblock Arxiv eprint 2110.11216.

\bibitem{ar20}
P.~Alquier and J.~Ridgway.
\newblock Concentration of tempered posteriors and of their variational
  approximations.
\newblock {\em Ann. Statist.}, 48(3):1475--1497, 2020.

\bibitem{bcg21}
S.~Banerjee, I.~Castillo, and S.~Ghosal.
\newblock Bayesian inference in high-dimensional models.
\newblock 2021.
\newblock Book chapter to appear in Springer volume on data science, Preprint
  arXiv:2101.04491.

\bibitem{bsw99}
A.~Barron, M.~J. Schervish, and L.~Wasserman.
\newblock The consistency of posterior distributions in nonparametric problems.
\newblock {\em Ann. Statist.}, 27(2):536--561, 1999.

\bibitem{BT67}
P.~Billingsley and F.~Tops{\o}e.
\newblock Uniformity in weak convergence.
\newblock {\em Z. Wahrscheinlichkeitstheorie und Verw. Gebiete}, 7:1--16, 1967.

\bibitem{bnj03}
D.~M. Blei, A.~Y. Ng, and M.~I. Jordan.
\newblock {L}atent {D}irichlet {A}llocation.
\newblock {\em J. Mach. Learn. Res.}, 3:993–1022, mar 2003.

\bibitem{blmbook}
S.~Boucheron, G.~Lugosi, and P.~Massart.
\newblock {\em Concentration inequalities. A non asymptotic theory of
  independence.}
\newblock Oxford University Press, 2013.

\bibitem{butucea18}
C.~Butucea, M.~Ndaoud, N.~A. Stepanova, and A.~B. Tsybakov.
\newblock Variable selection with {H}amming loss.
\newblock {\em Ann. Statist.}, 46(5):1837--1875, 2018.

\bibitem{carvalhopolsonscott}
C.~M. Carvalho, N.~G. Polson, and J.~G. Scott.
\newblock The horseshoe estimator for sparse signals.
\newblock {\em Biometrika}, 97(2):465--480, 2010.

\bibitem{ic08}
I.~Castillo.
\newblock Lower bounds for posterior rates with {G}aussian process priors.
\newblock {\em Electronic Journal of Statistics}, 2:1281--1299, 2008.

\bibitem{c12}
I.~Castillo.
\newblock {A} semiparametric {B}ernstein-von {M}ises theorem for {G}aussian
  process priors.
\newblock {\em Probability Theory and Related Fields}, 152(1-2):53--99, 2012.

\bibitem{ic12b}
I.~Castillo.
\newblock {S}emiparametric {B}ernstein--von {M}ises theorem and bias,
  illustrated with {G}aussian process priors.
\newblock {\em Sankhya A}, 74(2):194--221, 2012.

\bibitem{c17}
I.~Castillo.
\newblock P\'{o}lya tree posterior distributions on densities.
\newblock {\em Ann. Inst. Henri Poincar\'{e} Probab. Stat.}, 53(4):2074--2102,
  2017.

\bibitem{ckp13}
I.~Castillo, G.~Kerkyacharian, and D.~Picard.
\newblock Thomas {B}ayes' walk on manifolds.
\newblock {\em Probability Theory and Related Fields}, 158(3-4):665--710, 2014.

\bibitem{cm18}
I.~Castillo and R.~Mismer.
\newblock Empirical {B}ayes analysis of spike and slab posterior distributions.
\newblock {\em Electron. J. Stat.}, 12(2):3953--4001, 2018.

\bibitem{cm21}
I.~Castillo and R.~Mismer.
\newblock Spike and slab {P}\'{o}lya tree posterior densities: adaptive
  inference.
\newblock {\em Ann. Inst. Henri Poincar\'{e} Probab. Stat.}, 57(3):1521--1548,
  2021.

\bibitem{bvmnp}
I.~Castillo and R.~Nickl.
\newblock Nonparametric {B}ernstein--von {M}ises theorems in {G}aussian white
  noise.
\newblock {\em The Annals of Statistics}, 41(4):1999--2028, 2013.

\bibitem{bvmnp2}
I.~Castillo and R.~Nickl.
\newblock On the {B}ernstein--von {M}ises phenomenon for nonparametric {B}ayes
  procedures.
\newblock {\em The Annals of Statistics.}, 42(5):1941--1969, 2014.

\bibitem{cr22}
I.~Castillo and T.~Randrianarisoa.
\newblock Optional {P}\'{o}lya trees: posterior rates and uncertainty
  quantification.
\newblock {\em Electron. J. Stat.}, 16(2):6267--6312, 2022.

\bibitem{cr23}
I.~Castillo and T.~Randrianarisoa.
\newblock Deep {H}orseshoe {G}aussian {P}rocesses.
\newblock 2024.
\newblock manuscript in preparation.

\bibitem{cr20}
I.~Castillo and E.~Roquain.
\newblock On spike and slab empirical {B}ayes multiple testing.
\newblock {\em Ann. Statist.}, 48(5):2548--2574, 2020.

\bibitem{cr15}
I.~Castillo and J.~Rousseau.
\newblock A {B}ernstein–von {M}ises {T}heorem for smooth functionals in
  semiparametric models.
\newblock {\em Ann. Statist.}, 43(6), Dec 2015.

\bibitem{cr21}
I.~Castillo and V.~Ro\v{c}kov\'{a}.
\newblock Uncertainty quantification for {B}ayesian {CART}.
\newblock {\em Ann. Statist.}, 49(6):3482--3509, 2021.

\bibitem{spahd}
I.~Castillo, J.~Schmidt-Hieber, and A.~van~der Vaart.
\newblock Bayesian linear regression with sparse priors.
\newblock {\em Ann. Statist.}, 43(5):1986--2018, 2015.

\bibitem{cs20}
I.~Castillo and B.~Szab\'{o}.
\newblock Spike and slab empirical {B}ayes sparse credible sets.
\newblock {\em Bernoulli}, 26(1):127--158, 2020.

\bibitem{cv21}
I.~Castillo and S.~van~der Pas.
\newblock Multiscale {B}ayesian survival analysis.
\newblock {\em Ann. Statist.}, 49(6):3559--3582, 2021.

\bibitem{cv12}
I.~Castillo and A.~van~der Vaart.
\newblock Needles and straw in a haystack: posterior concentration for possibly
  sparse sequences.
\newblock {\em The Annals of Statistics}, 40(4):2069--2101, 2012.

\bibitem{catoni01}
O.~Catoni.
\newblock {\em Statistical learning theory and stochastic optimization}, volume
  1851 of {\em Lecture Notes in Mathematics}.
\newblock Springer-Verlag, Berlin, 2004.
\newblock Lecture notes from the 31st Summer School on Probability Theory held
  in Saint-Flour, July 8--25, 2001.

\bibitem{catoni07}
O.~Catoni.
\newblock {\em {PAC}-{B}ayesian supervised classification: the thermodynamics
  of statistical learning}, volume~56 of {\em Institute of Mathematical
  Statistics Lecture Notes---Monograph Series}.
\newblock Institute of Mathematical Statistics, Beachwood, OH, 2007.

\bibitem{cart1}
H.~Chipman, E.~I. George, and R.~E. Mc{C}ulloch.
\newblock Bayesian {CART} model search.
\newblock {\em Journal of the American Statistical Association}, 93:935--960,
  1997.

\bibitem{csr21}
G.~Clara, B.~Szabo, and K.~Ray.
\newblock sparsevb: {S}pike-and-{S}lab {V}ariational {B}ayes for linear and
  logistic regression.
\newblock 2021.
\newblock R package.

\bibitem{C93}
D.~D. Cox.
\newblock An analysis of {B}ayesian inference for nonparametric regression.
\newblock {\em Ann. Statist.}, 21(2):903--923, 1993.

\bibitem{damianou13}
A.~Damianou and N.~D. Lawrence.
\newblock Deep {G}aussian processes.
\newblock In C.~M. Carvalho and P.~Ravikumar, editors, {\em Proceedings of the
  Sixteenth International Conference on Artificial Intelligence and
  Statistics}, volume~31 of {\em Proceedings of Machine Learning Research},
  pages 207--215, Scottsdale, Arizona, USA, 29 Apr--01 May 2013. PMLR.

\bibitem{djvz13}
R.~de~Jonge and H.~van Zanten.
\newblock Semiparametric {B}ernstein--von {M}ises for the error standard
  deviation.
\newblock {\em Electron. J. Stat.}, 7:217--243, 2013.

\bibitem{dejongevz10}
R.~de~Jonge and J.~H. van Zanten.
\newblock Adaptive nonparametric {B}ayesian inference using location-scale
  mixture priors.
\newblock {\em Ann. Statist.}, 38(6):3300--3320, 2010.

\bibitem{cart2}
D.~Denison, B.~Mallick, and A.~Smith.
\newblock A {B}ayesian {CART} algorithm.
\newblock {\em Biometrika}, 85:363--377, 1998.

\bibitem{dr23}
N.~Deo and T.~Randrianarisoa.
\newblock {On adaptive confidence sets for the Wasserstein distances}.
\newblock {\em Bernoulli}, 29(3):2119 -- 2141, 2023.

\bibitem{df86}
P.~Diaconis and D.~Freedman.
\newblock On the consistency of {B}ayes estimates.
\newblock {\em Ann. Statist.}, 14(1):1--67, 1986.
\newblock With a discussion and a rejoinder by the authors.

\bibitem{doob}
J.~L. Doob.
\newblock Application of the theory of martingales.
\newblock In {\em Le {C}alcul des {P}robabilit\'es et ses {A}pplications},
  Colloques Internationaux du CNRS, no. 13, pages 23--27. CNRS, Paris, 1949.

\bibitem{D02}
R.~M. Dudley.
\newblock {\em Real analysis and probability}.
\newblock Cambridge, UK, 2002.

\bibitem{efron07}
B.~Efron.
\newblock Size, power and false discovery rates.
\newblock {\em Ann. Statist.}, 35(4):1351--1377, 2007.

\bibitem{efron08}
B.~Efron.
\newblock Microarrays, empirical {B}ayes and the two-groups model.
\newblock {\em Statist. Sci.}, 23(1):1--22, 2008.

\bibitem{efronetal01}
B.~Efron, R.~Tibshirani, J.~D. Storey, and V.~Tusher.
\newblock Empirical bayes analysis of a microarray experiment.
\newblock {\em Journal of the American Statistical Association},
  96(456):1151--1160, 2001.

\bibitem{fabius64}
J.~Fabius.
\newblock Asymptotic behavior of {B}ayes' estimates.
\newblock {\em Ann. Math. Statist.}, 35:846--856, 1964.

\bibitem{ferguson73}
T.~S. Ferguson.
\newblock A {B}ayesian analysis of some nonparametric problems.
\newblock {\em Ann. Statist.}, 1:209--230, 1973.

\bibitem{fsh23}
G.~Finocchio and J.~Schmidt-Hieber.
\newblock Posterior contraction for deep {G}aussian process priors.
\newblock {\em J. Mach. Learn. Res.}, 24:Paper No. [66], 49, 2023.

\bibitem{Free99}
D.~Freedman.
\newblock On the {B}ernstein-von {M}ises theorem with infinite-dimensional
  parameters.
\newblock {\em Ann. Statist.}, 27(4):1119--1140, 1999.

\bibitem{freedman63}
D.~A. Freedman.
\newblock On the asymptotic behavior of {B}ayes' estimates in the discrete
  case.
\newblock {\em Ann. Math. Statist.}, 34:1386--1403, 1963.

\bibitem{gvz20}
C.~Gao, A.~W. van~der Vaart, and H.~H. Zhou.
\newblock A general framework for {B}ayes structured linear models.
\newblock {\em Ann. Statist.}, 48(5):2848--2878, 2020.

\bibitem{gz16}
C.~Gao and H.~H. Zhou.
\newblock Bernstein--von {M}ises theorems for functionals of the covariance
  matrix.
\newblock {\em Electron. J. Stat.}, 10(2):1751--1806, 2016.

\bibitem{georgefoster}
E.~I. George and D.~P. Foster.
\newblock Calibration and empirical {B}ayes variable selection.
\newblock {\em Biometrika}, 87(4):731--747, 2000.

\bibitem{ggv00}
S.~Ghosal, J.~K. Ghosh, and A.~W. van~der Vaart.
\newblock Convergence rates of posterior distributions.
\newblock {\em Ann. Statist.}, 28(2):500--531, 2000.

\bibitem{gvni}
S.~Ghosal and A.~van~der Vaart.
\newblock Convergence rates of posterior distributions for non-i.i.d.
  observations.
\newblock {\em Ann. Statist.}, 35(1):192--223, 2007.

\bibitem{gvbook}
S.~Ghosal and A.~van~der Vaart.
\newblock {\em Fundamentals of nonparametric {B}ayesian inference}, volume~44
  of {\em Cambridge Series in Statistical and Probabilistic Mathematics}.
\newblock Cambridge University Press, Cambridge, 2017.

\bibitem{ginenickl09}
E.~Gin\'{e} and R.~Nickl.
\newblock Uniform limit theorems for wavelet density estimators.
\newblock {\em Ann. Probab.}, 37(4):1605--1646, 2009.

\bibitem{gn11}
E.~Gin\'{e} and R.~Nickl.
\newblock Rates of contraction for posterior distributions in {L}$^r$-metrics,
  $1 \le r \le \infty$.
\newblock {\em Ann. Statist.}, 39:2883--2911, 2011.

\bibitem{ginenicklbook}
E.~Gin\'{e} and R.~Nickl.
\newblock {\em Mathematical foundations of infinite-dimensional statistical
  models}.
\newblock Cambridge Series in Statistical and Probabilistic Mathematics, [40].
  Cambridge University Press, New York, 2016.

\bibitem{grs22}
M.~Giordano, K.~Ray, and J.~Schmidt{-}Hieber.
\newblock On the inability of {G}aussian process regression to optimally learn
  compositional functions.
\newblock In {\em Advances in Neural Information Processing Systems 35: NeurIPS
  2022}, 2022.

\bibitem{Grigoryan}
A.~Grigor'yan.
\newblock {\em Heat kernel and analysis on manifolds}, volume~47 of {\em AMS/IP
  Studies in Advanced Mathematics}.
\newblock American Mathematical Society, Providence, RI, 2009.

\bibitem{hrs15}
M.~Hoffmann, J.~Rousseau, and J.~Schmidt-Hieber.
\newblock On adaptive posterior concentration rates.
\newblock {\em Ann. Statist.}, 43(5):2259--2295, 2015.

\bibitem{js04}
I.~M. Johnstone and B.~W. Silverman.
\newblock Needles and straw in haystacks: empirical {B}ayes estimates of
  possibly sparse sequences.
\newblock {\em Ann. Statist.}, 32(4):1594--1649, 2004.

\bibitem{ksvv16}
B.~T. Knapik, B.~T. Szab\'{o}, A.~W. van~der Vaart, and J.~H. van Zanten.
\newblock Bayes procedures for adaptive inference in inverse problems for the
  white noise model.
\newblock {\em Probab. Theory Related Fields}, 164(3-4):771--813, 2016.

\bibitem{kvv11}
B.~T. Knapik, A.~W. van~der Vaart, and J.~H. van Zanten.
\newblock Bayesian inverse problems with {G}aussian priors.
\newblock {\em Ann. Statist.}, 39(5):2626--2657, 2011.

\bibitem{krv10}
W.~Kruijer, J.~Rousseau, and A.~van~der Vaart.
\newblock Adaptive {B}ayesian density estimation with location-scale mixtures.
\newblock {\em Electron. J. Stat.}, 4:1225--1257, 2010.

\bibitem{KuelbsLi93}
J.~Kuelbs and W.~V. Li.
\newblock Metric entropy and the small ball problem for {G}aussian measures.
\newblock {\em J. Funct. Anal.}, 116(1):133--157, 1993.

\bibitem{laplace1774}
P.-S. Laplace.
\newblock Mémoire sur la probabilité de causes par les évènements.
\newblock {\em Mémoire de l'Académie Royale des Sciences}, 1774.

\bibitem{lavine92}
M.~Lavine.
\newblock Some aspects of {P}\'olya tree distributions for statistical
  modelling.
\newblock {\em Ann. Statist.}, 20(3):1222--1235, 1992.

\bibitem{L11}
H.~Leahu.
\newblock On the {B}ernstein-von {M}ises phenomenon in the {G}aussian white
  noise model.
\newblock {\em Electron. J. Stat.}, 5:373--404, 2011.

\bibitem{Lenk88}
P.~J. Lenk.
\newblock The logistic normal distribution for {B}ayesian, nonparametric,
  predictive densities.
\newblock {\em J. Amer. Statist. Assoc.}, 83(402):509--516, 1988.

\bibitem{Leonard78}
T.~Leonard.
\newblock Density estimation, stochastic processes and prior information.
\newblock {\em J. Roy. Statist. Soc. Ser. B}, 40(2):113--146, 1978.
\newblock With discussion.

\bibitem{ltcr23}
A.~L'Huillier, L.~Travis, I.~Castillo, and K.~Ray.
\newblock Semiparametric inference using fractional posteriors.
\newblock {\em Journal of Machine Learning Research}, 24(389):1--61, 2023.

\bibitem{L83}
A.~Lo.
\newblock Weak convergence for {D}irichlet processes.
\newblock {\em Sankhy\=a}, 45(1):105--111, 1983.

\bibitem{msw92}
R.~D. Mauldin, W.~D. Sudderth, and S.~C. Williams.
\newblock P\'olya trees and random distributions.
\newblock {\em Ann. Statist.}, 20(3):1203--1221, 1992.

\bibitem{mcallester98}
D.~A. McAllester.
\newblock Some {PAC}-{B}ayesian theorems.
\newblock In {\em Proceedings of the {E}leventh {A}nnual {C}onference on
  {C}omputational {L}earning {T}heory ({M}adison, {WI}, 1998)}, pages 230--234.
  ACM, New York, 1998.

\bibitem{mb88}
T.~J. Mitchell and J.~J. Beauchamp.
\newblock Bayesian variable selection in linear regression.
\newblock {\em J. Amer. Statist. Assoc.}, 83(404):1023--1036, 1988.
\newblock With comments by James Berger and C. L. Mallows and with a reply by
  the authors.

\bibitem{mnp21}
F.~Monard, R.~Nickl, and G.~P. Paternain.
\newblock Statistical guarantees for {B}ayesian uncertainty quantification in
  nonlinear inverse problems with {G}aussian process priors.
\newblock {\em Ann. Statist.}, 49(6):3255--3298, 2021.

\bibitem{zn22}
Z.~Naulet.
\newblock Adaptive {B}ayesian density estimation in sup-norm.
\newblock {\em Bernoulli}, 28(2):1284--1308, 2022.

\bibitem{nicklschroedinger}
R.~Nickl.
\newblock Bernstein--von {M}ises theorems for statistical inverse problems {I}:
  {S}chr\"{o}dinger equation.
\newblock {\em J. Eur. Math. Soc. (JEMS)}, 22(8):2697--2750, 2020.

\bibitem{nickl23}
R.~Nickl.
\newblock {\em Bayesian non-linear statistical inverse problems}.
\newblock Zurich Lectures in Advanced Mathematics. European Mathematical
  Society (EMS) Press, 2023.

\bibitem{nr20}
R.~Nickl and K.~Ray.
\newblock Nonparametric statistical inference for drift vector fields of
  multi-dimensional diffusions.
\newblock {\em Ann. Statist.}, 48(3):1383--1408, 2020.

\bibitem{nicklsoehl19}
R.~Nickl and J.~S\"{o}hl.
\newblock Bernstein-von {M}ises theorems for statistical inverse problems {II}:
  compound {P}oisson processes.
\newblock {\em Electron. J. Stat.}, 13(2):3513--3571, 2019.

\bibitem{nv13}
R.~Nickl and S.~van~de Geer.
\newblock Confidence sets in sparse regression.
\newblock {\em Ann. Statist.}, 41(6):2852--2876, 2013.

\bibitem{nc23}
B.~Y.-C. Ning and I.~Castillo.
\newblock Bayesian multiscale analysis of the {C}ox model.
\newblock {\em Bernoulli}, 30(2):1525--1554, 2024.

\bibitem{petrone02}
S.~Petrone and L.~Wasserman.
\newblock Consistency of {B}ernstein polynomial posteriors.
\newblock {\em J. R. Stat. Soc. Ser. B Stat. Methodol.}, 64(1):79--100, 2002.

\bibitem{steph00}
J.~K. Pritchard, M.~Stephens, and P.~Donnelly.
\newblock {Inference of Population Structure Using Multilocus Genotype Data}.
\newblock {\em Genetics}, 155(2):945--959, 06 2000.

\bibitem{ray17}
K.~Ray.
\newblock Adaptive {B}ernstein--von {M}ises theorems in {G}aussian white noise.
\newblock {\em Ann. Statist.}, 45(6):2511--2536, 2017.

\bibitem{rs22}
K.~Ray and B.~Szab\'{o}.
\newblock Variational {B}ayes for high-dimensional linear regression with
  sparse priors.
\newblock {\em J. Amer. Statist. Assoc.}, 117(539):1270--1281, 2022.

\bibitem{rv20}
K.~Ray and A.~van~der Vaart.
\newblock Semiparametric {B}ayesian causal inference.
\newblock {\em Ann. Statist.}, 48(5):2999--3020, 2020.

\bibitem{rrv22}
T.~Rebafka, {\'E}.~Roquain, and F.~Villers.
\newblock {Powerful multiple testing of paired null hypotheses using a latent
  graph model}.
\newblock {\em Electronic Journal of Statistics}, 16(1):2796 -- 2858, 2022.

\bibitem{r10}
J.~Rousseau.
\newblock Rates of convergence for the posterior distributions of mixtures of
  {B}etas and adaptive nonparametric estimation of the density.
\newblock {\em Ann. Statist.}, 38(1):146--180, 2010.

\bibitem{rs17}
J.~Rousseau and B.~Szabo.
\newblock Asymptotic behaviour of the empirical {B}ayes posteriors associated
  to maximum marginal likelihood estimator.
\newblock {\em Ann. Statist.}, 45(2):833--865, 2017.

\bibitem{rs20}
J.~Rousseau and B.~Szabo.
\newblock Asymptotic frequentist coverage properties of {B}ayesian credible
  sets for sieve priors.
\newblock {\em Ann. Statist.}, 48(4):2155--2179, 2020.

\bibitem{jsh20}
J.~Schmidt-Hieber.
\newblock Nonparametric regression using deep neural networks with {R}e{LU}
  activation function.
\newblock {\em Ann. Statist.}, 48(4):1875--1897, 2020.

\bibitem{Schwartz}
L.~Schwartz.
\newblock On {B}ayes procedures.
\newblock {\em Z. Wahrscheinlichkeitstheorie und Verw. Gebiete}, 4:10--26,
  1965.

\bibitem{sw01}
X.~Shen and L.~Wasserman.
\newblock Rates of convergence of posterior distributions.
\newblock {\em Ann. Statist.}, 29(3):687--714, 2001.

\bibitem{storey03}
J.~D. Storey.
\newblock The positive false discovery rate: a {B}ayesian interpretation and
  the {$q$}-value.
\newblock {\em Ann. Statist.}, 31(6):2013--2035, 2003.

\bibitem{stuart10}
A.~M. Stuart.
\newblock Inverse problems: {A} {B}ayesian perspective.
\newblock {\em Acta Numerica}, 19:451–559, 2010.

\bibitem{sc09}
W.~Sun and T.~T. Cai.
\newblock Large-scale multiple testing under dependence.
\newblock {\em J. R. Stat. Soc. Ser. B. Stat. Methodol.}, 71(2):393--424, 2009.

\bibitem{svv15}
B.~Szab\'{o}, A.~W. van~der Vaart, and J.~H. van Zanten.
\newblock Frequentist coverage of adaptive nonparametric {B}ayesian credible
  sets.
\newblock {\em Ann. Statist.}, 43(4):1391--1428, 2015.

\bibitem{svvl2}
B.~Szabó, A.~{van der Vaart}, and H.~{van Zanten}.
\newblock Honest bayesian confidence sets for the l2-norm.
\newblock {\em Journal of Statistical Planning and Inference}, 166:36--51,
  2015.
\newblock Special Issue on Bayesian Nonparametrics.

\bibitem{gravi19}
E.~Thrane and C.~Talbot.
\newblock An introduction to {B}ayesian inference in gravitational-wave
  astronomy: {P}arameter estimation, model selection, and hierarchical models.
\newblock {\em Publications of the Astronomical Society of Australia}, 36:e010,
  2019.

\bibitem{saravdg-stf}
S.~van~de Geer.
\newblock {\em Estimation and testing under sparsity}, volume 2159 of {\em
  Lecture Notes in Mathematics}.
\newblock Springer, 2016.
\newblock Lecture notes from the 45th Probability Summer School held in
  Saint-Four, 2015.

\bibitem{vsv17}
S.~van~der Pas, B.~Szab\'{o}, and A.~van~der Vaart.
\newblock Uncertainty quantification for the horseshoe (with discussion).
\newblock {\em Bayesian Anal.}, 12(4):1221--1274, 2017.
\newblock With a rejoinder by the authors.

\bibitem{vss16}
S.~L. van~der Pas, J.-B. Salomond, and J.~Schmidt-Hieber.
\newblock Conditions for posterior contraction in the sparse normal means
  problem.
\newblock {\em Electron. J. Stat.}, 10(1):976--1000, 2016.

\bibitem{aadstflour}
A.~van~der Vaart.
\newblock Semiparametric statistics.
\newblock In {\em Lectures on probability theory and statistics
  ({S}aint-{F}lour, 1999)}, volume 1781 of {\em Lecture Notes in Math.}, pages
  331--457. Springer, Berlin, 2002.

\bibitem{vvvz07}
A.~van~der Vaart and H.~van Zanten.
\newblock Bayesian inference with rescaled {G}aussian process priors.
\newblock {\em Electron. J. Stat.}, 1:433--448, 2007.

\bibitem{aad98}
A.~W. van~der Vaart.
\newblock {\em Asymptotic statistics}, volume~3 of {\em Cambridge Series in
  Statistical and Probabilistic Mathematics}.
\newblock Cambridge University Press, Cambridge, 1998.

\bibitem{vvvz}
A.~W. van~der Vaart and H.~van Zanten.
\newblock Rates of contraction of posterior distributions based on {G}aussian
  process priors.
\newblock {\em Ann. Statist.}, 36(3):1435--1463, 2008.

\bibitem{rkhs}
A.~W. van~der Vaart and H.~van Zanten.
\newblock Reproducing kernel {H}ilbert spaces of {G}aussian priors.
\newblock {\em IMS Collections}, 3:200--222, 2008.

\bibitem{vvvz09}
A.~W. van~der Vaart and H.~van Zanten.
\newblock Adaptive {B}ayesian estimation using a {G}aussian random field with
  inverse gamma bandwidth.
\newblock {\em Ann. Statist.}, 37(5B):2655--2675, 2009.

\bibitem{vvvz11}
A.~W. van~der Vaart and H.~van Zanten.
\newblock Information rates of nonparametric {G}aussian process methods.
\newblock {\em J. Mach. Learn. Res.}, 12:2095--2119, 2011.

\bibitem{AadWell}
A.~W. van~der Vaart and J.~A. Wellner.
\newblock {\em Weak convergence and empirical processes}.
\newblock Springer Series in Statistics. Springer-Verlag, New York, 1996.

\bibitem{vanerven}
T.~van Erven and P.~Harremoes.
\newblock R{\'{e}}nyi divergence and {K}ullback--{L}eibler divergence.
\newblock {\em {IEEE} Transactions on Information Theory}, 60(7):3797--3820,
  2014.

\bibitem{gravi14}
J.~Veitch et~al.
\newblock Parameter estimation for compact binaries with ground-based
  gravitational-wave observations using the lalinference software library.
\newblock {\em Phys. Rev. D}, 91:042003, 2015.

\bibitem{walkerhjort01}
S.~Walker and N.~L. Hjort.
\newblock On {B}ayesian consistency.
\newblock {\em J. R. Stat. Soc. Ser. B Stat. Methodol.}, 63(4):811--821, 2001.

\bibitem{ypb20}
Y.~Yang, D.~Pati, and A.~Bhattacharya.
\newblock {$\alpha$}-variational inference with statistical guarantees.
\newblock {\em Ann. Statist.}, 48(2):886--905, 2020.

\bibitem{zg20}
F.~Zhang and C.~Gao.
\newblock Convergence rates of variational posterior distributions.
\newblock {\em Ann. Statist.}, 48(4):2180--2207, 2020.

\bibitem{tzhang06}
T.~Zhang.
\newblock From {$\epsilon$}-entropy to {KL}-entropy: analysis of minimum
  information complexity density estimation.
\newblock {\em Ann. Statist.}, 34(5):2180--2210, 2006.

\end{thebibliography}

\end{document}